\documentclass{article}

\usepackage{geometry}
\usepackage{cancel}
\usepackage{graphicx}
\usepackage{subcaption}
\usepackage{adjustbox}

\usepackage{algorithmicx}
\usepackage[ruled,vlined]{algorithm2e}
\usepackage{algpseudocode,float}
\algblock{ParFor}{EndParFor}

\usepackage{mathrsfs}  
\usepackage{sidecap}
\usepackage{amsfonts}
\usepackage{amssymb}
\usepackage{amsmath}
\usepackage{amsthm}
\usepackage{comment}
\usepackage{xcolor}
\usepackage{relsize}
\usepackage{multirow}
\usepackage{multicol}
\usepackage{diagbox}
\usepackage{overpic}
\usepackage{authblk}
\usepackage{bm}
\newcommand{\R}{{\mathbb R}}

\usepackage{amssymb}

\DeclareMathOperator*{\argmin}{arg\,min}
\newcommand{\STAB}[1]{\begin{tabular}{@{}c@{}}#1\end{tabular}}

\newtheorem{lemma}{Lemma}
\newtheorem{corollary}{Corollary}
\newtheorem{proposition}{Proposition}
\newtheorem{definition}{Definition}
\newtheorem{remark}{Remark}

\usepackage{cleveref}
\numberwithin{equation}{section}

\date{}

\title{On and beyond Total Variation regularisation in imaging:\\ the role of space variance}

\newcommand{\TV}{\mathrm{TV}}
\newcommand{\WTV}{\mathrm{WTV}}
\newcommand{\DTV}{\mathrm{DTV}}
\newcommand{\WDTV}{\mathrm{WDTV}}

\newcommand{\M}{\mathrm{M}}
\newcommand{\D}{\mathrm{D}}
\newcommand{\A}{\mathrm{A}}

\newcommand{\I}{\mathrm{I}}
\newcommand{\E}{\mathrm{E}}
\newcommand{\V}{\mathrm{V}}

\newcommand{\argmax}{\arg\max}
\usepackage{subcaption}
\renewcommand{\argmin}[1]{\underset{#1}{\operatorname{arg}\,\operatorname{min}}\;}
  \definecolor{cosmiclatte}{rgb}{1.0, 0.97, 0.91}
\definecolor{azure}{rgb}{0.94, 1.0, 1.0}
\definecolor{beaublue}{rgb}{0.74, 0.83, 0.9}
\usepackage{color, colortbl}
\definecolor{Gray}{rgb}{0.75,0.75,0.75}%{0.9}
\definecolor{apricot}{rgb}{0.98, 0.81, 0.69}
\definecolor{almond}{rgb}{0.94, 0.87, 0.8}
\definecolor{LightCyan}{rgb}{0.88,1,1}

\DeclareMathOperator{\diag}{diag}

\author[1]{Monica Pragliola\thanks{monica.pragliola2@unibo.it}}
\author[2]{Luca Calatroni\thanks{calatroni@i3s.unice.fr}}
\author[1]{Alessandro Lanza\thanks{alessandro.lanza2@unibo.it}}
\author[1]{Fiorella Sgallari\thanks{fiorella.sgallari@unibo.it}}
\affil[1]{Department of Mathematics, University of Bologna, Italy}
\affil[2]{CNRS, UCA, INRIA, Morpheme, I3S, Sophia-Antipolis, France}

\begin{document}

\maketitle

\begin{abstract}
	Over the last 30 years a plethora of  variational regularisation models for image reconstruction has been proposed and thoroughly inspected by the applied mathematics community. Among them, the pioneering prototype often taught and learned in basic courses in mathematical image processing is the celebrated Rudin-Osher-Fatemi (ROF) model \cite{ROF} which relies on the minimisation of the edge-preserving Total Variation (TV) semi-norm as regularisation term. 
	Despite its (often limiting) simplicity, this model is still very much employed in many applications and used as a benchmark for assessing the performance of modern learning-based image reconstruction approaches, thanks to its thorough analytical and numerical understanding. Among the many extensions to TV  proposed over the years, a large class is based on the concept of \emph{space variance}. Space-variant models can indeed overcome the intrinsic inability of TV to describe \emph{local} features (strength, sharpness, directionality) by means of an adaptive mathematical modelling which accommodates local regularisation weighting, variable smoothness and anisotropy.  Those ideas can further be cast in the flexible Bayesian framework of generalised Gaussian distributions and combined with maximum likelihood and hierarchical optimisation approaches for efficient hyper-parameter estimation. In this work, we review and connect the major contributions in the field of space-variant TV-type image reconstruction models, focusing, in particular, on their Bayesian interpretation which paves the way to new exciting and unexplored research directions.
\end{abstract}

\section{Introduction}
\label{sec:intro}

%\ldots\\

The technological developments which favoured the storage, the exploitation and the management of large and unstructured data over the last decades has been responsible of significant and fundamental advances in the field of applied mathematics. In particular, the field of mathematical image processing has undergone significant changes in its paradigm which has been shifted several times over the years. Historically, imaging problems have been formulated as specific instances of (linear) ill-posed inverse problems and studied by means of classical tools such as functional calculus, Partial Differential Equations (PDEs) and Fourier analysis. Later in the 90's, advances in the field of non-smooth variational calculus have drifted the attention towards the use of sparsity-promoting image regularisation models as well as the development of efficient optimisation algorithms tailored to compute the desired output as efficiently as possible. Over the last decade, a new class of models has attracted the attention of the applied mathematics community. Differently from the traditional formulation of imaging problems described above, whose ingredients are chosen \emph{a priori} following analytical, model-driven strategies, by capitalising on the improved technological advances, these new data-driven approaches exploit the large availability of imaging data and design \emph{a posteriori} image reconstruction models tailored to fit the specific application at hand. Understanding the (deep) reasons behind the outstanding performance of these models is nowadays among the (if not the) most prominent challenging tasks, with implications in fields such as artificial intelligence, human-to-robot interactions and bio-inspired computer designs. Alongside, the clever and efficient exploitation of training data has favoured the development of new and theoretically grounded branches of applied mathematics lying at the interface between analysis, variational calculus and statistics. Interestingly, in many imaging applications, the fusion of classical and modern approaches has been in fact capable of overcoming the intrinsic difficulties and rigidities of fully model-driven methods by incorporating appropriately data-driven information.

\medskip

In this spirit,  we present in the following a scientific travel across disciplines taking as example a standard problem in the context of mathematical image reconstruction in the attempt of highlighting for the popular and well-studied Total Variation (TV) regularisation model some of the many extensions proposed over the years featuring as least common denominator the description of local image features at a pixel scale. In the attempt of combining classical tools of variational calculus, optimisation, numerical analysis with more data-driven large-scale statistical approaches, we introduce a new flexible Bayesian interpretation of the imaging quantities into play and report on how such happy marriage can be efficiently exploited as a powerful tool for the exploration of new research directions in imaging.

With the intent of providing an as-exhaustive-as-possible review on the topic, we will start our discussion by recalling in the following introductory sections the main characters in our play, providing appropriate referencing and illustrations which, we hope, will help the inexpert reader to familiarise with the main notions introduced.

\subsection{Imaging inverse problems}
We start our discussion by setting up the scene. To do so, we consider the formulation of a general image reconstruction problem defined on an image domain $\bm{\Omega} := \left\{ (h,l):h=1,\ldots,n_1,~ l=1,\ldots, n_2 \right\}$ $\subset\R^2$ with $|\bm{\Omega}|=n_1n_2=:N$ given by
\begin{equation}
\text{find}\quad\bm{u}\quad\text{s.t.}\quad\bm{b} = \mathcal{N}(\bm{\A}\bm{u}) \,,
\label{eq:lin_model}
\end{equation}
where $\bm{u}\in\R^N$ and $\bm{b} \in \R^M$ are the vectorised unknown image and observed data, respectively, $\bm{\A}\in\R^{M\times N}$ is the (known) linear forward operator and $\mathcal{N}:\R^M\to\R^M$ stands for the degradation operator modelling the presence of noise in $\bm{b}$. Some classical examples for $\bm{\A}$ are, for instance, convolution (blurring), Radon transform and under-sampling operators. 
%in a nondeterministic and possibly nonlinear way. 
As it is known from standard books in inverse problems (see, e.g., \cite{engl2000regularization,stuartInversebook}), it is in general not possible to solve (\ref{eq:lin_model}) directly due to the lack of stability and/or uniqueness properties of the operators involved. As a remedy, problem \eqref{eq:lin_model} can be reformulated as the problem of finding an estimate $\bm{u}^*$ of $\bm{u}$ as accurate as possible by solving a new, well-posed problem where some \emph{a priori} information on $
\bm{u}^*$ is encoded in the form of a regularisation term. In this work, we will focus our attention on the the family of variational regularisation methods where the reconstructed image  $\bm{u}^*\in\mathbb{R}^N$ is computed as a minimiser of a suitable cost functional $\mathcal{J}: \mathbb{R}^N \to \mathbb{R}$ such that the problem can be formulated as 
\begin{equation}\label{eq:GVM}
\text{find }\;\bm{u}^* \:{\in}\; 
\argmin{\bm{u} \in \mathbb{R}^N}
\; \Big( \mathcal{J}(\bm{u};\mu) {:=} 
\mathcal{R}(\bm{u}) + \mu \mathcal{F}(\bm{\A}\bm{u};\bm{b}) \Big).
\end{equation}
The functionals $\mathcal{R}$ and $\mathcal{F}$ are commonly referred to as the \emph{regularisation} and the \emph{data fidelity} term, respectively. While $\mathcal{R}$ encodes prior information on the desired image $\bm{u}$ (such as, e.g, its regularity and/or its sparsity patterns), the term $\mathcal{F}$ measures the `distance' between the given image $\bm{b}$ and $\bm{u}$ after the action of the operator $\bm{\A}$ with respect to some functional describing noise statistics in the data. Finally, the regularisation parameter $\mu > 0$ controls the trade-off between the two terms.

\subsection{A leading actor: TV regularisation. Features, drawbacks and limitations} \label{sec:features_TV}
Probably (if not certainly) the most popular choice in the context of imaging for $\mathcal{R}$ in \eqref{eq:GVM} is the TV semi-norm which is suited for describing meaningful image contents such as image discontinuities (edges). In its simplest discrete form, it is defined as the following non-smooth and convex regulariser
%
%\begin{equation}
%\label{eq:TV_}
%\mathcal{R}(\bm{u}) = \TV(\bm{u}) = \sum_{i=1}^{N} % \| (\nabla\bm{ u})_{i} \|_2 \,   
%\end{equation}
%
%
\begin{equation}
\label{eq:TV_reg}
\tag{\text{TV}}
\mathcal{R}(\bm{u}) \,\;{=}\;\, 
\TV(\bm{u}) \,\;{:=}\;\, 
\sum_{i=1}^{N}  \| (\bm{\D u})_{i} \|_2 \, ,
\end{equation}
where, for each pixel  $\,i=1,\ldots,N,$  $\,(\bm{ \D u})_i = ( \bm{\D}_h \bm{u}, \bm{\D}_v \bm{u})_i \in \R^2$ stands for  the discrete gradient of image $\bm{u}$ at pixel $i$, with $\bm{\D}_h, \bm{\D}_v \in \R^{N \times N}$ suitable finite difference operators discretising the partial derivatives of image $\bm{u}$ along the horizontal and vertical directions, respectively. 
%{\color{red} We remark that $\alpha$ here is a local parameter (later it is $\alpha_i$) and describes local image scales in a statistical sense, as it will be explained in detail in the following, while the parameter $\mu $ in \eqref{eq:GVM}
%is global and codifies the discrepancy with respect to the given noise level of the image. The redundancy of such parameters is therefore only apparent in \eqref{eq:GVM}  and  \eqref{eq:TV_} as the  $\mu $ value is computed depending on the global noise statistics in comparison with the local regularization strength encoded  by the parameters _$\alpha$ or $\alpha_i$.
%}

%The TV regualriser has been
%This choice was originally proposed in \cite{ROF} for image reconstruction problems with Gaussian noise but it  has been successfully applied to a variety of imaging problems. Particularly since \cite{ROF}, TV plays a crucialrole for variational image denoising, deblurring, inpainting, segmentation, magnetic resonanceimage (MRI) reconstruction and many others - see \cite{ChCaCrNoPo10}.
The use of TV regularisation in imaging was firstly proposed by Rudin, Osher and Fatemi in \cite{ROF} which is nowadays probably among the most cited papers in mathematical imaging\footnote{15749 citations according to Google Scholar.}. In the 90's, the use of TV in the context of imaging paved the way for the development of mathematical approaches based on the use of nonlinear, edge-preserving, sparse gradient-based regularisation models and for their application in a variety of image reconstruction problems. Analytically, the fine properties of TV  in the context of image reconstruction have been thoroughly studied and understood over years (see, e.g., \cite{Caselles2015,ChCaCrNoPo10,TVzoo} for a review) and efficient algorithmic approaches have been developed for the efficient numerical solution of TV-based problems (see, e.g., the recent review \cite{chambolle2016introduction}).

Despite the large interest and thorough understanding towards the analytical and regularisation properties of TV regularisation (we will list the main contributions in both directions in due course), 
such regulariser also presents significant limitations. A major one is the so-called \emph{staircasing effect}, which consists of a tendency to promote edges at the expense of smooth structures, see, e.g., \cite{CCN, Nikolova2000,Jalazai2016} for some analytical studies. Moreover, as observed, e.g., in \cite{Strong_2003,Meyer}, TV reconstructions also suffer from loss of contrast artefacts, even in the case of noise-free observed images.
% {\color{red} but
% it has also the drawback of the contrast loss in the restoration \cite{Tang2016, ipol.2012.g-tvd}.
% Luca: non capisco queste referenze, il primo paper ha il titolo del mio lavoro sull'osmosi, il secondo e' un paper IPOL di codici...
% ho messo due riferimenti a mio parere pertinenti} 
Another major limitation of TV as is its \emph{global} or \emph{space-invariant} behaviour, that is the fact that the contribution at each pixel $i=1,\ldots,N$ in \eqref{eq:TV_reg} to the whole regularisation takes exactly the same functional form. Due to this `rigidity', the TV regulariser is not suited to describe possibly very heterogeneous local image structures encountered, for instance, in natural images. Furthermore, such form is not adapted to situations where clear  \emph{directionality} (either global \cite{BayramDTV2012,KonDonKnu17}  or local \cite{Zhang2013}) appears. 

\medskip

We will now review the main contributions proposed in the literature over the last decades to improve upon the regularisation capabilities of TV and, in particular, to reduce the aforementioned drawbacks  within the class of global and space-invariant regularisers. Next, we will discuss on the advantages that space-variant approaches bring along, pointing out how the notion of space-variance has been used under different names in different mathematical fields.

%; in addition, we will review several classes of real-worlds applications in which the space-variant paradigm has already been successfully employed or for which it may result to be particularly beneficial.

%In this paper, we do not aim to trace back every single contribution (both analytical and numerical) on the use of TV as image regularisation model. Rather, we prefer to focus on one specific feature which can be practically (and reasonably) modified so as to achieve better models while keeping essentially the model unchanged: the role of the space-variance. By simply  letting some of the 

%From a Bayesian modelling point of view and via standard Maximum A Posteriori (MAP) arguments, fixing \eqref{eq:TV} as the image regularisation model corresponds to impose a Markov Random Field (MRF) assumption on the objective image by which its gradient magnitude is distributed a half-Laplacian distribution with fixed scale parameter $\alpha>0$. 
%This choice is very rigid\ldots

%Within a Bayesian modelling context, one possibility to increase the flexibility of TV-type image reconstruction models consists in softening the rigidity due to the fixed scale, smoothness and isotropy features of the half-Laplacian distribution. 

\subsection{A partial remedy: space-invariant TV generalisations}
\label{subsec:globgen}
In order to reduce some of the TV reconstruction drawbacks highlighted above, several space-invariant generalisations have been proposed over the years, see, e.g., \cite{CEP, CMM2000, PS2014, Setzer05variationalmethods, WuTai2010} and the references therein. Within this class, we mention in particular two nowadays very popular extensions: the Infimal Convolution Total Variation (ICTV) proposed by Chambolle and Lions in \cite{chambollelions1997} and the Total Generalised Variation (TGV) regulariser introduced by Bredies, Kunisch and Pock in \cite{TGV}. This latter regulariser shares some favourable properties with TV, such as rotational invariance, lower semi-continuity and convexity. However, differently from TV, TGV involves and balances higher-order derivatives of the desired image, which reduces staircasing, while preserving sharp edges at the same time. 
%As noted in several works, ICTV appears to behave better for images with more complicated and varying contents and moderate noise levels, while second-order TGV behaves better for images with large smooth areas.

\smallskip

Generally speaking, while the use of higher-order extensions has shown to be very effective in practice and thoroughly analysed from both an analytical and numerical viewpoint, the question of how to overcome the intrinsic `rigidity' of first-order TV-type regularisation models is still very much open. 
%$\left(\mathrm{TGV}_{\alpha}^{q}\right)$, which has been  to balance the first $q$ derivatives of $\bm{u}$ with a regularisation parameter vector $\alpha$ .
%However, $p$ The parameter $p$, however, is fixed over the whole image domain and, hence, does not allow to capture locality in the image. 
%In \cite{vip,CMBBE} the authors consider a space-variant extension of the TV$_p$ regulariser in \eqref{eq:TVp}  which can better adapt to local image smoothness upon suitable parameter estimation. The new TV$_{\alpha,p}^{\mathrm{sv}}$ regulariser is there defined by
%
%\begin{equation}\label{eq:TVpsv}
%\mathrm{TV}_{\alpha,p}^{\mathrm{sv}}(u) := %\sum_{i=1}^{n} \alpha_i \| (\nabla u)_{i} \|_2^{p_i} \, ,
%\quad p_i \:{\in}\: (0,2], \quad \alpha_i > 0 
%\;\;\, \forall \, i=1,\ldots,n,
%\end{equation}
%
%and shown to be effective on several image restoration problems.
%%%%%%%%%%
With the intent of adapting the TV-type regularisation to structural image information, in \cite{BayramDTV2012} Bayram and Kamasak proposed a Directional TV (DTV) regulariser for image denoising, whose analytic form reads
\begin{equation}  \label{eq:DTV}
\DTV (\bm{u}):=  \sum_{i=1}^N \| \bm{\Lambda}_{a} \bm{\mathrm{R}}_{-\theta} (\bm{\D u})_{i} \|_{2},\qquad a\in(0,1],\qquad \theta\in[-\pi/2, \pi/2). %\tag{\text{$\DTV$}}
\end{equation}
Here, $\theta\in[-\pi/2, \pi/2)$ denotes the dominant orientation in the  target image $\mathbf{u}$,  $\bm{\mathrm{R}}_{-\theta}$ denotes the rotation matrix of angle $-\theta$, while  $\bm{\Lambda}_a=\text{diag}(1,a)$ is a diagonal matrix which encodes the strength of the regularisation along the direction $(\cos\theta,\sin\theta)$ and its orthogonal depending on a parameter $a\in(0,1]$ whose value range from enforcing a full isotropic modelling (no directional preference) for $a=1$ to a strong anisotropy one for $a\approx 0$.
The very same modelling can be analogously used to define a Directional variant of the TGV regulariser, dubbed DTGV, as done in \cite{KonDonKnuDTGV19} by Kongskov, Dong and Knudsen. The use DTV and DTGV is indeed beneficial in applications related to fibres, such as the study of glass fibres in wind-turbine blades \cite{Sandoghchi:14} and the scanning of optical fibres from computed tomography (CT) scans \cite{Jespersen2016}.

% {\color{red} \ldots blind image fusion for hyperspectral imaging \cite{Bungert_2018} and noise attenuation in seismic data \cite{Sismic2021}.  

% Luca: queste referenze sono relative a modelli direzionali e adattativi, NON globali. Ho aggiunto qualche applicazione per direzioni dominanti}

To illustrate the benefits of incorporating directional information defined in terms of a dominant orientation $\theta$ in the regulariser, we consider the test image in Figure \ref{fig:im_teas1}, previously used in \cite{KonDonKnuDTGV19}, in which the existing piece-wise constant regions as well as the smooth straight lines align along the edge direction $\mathbf{v}:=(\cos\theta,\sin\theta)$. The 2D histogram of the gradients reported in Figure \ref{fig:hist_teas1} and the scatter plot in Figure \ref{fig:scat_teas1} show, as expected, global alignment along the perpendicular direction $\mathbf{v}^\perp=(-\sin\theta,\cos\theta)$.
%,\textcolor{red}{ where one can notice the elongation along a direction orthogonal to the orientation of the edges}. 
It appears natural for this type of images to design a regularisation whose functional form could promote smoothing along the direction $\mathbf{v}$ only (i.e. anisotropically) and not indistinctly along all directions (i.e. isotropically) to  better adapt to the underlying geometrical image structures, thus avoiding staircasing.

%In order to encode information about dominant orientation in the image, Bayram et a. in \cite{BayramDTV2012} proposed a type of directional TV (DTV) for image denoising.
%In brief, the total variation of an image is computed by summing the norms of the gradient. Essentially, DTV idea is to first weight the gradients depending on their directions and then sum their norms: 
%\begin{equation}  \label{eq:DTV}
%\DTV (\bm{u}):= \sum_{i=1}^N  \| \bm{\Lambda}_i \bm{\mathrm{R}}_{\theta_i} (\bm{\D u})_i \|_{2} %\tag{\text{$\DTV$}}
%\end{equation}
%
%\begin{equation}  \label{eq:DTV}
%\DTV (\bm{u}):=  \| \bm{\Lambda} %\bm{\mathrm{R}}_{\theta} (\bm{\D u}) \|_{2} %\tag{\text{$\DTV$}}
%\end{equation}
%
%where  $\bm{\mathrm{R}}_{\theta}$
%and $\bm{\Lambda}=diag(1, \alpha)$ denote the rotation and scaling matrices. 

%Such a modification leads to increased sensitivity to variations at a certain direction oriented
%along the angle $\theta$.

\begin{figure}
	\centering
	\begin{subfigure}{0.32\textwidth}
		\centering
		\includegraphics[height=1.3in]{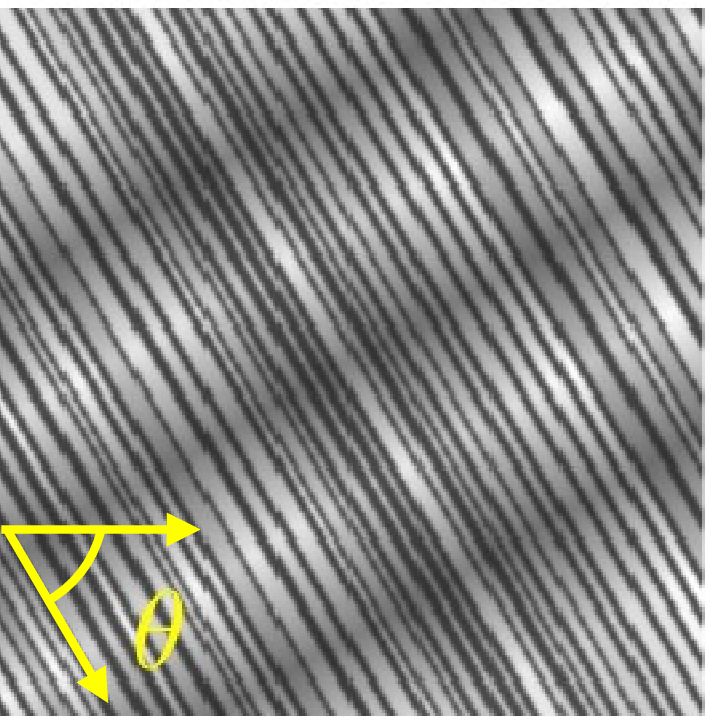}
		\caption{}
		\label{fig:im_teas1}
	\end{subfigure}
	\begin{subfigure}{0.32\textwidth}
		\centering
		\includegraphics[height=1.5in]{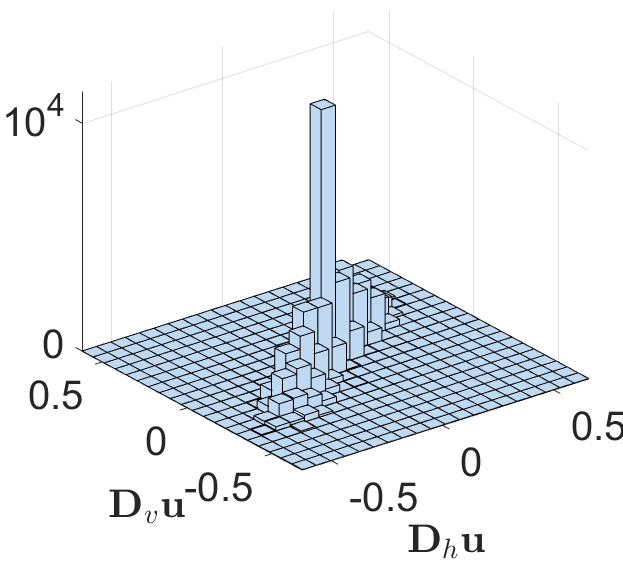}
		\caption{}
		\label{fig:hist_teas1}
	\end{subfigure}
	\begin{subfigure}{0.32\textwidth}
		\centering
		\includegraphics[height=1.5in]{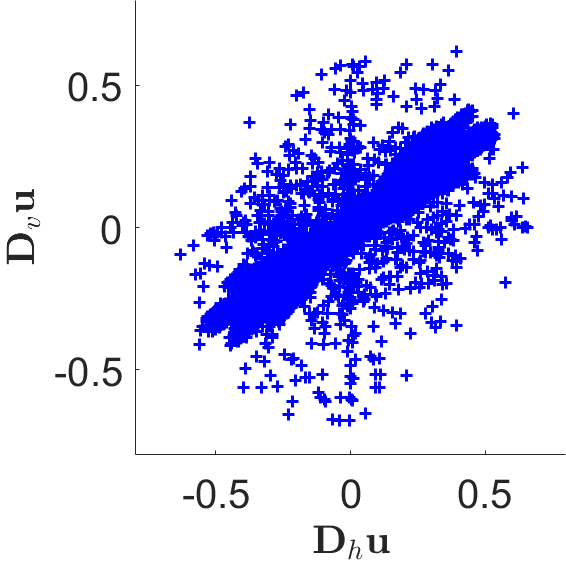}
		\caption{}
		\label{fig:scat_teas1}
	\end{subfigure}
	\caption{Test image from \cite{KonDonKnuDTGV19} with a clear dominant direction defined by angle $\,\theta \in [-\pi/2,\pi/2)$ (a), 2D histogram of image gradient components (b), scatter plot of image gradient components (c).}
	\label{fig:teas1}
\end{figure}

A different extension to TV regularisation has been explored in \cite{VB3,Krishan2009,tvpl2}.  By adopting a statistical viewpoint, in these works the authors show how the use of TV implicitly corresponds to assume a space-invariant one-parameter half-Laplacian distribution (hLd) for the gradient magnitudes of the target image $\bm{u}$, that appears to be in general too restrictive to model the distribution of gradient magnitudes in natural images for which a more general half-Generalised Gaussian distribution (hGGd) prior should be considered instead. This choice corresponds to employ the following TV$_p$ regularisation model
\begin{equation}\label{eq:TVp}
\mathrm{TV}_p(\bm{u}) := \sum_{i=1}^{N}  \|(\bm{\D u})_{i}  \|_2^p \, ,
\quad p \:{\in}\: (0,2],
\end{equation}
where the exponent $p\in(0,2]$ is a free parameter providing the TV$_p$ regulariser with higher flexibility than the TV regulariser. Its
setting is indeed related to the properties of the image of interest: it can be fixed either empirically to enhance sparsity ($p<1$, see \cite{Krishan2009}) or regarded as an information tailored to the image itself which should thus be estimated appropriately. The TV$_p$ regulariser has proved to be effective for the solution of several imaging problems ranging from  labelling and segmentation \cite{Li2020} to blind deblurring \cite{Krishan2009, Liu2017} and synthetic aperture radar (SAR) image despeckling  \cite{Guo2021} and many more. Its performance strongly depends on the selected/estimated value of $p$, whose setting may be hard in case of very heterogeneous images composed, for instance, by both smooth and piece-wise constant regions. 

%Similarly as for what observed for the DTV regulariser, the good performance of TV$_p$ and  (a fortiori) of TV regularisations is, generally speaking, limited to a specific class of image. The use of TV$_p$ with $p<1$, for example, appears to be a particular good choice in the case of images with extended homogeneous regions and limited amount of texture information, due to its enhanced gradient-sparsifying properties which, on the other hand, is poorly adapted to images rich of texture information. In the following we discuss on how the use of an adaptive space-variant modelling can improve upon such `rigid' behaviours by endowing the regularisation model with further degrees of freedom.

Finally, we recall that a further albeit classical limitation of global TV-type regularisation consists in the choice of the optimal regularisation parameter $\mu$ in (\ref{eq:GVM}). As a matter of facts, this is a common challenge for all regularised inverse problems in the form \eqref{eq:GVM}, not limited to the TV context. Among classical methods for parameter selection, we recall here those based on discrepancy principle \cite{hanke.ch8,APE}, generalised cross validation \cite{Golub79, Fenu2016}, L-curve analysis \cite{CALVETTI2000423} and unbiased risk estimators (SURE) \cite{SURE,Sure_app}. Note that while being effective in practice, these approaches often require the prior knowledge of the noise level in the data, which in many practical applications is a difficult information to obtain. To avoid this issue, different techniques based, for instance, on bilevel learning \cite{KunischPock2013,CalatroniBilevel,Van_Chung_2017,Hintermuller2017a,Hintermueller2017b} or on statistical whiteness principles can be used \cite{Lanzaetna,Lanza_JMIV}.

%\medskip

%In the following we discuss on how the use of an adaptive space-variant modelling can improve upon such `rigid' behaviours by endowing the regularisation model with further degrees of freedom.

\subsection{Incorporating space variance} 

%{\color{red} verificare figure (figs-> figures) + ref a fig multiple}

While the aforementioned modelling appears very restrictive from a \emph{global} viewpoint, it is a natural question wondering whether by considering \emph{local} image information, that is looking at image patches of suitably small size, image information can be `glued' together so as to define a better, more suitable image regulariser.
%Such assumption will be deeply discussed in this review paper and several perspectives (statistical, geometrical, analytical) will be provided to the reader to convince her/him on the flexibility of this modelling. 

We motivate this idea by showing in Figure \ref{fig:im_teas2} an enlightening example concerned with local description of directional features for the popular test image image \texttt{barbara}. %Here, we choose to give a closer look to the global and local behavior of the gradients, with the aim of getting more insights about the potential benefits of a space-variant approach. 
We start selecting three sub-regions of interest: two of them are characterised by geometric textures - see Figures \ref{fig:im_teas2_z1},\ref{fig:im_teas2_z2} - while the other presents smooth and homogeneous details - see Figure \ref{fig:im_teas2_z3}. The global image histogram and plot in Figures \ref{fig:hist_teas2}, \ref{fig:scat_teas2} clearly show that image gradients are not oriented along a single dominant direction; rather, they appear to be non-uniformly spread over a box-shaped neighbourhood of the origin. Different scenarios arise when considering the two textured regions, as the gradients therein show a clear directional bias, as displayed in the histograms in Figures~\ref{fig:hist_teas2_z1},\ref{fig:hist_teas2_z2} and in the scatter plots in Figures~\ref{fig:scat_teas2_z1},\ref{fig:scat_teas2_z2}. Finally, for the gradients computed within the smooth region, a different configuration occurs, being them not aligned along any preferred direction - see Figure \ref{fig:hist_teas2_z3}, but rather concentrated in a neighbourhood of the origin, due to the large homogeneous image content. As observable from the scatter plot in Figure \ref{fig:scat_teas2_z3},  dispersion around the origin is however smaller with respect to the global scatter plot.

\begin{figure}[!th]
	\centering
	\begin{subfigure}{0.32\textwidth}
		\centering
		\includegraphics[height=1.3in]{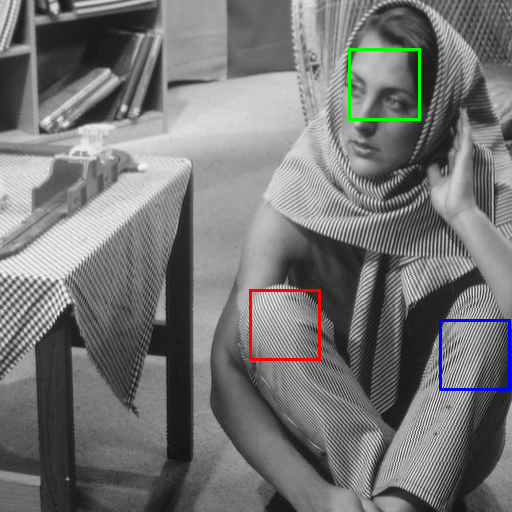}
		\caption{}
		\label{fig:im_teas2}
	\end{subfigure}
	\begin{subfigure}{0.32\textwidth}
		\centering
		\includegraphics[height=1.4in]{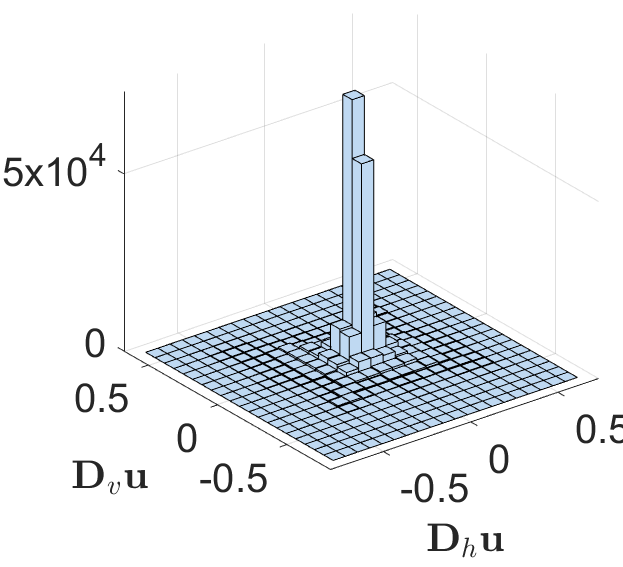}
		\caption{}
		\label{fig:hist_teas2}
	\end{subfigure}
	\begin{subfigure}{0.32\textwidth}
		\centering
		\includegraphics[height=1.4in]{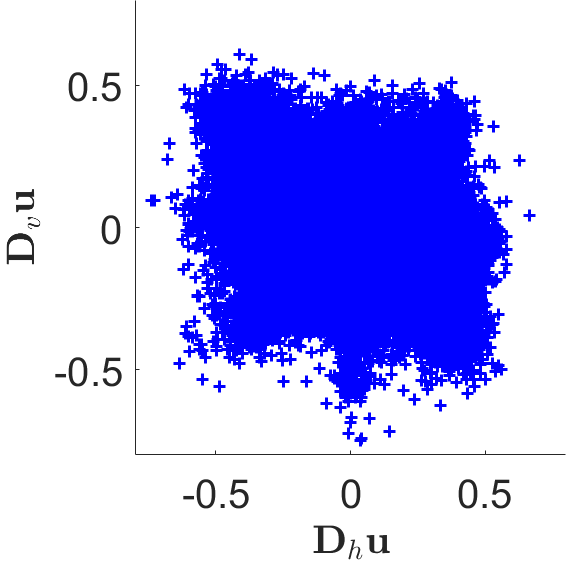}
		\caption{}
		\label{fig:scat_teas2}
	\end{subfigure}\\
	\begin{subfigure}{0.32\textwidth}
		\centering
		\includegraphics[height=1.3in]{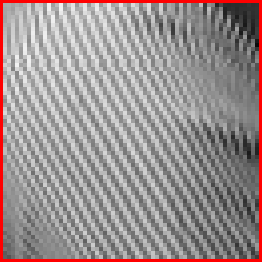}
		\caption{}
		\label{fig:im_teas2_z1}
	\end{subfigure}
	\begin{subfigure}{0.32\textwidth}
		\centering
		\includegraphics[height=1.4in]{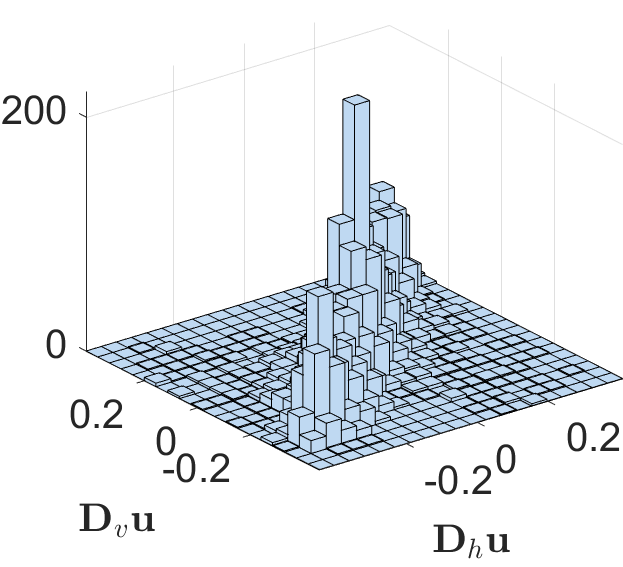}
		\caption{}
		\label{fig:hist_teas2_z1}
	\end{subfigure}
	\begin{subfigure}{0.32\textwidth}
		\centering
		\includegraphics[height=1.4in]{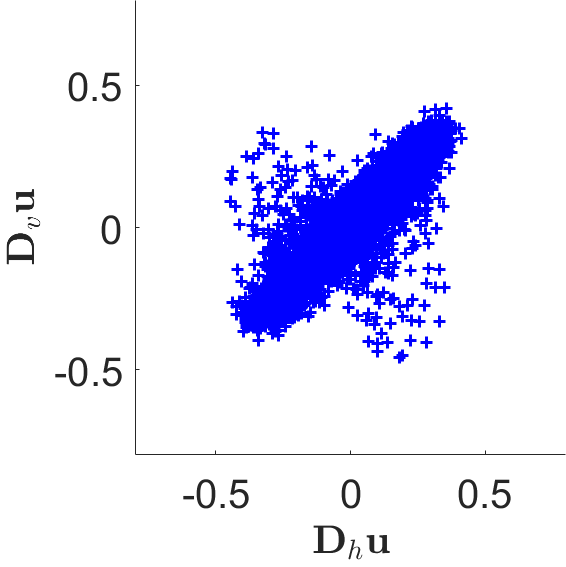}
		\caption{}
		\label{fig:scat_teas2_z1}
	\end{subfigure}\\
	\begin{subfigure}{0.32\textwidth}
		\centering
		\includegraphics[height=1.3in]{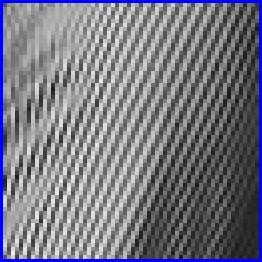}
		\caption{}
		\label{fig:im_teas2_z2}
	\end{subfigure}
	\begin{subfigure}{0.32\textwidth}
		\centering
		\includegraphics[height=1.4in]{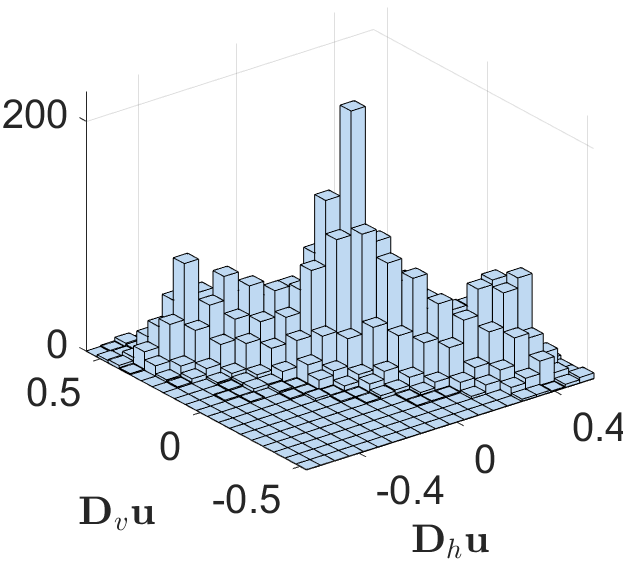}
		\caption{}
		\label{fig:hist_teas2_z2}
	\end{subfigure}
	\begin{subfigure}{0.32\textwidth}
		\centering
		\includegraphics[height=1.4in]{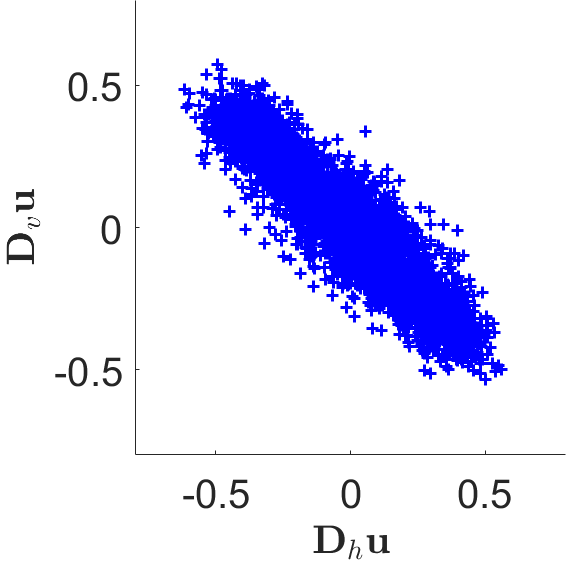}
		\caption{}
		\label{fig:scat_teas2_z2}
	\end{subfigure}\\
	\begin{subfigure}{0.32\textwidth}
		\centering
		\includegraphics[height=1.3in]{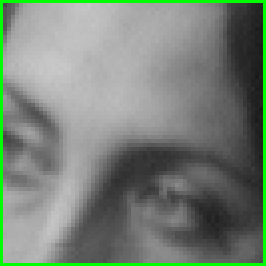}
		\caption{}
		\label{fig:im_teas2_z3}
	\end{subfigure}
	\begin{subfigure}{0.32\textwidth}
		\centering
		\includegraphics[height=1.4in]{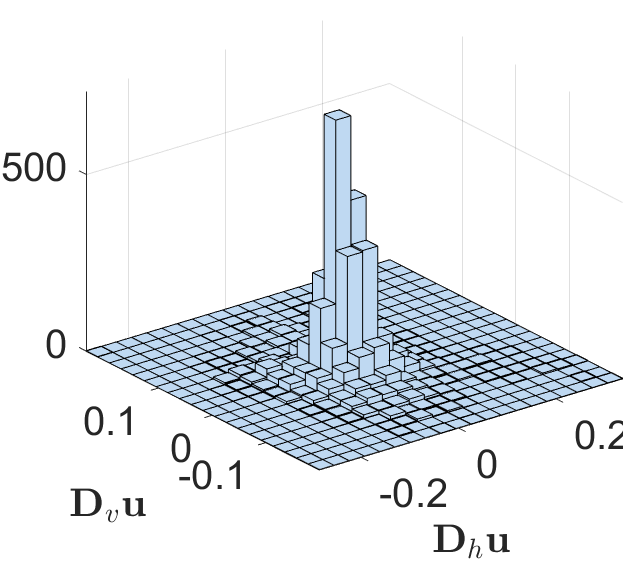}
		\caption{}
		\label{fig:hist_teas2_z3}
	\end{subfigure}
	\begin{subfigure}{0.32\textwidth}
		\centering
		\includegraphics[height=1.4in]{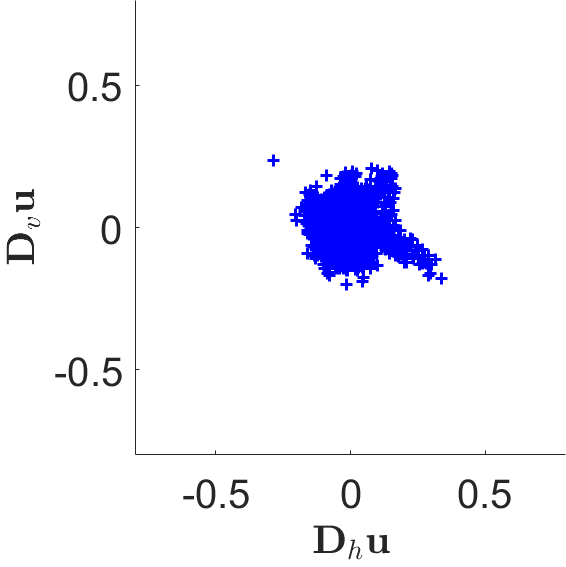}
		\caption{}
		\label{fig:scat_teas2_z3}
	\end{subfigure}
	\caption{Test image \texttt{barbara} and selected sub-regions (left column), 2D histograms of image gradient components (middle column), scatter plots of image gradient components (right column).}
	\label{fig:teas2}
\end{figure}

Similarly, examples illustrating the different local behaviour of the TV$_p$ regulariser \eqref{eq:TVp} depending on the particular (homogeneous VS. textured) image patch considered could be given, as we will extensively report in the following sections. Those examples motivate and justify the idea of formulating a new space-variant framework allowing for the description of image contents at a \emph{local} scale. 

The idea of incorporating space-variant information is in fact not new in the context of mathematical methods for image reconstruction.  Early approaches can already be traced back looking at the contributions in the field of diffusion-type PDEs for imaging \cite{PeronaMalik1990,weickert98,Roussos2010,Scharr2003, Bro02} and statistical approaches \cite{Descombes1999,Ramamurthy2009,Calvetti_2019,Calvetti_2020,Calvetti_2020spars,Pereyra2017FastUB,CS}. In the last couple of years and under a different perspective, few contributions have also been made in the context of (deep) learning approaches for imaging \cite{Goodfellow2014,Lucic2018,Lunz2018,miyato2018spectral,kurach2019the}. In the following, we summarise the ideas contained in these works by structuring our description by topical subsections, to favour readability. As the reader will notice, however, a detailed description of the (large) class of space-variant variational image regularisation methods extending those described in Section \ref{subsec:globgen} will be postponed to the later sections of this work, as those models will appear naturally as soon as our combined statistical/analytical modelling will be introduced. 
%Moreover in  real world inverse problems, e.g. medical imaging, video editing and cultural heritage, prior information and a priori-based regularization techniques play important roles.
%To simplify the description, we detail in the following the main mathematical fields where this idea was studied.
%\paragraph{Application with Adaptive modelling}

%\paragraph{Space-variance in PDEs}

\subsubsection{PDE approaches}
In the context of PDE approaches for imaging, the idea of space variance has enriched standard linear and non-linear diffusion-type models starting from the 
work of Perona and Malik in the 90's \cite{PeronaMalik1990}. There, the authors proposed  a nonlinear space-adaptive diffusion method for avoiding the blurring and localisation problems of linear diffusion filtering. To do so, an inhomogeneous and space-variant process reducing diffusivity at locations with high probability of being edges was considered. As a natural choice, for any point $\bm{x}\in \bm{O}\subset\mathbb{R}^2$, where by $\bm{O}$ we denote a regular image domain, an edge-stopping diffusion function depending locally on the quantity $| \bm{\nabla} u(\bm{x})|$ was used as a likelihood measure. Note that the Perona-Malik filter is a particular instance of the diffusion model
\begin{equation} \label{eq:anis_diff}
\begin{cases}
$\:$u_t = \mathrm{\bm{div}}\big({W(| \bm{\nabla} u |)}\bm{\nabla} u \big) &\text{on } \bm{O}\times (0,\infty], \\
$\:$u(x,0)=b(x) & \text{on } \bm{O},\\
$\:$\langle \bm{\nabla} u, \bm{n}\rangle=0  & \text{on }   \partial\bm{O}\times(0,T],
\end{cases}
\end{equation}
where $\bm{n}$ stands for the outward normal vector on $ \partial\bm{O} $ and $b$ stands for the observed image. Problem \eqref{eq:anis_diff} is a standard reference model for anisotropic image restoration PDE approaches. In the case $W\equiv 1$, it acts as a convolution model of the given function $b(\cdot)$ with a Gaussian kernel with standard deviation parameter $\sigma=\sqrt{2 t}$. Such operation corresponds to the well-known low-pass spectral filtering and it is commonly used for smoothing pictures by averaging values within a certain neighbourhood. In the general case, model \eqref{eq:anis_diff} produces a family of images parametrised by $t>0$,  each resulting in a combination between the original image and a filter that depends on the local content of the given image $b(\cdot)$.  
%As a consequence, anisotropic diffusion is a non-linear and space-variant transformation of the original image. 
For $\bm{x}\in\bm{O}$, the function  $W$ can be, for instance, chosen as 
\begin{equation}\label{eq:diffcoef}
W(|\nabla u(\bm{x})|)\,\;{=}\;\,\frac{1}{\displaystyle{1+\left(\frac{|\bm{\nabla} u(\bm{x})|}{K}\right)^{2}}}\,,
\end{equation}
where the parameter $K>0$ controls the sensitivity to edges.
%Its value is inversely proportional to the magnitude of the gradient. 
On uniform regions, where the magnitude of the gradients is weaker, the diffusion coefficient $W$ is close to $1$, so that \eqref{eq:anis_diff} turns into a heat equation which smooths out the noise. Close to edges and boundaries, the gradient magnitudes get larger instead, thereby the diffusion coefficient in \eqref{eq:diffcoef} vanishes; as a result, in correspondence of these pixels diffusion is not performed and meaningful structures are preserved.

However, despite their edge-adaptive behaviour, scalar diffusivity functions $W(\cdot)$ are intrinsically uncapable of adjusting the diffusion along the orientation of salient image structures. To do so, a diffusion tensor leading to anisotropic diffusion filters has to be
introduced.
%
%so the diffusion is stopped across boundaries and it preserves the edges.
%Near boundaries the magnitude of the gradient will be strong,
%within uniform regions, or inner regions, diffusion coefficient is almost 1, so it acts as heat equation, it smooths the inner region and removes the noise.
The most popular choice consists in replacing the scalar $W$ by the structure tensor $\bm{\mathrm{J}}(\bm{x})=\bm{\nabla} u(\bm{x}) \otimes \bm{\nabla} u(\bm{x}) \in \R^{2 \times 2}$. This matrix can be written in terms of its eigenvalues and eigenvectors, the latter encoding the dominant local orientation, the former representing the strength of the diffusion along the preferred direction and its orthogonal. As a result, the action of $\bm{\mathrm{J}}(\bm{x})$ can be somehow synthesised by the elongation of the associated elliptical level curves, see, e.g., \cite{weickert98,Roussos2010,Scharr2003, Bro02,Grasmair2010,Estellers2015}.% for analogous approaches. 
%... which locally adapts both the strength and the direction of the diffusion along edges where the elliptical curves representing the action of $\mathbf{J}(\bm{x})$ are elongated, see, e.g.,
%, while they get narrow across them %it has an orientation given by the structure such that it is elongated along the structure and narrow across 

%{\color{red} valutare se dire due parole in piu' su structure tensor

%To go further into the formalism, the structure tensor $J_{\rho}$ can be written over its eigenvalues $\left(\lambda_{+}, \lambda_{-}\right)$ and eigenvectors $\left(\theta_{+}, \theta_{-}\right),$ that is:
%$$
%J_{\rho}=\left(\begin{array}{ll}\theta_{+} & \theta_{-}
%\end{array}\right)\left[\begin{array}{cc}
%\lambda_{+} & 0 \\0 & \lambda_{-}\end{array}\right]\left[\begin{array}{l}\theta_{+} \\\theta_{-}
%\end{array}\right]=\lambda_{+} \theta_{+} \theta_{+}^{T}+\lambda_{-} \theta_{-} \theta_{-}^{T}
%$$
%The eigenvectors of $J_{\rho}$ give the preferred local orientations, and the corresponding eigenvalues denote the local contrast along these directions.
%}

%A more general formulation consider a tensor as $W(x)$, 

\subsubsection{Statistical approaches}  \label{sec:statistical}

%{\color{red} aggiungere stima MAP classica, MRF e alcune citazioni classiche tipo Geman/Geman e altre}

Statistical approaches for image processing have become very popular in the last decades 
%as they fill several gaps left behind by the fully deterministic methods, which mainly have to do with the difficulty in expressing qualitative information about the image of interest. In this perspective, Bayesian strategies 
due to their ability to %provide a particularly versatile framework of
incorporate non-deterministic information in the forward model - see \cite{CSstat}. Here, the core idea is to model the unknown image $\bm{u}$ as a random variable $\bm{U}$ to highlight the intrinsic uncertainty about its value, which is also related to possible approximations of the model operator $\bm{\A}$, and of the noise degradation model $\mathcal{N}$. 
The information or the beliefs available \emph{a priori} on the random variable $\bm{U}$ are encoded in the \emph{prior} probability density function (pdf) $\mathbb{P}(\bm{u})$. Analogously, the observed image $\bm{b}$ is regarded as a realisation of a random variable $\bm{B}$, whose behaviour for a fixed $\bm{u}$ is encoded in the \emph{likelihood} pdf $\mathbb{P}(\bm{b}\mid \bm{u})$. In this framework, the goal is to recover the distribution of $\bm{U}$ according to the given observation $\bm{b}$ and the underlying degradation model $\bm{\A}$; in terms of distributions, this translates in seeking the \emph{posterior} pdf $\mathbb{P}(\bm{u}\mid \bm{b},\bm{\A})$ which is related to prior and likelihood pdfs via the Bayes' formula:
\begin{equation} \label{eq:bayes}
\mathbb{P}(\bm{u}\mid\bm{b},\bm{\A}) = \frac{\mathbb{P}(\bm{b}\mid\bm{\A}\bm{u})\mathbb{P}(\bm{u})}{\mathbb{P}(\bm{b})} \,,
\end{equation}
where $\mathbb{P}(\bm{b})$ is often referred to as the evidence term and plays the role of a normalisation constant.

The poorer is the information on the degradation model, the more relevant is the design of a suitable prior for the unknown image. The prior pdf $\mathbb{P}(\bm{u})$ can model different characteristics of the image, ranging from the presence of textures to boundary configurations. Popular priors for image restoration problems encode information on the distribution of the grey levels within an image and the transition of grey-scale intensities between different areas of the image \cite{Li2009}.

In \cite{Geman}, the authors interpreted the pixel-grey levels as states of atoms in a lattice-like physical system, so that the unknown image is modelled as a Markov Random Field (MRF). This translates into the requirement that a selected feature at the generic pixel $i$ of $\bm{u}$ only depends on the behaviour of $\bm{u}$ at pixels belonging to a set $\mathcal{C}_i$ of neighbours of $i$, called \emph{clique}. When the selected feature is the grey level, the 2D Markovian property at pixel $i$ reads
\begin{equation}
\label{eq:mrf}
\mathbb{P}\big(\,U_i = u_i\mid U_j = u_j\,,\;j\neq i\,\big) 
\:\;{=}\;\:
\mathbb{P}\big(\,U_i = u_i\mid U_j = u_j\,,\;j \in \mathcal{C}_i\,\big),
\end{equation}
where by $\mathbb{P}(\cdot)$ we denote the  probability density function (pdf).
%with $\mathrm{P}$ denoting the probability mass function.
The prior distribution for a MRF is the so-called \emph{Gibbs prior}
\begin{equation}
\label{eq:mrfpr}
\mathbb{P}(\bm{u}) = \frac{1}{Z} \exp \bigg(- \sum_{i=1}^{n}V_{\mathcal{C}_{i}}(\bm{u})\bigg)\,,
\end{equation}
where $Z>0$ is a normalisation constant and $V_{\mathcal{C}_{i}}$ is referred to as the \emph{Gibbs potential function} defined on a clique of pixels centred at pixel $i$ - see also \cite{MRF1} for a more extensive discussion. The potential functions typically depend on a number of parameters that can be considered fixed in the case our prior beliefs are informative enough to allow for their manual setting, but which, in general, may vary from pixel to pixel. 

%where $Z>0$ is a normalisation constant and $V_{\mathcal{C}_{i}}$ is referred to as the \emph{Gibbs potential function} defined on a clique of pixels centred at pixel $i$ - see also \cite{MRF1} for a more extensive discussion. The parameters involved in the definition of the potential functions can be considered fixed in the case our prior beliefs are informative enough to allow for their manual setting, but in general may vary from pixel to pixel.  

%
%For a more extensive discussion on MRF we refer the reader to \cite{MRF1}.
%Observe that the $i$-th potential function does only depend on the pixels with indices in $\mathcal{C}_i$. 
%
As it will be explored in details in Section \ref{subsec:prior} there exists a strict relation between Gibbs' priors and many notable regularisers, such as Tikhonov and TV, 
%which arise as the argument of the negative exponential in \eqref{eq:mrfpr} when the priors
which are designed based on the properties of the discrete image gradients $\bm{\D u}$. Note that non-stationary (i.e. space-variant) MRF approaches have been also considered over the years (see, e.g., \cite{Descombes1999,Ramamurthy2009,Li2009} for some applications) in a purely statistical framework, for which the design of efficient sampling strategies and Maximum Likelihood (ML) approaches is required.

By \eqref{eq:bayes}, the sought unknown image $\bm{u}$ can be recovered as a single-point representative of $\mathbb{P}(\bm{u}\mid\bm{b},\bm{\A})$ moving from a purely statistical to a more optimisation-based framework. A very popular strategy for that goes under the name of Maximum A Posteriori (MAP) estimation \cite{KSstat}: it consists in summarising the posterior pdf with its mode, i.e.
\begin{eqnarray}
\bm{u}_{\mathrm{MAP}}
&\,\;{\in}\;\,&
\argmax_{\bm{u}\in\R^N}\left\{\,\mathbb{P}(\bm{b}\mid\bm{u},\bm{\A})\,\mathbb{P}(\bm{u})\,\right\}
\nonumber\\
&\,\;{=}\;\,& 
\arg\min_{\bm{u}\in\R^N}\!\left\{\,-\ln\mathbb{P}(\bm{b}\mid\bm{u},\bm{\A})-\ln\mathbb{P}(\bm{u})\,\right\}, 
\label{eq:mapfirst}
\end{eqnarray}
where the evidence term $\mathbb{P}(\bm{b})$ has been neglected since it does not depend on $\bm{u}$. We remark that connections between the statistical interpretation of the inverse problem \eqref{eq:lin_model} via MAP formulation \eqref{eq:mapfirst} and its analytical counterpart \eqref{eq:GVM} have been nowadays drawn for a large variety of \emph{space-invariant} regularisation models, while they are generally hard to be understood in the context of \emph{space-variant} models. We will try to fill this gap in the following sections, but for the moment we only warn the reader that moving from a global or space-invariant to a more-informative local or space-variant framework requires significant modelling as well as computational differences have to be taken into account as the explicit dependence on the distribution hyperparameters in \eqref{eq:mrfpr} cannot be neglected any longer.

In space-variant settings, a very natural alternative to a completely supervised strategy, i.e. the parameters of the non-stationary MRFs are specified a priori, is to model the parameters themselves as random variables following a hierarchical Bayesian approach. In this sense, the MAP paradigm represents a very versatile tool in terms of imaging applications and algorithmic optimisation. Among the many contributions in this field, we mention \cite{Calvetti_2019,Calvetti_2020,Calvetti_2020spars}, where the authors propose an iterative alternating scheme for the solution of the hierarchical MAP formulation for sparse recovery problems. A parameter marginalisation, followed by a small variance analysis is employed instead in \cite{Pereyra2017FastUB} for image inpainting applications. 

We further mention that the classical literature on \emph{compressed sensing} algorithms has been revisited and interpreted in probabilistic terms; in this perspective, the popular $\ell_1$-regularisation terms can be thought of as deriving from a support-informed or spatially-adaptive prior, where the local weights are typically estimated starting from the observable data following an empirical Bayesian approach - see \cite{CS} and references therein.

\subsubsection{Generative and unfolded learning-based approaches}  \label{sec:learning}

Many shortcomings are classically associated to the use of model-driven approaches.  Among them, the dependence on the (supposedly known) forward model operator $\bm{\A}$, the high complexity encountered when solving in practice large dimensional PDE systems or computing high-dimensional integrals and, mostly, the conceptual and intrinsic ideas of representing the unknown solution in terms of \emph{a-priori} fixed models and distributions have been shown to represent major limitations which have been overcome over the recent years  by replacing knowledge-based by data-driven designs. It is outside the scopes of this review providing an extensive state-of-the-art description of these from-shallow-to-deep-learning approaches as well as their connection with the world of inverse problem. For that, we refer the reader to the recent review paper \cite{SchoenliebDataDriven2019} where these questions are addressed in a thorough way and where an extensive literature review is given.

For the following description, we will limit ourselves to consider two nowadays extremely popular classes of data-driven approaches which, in some sense, are in close connection with the space-variant modelling discussed in this work. The former class, introduced firstly in  \cite{Goodfellow2014} under the name of Generative Adversarial Networks (GAN), exploit training examples to estimate, rather than the desired solution itself, the distribution it is sampled from (ideally \eqref{eq:mrfpr}) by means of the interplay between two adversarial networks which force the joint machinery to discriminate between `true' data distribution and its `opponent' (adversarial) attacks. Differently from the fully model-driven statistical approaches summarised in Section \ref{sec:statistical}, the solution computed by GANs  potentially offers a more precise way to describe local image features in natural images. However, given their fully data-driven modelling, the interpretation of GANs within both a statistical and analytical framework remains somehow unclear. The outstanding performance of GANs in the field of imaging, however, has favoured their use in a large variety of imaging problems, see, e.g., \cite{Lucic2018,Lunz2018,miyato2018spectral,kurach2019the} and is still a growing research area in the field which could inspire and complement the classical analytical/statistical knowledge-based modelling in future work (see Section \ref{sec:challenges} for research outlooks).

We further recall a different class of learning-based methods which have become very popular due to their easy interpretation from an optimisation viewpoint. Heuristically, such approaches are indeed inspired and thus made understandable by the operative expression of the iterative solution of problem \eqref{eq:GVM} whose algorithmic solver can be \emph{unrolled}/\emph{unfolded} by means of deep-learning architectures (see \cite{Monga2021} for a review). As showed recently in some works, see, e.g., \cite{ChenPock2017,Bertocchi_2020,Hinz2018}, this approach provides indeed an interpretable framework for the data-driven estimation for several model hyperparameters, such as diffusion filters, algorithmic step-sizes and many more, thus combining well-known notions of optimisation (such as gradient/proximal-gradient updates, algorithmic parameters) with more learning-based concepts (activation functions, learning rates\ldots).

% hese models achieve increasingly better predictive performance, their structures have also become much more complex. A common and difficult problem for complex models is overfitting. Regularisation is used to penalise the complexity of the model in order to avoid overfitting. However, in most learning frameworks, regularisation function is usually set with some hyper-parameters where the best setting is difficult to find. 
% The above difficulty is addressed by adaptive regularisation methods \cite{IEEE_2021} based on Gaussian Mixture (GM) to learn the best regularisation function/prior according to the observed parameters. 

\subsubsection{Exploiting space variance in applications}

Space-variant models have been shown to be effective not only on synthetic and/or `didactic' examples, but also on a wide class of real-world applications, such as, for instance, medical imaging.
%In Computed Tomography (CT), a beam of X-rays is  sent through an object of interest. When the beam has crossed the object, the reduced intensity of the rays are measured so as to give, by means of reconstruction techniques, a 2D or a 3D image of the attenuation coefficients.
%in CT by sending X-rays through the object of interest, we are able to measure the reduction in intensity on the opposite side due to the attenuation of the X-ray when travelling through the object. Then using reconstruction methods we will obtain a 2D or 3D image of  the X-ray attenuation coefficients in the object.
%In recent years, much attention has been given to variational models of the form \eqref{eq:GVM}, upon an appropriate choice of the operator $\bm{\A}$ which is typically set as the Radon transform in the case of CT or an undersampled Fourier transform operator for MRI.

%while $\bm{b}$ denotes the CT data, typically referred to as \emph{sinogram}.
Regarding CT applications, for instance, several image  reconstruction approaches relying intrinsically on a space-variant estimation of models hyperparameters have been considered. In \cite{Tovey2019}, for instance, directional TV regularisation on the sinogram data is proposed to inpaint the missing range of angles and improve the inversion process, while in \cite{Dong_CS_2020} an automatic selection strategy for a weighted reconstruction model is considered and validations on both synthetic and real data are reported.

%The choice of the regularisation parameter $\mu$ is crucial to get high-quality reconstructions. When removing the noise and the artifacts typically arising in the reconstruction, one often risks to damage details possibly appearing at different scales. For this reason, a remedy proposed in \cite{Dong_CS_2020} consists in assigning different regularisation parameters to different structural scales, so as to perform a regularisation that preserves the meaningful information in the object.
%The noise and artremoval in which details at different scalesfor receiv%ing a desirable reconstruction and a spatially varying regularization parameter  is in accordance to the fact that objects imaged with CT usually contain structures of different scales. In order to preserve these scales in the CT image reconstruction while still removing noise and other artefacts, different regularization parameters should be assigned to different structural scales in the reconstruction \cite{Dong_CS_2020}.

As far as (multicontrast) Magnetic Resonance Imaging (MRI) applications are concerned, we refer here to \cite{Arridge2007,Ehrhardt2016} where local regularisation weighting as well as directional (therein often called \emph{structural}) strategies extending those used in \cite{Kaipio_1999} for Electrical Impedence Tomography (EIT) are used.
Analogous approaches have further been considered in \cite{Ehrhardt_2014,Ehrhardt_2016} for improving the quality of Positron Emission Tomography (PET) imaging data by an appropriate fusion driven by structural MRI data. Similar approaches have further been considered in \cite{Bungert_2018} for blind hyperspectral imaging and in \cite{Bathke2017} for magnetic particle imaging.
%{\color{red} l'applicazione PAT qui sotto use delle wavelet direzionali, non c'entra molto con il nostro framework}

Within the class of medical imaging applications, we further  mention Photo-Acoustic Tomography (PAT)  %also known as opto-acoustic tomography 
%PAT is an intrinsically hybrid imaging technique that combines the high spectroscopic contrast of tissue constituents to light, with the high resolution of ultrasound imaging technique, see, e.g., \cite{} for a review.
%Recently, there has been intense interest in solving the PAT reconstruction problem iteratively with a specific focus on regularised reconstruction. The regularisers used for these reconstructions have many different forms, depending on prior assumptions on the image, e.g.
%TV, second order TGV and the complex -tree wavelet transform 
imaging, for which in \cite{app_2018} a space-variant modelling well-adapted to the composite and heterogeneous nature of the target is proposed.
%In \cite{ParMasSch18applied}, applications to applications to surface reconstruction problems from scattered height data were considered.

Among the many other real-world applications which significantly benefit from the use of space-variant approaches in terms, in particular, of  local directional dependence, we mention here the work carried out in \cite{Giakoumis,Lozes2015,Calatroni2018Heritage,Heritage} where non-invasive digital reconstruction models based on anisotropic diffusion and transport PDEs have been effectively used in the context of digital reconstruction of ancient frescoes, illuminated manuscripts, surface colorisation, and inpainting to unveil missing or occluded contents via the use of inpainting, image fusion and/or image enhancement techniques.

\subsection{Motivation and contribution of this work}
From the aforementioned sections we have seen that, though following different paths, different communities focused on the mathematical modelling of local image features. Depending on the scientific community considered and when looking at all these works, though, it is not very clear how to connect and compare these different findings, since very similar properties interpreted in different fields may be called with very different names (e.g. \emph{non-stationary} models in a Bayesian framework, \emph{structural} or \emph{adaptive} approaches in an analytical context\ldots). The objective of this work is to provide a unified view of many of these many different models in terms of a new, generalised Bayesian modelling which allows also for some original extensions which have not explicitly studied before. The Bayesian framework described in this work paves indeed the way for the design of new, unexplored strategies helpful to design flexible and adaptive image regularisation functionals whose hyperparameters can be estimated by taking advantage of the form of the underlying gradient distributions through statistical approaches. In order to present the framework in its full generality and in view of its application to a larger class of image reconstruction models, we will not omit to provide details on the use of a (generalised) discrepancy principle strategy needed to compute the hyperparameters associated to the likelihood functionals. We further stress that TV regularisation \eqref{eq:TV_reg} is here taken as a reference regularisation model in this work due to the incredible amount of contributions developed over the last thirty years, as we have discussed and will discuss thoroughly in the following. However, the reader should be reassured that analogous considerations could (and should!) still be exploited for different type of regularisation functionals, as we will shortly comment in the final Section \ref{sec:challenges} of this work.

\subsection{Structure of the paper} 
The paper is organised as follows. In Section \ref{sec:prel}, we set the notations and recall the main notions and definitions which will be useful in the rest of the article. %Besides this, we also extend the formulation of the popular Discrepancy Principle (DP) for the automatic selection of $\mu$ in \eqref{eq:GVM} under more general noise models. 
Then, in Section \ref{sec:Bayes}, we set the Bayesian probabilistic scene by introducing the main actors, namely the space-variant (non-stationary) priors together with the likelihood pdf corresponding to the class of noise models considered in this review. The properties of the regularisers induced by the space-variant priors will be analysed from a modelling and optimisation viewpoint in Section \ref{sec:map}, while in Section \ref{sec:geom} we provide some useful insights on their geometric interpretation. Next, in Section \ref{sec:joint}, we formulate the final joint image and hyperparameter estimation models, where the different space-variant regularisers proposed  are combined with  general data fidelity terms and the suitable prior distributions on the model hypeparameters (i.e. \emph{hyperpriors}). %after deriving the general fidelity term and selecting a suitable hyperprior for the unknown parameters, the final joint hypermodels are formulated.
%, analysing their potentials from the modelling point of view. In Section \ref{sec:map}, we derive the space variant regularisers that will be analysed in terms of their regularisation effects. Next, in Section \ref{sec:geom}, we provide some useful insights on the local behaviour of the previously introduced regularisers. In Section \ref{sec:joint}, the derived regularisers will be embedded in the joint Bayesian hypermodel.
In Section \ref{sec:parameter_estimation}, we address the hyperparameter estimation problem by designing robust maximum likelihood-type strategies that will be tested on synthetic and natural examples. Then, in Section \ref{sec:admm} the numerical solution of the general variational model in the form \eqref{eq:GVM} upon the selected choices of regularisers and fidelity terms is addressed by means of an Alternating Direction Method of Multipliers (ADMM). In Section \ref{sec:restoration}, the effectiveness of the space-variant approach is finally assessed by applying the designed framework to the restoration of different synthetic and natural images. To conclude, in Section \ref{sec:challenges} we discuss some open questions and challenges representing natural extensions of this work. Finally, we report in Section~\ref{sec:conc} some final considerations and remarks.

\section{Notations and preliminaries}
\label{sec:prel}

%In this section, we set the notations and introduce useful definitions for the sequel.
%that will be used throughout the paper and we introduce useful notions on which we will largely rely.
We will use the notation $\R_+$ and $\R_{++} = \R_{+}\setminus \{0\}$ for the set of non-negative and positive real numbers, respectively, and denote by $\bm{0}_d$, $\bm{1}_d$, $\bm{\I}_d$ the $d$-dimensional vectors of all zeros and ones and the identity matrix of size $d\times d$, respectively. In the case of a matrix $\mathbf{M}\in\mathbb{R}^{d\times d},d>1$, we will denote by $|\mathbf{M}|$ the determinant of $\mathbf{M}$.

To indicate multi-variate random variables and their realisations we will use bold capital/lower-case letters, e.g. $\bm{X}$ and $\bm{x}$, and we denote by $\mathrm{P}_{\bm{X}}$, $\mathbb{P}_{\bm{X}}$, $\bm{\eta}_{\bm{X}}=\mathbb{E}(\bm{X})$, $\bm{\Sigma}_{\bm{X}}$ the probability mass function, pdf, mean and covariance matrix of the random variable $\bm{X}$, respectively. We will omit the subscript $\bm{X}$ if not necessary.
The characteristic and the indicator function of a set $\bm{S}$ are defined as
\begin{equation}
\chi_{\bm{S}}(\bm{x}) \;{:=}\; \left\{\begin{array}{ll}
1&\text{if  }\bm{x}\in \bm{S} \\
0&\text{otherwise} 
\end{array} \right.  \,,\quad \iota_{\bm{S}}(\bm{x})\;{:=}\;-\ln\chi_{\bm{S}}(\bm{x}) = \left\{\begin{array}{ll}
0&\text{ if  }\bm{x}\in \bm{S} \\
+\infty&\text{ otherwise} 
\end{array} \right.\,,
\end{equation}
respectively.
%As we poined out in Sec.~\ref{sec:intro}, this review will start by deriving most of the different space-variant TV-based regularisers proposed in literature in probabilistic terms. 
Moreover, we denote by $\Gamma$ the Gamma function, which is defined as follows:
\begin{definition}[Gamma and incomplete Gamma functions]
	The lower and upper incomplete Gamma functions, $\underline{\Gamma}$ and $\overline{\Gamma}$ respectively, are defined by
	\begin{equation}
	\underline{\Gamma}(x,y)
	\,\;{=}\: 
	\int_0^y t^{x-1} e^{-t} d t \, , 
	\quad\;\:
	\overline{\Gamma}(x,y)
	\,\;{=}\: 
	\int_y^{+\infty} t^{x-1} e^{-t} d t \, ,
	\quad\;\:
	(x,y) \;{\in}\;\, \R_{++} \,{\times}\; \R_+ \, .
	\label{eq:Ginc}
	\end{equation}
	The (complete) Gamma function $\Gamma$ is
	\begin{equation}
	\Gamma(x) 
	\,\;{=}\;\,
	\lim_{y \to +\infty} \, \underline{\Gamma}(x,y)
	\,\;{=}\;\,
	\overline{\Gamma}(x,0)
	\,\;{=}\: 
	\int_0^{+\infty} t^{x-1} e^{-t} d t \, ,
	\quad\;\:
	x \;{\in}\; \R_{++} \, .
	\label{eq:Gamma}
	\end{equation}
\end{definition}

Now, we recall the definitions of few well-known distributions to which we are often referring throughout the discussion and that will be mainly employed in the modelling of the regularisation terms reviewed here.

\begin{definition}[Univariate Laplacian distribution]
	\label{def:Ld}
	A scalar random variable $X$ is Laplacian-distributed with mean $\eta\in\R$ and scale parameter $\gamma\in\R_{++}$, denoted by $X\sim \mathrm{L}(\eta,\gamma)$, if its pdf has the form
	\begin{equation}
	\label{eq:Lpdf}
	\tag{Ld}
	\mathbb{P}(x|\eta,\gamma) 
	\;{=}\;
	\frac{\gamma}{2}
	\exp\left(-\gamma \, |x-\eta|\,\right) \, ,
	\quad x \in \R \, .
	\end{equation}
\end{definition}

\begin{definition}[Univariate Generalised Gaussian distribution]
	\label{def:GGd}
	A scalar random variable $X$ is generalised Gaussian-distributed with mean $\eta \in \R$, scale parameter $\gamma\in\R_{++}$ and shape parameter $s \in \R_{++}$, denoted by $X \sim \mathrm{GG}(\eta,\gamma,s)$, if its pdf has the form
	\begin{equation}
	\label{eq:GGpdf}
	\tag{GGd}
	\mathbb{P}\left(x|\eta,\gamma,s\right) 
	\,\;{=}\;\, 
	\frac{\gamma}{2} \frac{s}{\Gamma(1/s) }
	\,
	\exp\left( - \gamma^s\left|x-\eta \right|^s \right) \!, 
	\quad x \;{\in}\; \R ,
	\end{equation}
	%
	%with $\Gamma$ denoting the Gamma function.
	%and $\phi:\R_{++}\to\R$ defined as
	%	\begin{equation}
	%	   \phi(s) \,\;{:=}\;\, 
	%	\sqrt{
	%		\Gamma(1/s) \,/\, \Gamma(3/s)
	%\frac{\Gamma(1/x)}{\Gamma(3/x)}
	%	}\,.
	%	\end{equation}
	with $\Gamma(\cdot)$ denoting the Gamma function defined in \eqref{eq:Gamma}.
	In particular, for any fixed $\eta \in \R$, $\gamma \in \R_{++}$, the pdf in \eqref{eq:GGpdf} converges pointwise to a uniform distribution as $s \to +\infty$, namely
	\begin{equation}
	\lim_{s \to +\infty}
	\mathbb{P}(x|\eta,\gamma,s) \,\;{=}\;\, 
	\frac{\gamma}{2}\,
	\chi_{\,[0, 1/\gamma]}
	\left( \left| x - \eta \right| \right).
	\label{eq:GGpdflim}
	\end{equation}
	Finally, the standard deviation $\sigma\in\R_+$ of the $\mathrm{GG}$ pdf in \eqref{eq:GGpdf} can be written in terms of the scale parameter $\gamma$ as follows
	\begin{equation}
	\label{eq:GNstd}
	\sigma = (1/\gamma) \sqrt{\,\Gamma(3/s)\,/\,\Gamma(1/s)}\,.
	\end{equation}
\end{definition}

%As a special case of the GN pdf when $s=1$ and the random variable $X$ admits only positive realizations, we consider the following definition:

%Definition \ref{def:GGd} can be extended to the case of bivariate random variables as follows:

The following definition extends the GG distribution to the bivariate case.
\begin{definition}[Bivariate Generalised Gaussian distribution]
	\label{def:BGGD}
	A bivariate random variable $\bm{X}$ is generalised Gaussian-distributed with mean $\bm{\eta}\in\R^2$, symmetric positive definite covariance matrix $\bm{\Sigma}\in\R^{2\times 2}$ and shape parameter $s \in \R_{++}$, denoted by $\bm{X}\sim \mathrm{BGG}(\bm{\eta},\bm{\Sigma},s)$, if its pdf has the form
	\begin{equation}
	\label{eq:BGGpdf}
	\tag{BGGd}
	\mathbb{P}(\bm{x}|\bm{\eta},\bm{\Sigma},s) 
	\;{=}\;
	\frac{1}{2 \pi |\bm{\Sigma}|^{1/2}} \, 
	\frac{s}
	{\Gamma(2/s) 2^{2/s}}
	\: 
	\exp\left(-\frac{1}{2}
	\left(
	(\bm{x}-\bm{\mu})^{T}
	\bm{\Sigma}^{-1}
	(\bm{x}-\bm{\mu})\right)^{\textstyle{\frac{s}{2}}}\right), 
	%\quad \bm{x} \in \R^2 ,
	\end{equation}
\end{definition}
We now provide definitions for the Laplace distribution \eqref{eq:Lpdf} and the Generalised Gaussian one \eqref{eq:GGpdf} when the scalar random variable $X$ is known to be non-negative.
\begin{definition}[Univariate Half Laplacian distribution]
	\label{def:hLd}
	A scalar random variable $X$ is Half Laplacian-distributed with scale parameter $\gamma \in \R_{++}$, denoted by $X \sim \mathrm{hL}(\gamma)$, if $X = |Y|$, with $Y \sim \mathrm{L}(0,\gamma)$. The pdf of $X$ takes the form
	\begin{equation}
	\label{eq:hLpdf}
	\tag{hLd}
	\mathbb{P}(x|\gamma) 
	\;{=}\;
	\left\{ 
	\begin{array}{lc}
	\displaystyle{
		\gamma \,
		\exp\left(-\gamma \, |x|\,\right)}&  
	\text{if} \;\; x \in \R_+ \, ,\\
	0&\text{otherwise}. 
	\end{array}
	\right.
	\end{equation}
\end{definition}
\begin{definition}[Univariate Half Generalised Gaussian distribution]
	\label{def:hGGd}
	A scalar random variable $X$ is Half Generalised Gaussian-distributed with scale parameter \mbox{$\gamma \in \R_{++}$} and shape parameter $s \in \R_{++}$, denoted by $X \sim \mathrm{hGG}(\gamma,s)$, if $X = |Y|$, with $Y \sim \mathrm{GG}(0,\gamma,s)$. The pdf of $X$ takes the form
	\begin{equation}
	\label{eq:hGGpdf}
	\tag{hGGd}
	\mathbb{P}(x|\gamma,s) 
	\;{=}\;
	\left\{ 
	\begin{array}{ll}
	\displaystyle{
		\gamma \, \frac{s}{\Gamma(1/s) }
		\,
		\exp\left( - \gamma^s\left|x\right|^s \right)}&  
	\text{if} \;\; x \in \R_+ \, ,\\
	0&\text{otherwise}. 
	\end{array}
	\right.
	\end{equation}
\end{definition}
We also give the definition of Gamma distribution which will be used in the next sub-section and in the Appendix.
\begin{definition}[Univariate Gamma distribution] 
	A scalar random variable $X$ is Gamma-distributed with scale parameter $\nu\in\R_{++}$ and shape parameter $z\in\R_{++}$, denoted by $X \sim \mathrm{Gamma}(\nu,z)$,  if its pdf has the form
	\begin{equation}
	\mathbb{P}(x|\nu,z) \,\;{=}\;\, 
	\left\{
	\begin{array}{ll}
	\frac{1}{\nu^z \Gamma(z)}
	\, x^{z-1} 
	\exp\left( -\frac{x}{\nu}\right) 
	& \mathrm{for}\;\;
	x \;{\in}\; \R_{++}
	\\
	0 & \mathrm{otherwise}
	\end{array}
	\right.
	\,.
	\label{eq:Gamma_pdf}
	\end{equation}
\end{definition}

Finally, we recall the definition of proximal opearator to which we will extensively refer in Sections \ref{sec:map}-\ref{sec:admm}:

\begin{definition}[proximal operator]
	\label{def:proxx}
	Let $f: \R^N \to \R$ be a proper, lower semi-continuous and possibly non-convex function and let $\beta \in \R_{++}$. The proximal operator of $f$ with proximity parameter $\beta$ is the set-valued function $\mathrm{prox}_f^{\beta}: \R^N \rightrightarrows \R^N$ defined for any $\bm{w}\in\R^N$ by
	\begin{equation}  
	\label{eq:proxgen}
	\mathrm{prox}_f^{\beta}(\bm{w}) \;{:=}\; \argmin{\bm{x} \in \R^N} 
	\left\{
	f(\bm{x}) + \frac{\beta}{2} \, \left\| \bm{x} - \bm{w} \right\|_2^2 
	\right\} \subset \mathbb{R}^N.
	\end{equation}
\end{definition}
Note that if $f$ in the Definition above is convex, then the minimisation problem in \eqref{eq:proxgen} is strongly convex hence it admits a unique minimiser. In this case, $\mathrm{prox}_f^{\beta}(\cdot)$ is a well-defined function from $\mathbb{R}^N$ to itself which coincides with the well-studied proximal operator frequently encountered in convex optimisation contexts (see, e.g., \cite{CombPes11}).

\subsection{Generalised Discrepancy Principle}\label{sec:gdp}

%  We detail in this Section the general automatic parameter selection strategy which will be used in the following  to compute the regularisation parameter $\mu$ appearing in the general variational model \eqref{eq:GVM}. Note that in the case of additive white Gaussian noise, this procedure coincides with the standard discrepancy principle strategy studied, e.g. in \cite{Morozov,engl2000regularization} in the context of regularised inverse problems. 

Assuming that the noise degradation operator $\mathcal{N}$ in \eqref{eq:lin_model} models the action of an additive, zero-mean, independent and identically distributed (i.i.d.) generalised Gaussian (in short, AIGG) noise, we have that \eqref{eq:lin_model} can be rewritten as
\begin{equation}
\label{eq:linmod_2}
\bm{b} = \bm{\A u} +\bm{ e}\,,\;\, \text{with  } e_j \text{ realisation of }E_j\sim \mathrm{GG}(0,\omega,q)\,,\;\: j = 1,\ldots,M,
\end{equation}
where, based on Definition \ref{def:GGd}, $\,\bm{e} \in \R^M$ is the vector of realisations of the $M$-variate random variable $\bm{E}$ whose components are i.i.d. GG random variables with shape parameter $q \in \R_{++}$
and scale parameter $\omega \in \R_{++}$, the latter encoding information on the noise standard deviation according to \eqref{eq:GNstd}.

\begin{comment}

Under these settings, a popular approach for automatically setting $\mu$ in \eqref{eq:GVM} goes under the name of Discrepancy Principle (DP) \cite{Morozov,engl2000regularization}.

%, originally introduced for the automatic selection of $\mu$ in \eqref{eq:GVM} under the action of a Gaussian degradation, can be generalised so as to deal with noise models following a GG distribution.

%Before recalling the DP formulation, we remark that the solution $\bm{u}^*(\mu)$ of the general variational model in \eqref{eq:GVM} can be regarded as depending on the regularisation parameter $\mu$ selected, so that it can be equivalently denoted by $\bm{u}^*(\mu)$. We have:

%\textcolor{red}{Definire prima observation model}

\begin{definition}[Discrepancy Principle]
Let $\bm{u}^*(\mu)$ be the solution of \eqref{eq:GVM} under the adoption of the observation model in \eqref{eq:linmod_2}. We denote by $\bm{r}^*(\mu)=\bm{\A u}^*(\mu)-\bm{b}$ the residual image corresponding to $\mu$. The DP can be formulated as follows:
\begin{equation}
\text{Select  }\mu\;{=}\;\mu^*\text{  such that  }\|\bm{r}^*(\mu)\|_2 = \delta \;{:=}\; \tau\mathbb{E}(\|\bm{E}\|_2)\;{=}\;\tau\sqrt{N}\sigma,
\end{equation}
where $\sigma$ is computed as in \eqref{eq:GNstd} and $\tau$, to which we refer as the \emph{discrepancy parameter} is slightly greater or equal than $1$.
\end{definition}

Notice that the discrepancy parameter has the effect of avoid under-estimation of the noise.

\end{comment}
%

Before detailing a Generalised Discrepancy Principle (in short, GDP) useful to define an automatic selection strategy for the regularisation parameter $\mu$ in \eqref{eq:GVM} under the modelling assumption \eqref{eq:linmod_2}, we report the following result, whose proof is based on classical probability arguments and given for completeness in the Appendix.
\begin{proposition}
	\label{prop:Lq_pdf_pap}
	If $\,X_i \sim \mathrm{GG}(0,\omega,q)$, $i = 1,\ldots,M$, with $\omega,q\in\R_{++}$, are independent random variables, then we have
	\begin{equation}
	Y 
	\,\;{=}\;\, 
	\left\| \left(X_1,\ldots,X_M\right) \right\|_q^q
	\,\;{=}\;\, 
	\sum_{i=1}^M \left|X_i\right|^q
	\,\;{\sim}\;\, 
	\mathrm{Gamma}(\nu,z), \quad \nu = \frac{1}{\omega^q}\,,\;z = \frac{M}{q}\,.
	\label{eq:Lq_pdf}
	\end{equation}
	The random variable $Y$ has mean $\eta_Y$ and variance $\sigma_Y^2$ whose expressions are given by
	\begin{equation}
	\eta_Y \,\;{=}\;\, \frac{M}{q}\, \frac{1}{\omega^q} \, , \quad\quad
	\sigma_Y^2 \,\;{=}\;\, \frac{M}{q} \, \frac{1}{\omega^{2q}} \, .
	\label{eq:Lq_moms}
	\end{equation}
\end{proposition}
Thanks to Proposition~\ref{prop:Lq_pdf_pap}, we can now provide the following GDP.

%We are now ready to introduce the Generalised Discrepancy Principle (GDP) which represents a criterion for the automatic selection of $\mu$ in \eqref{eq:GVM} when the additive noise model is known to be GG.
%when the noise model in \eqref{eq:linmod_2} is GG with a generic shape parameter $q\in\R_{++}$.

%, whose proof is reported in the appendix, allows to generalise the popular discrepancy principle, which in principle stands for Gaussian noise, to more general noise distributions belonging to the GG family.

\begin{definition}[Generalised Discrepancy Principle]\label{def:gdp}
	Let $\bm{u}^*(\mu)$ be the parameter-dependent solution of \eqref{eq:GVM} for model \eqref{eq:linmod_2}; denoting by $\bm{r}^*(\mu)=\bm{\A u}^*(\mu)-\bm{b}$ the associated residual image, we have that the Generalised Discrepancy Principle can be formulated as:
	%Let $\bm{r}^*(\mu)=\bm{\A u}^*-\bm{b}\in\R^M$ denote the residual corresponding the solution $\bm{u}^*(\mu)$ of the general variational model \eqref{eq:GVM}. The Generalised Discrepancy Principle (GDP) can be expressed as follows:
	%
	\begin{equation}\label{eq:GDP1}
	\text{Select  }\:\mu\,\;{=}\;\,\mu^*\text{  such that  }\,\,\|\bm{r}^*(\mu^*)\|_q \,\;{=}\;\, \delta_q  \, ,
	\end{equation}
	where
	\begin{equation}\label{eq:GDP2}
	\delta_q\,\;{:=}\;\,\tau \, \mathbb{E}(\|\bm{E}\|_q)\,\;{=}\;\left\{\begin{array}{rl}
	\displaystyle{  \tau\left(M/q\right)^{1/q}(1/\omega)}  &\text{  if  }\quad q<+\infty  \\
	\displaystyle{  \tau(1/\omega)}&\text{  if  }\quad q=+\infty   
	\end{array}  \right.\,,
	\end{equation}
	%\begin{align}\label{eq:GDP2a}
	%\delta_q\;{=}\;\tau \mathbb{E}(\|\bm{E}\|_q)\;{=}\;&\tau\left(N/q\right)^{1/q}\frac{1}{\beta}=\tau \left(N/q\right)^{1/q}\sigma\sqrt{{\Gamma(1/q)}/{\Gamma(3/q)}}\,,\; q\;{<}\;+\infty\,,\\
	%\label{eq:GDP2b}
	%\delta_q\;{=}\;  %\tau\mathbb{E}(\|\bm{E}\|_q)\;{=}\;&\tau\s%qrt{3}\sigma\,,\; q\;{=}\;+\infty\,,\\
	%\end{align}
	with $\tau\approx 1$.
\end{definition}

Note that in the case $q=2$, the strategy reduces to the classical discrepancy principle strategy detailed, e.g., in \cite{Morozov,engl2000regularization} for i.i.d. Gaussian noise.

\section{A flexible Bayesian framework}
\label{sec:Bayes}

In this section, we recall the general Bayesian framework outlined in Section \ref{sec:statistical} and adapt it to our purposes and considerations. We start specifying the different noise degradation models $\mathcal{N}$ considered in \eqref{eq:lin_model} and define suitable likelihood pdfs accordingly. Next, we specify the flexible space-variant priors focus of this work by defining a class of increasingly general distributions. 
%We remark that a significant difference with the strategies to which we referred in Sec.~\ref{sec:statistical} is that in what follows we will make explicit the dependence of the prior and of the likelihood pdfs on the parameters. 
Likelihoods and priors are then combined by means of a suitable MAP estimate suited to describe the case where the prior hyperparameters are unknown.

%A widely adopted approach for the design of variational models aimed at solving the generic image reconstruction inverse problem consists in recasting the original model \eqref{eq:lin_model} in probabilistic terms. To this purpose, the unknown quantities involved in \eqref{eq:lin_model} are modelled as random variables. The goal is to find the distribution of $\bm{u}$, combining the information encoded in the observed data $\bm{b}$ with \emph{a priori} beliefs or information that we may have on $\bm{u}$. In terms of probability density functions, we seek for the analytic expression of the posterior pdf, once the likelihood and the prior pdfs are given.

\subsection{Likelihoods}
\label{subsek:like}

$\,$ In the following, $\,$the likelihood pdf will be indicated by  \mbox{$\mathbb{P}(\bm{b}\mid \bm{\A u}, \bm{\Phi})$,} where, in addition to $\bm{\A u}\in\mathbb{R}^M$ the dependence on the (generally unknown) likelihood hyperparameter vector $\bm{\Phi}\in\R^{k}$ involved in the analytic expression of the pdf is here explicitly taken into account.
%Note that, in general, the entries of the vector $\bm{\Phi}$ are unknown.

In order to benefit from the automatic parameter selection strategy provided by the GDP detailed in Definition \ref{def:gdp}, we will focus our attention to the class of AIGGN corresponding to model \eqref{eq:linmod_2}.
%

\begin{comment}
In principle, our discussion can be easily declined so as to address any noise degradation model; therefore, the design of the likelihood pdf is not particularly relevant in this context.
However, in what follows we choose to limit our formulation to the class of white additive noises; more specifically, the noise degradation operator $\mathcal{N}$ is set so as to model the action of an additive white generalised Gaussian noise (AWGGN). The original linear inverse problem \eqref{eq:lin_model} thus turns into:
\begin{equation}
\label{eq:lin_model_lin}
\bm{b} = \bm{\A u} +\bm{ e}\,,\; \text{with  } e_j \text{ realisation of }E_j\sim \mathrm{GG}(0,\omega,q)\,,\;j = 1,\ldots,M\,,
\end{equation}
where $\omega\in\R_{++}$ denotes the scale parameter of the GG distribution and is related to the noise standard deviation $\sigma$ via \eqref{eq:GNstd}, while $q$ represents the noise shape parameter, which will be assumed to be greater or equal than 1, i.e. $q\in[1,+\infty]$. 

\end{comment}
%
Note that, although not exhaustive, this  class is very general as it contains some commonly-used noise models, such as, e.g., the additive i.i.d. Laplacian (AIL) noise ($q=1$), the additive i.i.d. Gaussian (AIG) noise ($q=2$) and the additive i.i.d. uniform (AIU) noise ($q=+\infty$).

Due to the independence assumption for the univariate random variables $E_j$ in \eqref{eq:linmod_2}, the $M$-variate likelihood pdf can be written as the product of $M$ identical univariate GG pdfs (see Definition \ref{def:GGd}). When $q<+\infty$, it thus takes the form:
\begin{align}
\nonumber   \mathbb{P}(\bm{b}\mid \bm{\A u},\bm{\Phi}) =& \prod_{j=1}^{M} \left(\frac{\omega}{2}\frac{q}{\Gamma(1/q)}\exp\left(-\omega^q|(\bm{\A u})_j-b_j|^q\right)\right)\\
\label{eq:likeli}
=&\left(\frac{\omega}{2}\frac{q}{\Gamma(1/q)}\right)^M \exp\left(-\omega^q\left\|\bm{\A u}-\bm{b}\right\|_q^q\right),\,\text{with }\;\bm{\Phi} {=}(\omega,q)\in\R_{++}^2\,,
\end{align}
while for $q=+\infty$ it reads
\begin{align}
\nonumber   \mathbb{P}(\bm{b}\mid \bm{\A u},\bm{\Phi}) =& \prod_{j=1}^{M} \left(\frac{\omega}{2}\right)\chi_{[0,1/\omega]}\left(|(\bm{\A u})_j-\bm{b}_j|\right)\\
\label{eq:likeli_inf}
=&\left(\frac{\omega}{2}\right)^M\chi_{[0,1/\omega]}\left(\|\bm{\A u}-\bm{b}\|_{\infty}\right),\,\text{with }\;\bm{\Phi} {=}(\omega,q)\in\R_{++}\times \{+\infty\}\,.
\end{align}
In our settings, we will assume that both parameters $\bm{\Phi} {=}(\omega,q)$ are known. As a consequence, from now on, the dependence on $\bm{\Phi}$ in the expression of the likelihood pdf will be omitted. 

\subsection{Priors}
\label{subsec:prior}

Recalling the statistical modelling introduced in Section \ref{sec:statistical} and in particular the MRF structure in \eqref{eq:mrfpr}, we proceed similarly as in the previous section and make explicit the dependence of the prior pdf on the vector of prior hyperparameters $\bm{\Theta}$ involved in its expression, which here are assumed to be unknown. The corresponding Gibbs' prior reads
%Notice that, unlike the Gibbs' prior introduced in \eqref{eq:mrfpr}, we are now making explicit the dependence of the prior on the parameters.
\begin{equation}
\label{eq:mrfpr2}
\mathbb{P}(\bm{u}\mid \bm{\Theta}) = \frac{1}{Z(\bm{\Theta})} \exp \bigg(- \sum_{i=1}^{n}V_{\mathcal{C}_{i}}(\bm{u};\bm{\Theta})\bigg)\,,
\end{equation}
where $Z(\bm{\Theta}) \in \R_{++}$ is a normalisation constant depending on the unknown parameters $\bm{\Theta}$ while $V_{\mathcal{C}_{i}}$ is the Gibbs' potential on the $i$-th clique $\mathcal{C}_i$.

Recalling the general Markovian property \eqref{eq:mrf} and thinking of the description of the image in terms of its local gradients $(\bm{\D u})_i$  discretised by standard first-order forward finite differences, we have  that \eqref{eq:mrf} turns into
\begin{equation}
\label{eq:gradmrf}
\mathbb{P}(\,U_i \,{=}\: u_i\mid U_j \,{=}\: u_j\,,\;j\neq i\,) \;{=}\;
\mathbb{P}(\,U_i \,{=}\: u_i\mid U_{i,\mathrm{right}} \,{=}\: u_{i,\mathrm{right}}\,,\, U_{i,\mathrm{down}} \,{=}\: u_{i,\mathrm{down}}).
\end{equation}
For better illustration, we show the corresponding configuration of the generic clique in Figure \ref{fig:clique}.  Condition \eqref{eq:gradmrf} states that the potential function $V_{\mathcal{C}_i}$ is defined over a discrete set of cardinality $3$, namely $\{u_i,u_{i,\mathrm{right}},u_{i,\mathrm{down}}\}$, which are indeed the values involved in the computation of the discrete gradient at pixel $i$. 
\begin{SCfigure*}[][!t]  
	\centering
	\includegraphics[width=2in]{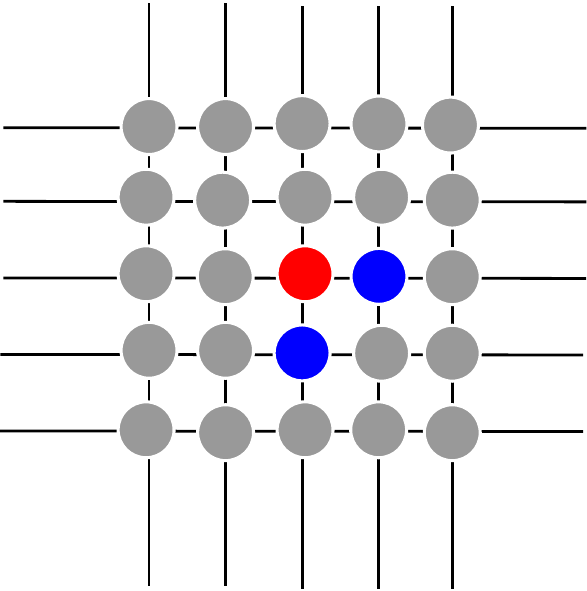}
	\caption{Pixels represented as atoms in a lattice. The coloured ones belong to the clique related to red atom. In particular, the blue atoms are involved in the computation of the finite difference  gradient evaluated at the red atom.}
	\label{fig:clique}
\end{SCfigure*}
%A very popular choice is to assume %The easiest choice to model \emph{a-priori} gradient information consists in choosing one single prior distribution looking only at information concentrated at each pixel $i=1,\ldots, N$. %One choice is, for instance, assuming 

Introducing the function $\bm{z}:\R^N\to \R_{+}^N$ defined by
\begin{equation}\label{eq:zfun}
\bm{z}(\bm{u}) = \left[z_1(\bm{u}),\ldots,z_N(\bm{u})\right]^T , \quad 
z_i(\bm{u}) = \|(\bm{\D u})_i\|_2\,,
\end{equation}
and assuming that each image gradient magnitude $x_i:=z_i(\bm{u}) = \| (\bm{\D u})_i \|_2 \in\R_{+}$ is the realisation of the same univariate half-Laplacian (or exponential) distribution \eqref{eq:hLpdf} with scale parameter $\alpha\in\R_{++}$ % (see Figure \ref{fig:hld} for different values of $\alpha$) 
%, which reads
%
%\begin{equation}  
%\label{eq:hLD}
%\mathbb{P}(x_i\mid\alpha) = \begin{cases}  
%\alpha \exp (-\alpha x_i),\quad & x_i\geq 0\\
%0, \quad & x_i<0
%\end{cases},
%%\quad i = 1,\ldots,N, 
%\quad \mathrm{with}\;\,
%\alpha \in \R_{++}\,,
%\tag{hLD}
%\end{equation}
%
and that the magnitudes at different pixels are independent, we have that $\mathbb{P}(\bm{u}\mid \bm{\Theta})$ takes the form of the Gibbs' TV prior
\begin{align}
\mathbb{P}(\bm{u}\mid \bm{\Theta})
\,\;{=}\;\,& c(\bm{\Theta}) \,\mathbb{P}(\bm{z}(\bm{u})\mid \bm{\Theta})
\,\;{=}\;\,
c(\bm{\Theta}) \prod_{i=1}^N 
\mathbb{P}(z_i(\bm{u}) \mid \bm{\Theta})\\
\,\;{=}\;\,&c(\bm{\Theta})\,
\prod_{i=1}^{N}\left(
\alpha\exp\left(-\alpha\left\|(\bm{\D u})_i\right\|_2\right)\right) 
\\
\,\;{=}\;\,&c(\bm{\Theta})\,
\alpha^N \exp\left(-\alpha\sum_{i=1}^{N}\left\|(\bm{\D}\bm{u})_i\right\|_2\right),
\; \text{with }\;\, \bm{\Theta} = \alpha \in \R_{++} \,,
\label{eq:TV_prior}
\end{align}
where the scalar normalisation function $c:\R_{++}\to\R_{++}$, depending only on $\bm{\Theta}$, reads
\begin{equation}
c(\bm{\Theta}) = \frac{1}{\displaystyle{ \int_{\bm{u}\in\R^N} \alpha^N \exp\left(-\alpha\sum_{i=1}^{N}\left\|(\bm{\D}\bm{u})_i\right\|_2\right)d\bm{u}   }}\,.
\end{equation}
Notice that the presence of $c(\bm{\Theta})$ guarantees that the prior pdf in \eqref{eq:TV_prior} sums up to one when considering the space of all possible configurations.

% A typical drawback that the TV prior brings along is the lack of flexibility due to the presence of a global scale parameter in the pdf \eqref{eq:TV_prior}. Indeed, this makes unfeasible to model in a different manner regions in the image presenting different gradient structures.

A way to improve upon the intrinsic rigidity of \eqref{eq:TV_prior}, due to the dependence on the single scale parameter $\alpha\in\mathbb{R}_{++}$, consists in letting it vary at any pixel, so as to maintain the same prior hypothesis on the image gradient magnitudes $x_i$, while enriching it with further flexibility depending on the local scale $\alpha_i$. The corresponding space-variant hLd probability density thus reads in this case
\begin{equation}  \label{eq:hLD-sv}
\mathbb{P}(x_i\mid\alpha_i) = \begin{cases}  
\alpha_i \exp (-\alpha_i x_i),\quad & x_i\geq 0\\
0, \quad & x_i<0
\end{cases},
\quad i = 1,\ldots,N, 
\quad\, 
\alpha_i \in \R_{++} \,,
\tag{hLd-sv}
\end{equation}
and yields the following non-stationary prior pdf on $\bm{u}$
\begin{align}
\mathbb{P}(\bm{u}\mid \bm{\Theta})
\,\;{=}\;\,&c(\bm{\Theta})\,\mathbb{P}(\bm{z}(\bm{u})\mid\bm{\Theta})\\
\;{=}\;&c(\bm{\Theta})\,\prod_{i=1}^{N}\left(
\alpha_i\exp\left(-\alpha_i\left\|(\bm{\D u})_i\right\|_2\right)\right) 
\nonumber \\
\,\;{=}\;\,&c(\bm{\Theta})\,
\left(\prod_{i=1}^{N}\alpha_i\right)
\exp\left(-\sum_{i=1}^{N} \alpha_i \left\|(\bm{\D}\bm{u})_i\right\|_2\right),
\quad \text{with }\;\, 
\bm{\Theta} = \bm{\alpha} \in\R_{++}^{N}\,,
\label{eq:WTV_prior}
\end{align}
with $z$ defined as in \eqref{eq:zfun} and the normalisation function $c(\bm{\Theta})$ defined by
\begin{equation}
c(\bm{\Theta}) = \frac{1}{\displaystyle{ \int_{\bm{u}\in\R^N} \left(\prod_{i=1}^{N}\alpha_i\right) \exp\left(-\sum_{i=1}^{N}\alpha_i\left\|(\bm{\D}\bm{u})_i\right\|_2\right)d\bm{u}   }}\,.
\end{equation}

\begin{figure}[!t]
	\centering
	\begin{subfigure}{0.45\textwidth}
		\centering
		\includegraphics[scale=0.2]{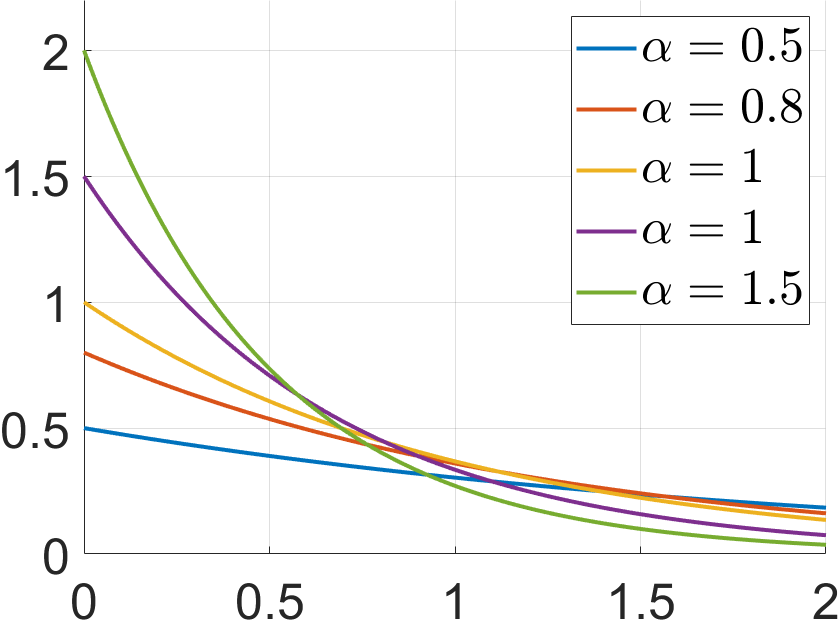}
		\caption{hLD}
		\label{fig:hld}
	\end{subfigure}
	\begin{subfigure}{0.45\textwidth}
		\centering
		\includegraphics[scale=0.2]{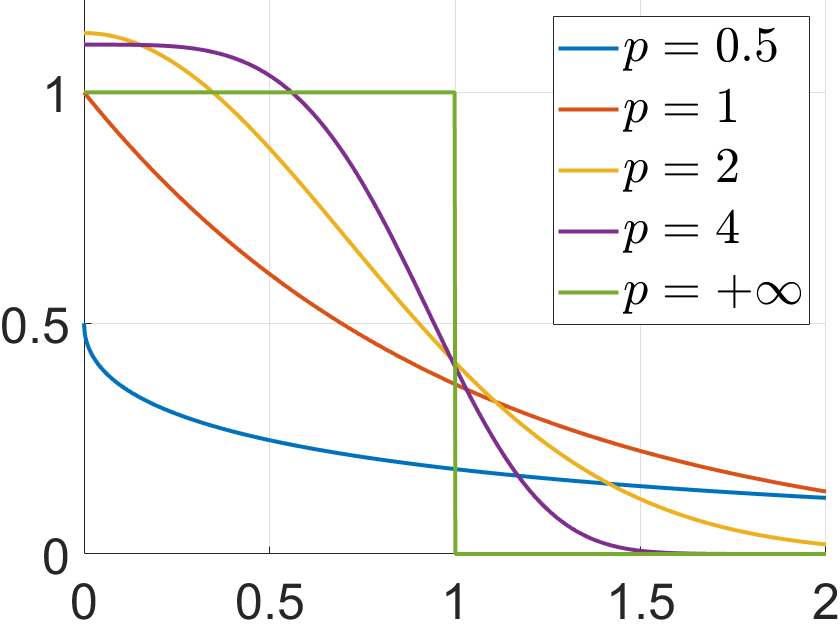}
		\caption{hGGD}
		\label{fig:hggd}
	\end{subfigure}\\
	\begin{subfigure}{0.32\textwidth}
		\centering
		\includegraphics[width=4.1cm]{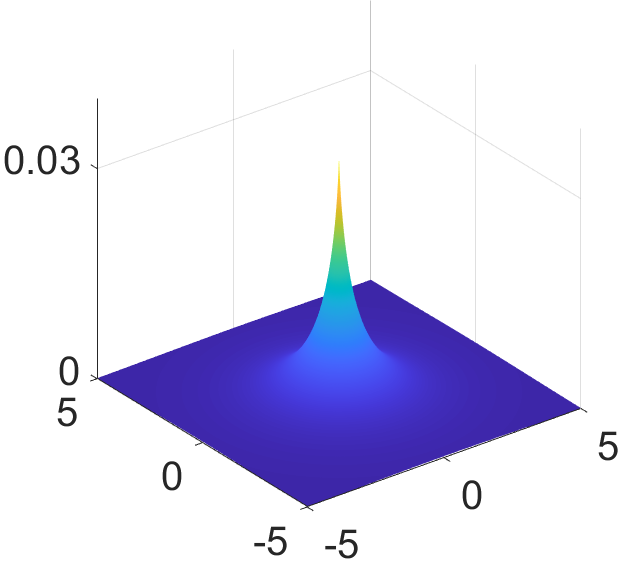}
		\caption{$(\alpha,p,\theta,a){=}(2,0.7,0,1)$}%\alpha{=}1,p{=}0.7,\theta{=}0,a{=}1$}
		\label{fig:bggd1}
	\end{subfigure}
	\begin{subfigure}{0.32\textwidth}
		\centering
		\includegraphics[width=4.1cm]{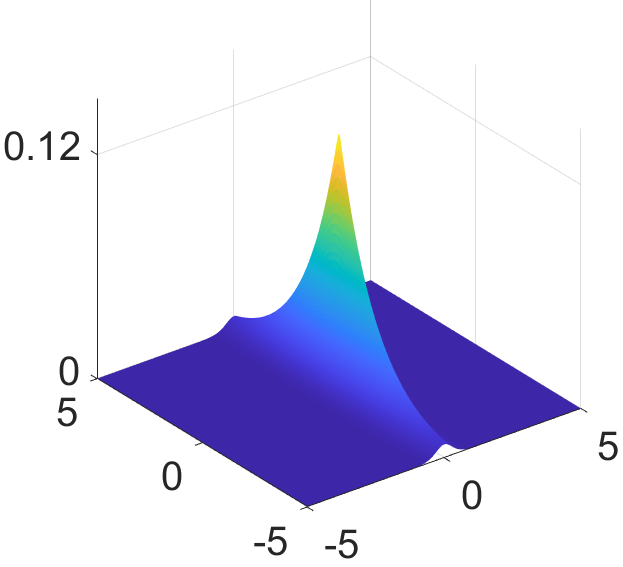}
		\caption{ $(\alpha,p,\theta,a){=}(6,1,0,0.1)$}%\alpha{=}6,p{=}1,\theta{=}0,a{=}0.1$}
		\label{fig:bggd2}
	\end{subfigure}
	\begin{subfigure}{0.32\textwidth}
		\centering
		\includegraphics[width=4.1cm]{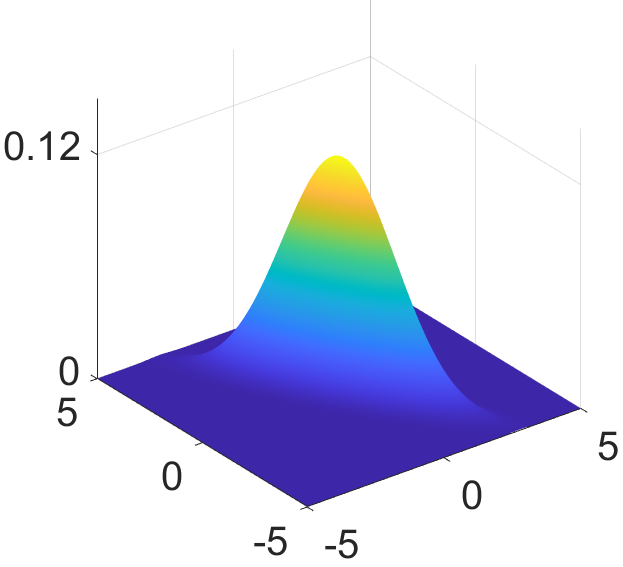}
		\caption{$(\alpha,p,\theta,a){=}(2,2,\pi/6,0.2)$}%\alpha{=}2,p{=}2,\theta{=}\pi/6,a{=}0.2$}
		\label{fig:bggd3}
	\end{subfigure}\\
	\begin{subfigure}{0.32\textwidth}
		\centering
		\includegraphics[width=4cm]{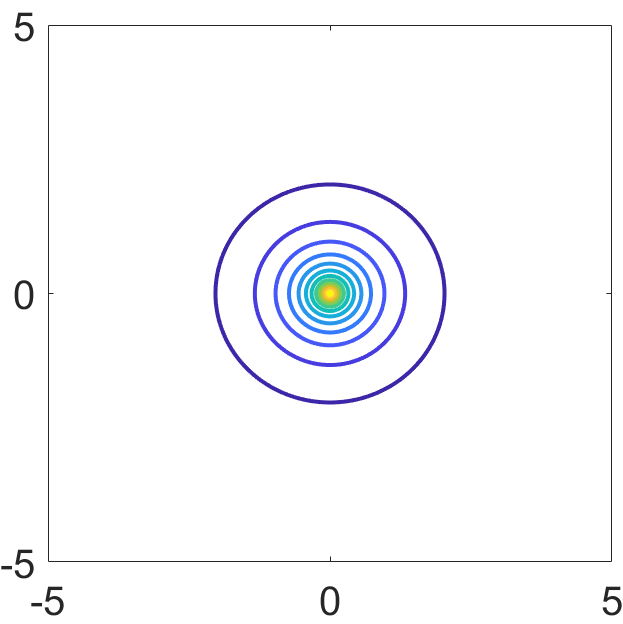}
		\caption{}
		\label{fig:bggd_cont1}
	\end{subfigure}
	\begin{subfigure}{0.32\textwidth}
		\centering
		\includegraphics[width=4cm]{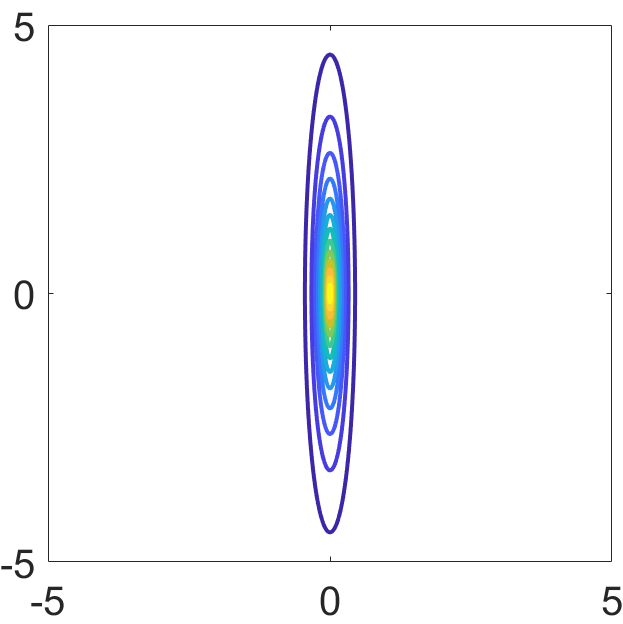}
		\caption{}
		\label{fig:bggd_cont2}
	\end{subfigure}
	\begin{subfigure}{0.32\textwidth}
		\centering
		\includegraphics[width=4cm]{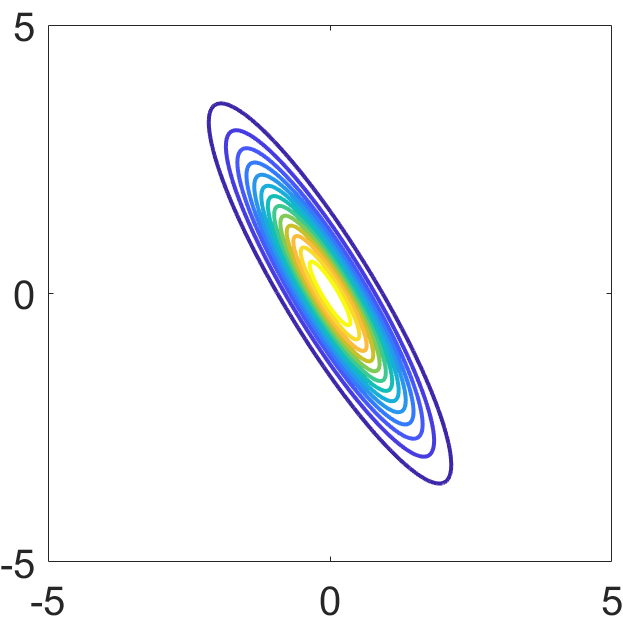}
		\caption{}
		\label{fig:bggd_cont3}
	\end{subfigure}
	\caption{\emph{First row}: half Laplacian distribution for different values of $\alpha\in\R_{++}$ (a), half Generalised Gaussian distribution for different values of $p\in\R_{++}$ and $\alpha=1$ (b). \emph{Second row}: bivariate Generalised Gaussian distribution for different values of $\alpha\in\R_{++}$, $p\in\R_{++}$, $\theta\in [-\pi/2,\pi/2)$, $a\in(0,1]$. \emph{Third row}: contour plots of the bivariate Generalised Gaussian pdfs displayed in the second row.}
	\label{fig:my_label_1}
\end{figure}

As shown in Figure \ref{fig:hld}, this choice allows for more flexibility in the description of local gradient contents; nonetheless, it has the major drawback of still limiting to the family of half-Laplacian distributions the choice of the local probability density function considered. 
%This can cause somehow some rigidity in the local functional shape of the prior. 

To overcome this, one can leave further freedom to the heavy- VS. light-tailed behaviour of the exponential distribution considered. This can be done in practice by allowing, along with a space-variant pdf scale $\alpha_i$, a different exponential behaviour depending on a ``sharpness" (shape) parameters $p_i$, still possibly varying at any $i=1,\ldots,N$. This choice corresponds to consider a space-variant half-Generalised Gaussian Distribution (hGGD-sv) (see \eqref{eq:GGpdf}), whose expression for  $i=1,\ldots,N$ reads:
\begin{equation}  \label{eq:hGGD-sv}
\mathbb{P}(x_i;\alpha_i, p_i) = \begin{cases}  
\frac{\alpha_i p_i}{\Gamma(1/p_i)} \exp (-(\alpha_i x_i)^p_i),\quad & x_i\geq 0\\
0 \quad & x_i<0
\end{cases},\quad\alpha_i,p_i\in\R_{++}\,. \tag{hGGd-sv}
\end{equation}
In Figure \ref{fig:hggd} we show the plot of the hGGd pdf for different values of the shape parameter $p$ while leaving the scale parameter $\alpha=1$ fixed. One can easily notice 
%how the introduction of a further parameter produces a gain in terms of flexibility. In fact, 
that the family of hGG distributions is particularly rich, ranging from hyper-Laplacian distributions for $p<1$ to uniform distributions for $p=+\infty$. 
The prior on $\bm{u}$ corresponding to the pdf in \eqref{eq:hGGD-sv} reads
\begin{align}
\!\mathbb{P}(\bm{u}\mid \bm{\Theta})
{=}& c(\bm{\Theta})\mathbb{P}(\bm{z}(\bm{u})\mid\bm{\Theta})\\
\!{=}& c(\bm{\Theta})\prod_{i=1}^{N}\left(
\frac{\alpha_i p_i}{\Gamma(1/p_i)}
\exp\left({-}\left(
\alpha_i\left\|(\bm{\D u})_i\right\|_2\right)^{p_i}\right)\right)
\nonumber \\
\!\,{=}\,&c(\bm{\Theta})\,
\left(\prod_{i=1}^{N}\frac{\alpha_i p_i}{\Gamma(1/p_i)}\right)  \exp\!\left({-}\sum_{i=1}^{N}\alpha_i^{p_i}\left\|(\bm{\D u})_i\right\|_2^{p_i}\right),
\;\text{ with } \bm{\Theta} {=} \left(\bm{\alpha},\bm{p}\right)\in\R_{++}^{N\times 2}\,,
\label{eq:WTVp_prior}
\end{align}
where $z$ is defined in \eqref{eq:zfun}, while $c(\bm{\Theta})$ now takes the form
\begin{equation}
c(\bm{\Theta}) = \frac{1}{\displaystyle{ \int_{\bm{u}\in\R^N} \left(\prod_{i=1}^{N}\frac{\alpha_i p_i}{\Gamma(1/p_i)}
		\right) \exp\left(-\sum_{i=1}^{N}\alpha_i^{p_i}\left\|(\bm{\D}\bm{u})_i\right\|_2^{p_i}\right)d\bm{u}   }}\,.
\end{equation}

\medskip

We stress that, despite their differences, the choices \eqref{eq:hLpdf}, \eqref{eq:hLD-sv} and \eqref{eq:hGGD-sv} correspond to the formulation of prior assumptions on the quantities $\| (\bm{\D u})_i \|_2$, for $i=1,\ldots,N$. Clearly, this is somehow a `rigid' choice since it does not exploit the two-dimensionality of local image gradients $(\bm{\D u})_i = ( (\bm{\D_h u})_i, (\bm{\D_v u})_i)$ and, possibly, the correlation between their horizontal and vertical components $(\bm{\D_h u})_i$ and $(\bm{\D_v u})_i$. To do so, a different, possibly space-variant, prior assumption imposing \emph{a-priori} information on the local image gradient can be made. Namely, for all $\bm{y}_i := ( (\bm{\D_h u})_i, (\bm{\D_v u})_i)$ and for $i=1,\ldots,N$ one can assume that $\bm{y}_i$ follows a Bivariate Generalised Gaussian distribution (BGGd, see \eqref{eq:BGGpdf}) which is space-variant with respect to shape and scale and takes locally the form:
\begin{equation}\label{eq:MGGDi}
\mathbb{P}(\bm{y}_i ;p_i,\bm{\Sigma}_i)
\;{=}\;\frac{1}{2 \pi |\bm{\Sigma}_i|^{1/2}} \, \frac{p_i}{\Gamma(2/p_i) \, 2^{\:\!2/p_i}}
\: 
\exp\left(-\frac{1}{2}(\bm{y}_i^{T}\bm{\Sigma}_i^{-1}\bm{y}_i)^{p_i/2}\right)\,.
%\frac{1}{2 \pi |\bm{\Sigma}_i|^{1/2}} \, \frac{p_i}{\Gamma(1/p_i) \, 2^{\:\!1/p_i}}
%\: 
%\text{exp}\left(-\frac{(\bm{y}_i^{T}\bm{\Sigma}_i^{-1}\bm{y}_i)^{\textstyle{\frac{p_i}{2}}}}{2}\right),
\tag{BGGd-sv}
\end{equation}
where, for every $i$, the  covariance matrix $\bm{\Sigma}_i \in \mathbb{R}^{2 \times 2}$ is symmetric positive definite with determinant $|\bm{\Sigma}_i|>0$. 
The associated non-stationary prior pdf on $\bm{u}$ can thus be written as
\begin{align}
\mathbb{P}(\bm{u}\mid\bm{\Theta}) 
\,\;{=}\;\,&c(\bm{\Theta})\,\mathbb{P}(\bm{z}(\bm{u})\mid\bm{\Theta})\\
\,\;{=}\;\,&c(\bm{\Theta})\,\prod_{i=1}^{N}\left(\frac{1}{2 \pi |\bm{\Sigma}_i|^{1/2}} \, \frac{p_i}{\Gamma(2/p_i) \, 2^{\:\!2/p_i}}
\exp\left(-\frac{1}{2}((\bm{\D u})_i^{T}\bm{\Sigma}_i^{-1}(\bm{\D u})_i)^{\textstyle{\frac{p_i}{2}}}\right)
\right)
\nonumber \\
\,\;{=}\;\,&c(\bm{\Theta})\,
\prod_{i=1}^{N}\left(\frac{1}{2 \pi |\bm{\Sigma}_i|^{1/2}} \, \frac{p_i}{\Gamma(2/p_i) \, 2^{\:\!2/p_i}}\right)
\exp\left(
-\frac{1}{2} \sum_{i=1}^{N}\left(
(\bm{\D u})_i^{T}\bm{\Sigma}_i^{-1}(\bm{\D u})_i
\right)^{\textstyle{\frac{p_i}{2}}}\right)\,,
\label{eq:bivprior}
\end{align}
where now $\bm{z}(\bm{u})=\bm{\D u}$, and the normalisation function $c$ is defined by
\begin{equation}
c(\bm{\Theta}) = \frac{1}{\displaystyle{ \int_{\bm{u}\in\R^N} \prod_{i=1}^{N}\left(\frac{1}{2 \pi |\bm{\Sigma}_i|^{1/2}} \, \frac{p_i}{\Gamma(2/p_i) \, 2^{\:\!2/p_i}}\right)
		\exp\left(
		-\frac{1}{2} \sum_{i=1}^{N}\left(
		(\bm{\D u})_i^{T}\bm{\Sigma}_i^{-1}(\bm{\D u})_i
		\right)^{\textstyle{\frac{p_i}{2}}}\right)d\bm{u}   }}\,.
\end{equation}

% Before detailing explicitly the vector of parameters $\bm{\Theta}$, at the aim of showing that the \eqref{eq:MGGDi} can be thought as a generalisation of the previous discussed distributions in \eqref{eq:hLD-sv} and \eqref{eq:hGGD-sv} - hence, the prior \eqref{eq:bivprior} can be thought as a generalisation of priors \eqref{eq:WTV_prior} and \eqref{eq:WTVp_prior} - 

For a better interpretation of such choice, we now perform some simple manipulations to the generic $i$-th term of the sum appearing in \eqref{eq:MGGDi} so as to highlight how the information related to the local image scale and orientations are all encoded in the local covariance matrices $\bm{\Sigma_i}$.  To this purpose, we consider the following eigenvalue decomposition 
\begin{equation}
\bm{\Sigma_i}
\;{=}\;\bm{
	\V_i}^T\bm{ \E_i \V_i},
\;\;
\bm{\E_i} \;{=}\: \begin{pmatrix}
e^{(1)}_i &\!\! 0\\
0 &\!\! e^{(2)}_i
\end{pmatrix}\!, 
\;\;
e^{(2)}_i \geq e^{(1)}_i > 0, 
\;\;
\bm{\V_i}^T\! \bm{\V_i} = \bm{\V_i} \bm{\V_i}^T = \bm{\I}_2 , 
\label{eq:FF4}
\end{equation}
where for every $i=1,\ldots,N$, $e^{(1)}_i, e^{(2)}_i$ are the (positive) eigenvalues of $\bm{\Sigma_i}$ and $\bm{\V_i}$ is an orthonormal (rotation) matrix to be made precise. We can thus rewrite the $i$-th term of the sum in \eqref{eq:bivprior} as
\begin{equation}
\label{eq:comp1}
\Big( \:\! (\bm{\D u})_{i}^T \bm{\Sigma_i}^{-1} (\bm{\D u})_{i} \:\! \Big)^{\textstyle{\frac{p_{i}}{2}}} \! = \: 
\Big( \:\! (\bm{\D u})_{i}^T \bm{\V_i}^T \bm{\E_i}^{-1} \bm{\V_i} (\bm{\D u})_{i} \:\! \Big)^{\textstyle{\frac{p_{i}}{2}}} \! = \:
\left\| \, \widetilde{\bm{\Lambda}}_i \bm{\mathrm{R}}_{\:\!\text{-}\theta_i}\, (\bm{\D u})_i \right\|_2^{p_{i}}\,,
\end{equation}
where 
\begin{eqnarray}
&\widetilde{\bm{\Lambda}}_i 
\;{=}\: 
\begin{pmatrix}
{\widetilde{\lambda}_i^{(1)}} & 0\\
0 & \widetilde{\lambda}_i^{(2)}
\end{pmatrix}
\,\;{:=}\;\,\, 
\bm{\E_i}^{-1/2}
\:{=}\:
\begin{pmatrix}
1/\sqrt{e_i^{(1)}} &\!\!\!0 \\
0 & \!\!\!1/\sqrt{e_i^{(2)}}
\end{pmatrix}
,
\label{siginvdec0}
\\
&\bm{\mathrm{R}}_{\:\!\text{-}\theta_i} 
\;{=}\: 
\begin{pmatrix}
\cos\theta_i & \sin\theta_i\\
-\sin\theta_i & \cos\theta_i
\end{pmatrix} 
\,\;{=}\;\, \bm{\V_i}\,,\quad \theta_i\in[-\pi/2,\pi/2)\,,
\label{siginvdec}
\end{eqnarray}
and $\theta_i\in[-\pi/2,\pi/2)$ denotes the angle drawn locally with respect to the horizontal axis, as simple geometrical considerations show.
By now introducing the two parameter vectors $\bm{\alpha} \in \R_{++}^{N}$ and  $\bm{a}\in(0,1]^N$ with components
\begin{equation}
\alpha_i := \widetilde{\lambda}_i^{(1)}
%\left( \widetilde{\lambda}_i^{(1)} \right)^{p_i} \in \R_{++}, 
\in \R_{++},
\quad\; a_i 
\;{:=}\; 
\frac{\widetilde{\lambda}_i^{(2)}}{\widetilde{\lambda}_i^{(1)}} \in (0,1], 
\quad\; i = 1,\ldots,N\,,
\label{eq:parss}
\end{equation}
we have that the matrix $\widetilde{\bm{\Lambda}}_i$ in \eqref{siginvdec} can be equivalently rewritten as
\begin{equation}
\widetilde{\bm{\Lambda}}_i 
\,\;{=}\;\, 
\widetilde{\lambda}_i^{(1)} \,
\left(
\begin{array}{cc}
1&0\\
0&\displaystyle{
	\widetilde{\lambda}_i^{(2)} / \widetilde{\lambda}_i^{(1)}
	%\frac{\widetilde{\lambda}_i^{(2)}}{\widetilde{\lambda}_i^{(1)}}
}
\end{array}
\right)
\;{=}\;\,
%\alpha_i^{1/p_i}
\alpha_i
\bm{\Lambda}_{a_i}
\,,\quad
\text{with}\;\: 
\bm{\Lambda}_{a_i} = 
\left(
\begin{array}{cc}
1&0\\
0&a_i
\end{array}
\right)\,.
\label{eq:comp2}
\end{equation}
Combining altogether, we have that \eqref{siginvdec}-\eqref{eq:comp2} entail that the term in \eqref{eq:comp1} can be indeed written as
\begin{equation}
\Big( \:\! (\bm{\D u})_{i}^T \bm{\Sigma_i}^{-1} (\bm{\D u})_{i} \:\! \Big)^{\textstyle{\frac{p_{i}}{2}}} 
=\;
\alpha_i^{p_i} \left\| \, 
\bm{\Lambda}_{a_i} \bm{\mathrm{R}}_{\:\!\text{-}\theta_i} (\bm{\D u})_{i} 
\right\|_2^{p_i} \, .
\label{eq:serve0}
\end{equation}
Furthermore, based on \eqref{eq:FF4} and \eqref{siginvdec}-\eqref{eq:comp2}, we observe that:
\begin{equation}
\big| \bm{\Sigma}_i \big|^{-1/2} 
\;{=}\;
\big| 
\bm{\V}_i^T\bm{\E}_i \bm{\V}_i
\big|^{-1/2} 
\;{=}\;
\big| \bm{\E}_i \big|^{-1/2} 
\;{=}\;
\big| \widetilde{\bm{\Lambda}}_i \big|
\;{=}\;
\big| \alpha_i^2 \bm{\Lambda}_{a_i} \big|
\;{=}\;
\alpha_i^{2} a_i >0.
\label{eq:serve}
\end{equation}
Plugging now \eqref{eq:serve0} and \eqref{eq:serve} into the expression  \eqref{eq:bivprior}, we obtain the following equivalent form
\begin{equation}
\mathbb{P}(\bm{u}\mid\bm{\Theta}) 
\;{=}\;
\frac{c(\bm{\Theta})}{(2\pi)^N}\left(
\prod_{i=1}^{N} 
\frac{\alpha_i^{2} \, p_i \, a_i}{\Gamma(2/p_i) \, 2^{\:\!2/p_i}}
\right)
\exp\left(
-\sum_{i=1}^{N} \alpha_i^{p_i} \left\| \, \bm{\Lambda}_{a_i} \bm{\mathrm{R}}_{\:\!\text{-}\theta_i} \, (\bm{\D u})_i \right\|_2^{p_{i}}\right)\, 
% \nonumber\\%\label{eq:propreg} \\
% &\,\;{=}\;\,& 
% \frac{1}{(2\pi)^N}\left(
% \prod_{i=1}^{N} 
% \frac{\alpha_i^{2} \, p_i \, a_i}{\Gamma(2/p_i) \, 2^{\:\!2/p_i}}
% \right)
% \exp\left(-\sum_{i=1}^{N}\alpha_i^{p_i} \left\lVert \begin{pmatrix}
%     D_{\theta_i} u_{i} \\
%     a_i D_{\theta_i^\perp} u_{i}
%     \end{pmatrix} \right\rVert_2^{p_i}\right)\,, 
\label{eq:propreg}
\end{equation}
where the vector of hyperparameters is here:
\begin{equation}
\label{eq:def_alpha_theta}
\bm{ \Theta} = \left(\bm{\alpha},\bm{p},\bm{\theta,\bm{a}}\right)\in\R_{++}^{N\times 2} \times [-\pi/2,\pi/2)^N \times (0,1]^N\,.
\end{equation}
%
% where $D_{\theta_i} u_{i}$ and $D_{\theta_i^\perp}u_{i}$ denote the directional derivatives along the direction $\bm{v}_i=(\cos\theta_i,\sin\theta_i)$ and its orthogonal $\bm{v}_i^\perp=(-\sin\theta_i,\cos\theta_i)$, i.e. $D_{\theta_i} u_{i}= (\bm{\D u})_{i}\cdot \bm{v}_i$ and $D_{\theta_i^\perp} u_{i}= (\bm{\D u})_{i}\cdot \bm{v}_i^\perp$.

%It can now be easily seen that the obtained prior \eqref{eq:propreg}-\eqref{eq:def_alpha_theta} represents a generalisation of all previously introduced ones. In particular, it
Compared to the univariate prior \eqref{eq:WTVp_prior}, prior \eqref{eq:propreg}-\eqref{eq:def_alpha_theta} is characterised by two additional vectors of (space-variant) parameters $\theta_i \in [-\pi/2,\pi/2)$ and $a_i \in (0,1]$, $i = 1.\ldots,N$. These parameters relate in fact to the bivariate nature of the BGGd in \eqref{eq:MGGDi}. In particular,
The parameter $\theta_i$ represents the direction of the major axis of elliptical contour lines of the local BGGd, while $a_i$ describes locally the eccentricity of the contour lines. More precisely, $a_i=1$ corresponds to circular contour lines, i.e. to a locally maximal isotropic pdf, whereas for $a_i\approx 0$ the contour lines approach lines drawing the angle $\theta_i$ w.r.t. to the horizontal axis, hence they are maximally anisotropic. 
The great flexibility of distribution in \eqref{eq:MGGDi} is highlighted in Figures \ref{fig:bggd1}-\ref{fig:bggd3}, where the pdfs corresponding to the choice of different scalar parameters $\alpha_i$, $p_i$, $a_i$ and $\theta_i$ are shown, while the corresponding contour plots are displayed in Figures \ref{fig:bggd_cont1}-\ref{fig:bggd_cont3}.

\begin{remark} 
	Note that the non-stationary prior in \eqref{eq:WTV_prior} reduces to the stationary TV prior in \eqref{eq:TV_prior} for constant choices of the scale parameters $\alpha_i= \alpha$, $\forall i$. Analogously, by setting $\alpha_i= \alpha$ and $p_i= p$, $\forall i$, in \eqref{eq:WTVp_prior}, we recover the space-invariant prior corresponding to the TV$_p$ regulariser in \eqref{eq:TVp}. The same consideration holds for the DTV regularisation term in \eqref{eq:DTV}, whose statistical counterpart is obtained starting from \eqref{eq:propreg} and setting $\alpha_i= \alpha$, $p_i= 1$, $\theta_i= \theta$ and $a_i= a$, $\forall i$.
	
	%for specific choices of the space-variant parameters previously introduced, one can recover the prior correspoding to 

	%We point out that the non-stationary prior in \eqref{eq:WTVp_prior} reduces to the stationary prior corresponding to the TV$_p$ regulariser in \eqref{eq:TVp} under specific choices of the space-variant parameters $\alpha_i,p_i$, namely $(\alpha_i,p_i)=(\alpha,p)$, $i=1,\ldots,N$. The same consideration holds for the non-stationary prior in \eqref{eq:bivprior} from which the space-invariant DTV regularisation term in \eqref{eq:DTV} can be derived by setting, $(\alpha_i,p_i,\theta_i,a_i)=(\alpha,p,\theta,a)$, $i=1,\ldots,N$.
	
	%{\color{red} togliere e spostare nel remark finale
	%Note that for constant choices of the scale parameters $\alpha_i \equiv\alpha$ and/or the shape parameters $p_i\equiv p$, one retrieves the distributions above (notice, in particular, that \eqref{eq:hLpdf} and \eqref{eq:hLD-sv} can be retrieved  by setting $p=1$ and letting the scale parameter constant/to vary, respectively).}

\end{remark}

%Note that when  in \eqref{eq:MGGDi} $p_i=2$ for every $i=1,\ldots,n$, then the BGGD reduces to a standard bivariate Gaussian distribution with pixel-wise covariance matrices $\Sigma_i$.

\subsection{Hierarchical modelling}
\label{subsec:hierarchcal}

The effort made in deriving the highly-parametric prior distributions in the previous section would be vain if not coupled with an automatic and robust procedure for the estimation of the unknown parameters $\bm{\Theta}$. The choice of recasting the original problem in probabilistic terms makes  very natural to model the unknown vector $\bm{\Theta}$ as well as the unknown  $\bm{u}$, as random variables. To do so, we thus need to introduce a further pdf encoding the \emph{a priori} beliefs on $\bm{\Theta}$, which, in the following, will be denoted by $\mathbb{P}(\bm{\Theta})$ and which will be referred to as \emph{hyperprior}. %This strategy goes under the name of \emph{hierarchical modelling}. 

% The joint prior on $(\bm{u},\bm{\Theta})$ is thus given by the product of the conditional prior $\mathbb{P}(\bm{u}\mid\bm{ \Theta})$ and the hyperprior $\mathbb{P}(\bm{\Theta})$. 
By proceeding as in \eqref{eq:bayes}, we seek for the analytic expression of the joint posterior pdf, which, by leaving the dependence on $\bm{\Theta}$ explicit, is related to the prior and likelihood pdf through
\begin{equation}
\label{eq:bayes2}
\mathbb{P}(\bm{u},\bm{\Theta} \mid\bm{ b}) \;{=}\;\frac{\mathbb{P}(\bm{u},\bm{\Theta})\mathbb{P}(\bm{b}\mid\bm{ \A u})}{\mathbb{P}(\bm{b})}\;{=}\;\frac{\mathbb{P}(\bm{u}\mid \bm{\Theta})\mathbb{P}(\bm{\Theta})\mathbb{P}(\bm{b}\mid \bm{\A u})}{\mathbb{P}(\bm{b})}\,,
\end{equation}
where we have used  $\mathbb{P}(\bm{u},\bm{ \Theta})=\mathbb{P}(\bm{u}\mid\bm{ \Theta})\mathbb{P}(\bm{\Theta})$.
Proceeding by standard MAP estimation, we thus have that the sought solution pair $\left\{\bm{u}^*,\bm{\Theta}^*\right\}$ is the one maximising $\mathbb{P}(\bm{u},\bm{\Theta}\mid\bm{ b})$, i.e.:
\begin{equation}
\label{eq:map_0}
\left\{\bm{u}^*,\bm{\Theta}^*\right\}\in\argmax_{\bm{u},\bm{\Theta}}\left\{
\mathbb{P}(\bm{u}\mid \bm{\Theta})\mathbb{P}(\bm{\Theta})\mathbb{P}(\bm{b}\mid \bm{\A u})\right\}\,,
\end{equation}
or, equivalently,
\begin{align}
\label{eq:map}
\begin{split}
\left\{\bm{u}^*,\bm{\Theta}^*\right\}\;{\in}\;&\argmin{\bm{u},\bm{\Theta}}\left\{-\ln\mathbb{P}(\bm{u}\mid\bm{ \Theta})-\ln\mathbb{P}(\bm{\Theta})-\ln\mathbb{P}(\bm{b}\mid \bm{\A u})\right\}\\
\;{=}\;&\argmin{\bm{u},\bm{\Theta}}\left\{-\ln c(\bm{\Theta})-\ln\mathbb{P}(\bm{z}(\bm{u})\mid\bm{ \Theta})-\ln\mathbb{P}(\bm{\Theta})-\ln\mathbb{P}(\bm{b}\mid \bm{\A u})\right\}
\end{split}
\end{align}
where the evidence term $\mathbb{P}(\bm{b})$ has been dropped as it does not depend either on $\bm{u}$ or $\bm{\Theta}$.

When tackling the joint model \eqref{eq:map}, two major difficulties arise, namely the computation of the highly-dimensional constant $c(\bm{\Theta})$ and the choice of an efficient algorithmic scheme for the numerical solution of the minimisation problem \eqref{eq:map}. Different strategies have been designed to overcome the former issue: most of them are based on a modification of the conditional prior $\mathbb{P}(\bm{u}\mid\bm{\Theta})$ which comes from either approximating $c(\bm{\Theta})$ (see \cite{VB1,VB2,VB3}) or neglecting it (see \cite{robust}). Here, we adopt this latter approach so that the joint hypermodel \eqref{eq:map} takes the form:
\begin{equation}
\label{eq:map_last}
\left\{\bm{u}^*,\bm{\Theta}^*\right\}\;{\in}\;\argmin{\bm{u},\bm{\Theta}}\left\{-\ln\mathbb{P}(\bm{z}(\bm{u})\mid\bm{ \Theta})-\ln\mathbb{P}(\bm{\Theta})-\ln\mathbb{P}(\bm{b}\mid \bm{\A u})\right\}\,.
\end{equation}
Neglecting $c(\bm{\Theta})$ provides a significant simplification of the problem of interest. Nonetheless, as we will show in Sections \ref{sec:parameter_estimation} and \ref{sec:restoration}, such simplification will result in an efficient ML-type parameter estimation strategy which will be shown to produce meaningful results. Clearly, a more accurate study of \eqref{eq:map} will require to deal explicitly with the computation of such constant by means, for instance, of analogous approches as those described in \cite{PereyraPartI,PereyraPartII}.

From a numerical perspective, the solution of problem \eqref{eq:map_last} can be addressed in different manners. A standard strategy illustrated in \cite{Calvetti08} is based on the design of an Iterated Sequential Algorithm (IAS) which, for $k\geq 0$ and upon a suitable initialisation for $\bm{u}^{(0)}$ reads:
\begin{align}
\label{eq:sub_th}
\bm{\Theta}^{(k+1)} \;{\in}\;& \argmin{\bm{\Theta}}\left\{-\ln\mathbb{P}(\bm{z}(\bm{u}^{(k)})\mid \bm{ \Theta})-\ln\mathbb{P}(\bm{\Theta})\right\}\\
\label{eq:sub_u}
\bm{u}^{(k+1)} \;{\in}\;& \argmin{\bm{u}}\left\{-\ln\mathbb{P}(\bm{z}(\bm{u})\mid\bm{\Theta}^{(k+1)})-\ln\mathbb{P}(\bm{b}\mid\bm{ \A u})\right\}\,,
\end{align}
where the function $\bm{z}(\cdot)$ has been defined in Section \ref{subsec:prior} depending on the specific form of the prior distribution at hand.

\section{The anatomy of space-variant regularisation models}
\label{sec:map}

In this section, we derive the explicit expressions of the negative log-prior term $-\ln\mathbb{P}(\bm{z}(\bm{u})\mid\bm{\Theta})$, appearing in the cost function of \eqref{eq:map_last}, depending on the particular choice of the prior pdf among the ones described in Section \ref{subsec:prior}.
%namely the negative log-prior $-\ln\mathbb{P}(\bm{u}\mid\bm{ \Theta})$, negative log-likelihood $-\ln\mathbb{P}(\bm{b}\mid \bm{\A u})$ and negative log-hyperprior $-\ln\mathbb{P}(\bm{\Theta})$. 
For each considered prior, we will write explicitly the analytical form of the corresponding image regulariser, dissecting its properties in terms of regularisation features and providing some intuitions on their sparsity promoting behaviour.

%in Section \ref{subsec:logpr} we start from %$-\ln\mathbb{P}(\bm{b}\mid \bm{\A u})$ 
%computing the negative logarithm of the different priors $\,\mathbb{P}(\bm{u}\mid\bm{ \Theta})$ introduced in Section \ref{subsec:prior}. 
%This leads to a set of different space-variant regularisation terms, which is the main focus of this review and will be motivated and discussed. 
%Then, the negative log-likelihood will be derived in Section \ref{subsec:loglik}: it will lead to the so-called L$_q$ data fidelity term.

\subsection{From non-stationary priors to space-variant regularisers}
\label{subsec:logpr}
Recalling \eqref{eq:TV_reg},  we start computing the negative logarithm of the stationary Gibbs' TV prior in \eqref{eq:TV_prior}. We have:
\begin{equation}
-\ln\mathbb{P}(\bm{z}(\bm{u})\mid \bm{\Theta})
\;{=}\; 
\alpha \mathrm{TV}(\bm{u})
- N \ln \alpha.
\label{eq:nlp_TV}
\end{equation}
%
%with the (space-invariant) isotropic TV semi-norm regulariser defined in . 
We now rewrite TV as
\begin{equation}
\label{eq:TV_reg2}
\TV(\bm{u}) 
\,\;{=}\;\, 
\sum_{i=1}^{N} f_{\mathrm{TV}}\left((\bm{\D u})_i\right),
%\quad h(g_1,g_2) = \|(\bm{\D u})_i\|_2\,,
\end{equation}
where the space-invariant and non-parametric function $f_{\mathrm{TV}}: \R^2 \to \R_+$ is defined by
\begin{equation}
f_{\mathrm{TV}}(\bm{g}_i) \,\;{:=}\;\, \left\| \bm{g}_i \right\|_2 \, ,
\quad \bm{g}_i = (g_{i,1},g_{i,2}) \in \R^2 \, ,
\label{eq:TV_h}
\end{equation}
and is referred to in the following as the TV \emph{gradient penalty function}. As it is well-known, TV is bounded from below by zero, (non-strictly) convex, non-coercive due to $\mathrm{null} (\bm{\D}) \neq \{\bm{0}_2\}$ 
%and infinitely many times differentiable except at points $\bm{u} \in \R^N$ such that $\exists\, i: \; (\bm{\D u})_i = \bm{0}$, where it is only continuous.
and non-smooth. This last property is indeed responsible of the good gradient sparsity-promoting effect of TV, which favours piece-wise constant solutions.
%By looking at the shape of t
The TV gradient penalty function $f_{\mathrm{TV}}$ in \eqref{eq:TV_h} is shown in Figure \ref{fig:TV_h}(a).

% , one can notice that the origin, corresponding to a null local gradient, represents a {\color{red}quite sharp potential sink capable of attracting the solution gradients.}

%
\begin{figure}[!h]
	\centering
	\begin{tabular}{cc}
		\includegraphics[trim={0.85cm 0.62cm 2.17cm 0.78cm},clip,width=5.2cm]{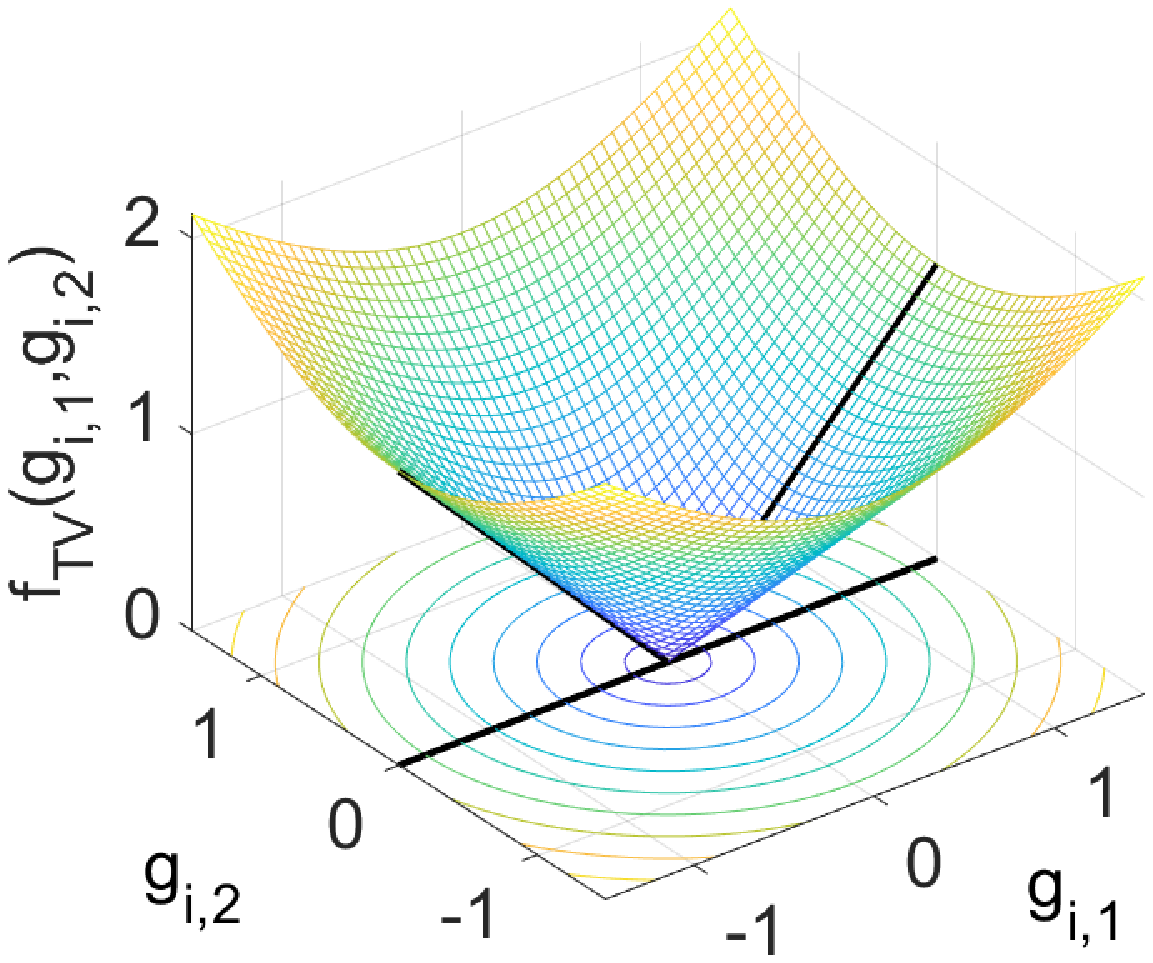}
		&
		\includegraphics[trim={0cm 0.0cm 0cm 0.0cm},clip,width=6.5cm]{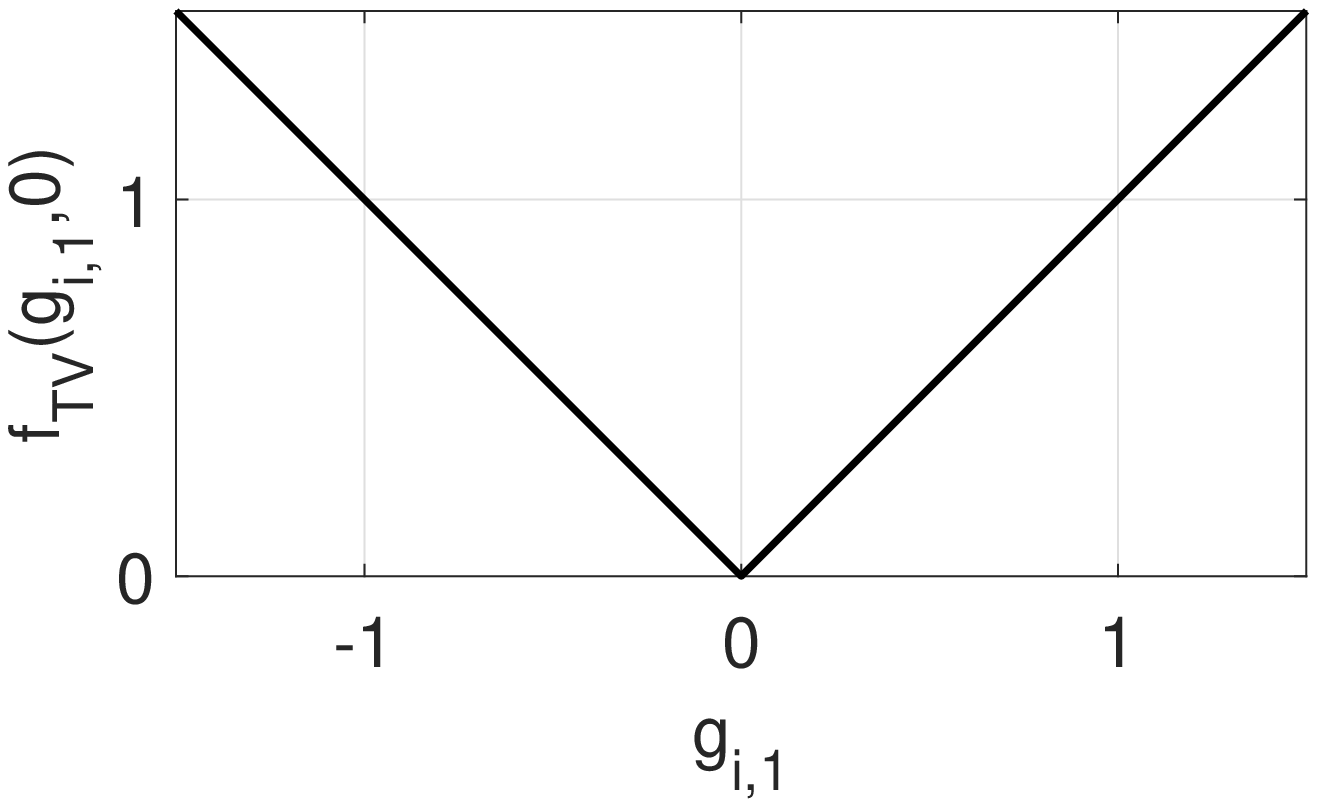}
		\\
		(a) $f_{\mathrm{TV}}$ &
		(b) 1D section of $f_{\mathrm{TV}}$ along the $x$-axis 
	\end{tabular}
	\caption{Space-invariant gradient penalty function $f_{\mathrm{TV}}$ defined in \eqref{eq:TV_h} for the TV regulariser \eqref{eq:TV_reg2}.}
	\label{fig:TV_h}
\end{figure}
To analyse in detail the properties of the TV regulariser,
%- and of subsequent space-variant regularisers as well - 
it is useful to consider the 1D sections of the gradient penalty function $f_{\mathrm{TV}}$ along straight lines passing through the origin of the penalty domain and having direction defined by the angle \mbox{$\varphi \in [-\pi,+\pi)$.} Using a standard (arc-length) parametrisation for straight lines, namely $\left\{\,g_{i,1}(t;\varphi) = t \, \cos(\varphi), \;\: g_{i,2}(t;\varphi) = t \, \sin(\varphi), \;\: t \in \R \, \right\}$, the sections of $f_{\mathrm{TV}}$ in \eqref{eq:TV_h} read
\begin{equation}
s_i(t;\varphi) \;{=}\; \left| t \right|, \quad t \in \R, \quad i = 1,\ldots,N \, .
\label{eq:TV_sec}
\end{equation}
In Figure \ref{fig:TV_h}(b) we show one section of $f_{\mathrm{TV}}$ along the direction defined by the angle \mbox{$\varphi = 0$,} i.e., the $x$-axis.  However, as the expression \eqref{eq:TV_sec} does not depend on $\varphi$, we deduce that the same Figure could be obtained by representing the section corresponding to \emph{any} $\varphi$, for any pixel location $i$. 
The TV penalty $f_{\mathrm{TV}}$ in \eqref{eq:TV_h} - whence, the \eqref{eq:TV_reg} regulariser - is in fact space and rotationally-invariant (i.e. \emph{isotropic}).

Being isotropic, TV does not take explicitly into account directionality properties in the image. Moreover, the presence of a fixed, global exponent $p=1$ for the norms in the penalty \eqref{eq:TV_h} and of a global scale parameter $\alpha>0$ in \eqref{eq:TV_prior} and, hence, in the negative log-prior \eqref{eq:nlp_TV} makes the TV regulariser not capable to adapt the strength (associated to $\alpha$ in \eqref{eq:nlp_TV}) nor the nature (associated to the exponent of the norm in \eqref{eq:TV_h}) of the gradient sparsity-promotion effect to the local contents of the image to be recovered. 

%It is worth noting that all these (negative) properties are due to the gradient penalty function $f$ defined in \eqref{eq:TV_h} and shown in Fig.~\ref{fig:TV_h} being space-invariant (the same penalty is adopted for any pixel $i$) and rotationally-invariant (see the circular level curves).

\medskip 

In the following, we inspect how the non-stationary priors introduced in Section \ref{subsec:prior} can favour local regularisation features, namely strength, sharpness and directionality.

\begin{figure}[!t]
	\centering
	\renewcommand{\arraystretch}{0.0}
	\renewcommand{\tabcolsep}{0.08cm}
	\begin{tabular}{ccccc}
		\!\!\!\!\!\!\!\!&$\;$& WTV & WTV$_{\bm{p}}^{\mathrm{sv}}$ & WDTV$_{\bm{p}}^{\mathrm{sv}}$
		\vspace{0.2cm}\\
		\!\!\!\!\!\!\!\!{\rotatebox{90}{$\qquad\quad$ pixel $j$}}&&
		%$i$&
		\includegraphics[trim={2.7cm 0.58cm 2.81cm 0.78cm},clip,width=3.85cm]{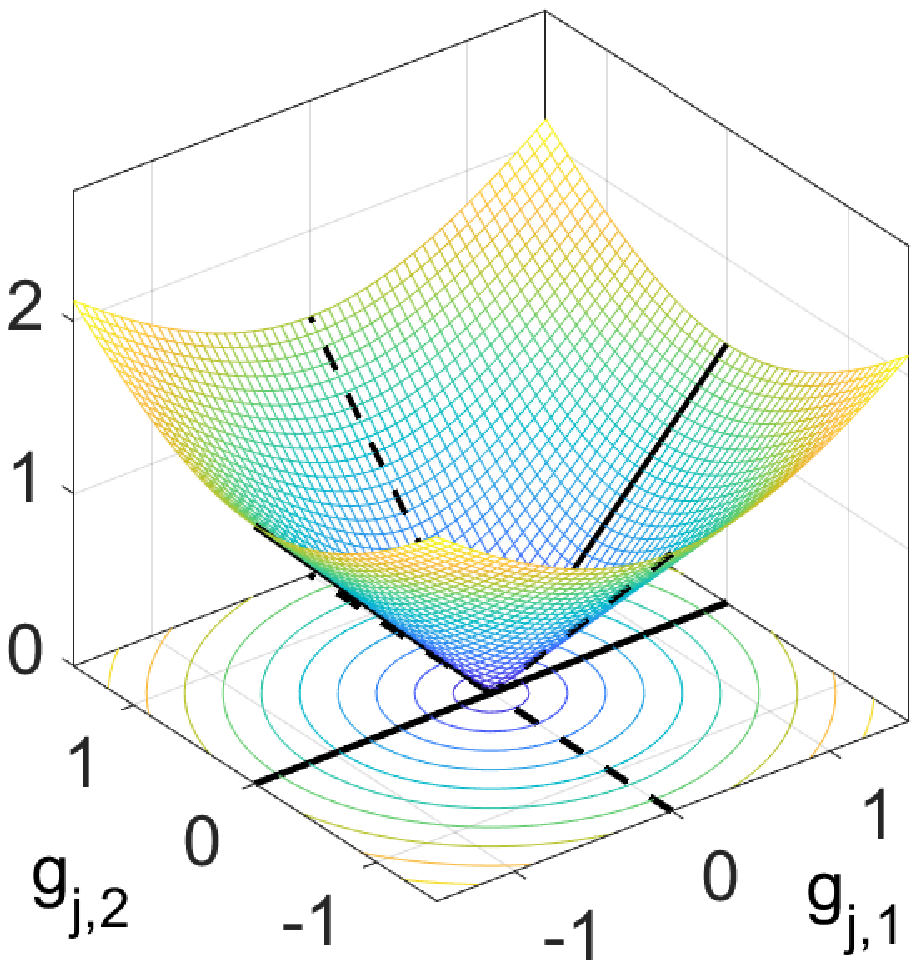}&
		\includegraphics[trim={2.7cm 0.58cm 2.81cm 0.78cm},clip,width=3.85cm]{images/WTVreg1.eps}&
		\includegraphics[trim={2.7cm 0.58cm 2.81cm 0.78cm},clip,width=3.85cm]{images/WTVreg1.eps}
		\vspace{0.12cm}\\
		\!\!\!\!\!\!\!\!&& \footnotesize{(a) $\alpha_j{=}1.0$} &\footnotesize{(b) $(\alpha,p)_j{=}(1.0,1.0)$}&
		\footnotesize{(c)\!\! $(\alpha,p,\theta,a)_j{=}(1.0,1.0,0,1.0)$}
		\vspace{0.17cm}\\
		\!\!\!\!\!\!\!\!{\rotatebox{90}{$\qquad\quad$ pixel $k$}}&&
		%$j$&
		\includegraphics[trim={2.7cm 0.58cm 2.81cm 0.78cm},clip,width=3.85cm]{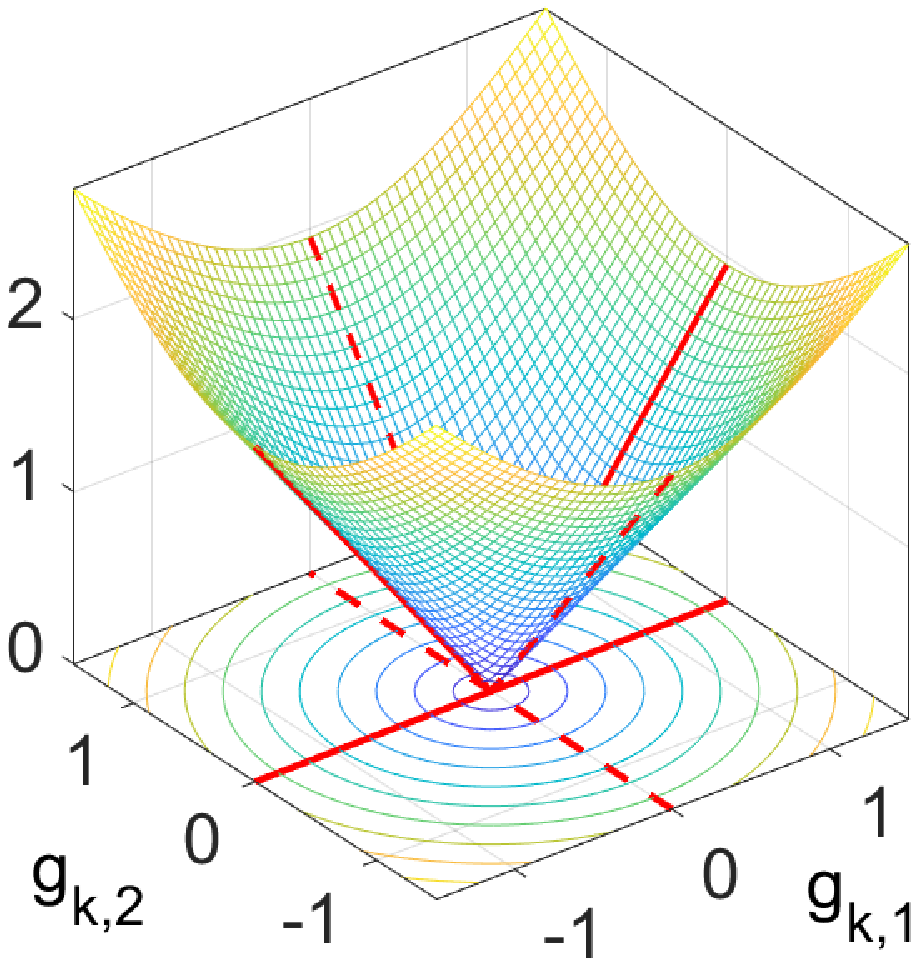}&
		\includegraphics[trim={2.7cm 0.58cm 2.81cm 0.78cm},clip,width=3.85cm]{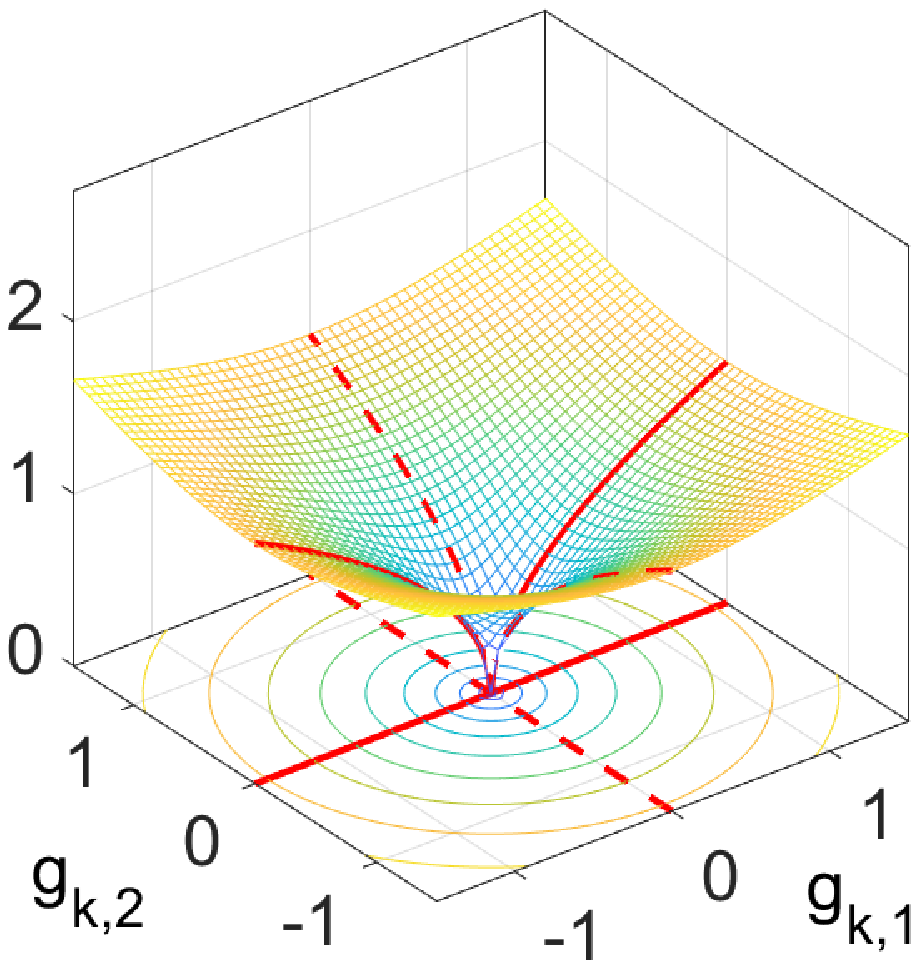}&
		\includegraphics[trim={2.7cm 0.58cm 2.81cm 0.78cm},clip,width=3.85cm]{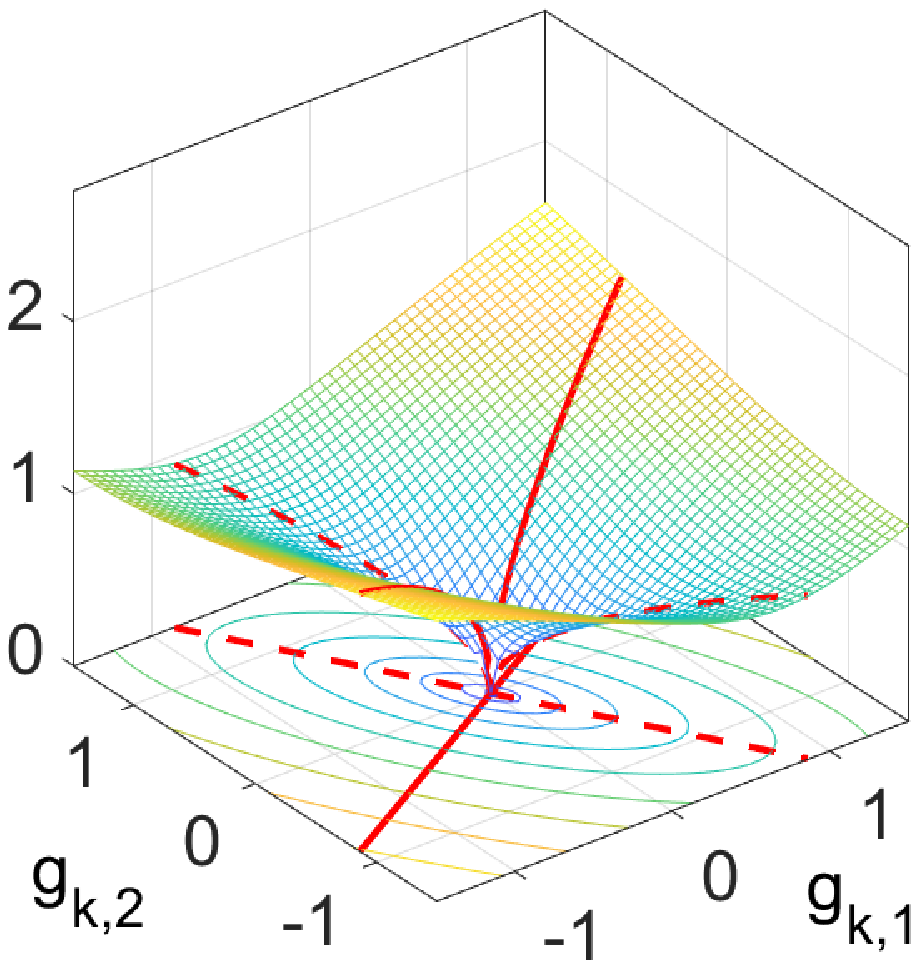}
		\vspace{0.12cm}\\
		\!\!\!\!\!\!\!\!&& \footnotesize{(d) $\alpha_k{=}1.3$}&
		\footnotesize{(e) $(\alpha,p)_k{=}(1.3,0.5)$}&\footnotesize{(f)$(\alpha,p,\theta,a)_k{=}(1.3,0.5,\frac{\pi}{6},0.4)$}
		\vspace{0.17cm}\\
		\!\!\!\!\!\!\!\!{\rotatebox{90}{$\qquad\quad$ pixel $l$}}&&
		%$k$&
		\includegraphics[trim={2.7cm 0.58cm 2.81cm 0.78cm},clip,width=3.85cm]{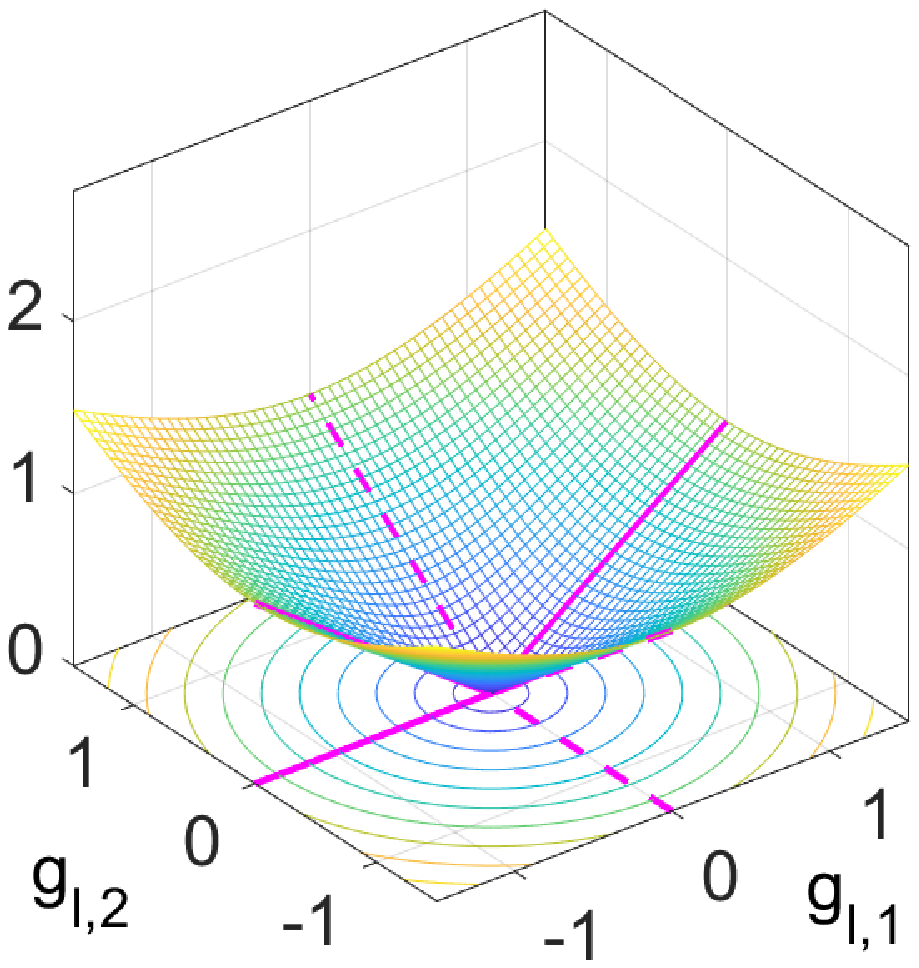}&
		\includegraphics[trim={2.7cm 0.58cm 2.81cm 0.78cm},clip,width=3.85cm]{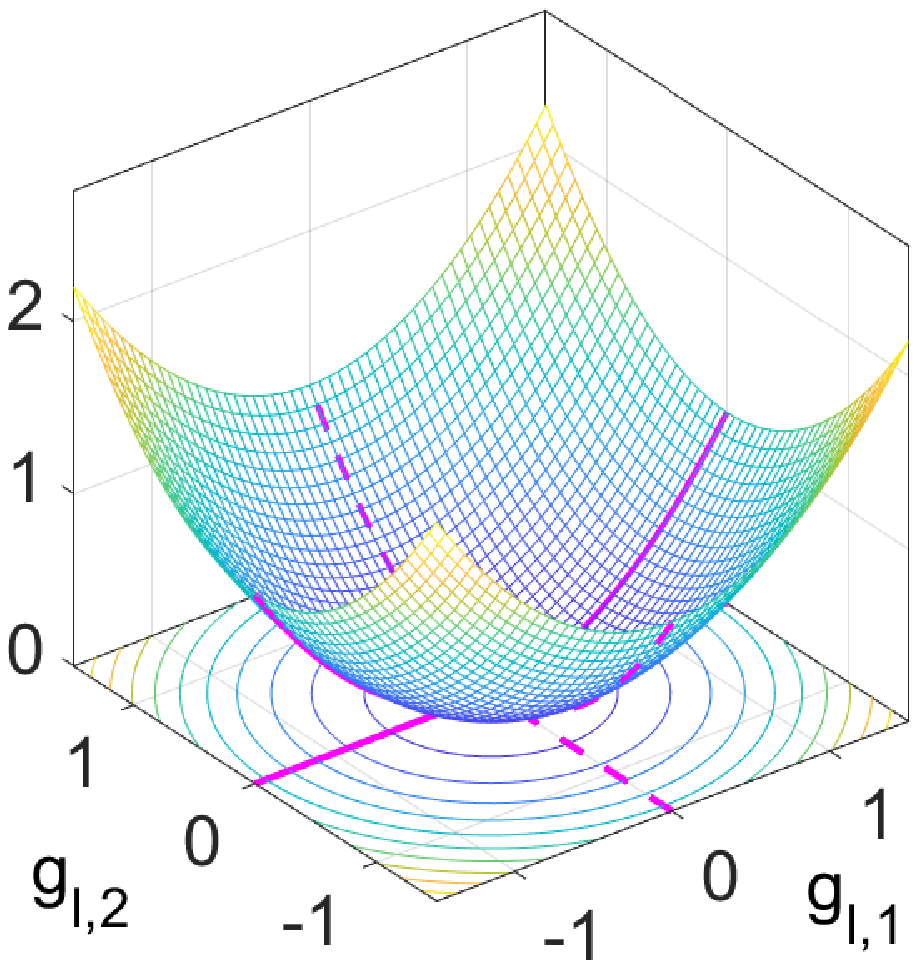}&
		\includegraphics[trim={2.7cm 0.58cm 2.81cm 0.78cm},clip,width=3.85cm]{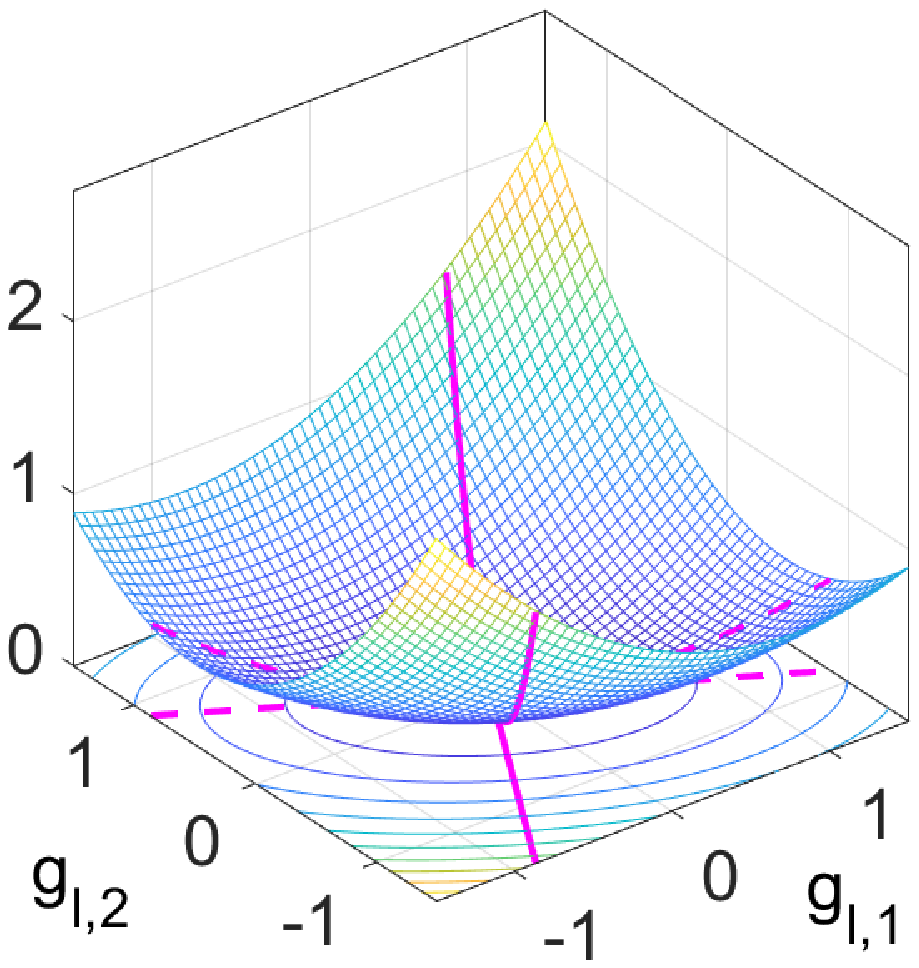}
		\vspace{0.12cm}\\
		\!\!\!\!\!\!\!\!&&\footnotesize{(g) $\alpha_l{=}0.7$}  &\footnotesize{(h)$(\alpha,p)_l{=}(0.7,2)$}&\footnotesize{(i)$(\alpha,p,\theta,a)_l{=}(0.7,2,\frac{\pi}{3},0.6)$}
		\vspace{0.17cm}\\
		%&&
		%\includegraphics[height=2.6cm]{images/WTVreg4.eps}&
		%\includegraphics[height=2.6cm]{images/WTVpreg4.eps}&
		%\includegraphics[height=2.6cm]{images/WDTVpreg4.eps}
		%\\
		&&
		\includegraphics[trim={0.67cm 0.2cm 1.28cm 0.75cm},clip,width=3.75cm]{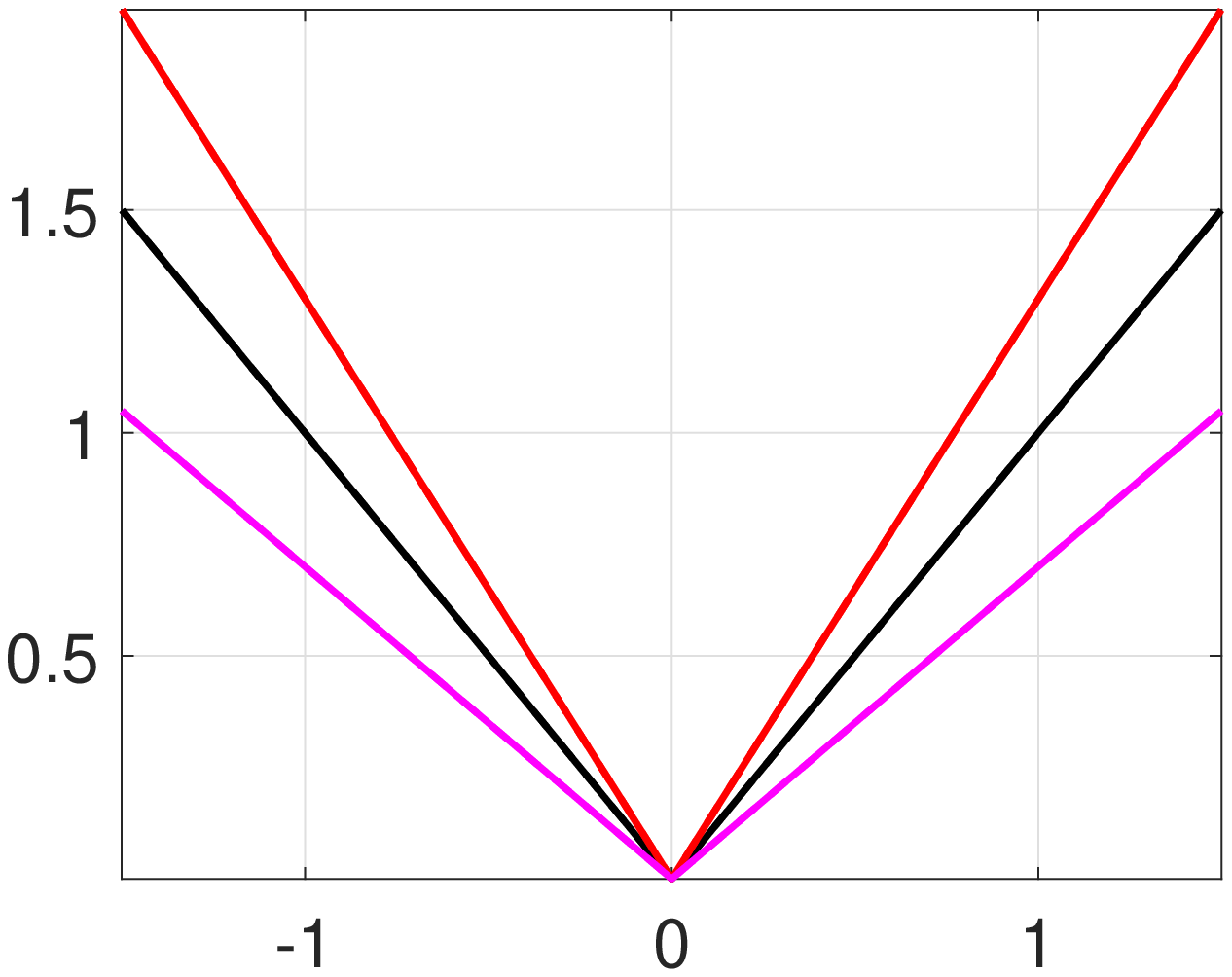}&
		\includegraphics[trim={0.67cm 0.2cm 1.28cm 0.75cm},clip,width=3.75cm]{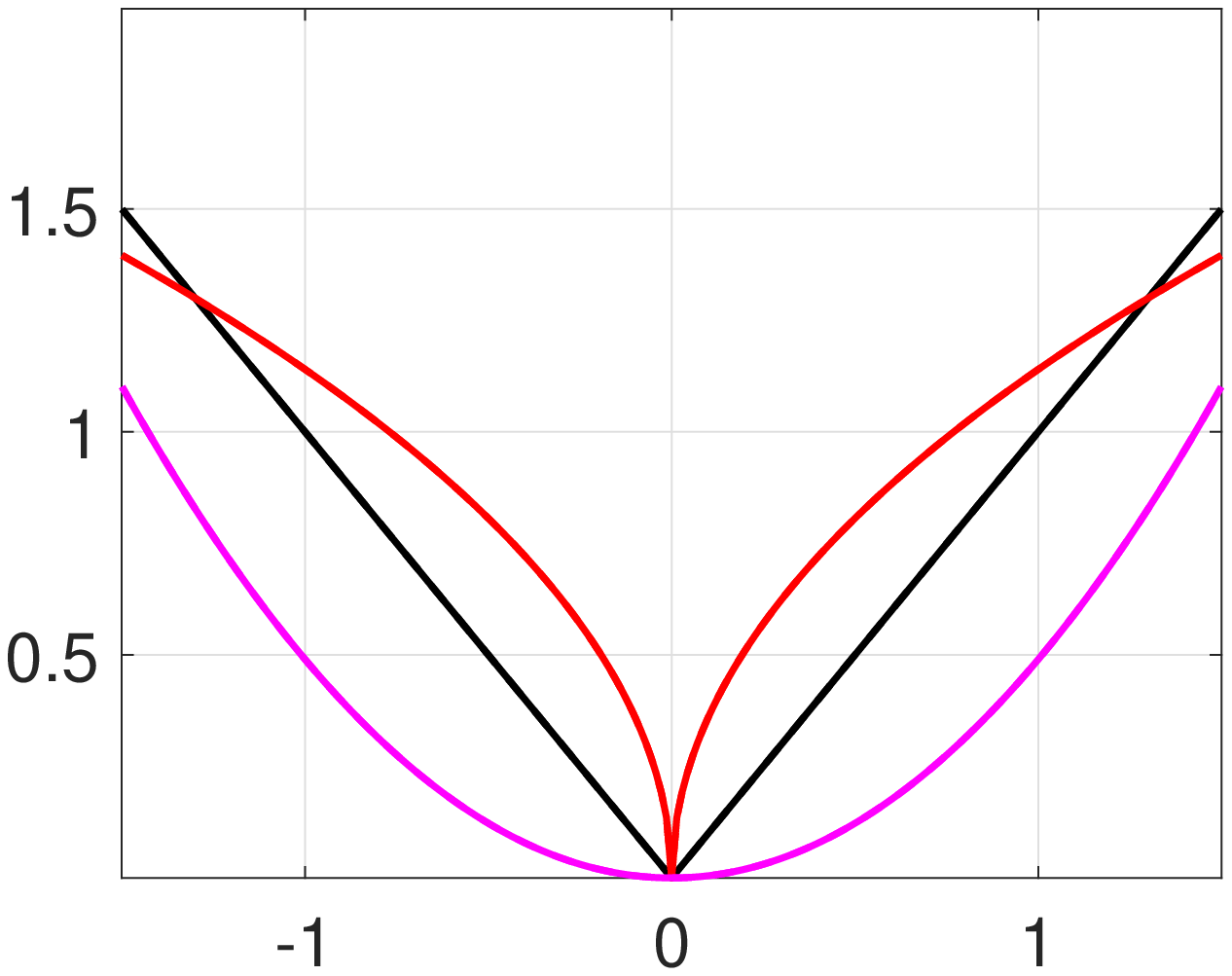}&
		\includegraphics[trim={0.67cm 0.2cm 1.28cm 0.75cm},clip,width=3.75cm]{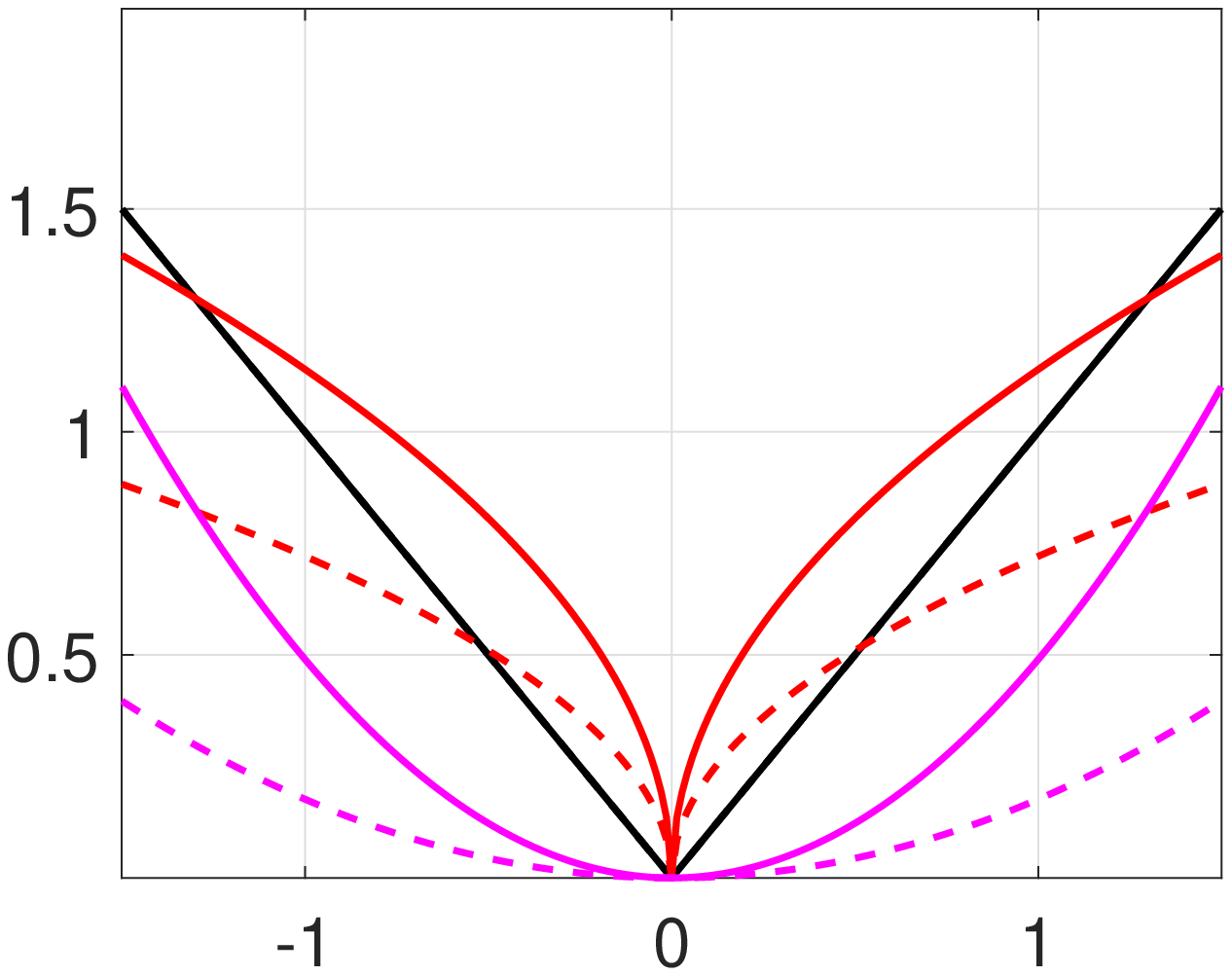}
		\vspace{0.12cm}\\
		&&
		\multicolumn{3}{l}{$\;\;\,$ (j) $\,$ 1D sections for angles $\varphi$ (solid lines) and $\,\varphi + \pi/2$ (dashed lines), with}
		\vspace{0.09cm}\\
		&& 
		\multicolumn{3}{l}{$\;\;\:\:\!$ $\varphi = 0 \,$ for WTV and WTV$_{\bm{p}}^{\mathrm{sv}}$, $\,\varphi = \theta_i\,$ for WDTV$_{\bm{p}}^{\mathrm{sv}}$. Notice that solid and}
		\vspace{0.02cm}\\
		&& 
		\multicolumn{3}{l}{$\;\;\:\:\!$ dashed lines coincide for isotropic penalties in (a)-(c),(d),(e),(g),(h)}
		%for isotropic penalty functions, namely ..., the solid and dashed section plots coincide.}
	\end{tabular}
	\caption{Graphs of the gradient penalty functions defined in \eqref{eq:WTV_h}, \eqref{eq:WTVp_h} and \eqref{eq:WDTVp_h} for the space-variant \eqref{eq:WTV_reg}, \eqref{eq:WTVp_reg} and \eqref{eq:WDTVp_reg} regularisers, respectively.}
	\label{fig:regs_h}
\end{figure}

%\textcolor{blue}{We will now introduce some space-varying extensions of the TV regularisation model \eqref{eq:TV} which can be derived via standard MAP estimation from the generalised Gaussian distributions introduced in Section \ref{sec:Bayes}. We will allow further flexibility by allowing a possible `isotropic' ($\ell^2$) VS. `anisotropic' ($\ell^1$) computation of gradient magnitude $\| \bm{ \D u} \|$ in \eqref{eq:TV}, thus considering the following norms (written in compact notation in terms of the index $t\in\left\{1,2\right\}$):\begin{equation}  \label{eq:iso_aniso}\| \bm{ \D u} \|_{t,1} = \sum_{i=1}^N \| (\bm{ \D u})_i \|_{t} =  \sum_{i=1}^N \left( (\D_h u)^t_i + (\D_v u)^t_i \right)^{1/t}, \quad t\in \left\{1, 2\right\}\end{equation}}

\subsubsection{Local regularisation strength}

%The first and probably the easiest way to make the \eqref{eq:TV_reg} regulariser spatially flexible  consists in allowing for a different amount of regularisation at every pixel in the image. From a variational viewpoint, this corresponds to relaxing the assumption of having a single regularisation parameter $\alpha$ which weighs identically the different local pixel-wise contributions combined in the definition of TV via a sum - see (\ref{eq:nlp_TV})-(\ref{eq:TV_h}).
Recalling Section \ref{subsec:prior}, the first and probably the easiest way to make the \eqref{eq:TV_reg} regulariser spatially flexible  consists in allowing for a different amount of regularisation at every pixel in the image. From a Bayesian perspective, this corresponds to assuming a non-stationary hL prior distribution for the gradient magnitudes of $\bm{u}$.
%This feature clearly does not take into account the need, often frequent in practice, to locally refine or coarsen the amount of the regularisation depending on the local characteristic in the image to reconstruct (for instance in the case of non-uniform noise, high-detail level\ldots).
By computing the negative logarithm in \eqref{eq:WTV_prior}, we have
\begin{equation}
-\ln\mathbb{P}(\bm{z}(\bm{u})\mid \bm{\Theta})
\;{=}\; 
\WTV(\bm{u};\bm{\Theta})
-  \sum_{i=1}^N \ln \alpha_i \, , 
\label{eq:nlp_WTV}
\end{equation}
where the space-variant WTV regulariser is defined in terms of hyperparameters \mbox{$\bm{\Theta}= \bm{\alpha}\,$} and reads
\begin{equation}
\WTV(\bm{u};\bm{\alpha}) 
\,\;{:=}\;\, 
\sum_{i=1}^{N}
\alpha_i\|(\bm{\D u})_i\|_2\,,
\quad
\bm{\alpha}\in \R_{++}^{N}\,.
\tag{\text{$\WTV$}}
\label{eq:WTV_reg}
\end{equation}
Analogously to \eqref{eq:TV_reg}, the \eqref{eq:WTV_reg} regulariser can be equivalently rewritten as
\begin{equation}
\label{eq:WTV_reg2}
\WTV(\bm{u};\bm{\alpha}) 
\,\;{=}\;\, 
\sum_{i=1}^{N} f_{\mathrm{WTV}}\left((\bm{\D u})_i;\alpha_i\right), \quad 
\alpha_i\in \R_{++},
%\quad h(g_1,g_2) = \|(\bm{\D u})_i\|_2\,,
\end{equation}
where the gradient penalty function $f_{\mathrm{WTV}}: \R^2 \to \R_+$ depends now locally on the parameter $\alpha_i$ and reads
\begin{equation}
f_{\mathrm{WTV}}(\bm{g}_i;\alpha_i) \,\;{:=}\;\, \alpha_i \left\| \bm{g}_i \right\|_2 \, ,
\quad \bm{g}_i = (g_{i,1},g_{i,2}) \in \R^2 \, .
\label{eq:WTV_h}
\end{equation}
%
%\textcolor{blue}{By allowing for a half-Laplacian distribution on $ \| (\D u)_i \|_{s,2}$, with $s\in\left\{1,2\right\}$ with space-variant scale-parameter $\alpha_i$ for any $i=1,\ldots, N$, we can derive via MAP estimation the following weighted-TV regularisation model:\begin{equation}  \label{eq:WTV}\WTV (u):= \sum_{i=1}^N \alpha_i \| (\bm{ \D u})_i \|_{t}\qquad t\in\left\{1,2\right\}  \tag{WTV}\end{equation}which is nothing but a space-variant extension of \eqref{eq:TV} where the parameters  $\alpha_i$ weight locally the amount of local TV regularisation at any pixel.}
%
%Since the local weights - or scale parameters - $\alpha_i$ are all positive, the \eqref{eq:WTV_reg} regulariser exhibits the same analytical properties of the \eqref{eq:TV_reg} regulariser - note that \eqref{eq:WTV_reg} reduces to \eqref{eq:TV_reg} in the special case that $\,\alpha_i = \alpha \in \R_{++}$ for any $i$. In particular, 
Similarly as for the \eqref{eq:TV_reg} regulariser, \eqref{eq:WTV_reg} is still convex and non-differentiable.
%with the fixed global exponent $1$ of the gradient norms yielding potentially the same good sparsity-promoting effect as the TV regulariser on the solution gradients.
However, the sparsity-promoting effect can now be locally modulated thanks to the presence of the local weights $\alpha_i$. 
%This makes the WTV regulariser potentially (if weights $\alpha_i$ are suitably set) capable of adapting its effect to the local image contents. 
To highlight this feature, we report in the first column of Figure \ref{fig:regs_h}  (i.e. Figures \ref{fig:regs_h}(a),(d),(g))  the graphs of the WTV gradient penalty function $f_{\mathrm{WTV}}$ defined in \eqref{eq:WTV_h} for three different values $\alpha_j = 1$, $\alpha_k = 1.3$, $\alpha_l = 0.7$ of the scale parameter, respectively, assuming that they represent the local weights of the WTV regulariser at different pixel positions $i \in \{j , k , l\}$. These three graphs 
%- actually, the graphs associated to any penalty of the form in \eqref{eq:WTV_h} -
share the same inverted right-circular conical shape with vertex at the origin as the TV penalty drawn in Figure \ref{fig:TV_h}, with the one in Figure \ref{fig:regs_h}(a) coinciding with the TV penalty.
Different values of the weight yield different slopes of the conical lateral surface - note that $\left\| \bm{\nabla} f_{\mathrm{WTV}} \left(g_{i,1},g_{i,2}\right)\right\|_2 = \alpha_i$ for any $\left(g_{i,1},g_{i,2}\right) \in \R^2 \setminus \{(0,0)\}$ - and, hence, different local regularisation strengths. The larger (smaller) is the local weight $\alpha_i$, the more (less) strongly the WTV regulariser will force $\left\| \left(\bm{\D u}\right)_i\right\|_2$, to be small. 

Similarly as before, we show the 1D sections of the three WTV penalty functions along the two directions defined by angles $\varphi = 0$ (solid lines) and $\varphi = \pi/2$ (dashed lines), corresponding to the x- and y-axis in the 3D plots in Figures \ref{fig:regs_h}(a),(d),(g) and, for better readability, in Figure \ref{fig:regs_h}(j), left. Like TV, the WTV regulariser is isotropic, hence the two sections - actually, any section along straight lines passing through the origin - of each of the three penalties coincide. Despite their space-variant feature, these sections are in fact still rotationally invariant as they take the form
\begin{equation}
s_i(t;\varphi) \;{=}\; \alpha_i \left| t \right|, \quad t \in \R, \quad i = 1,\ldots,N \, .
\label{eq:WTV_sec}
\end{equation}
Finally, one can notice from \eqref{eq:WTV_sec} and from Figure \ref{fig:regs_h}(j), left, that all sections are nothing but positively-scaled versions  of the absolute value function of scale parameter $\alpha_i$, i.e. of the TV sections in \eqref{eq:TV_sec}.

Due to its ability of promoting local TV smoothing, we remark that the WTV regulariser has been proposed and studied in several papers (e.g. \cite{Hintermueller2016,Hintermuller2017a,HintPapaf2019} and many more) from an analytical point of view and  motivated by means of  analogous probabilistic arguments in \cite{CLPS}.

\subsubsection{Local regularisation sharpness}  \label{sec:sharpness}

As previously mentioned, the weights $\alpha_i$ in the \eqref{eq:WTV_reg} regulariser allow to locally tune the strength of the gradient-sparsity promotion effect of the regularisation which, by construction, is of fixed TV type.  In fact, the presence of a global exponent $1$ for the gradient norms in definition of WTV does not allow to change, neither globally nor locally, the \emph{sharpness} of the associated gradient penalty functions, hence the nature of the involved sparsity-promotion. 
%This can be a quite severe limitation, for instance in case of images containing smooth regions, where a smooth quadratic regularisation - i.e., exponent $2$ - is well-known to be much preferable as it inherently avoids any possible staircasing effect. Also in piece-wise constant regions, it is well-known that a TV type regularisation can yield the undesirable contrast loss effect, which can be mitigated by changing the exponent $1$ into an exponent less that $1$. 

This motivates the introduction of a second set of  space-variant parameters $p_i>0$, $i = 1,\ldots,N$, being them the exponents of the gradient norms in the \eqref{eq:WTV_reg} definition and corresponding to the local shape parameters of the associated hGG pdf - see Definition \ref{def:hGGd}.
%
%In fact, as outlined in Sec. \ref{subsec:prior}, from a Bayesian viewpoint this corresponds to assuming a non-stationary hGG - instead of hL as for \eqref{eq:WTV_reg} - prior for the gradient magnitudes of $\bm{u}$.

We proceed as above and compute the negative logarithm of the non-stationary hGG prior \eqref{eq:WTVp_prior}, thus getting
\begin{equation}
-\ln\mathbb{P}(\bm{u}\mid \bm{\Theta})
\;{=}\; 
\WTV^{\mathrm{sv}}_{\bm{p}}(\bm{u};\bm{\Theta})
-  \sum_{i=1}^N \ln \frac{\alpha_i p_i}{\Gamma(1/p_i)} \, , 
\label{eq:nlp_WTVp}
\end{equation}
where $\,$the $\,$space-variant $\,$WTV$_p^{\mathrm{sv}}\,$ regulariser, $\,$depending on the hyperparameters \mbox{$\bm{\Theta} =
	(\bm{\alpha},\bm{p})$,} is defined by
\begin{equation}
\WTV^{\mathrm{sv}}_{\bm{p}}(\bm{u};\bm{\alpha},\bm{p}) 
\,\;{:=}\;\, 
\sum_{i=1}^{N}\alpha_i^{p_i}\|(\bm{\D u})_i\|_2^{p_i},
\quad
(\bm{\alpha},\bm{p}) \in \R_{++}^{N\times 2}.
\label{eq:WTVp_reg}
\tag{\text{$\WTV^{sv}_{\bm{p}}$}}
\end{equation}
Like \eqref{eq:WTV_reg}, the \eqref{eq:WTVp_reg} regulariser can be rewritten in terms of a parametric, space-variant gradient penalty function, namely
\begin{equation}
\label{eq:WTVp_reg2}
\WTV_{\bm{p}}^{\mathrm{sv}}(\bm{u};\bm{\alpha},\bm{p}) 
\,\;{=}\;\, 
\sum_{i=1}^{N} f_{\WTV^{\mathrm{sv}}_{\bm{p}}}\left((\bm{\D u})_i;\alpha_i,p_i\right), \quad 
(\alpha_i,p_i) \in \R_{++}^2,
%\quad h(g_1,g_2) = \|(\bm{\D u})_i\|_2\,,
\end{equation}
where the function $f_{\WTV^{\mathrm{sv}}_{\bm{p}}}(\cdot)$ now reads:
\begin{equation}
f_{\WTV^{\mathrm{sv}}_{\bm{p}}}(\bm{g}_i;\alpha_i,p_i) \,\;{=}\;\, \alpha_i^{p_i} \left\| \bm{g}_i \right\|_2^{p_i} \, ,
\quad \bm{g}_i = (g_{i,1},g_{i,2}) \in \R^2 \, .
\label{eq:WTVp_h}
\end{equation}

Like \eqref{eq:TV_reg} and \eqref{eq:WTV_reg}, the \eqref{eq:WTVp_reg} regulariser is bounded below by zero, continuous and non-coercive. However, its other regularity properties depend on the actual values of the parameters $p_i$. If $p_i \geq 1$ for any $i = 1,\ldots,N$, then  WTV$_p^{\mathrm{sv}}$ is convex, whereas it is non-convex if there exists at least one $i$ such that $p_i < 1$. Then, it is differentiable whenever $p_i > 1$ for any $i$, otherwise it is non-smooth.

In the second column of Figure \ref{fig:regs_h} we show the graph of the WTV$_p^{\mathrm{sv}}$  gradient penalty function $f_{\WTV^{\mathrm{sv}}_{\bm{p}}}$ defined in \eqref{eq:WTVp_h} for three different parameter configurations $(\alpha_j,p_j) = (1,1)$, $(\alpha_k,p_k) = (1.3,0.5)$ and $(\alpha_l,p_l) = (0.7,2)$, where, we remark, the scale parameter values $\alpha_j$, $\alpha_k$, $\alpha_l$ are the same as for the WTV penalties reported in the first column. In case of unitary scale and shape parameters - see Figure \ref{fig:regs_h}(b) - the WTV$_p^{\mathrm{sv}}$ penalty coincides with the TV penalty and, more in general, for $p_j=p_k=p_l=1$, the WTV$_p^{\mathrm{sv}}$ penalties coincide with the WTV penalties. 
For non unitary shape parameters, the WTV$_p^{\mathrm{sv}}$ penalty function can assume different shapes, ranging from non-convex and non-differentiable ones ($p_i < 1$, Figure \ref{fig:regs_h}(e)), to strongly convex and differentiable ones ($p_i > 1$, Figure \ref{fig:regs_h}(h)). The degree of freedoms encoded by the shape parameters $p_i$ thus provide the  WTV$_p^{\mathrm{sv}}$ regulariser with the ability to adapt its gradient sparsity-promoting effect to the local image content. In particular, $p_i > 1$ - typically, $p_i \geq 2$ - should be used to avoid TV staircasing  in correspondence of smooth image regions, whereas $p_i < 1$ - typically, $p_i \leq 0.5$ - should be used in piece-wise constant regions to mitigate the undesirable contrast loss effect of TV. 

%To better understand the reasons of such potential benefits of the \eqref{eq:WTVp_reg} regulariser as well as to highlight its limitations, it is useful to 
Similarly as for the previous  \eqref{eq:TV_reg} and \eqref{eq:WTV_reg} regularisers, let us now take a look at the 1D sections of the WTV$_p^{\mathrm{sv}}$ penalty in \eqref{eq:WTVp_h} for $\varphi\in[-\pi,\pi)$, which read
\begin{equation}
s_i(t;\varphi) \;{=}\; \alpha_i^{p_i} \left| t \right|^{p_i}, \quad t \in \R, \quad i = 1,\ldots,N \, .
\label{eq:WTVp_sec}
\end{equation}
By looking at the plot of such sections shown in Figure \ref{fig:regs_h}(j), centre, it is clear how the value of parameter $p_i$ can substantially change the regularisation effect at each pixel. In particular, by comparing the red, black and magenta sections in Figure \ref{fig:regs_h}(j), one can notice that for $p_i < 1$ small gradients are induced to be zero in a stronger way than for $p_i = 1$, but large gradients are less penalised (weaker contrast loss effect). On the other hand, for $p_i > 1$ 
%the angular point at $t=0$ turns into a stationary point, hence 
the sparsity-promoting effect
%- if understood in strict sense - 
is no longer present as the gradient penalty function is differentiable in $t=0$. More  generally, for $p_i > 1$ small gradients are less penalised than for $p_i = 1$, whereas large gradients are more penalised. We finally remark that, like for TV and WTV, the WTV$_p^{\mathrm{sv}}$ sections in \eqref{eq:WTVp_sec} do not depend on the direction angle $\varphi$, hence the \eqref{eq:WTVp_reg} regulariser still falls in the class of isotropic regularisers. 
This is visually confirmed by the  WTV$_p^{\mathrm{sv}}$ penalties shown in Figsures \ref{fig:regs_h}(b),(e),(h), which are  rotationally invariant (i.e. have circular level curves), and by the penalty sections along the $x$-axis and the $y$-axis, which coincide as it is evident from Figure \ref{fig:regs_h}(j), centre.

The WTV$_{\bm{p}}^{sv}$ regulariser has been first introduced in a simplified version, i.e., with $\alpha_i= \alpha$, $\forall i$, and interpreted in a probabilistic framework in \cite{VIP}. The general case with space-variant weights has been discussed in \cite{CMBBE}.

% {\color{red}
% The WTV$_p^{\mathrm{sv}}$ regulariser has been proposed in \cite{vip,CMBBE}.

%LUCA: dire qualcosina in piu'......
%}

\begin{comment} 
we get
%
\begin{align}
\label{eq:WTVp_joint}
\tag{\text{$\WTV^{sv}_{\bm{p}}$}-\text{$\mathrm{L}_q$}}
\begin{split}
\left\{\bm{u}^*,\bm{\Theta}^*\right\}\;{\in}\;&\argmin{\bm{u},\bm{\Theta}}\Bigg\{\,\WTV^{sv}_{\bm{p}}(\bm{u};\bm{\Theta}) \;{-}\; N\ln\prod_{i=1}^{N}\frac{\alpha_i p_i}{\Gamma(1/p_i)}\\ &\phantom{XXXX}+\mu\,\mathrm{L}_q(\bm{\A u};\bm{b})\,\Bigg\} \,.
\end{split}
\end{align}
\end{comment}

%\textcolor{blue}{In the case of a fixed regularisation weight for all pixels that corresponds to:\begin{equation}  \label{eq:TV_pi}\alpha\mathrm{TV}_{p_i} = \alpha \sum_{i=1}^N  \| (\bm{ \D u})_i \|^{p_i}_{t}\qquad t\in\left\{1,2\right\},   \tag{\text{$\TV_{p}$}}\end{equation}where $p_i\geq 1$ at every pixel in order to model different regularisation smoothness (not weight).For greater flexibility, we could also let the local regularisation weights vary, so as to obtain a fully space-variant TV regularisation models with adaptive weight and regularisation, which reads:\begin{equation}  \label{eq:WTV_pi}\WTV_{p_i} (u):= \sum_{i=1}^N \alpha_i \| (\bm{ \D u})_i \|^{p_i}_{t}\qquad t\in\left\{1,2\right\}   \tag{\text{$\WTV_{p}$}}\end{equation}}

\subsubsection{Local anisotropy}
\label{sec:log_loc_dir}

As shown above, the \eqref{eq:TV_reg}, \eqref{eq:WTV_reg} and \eqref{eq:WTVp_reg} regularisers are isotropic.
% , in the sense that they all penalise - albeit in different manners - the $\ell_2$ norm of the gradients of the sought image which, in the continuous setting, is a rotation invariant. In other but equivalent words, the associated gradient penalty functions $h$ in (\ref{eq:TV_h}), (\ref{eq:WTV_h}), (\ref{eq:WTVp_h}) are all rotationally invariant, as evident from the circular level curves of the penalties shown in Fig.~\ref{fig:TV_h} and in the first two columns of Fig.~\ref{fig:regs_h}, respectively. 
For this reason, such regularisers are not able to exploit any information on the directionality of local image structures and, hence, to drive their local nonlinear diffusion effect along specific directions only. As motivated in Section \ref{subsec:globgen}, this can be a limitation, especially for images presenting local structures characterised by well-defined orientations. 
%For this reason, it is useful to introduce two further sets of space-variant parameters $\theta_i$ and $a_i$, $i = 1,\ldots,N$, representing the dominant local image directions and the relative regularisation strengths along the orthogonal directions $\theta_i + \pi/2$, respectively.
As illustrated in Section \ref{subsec:prior}, to circumvent this limitation, a non-stationary BGG prior can be assumed for modelling the local distribution of gradients of $\bm{u}$.
%\ldots local frames, anisotropy. Specify the fact that re-defining the frame wrt to tangent/normal component is just a local change of basis

\begin{comment}

\begin{equation}  \label{eq:DTV_dir}
\DTV (\bm{u}):= \sum_{i=1}^N  \| \bm{\Lambda}_i \bm{\mathrm{R}}_{\theta_i} (\bm{\D u})_i \|_{2} \tag{\text{$\DTV$}}
\end{equation}

This modelling allows to define an even more flexible regulariser, based on local directional properties of the image, which, however, in order to be made more effective, shall also be combined with a local adaptation of regularisation smoothness.
\end{comment}

By computing the negative logarithm of the non-stationary BGG prior in \eqref{eq:propreg}, we have
\begin{equation}
-\ln\mathbb{P}(\bm{z}(\bm{u})\mid \bm{\Theta})
\;{=}\; 
\WDTV^{\mathrm{sv}}_{\bm{p}}(\bm{u};\bm{\Theta})
-  \sum_{i=1}^N \ln \left(\frac{\alpha_i^2 \, p_i\,a_i}{\Gamma(2/p_i)\,2^{2/p_i}}\right)
+ N \ln (2\pi)
\, , 
\label{eq:nlp_WDTVp}
\end{equation}
where the space-variant WDTV$_p^{\mathrm{sv}}$ regulariser is defined in terms of the hyperparameters $\,\bm{\Theta} = \left(\bm{\alpha},\bm{p},\bm{\theta},\bm{a}\right)$ and reads
\begin{align} 
\label{eq:WDTVp_reg}
\tag{\text{$\WDTV^{sv}_{\bm{p}}$}}
\WDTV_{\bm{p}}^{\mathrm{sv}}(\bm{u};\bm{\alpha},\bm{p},\bm{\theta},\bm{a}) 
\,\;{=}\;\, 
\sum_{i=1}^{N} 
\alpha_i^{p_i} \left\| \bm{\Lambda}_{a_i}\bm{\mathrm{R}}_{-\theta_i}(\bm{\D u})_i \right\|_2^{p_i}, \quad
\\
\left(\bm{\alpha},\bm{p},\bm{\theta},\bm{a}\right)
\,\;{\in}\;\,
\R_{++}^{N\times 2} \times [-\pi/2,\pi/2)^N \times (0,1]^N\,, \quad\quad\;\,
\nonumber
\end{align}
where the orthogonal (rotation) matrices $\bm{\mathrm{R}}_{-\theta_i}$ and the diagonal matrices $\bm{\Lambda}_{a_i}$ have been defined in \eqref{siginvdec} and (\ref{eq:comp2}), respectively.

Note that the \eqref{eq:WDTVp_reg} regulariser can also be written in terms of its parametric, space-variant gradient penalty functions as
\begin{eqnarray}
\label{eq:WTDVp_reg2}
\WDTV_{\bm{p}}^{\mathrm{sv}}(\bm{u};\bm{\alpha},\bm{p},\bm{\theta},\bm{a}) 
&\,\;{=}\;\,& 
\sum_{i=1}^{N} f_{\WDTV_{\bm{p}}^{\mathrm{sv}}}\left((\bm{\D u})_i;\alpha_i,p_i,\theta_i,a_i\right), \\ 
(\alpha_i,p_i,\theta_i,a_i) 
&\,\;{\in}\;\,&  \R_{++}^2\times [-\pi/2,\pi/2)\times(0,1]\,, \nonumber
%\quad h(g_1,g_2) = \|(\bm{\D u})_i\|_2\,,
\end{eqnarray}
with
\begin{equation}
f_{\WDTV_{\bm{p}}^{\mathrm{sv}}}(\bm{g}_i;\alpha_i,p_i,\theta_i,a_i) \,\;{=}\;\, \alpha_i^{p_i} \left\| \bm{\Lambda}_{a_i}\bm{\mathrm{R}}_{-\theta_i}\bm{g}_i \right\|_2^{p_i} \, ,
\quad \bm{g}_i = (g_{i,1},g_{i,2}) \in \R^2 \, .
\label{eq:WDTVp_h}
\end{equation}

%\begin{equation}  
%\!\!\!\!\!\!\!\!\!\!\!\!\!\!
%\WDTV^{sv}_{\bm{p}} (\bm{u};\bm{\Theta})
%\;{:=}\;
%\sum_{i=1}^N \alpha_i^{p_i} \| \bm{\Lambda}_{a_i} \bm{\mathrm{R}}_{-\theta_i} (\bm{\D u})_i \|^{p_i}_{2},
%\label{eq:DTV_p}
%\tag{\text{$\WDTV^{sv}_{\bm{p}}$}}
%\end{equation}
%
%%\begin{equation} 
%\qquad\text{with}\quad \, \bm{\Theta} \;{:=}\;
%(\boldsymbol{\alpha},\boldsymbol{p},\boldsymbol{\theta},\boldsymbol{a})\in\mathcal{D}_{\bm{\Theta}} = \R_{++}^{N\times 2} \times [0,\pi)^N \times (0,1]^N
%\end{equation}
%

Since $a_i \in (0,1]$ for any $i$, matrices $\bm{\M}_i := \bm{\Lambda}_{a_i} \bm{\mathrm{R}}_{-\theta_i} \in \R^{2 \times 2}$ are all non-singular. As a consequence, the \eqref{eq:WDTVp_reg} regulariser shares the same analytical properties as the \eqref{eq:WTVp_reg} regulariser. In particular, it is worth noting that the \eqref{eq:WDTVp_reg} regulariser reduces to the rotationally-invariant \eqref{eq:WTVp_reg} regulariser in the special case  $\,a_i = 1$ for any $i$, independently of the directionality parameters $\theta_i$.

In the last column of Figure \ref{fig:regs_h} we show the graph of the WDTV$_{\bm{p}}^{\mathrm{sv}}$ gradient penalty function $f_{\WDTV_{\bm{p}}^{\mathrm{sv}}}$ defined in \eqref{eq:WDTVp_h} for three different parameter configurations $(\alpha_j,p_j,\theta_j,a_j) = (1,1,0,1)$, $(\alpha_k,p_k,\theta_k,a_k) = (1.3,0.5,\pi/6,0.4)$ and $(\alpha_l,p_l,\theta_l,a_l) = (0.7,2,\pi/3,0.6)$,  where the scale and shape parameter values $(\alpha_j,p_j)$, $(\alpha_k,p_k)$ and $(\alpha_l,p_l)$ are the same as for the WTV$_{\bm{p}}^{\mathrm{sv}}$ penalties (second column). It is clear from these figures that the degrees of freedom represented by parameters $a_i$ allow to make the WDTV$_{\bm{p}}^{\mathrm{sv}}$ regulariser locally anisotropic, in the sense that it can locally 
%- i.e, independently at any pixel $i$ - 
penalise the gradient $\bm{g}_i = \left(\bm{\D u}\right)_i$ with different strength 
%(not with different sharpness, fixed by the local shape parameter $p_i$) 
according to its direction. %In fact, as expected from definition \eqref{eq:WDTVp_h}, 
The level curves of the penalties in Figures \ref{fig:regs_h}(f),(i) corresponding to $a_i < 1$ are elliptical and not circular as for the case $a_i = 1$ in Figure \ref{fig:regs_h}(c). Furthermore, the smaller $a_i$, the more eccentric the ellipses and, hence, the more anisotropic the regulariser. %We thus refer to $a_i$ as the local anisotropy parameters. Then, 
The local directional parameters $\theta_i$ represent local image directions along which a stronger regularisation effect is typically desired (typically, edge direction). We observe that the elliptical level curves of the penalties in Figures \ref{fig:regs_h}(f),(i) are rotated of angle $\theta_i$ counterclockwise, with the minor and major axes aligned along the directions defined by $\theta_i$ and $\theta_i + \pi/2$, respectively, and that the 1D sections of the WDTV$_{\bm{p}}^{\mathrm{sv}}$ penalty functions in \eqref{eq:WDTVp_h} along directions defined by angle $\varphi$ take the form
\begin{equation}
s_i(t;\varphi) \;{=}\; 
\left(
\cos^2(\varphi - \theta_i) + a_i^2 
\sin^2(\varphi - \theta_i) 
\right)^{p_i/2} \, 
\alpha_i^{p_i} \left| t \right|^{p_i}, \quad t \in \R, \quad i = 1,\ldots,N \, .
\label{eq:WDTVp_sec}
\end{equation}
It is a simple calculation verifying that, for any fixed $\theta_i$, $a_i$, the positive real coefficient in brackets takes its maximum (equal to 1) and minimum (equal to $a_i^{p_i}$) values for $\varphi = \theta_i$ and $\varphi = \theta_i + \pi/2$, respectively. This entails that the 1D sections of the WDTV$_{\bm{p}}^{\mathrm{sv}}$ penalty along the dominant direction $\theta_i$ and its orthogonal $\theta_i + \pi/2$ are those characterised by the strongest and the weakest regularisation effect, respectively. Note that the sections exhibit the same sharpness - i.e., the same shape - but they are differently scaled; see the pairs of solid/dashed red and magenta curves in Figure \ref{fig:regs_h}(j), right. 

The WDTV$_{\bm{p}}^{sv}$ regulariser has been first introduced and analysed in probabilistic settings in \cite{DTVp_siam}.

% {\color{red}
% The WDTV$_{\bm{p}}^{\mathrm{sv}}$ regulariser has been proposed in \cite{PANG2020} and in \cite{DTVp_siam}, where it has been motivated by means of the same probabilistic argument as in Sec.~\ref{subsec:prior}.

% LUCA: expand a bit$\ldots$

% }

\subsubsection{Comparing regularisers: proximal operators}  \label{sec:prox}

In order to gain more insights on the regularisation effects yielded by the different gradient penalty functions introduced in the previous sections and, consequently, on the different space-variant regularisers considered, we compare in this section the proximal operators $\mathrm{prox}_f^{\beta}: \R^2 \rightrightarrows \R^2$, $\,f \in \big\{ f_{\mathrm{TV}},$ $ f_{\mathrm{WTV}}, f_{\mathrm{WTV}_p^{\mathrm{sv}}}, f_{\mathrm{WDTV}_p^{\mathrm{sv}}} \big\}$ (see Definition \ref{def:proxx}) associated to the TV, WTV, WTV$^{sv}_{\bm{p}}$ and WDTV$_{\bm{p}}^{sv}$ penalty functions defined in (\ref{eq:TV_h}), (\ref{eq:WTV_h}), (\ref{eq:WTVp_h}) and (\ref{eq:WDTVp_h}), respectively.

% - see Definition \ref{def:proxx} - namely the (possibly set-valued) function $\,\mathrm{prox}_f^{\beta}: \R^2 \rightrightarrows \R^2$ defined by
% %
% \begin{equation}
% \mathrm{prox}_f^{\beta}(\bm{w})
% \,\;{=}\;\,
% \argmin{\bm{g} \in \R^2}
% \left\{
% f(\bm{g})
% + \frac{\beta}{2} \left\| \bm{g} - \bm{w} \right\|_2^2
% \right\}, \quad
% \bm{w} \in \R^2 \, .
% \label{eq:prox_init}
% \end{equation}
%
Having fixed the value of parameter $\beta \in \R_{++}$, and regarding $\bm{w}$ as an input image gradient vector to be regularised, the 2D vector field $\,\bm{e}_f^{\beta}: \R^2 \rightrightarrows \R^2$ defined by
\begin{equation}
\bm{e}_f^{\beta}(\bm{w}) 
\,\;{:=}\;\,
\mathrm{prox}_f^{\beta}(\bm{w})
\;{-} \bm{w}, \quad\;
\bm{w} \in \R^2 \, ,
\label{eq:prox_ef}
\end{equation}
can be studied to represent the regularisation effect of the gradient penalty function considered on $\bm{w}$.

Analytical expressions of the proximal operators of the gradient penalty functions in (\ref{eq:TV_h}), (\ref{eq:WTV_h}), (\ref{eq:WTVp_h}) and (\ref{eq:WDTVp_h}) associated to the (\ref{eq:TV_reg}), (\ref{eq:WTV_reg}), (\ref{eq:WTVp_reg}) and (\ref{eq:WDTVp_reg}) regularisers have been previously studied in \cite{HWTV}, \cite{tvpl2} and \cite{DTVp_siam}, respectively, and are discussed (for completeness) in Section \ref{sec:t_admm} of this review.
%, where they are used for solving one of the ADMM optimisation sub-problems.
%
Based on those expressions, we thus compute the vector field $\bm{e}_f^{\beta}$ in \eqref{eq:prox_ef} for each of the nine penalty functions considered in Figure \ref{fig:regs_h} and report the results in Figure \ref{fig:regs_prox}. In order to allow for a meaningful comparison between penalties, the same proximal parameter value $\beta = 3$ has been used.

\begin{figure}[!t]
	\centering
	\renewcommand{\arraystretch}{0.0}
	\renewcommand{\tabcolsep}{0.1cm}
	\begin{tabular}{ccccc}
		\!\!\!\!\!\!\!\!&\!\!\!\!& $\;\;\,$ WTV & $\quad$ WTV$_p^{\mathrm{sv}}$ & $\quad\:$ WDTV$_p^{\mathrm{sv}}$
		\vspace{0.3cm}\\ 
		\!\!\!\!\!\!\!\!{\rotatebox{90}{$\qquad\qquad\;\;$ pixel $j$}}&\!\!\!\!&
		%$i$&
		\includegraphics[height=4cm]{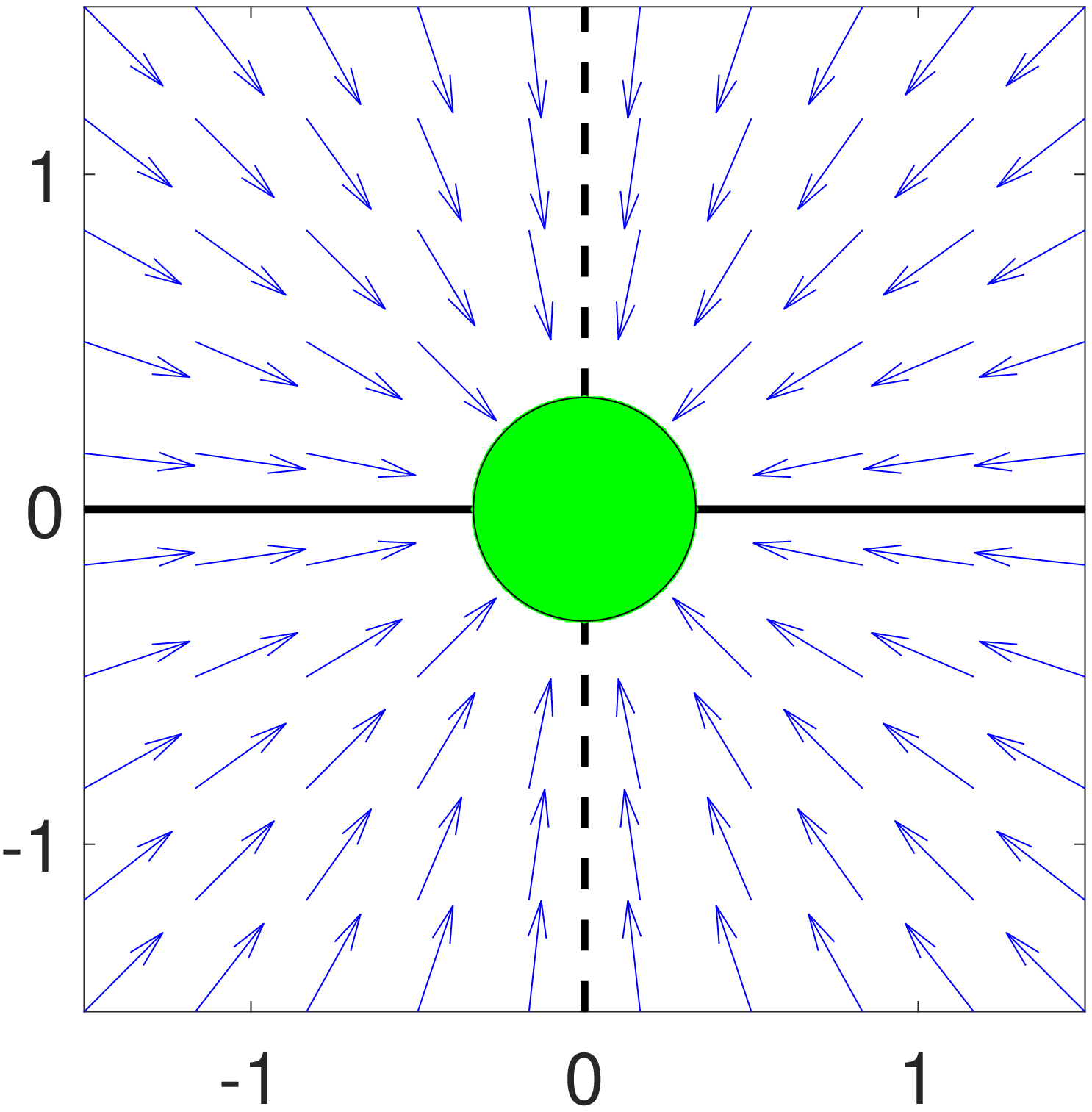}&
		\includegraphics[height=4cm]{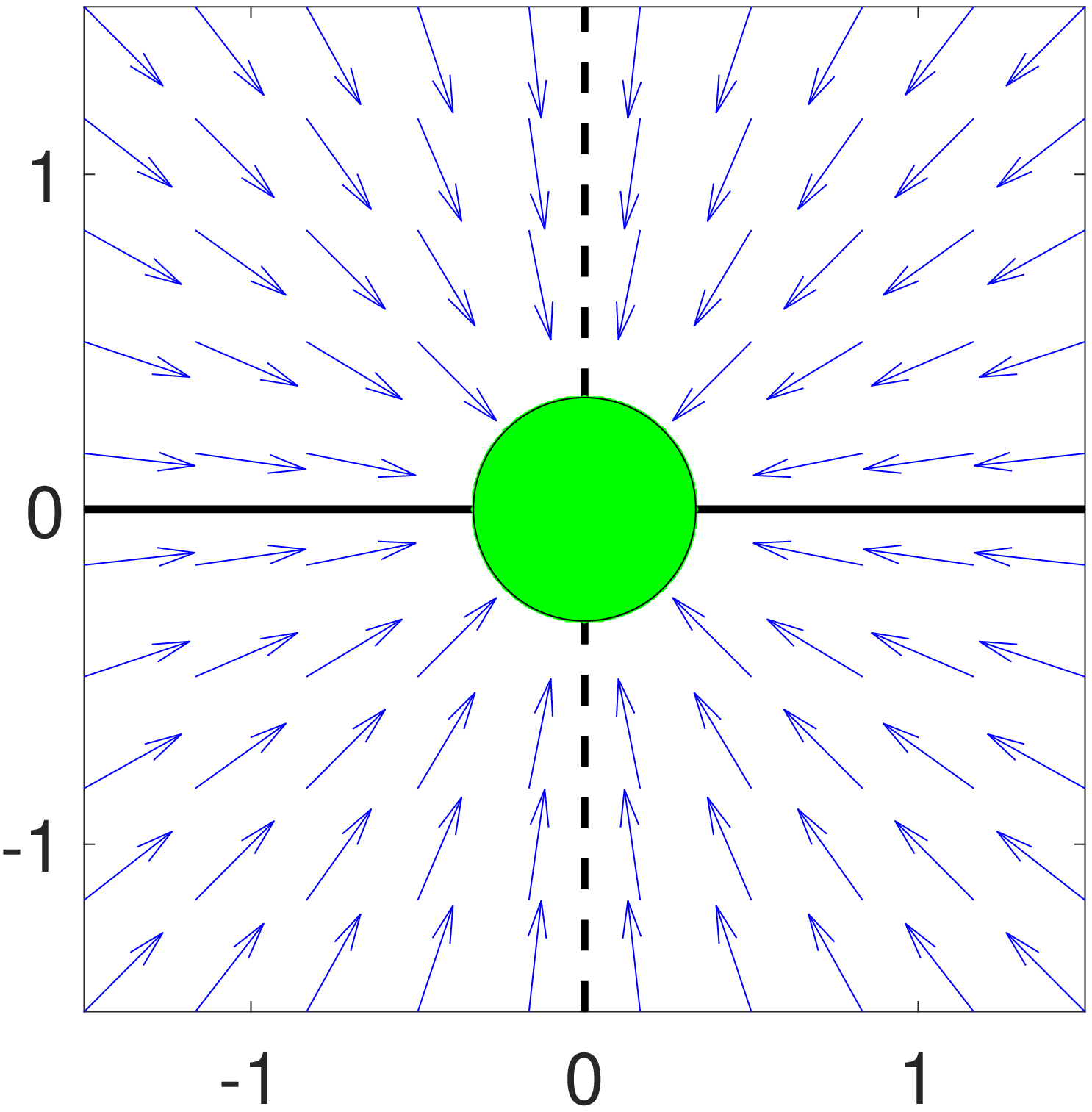}&
		\includegraphics[height=4cm]{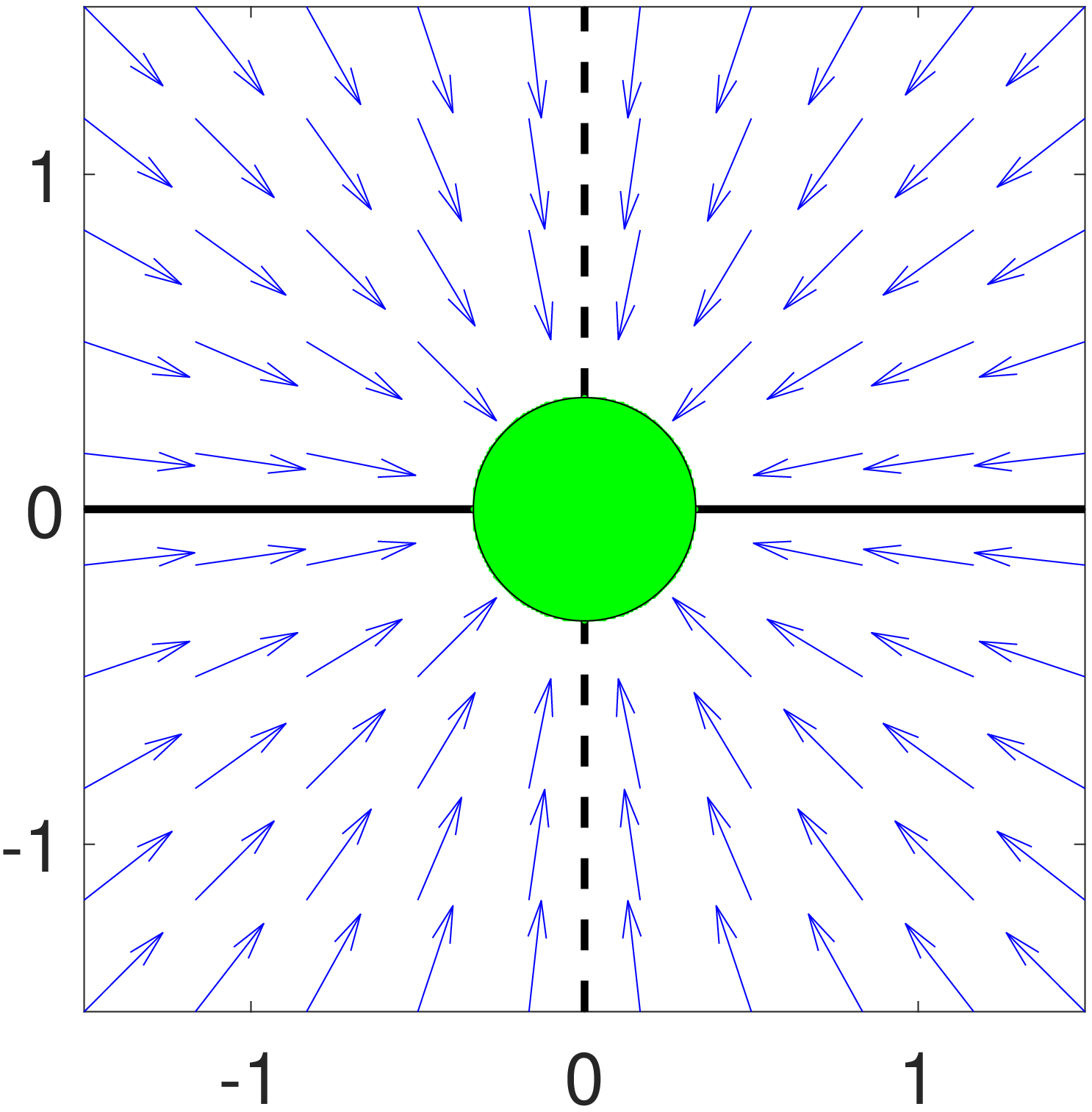}
		\vspace{0.1cm}\\
		\!\!\!\!\!\!\!\!&\!\!\!\!& $\;$ \footnotesize{(a) $\alpha_j{=}1.0$} &
		{\footnotesize (b) $(\alpha,p)_j{=}(1.0,1.0)$ }&\footnotesize{(c)\!\! $(\alpha,p,\theta,a)_j{=}(1.0,1.0,0,1.0)$ }
		%$\;\:$ (c) $\,a_j \;{=}\; 1.0$, $\theta_j \;{=}\; 0^{\circ}$ 
		\vspace{0.3cm}\\
		%\!\!\!\!\!\!\!\!&\!\!\!\!& $\;$ \small{(a) $\,\alpha_j{=}1.0$} &
		%\small{(b) $\,\alpha_j{=}1.0,p_j{=}1.0$ }&
		%$\;\:$ (c) $\,a_j \;{=}\; 1.0$, $\theta_j \;{=}\; 0^{\circ}$ 
		%\vspace{0.3cm}\\
		\!\!\!\!\!\!\!\!{\rotatebox{90}{$\qquad\qquad\;\;$ pixel $k$}}&\!\!\!\!&
		%$j$&
		\includegraphics[height=4cm]{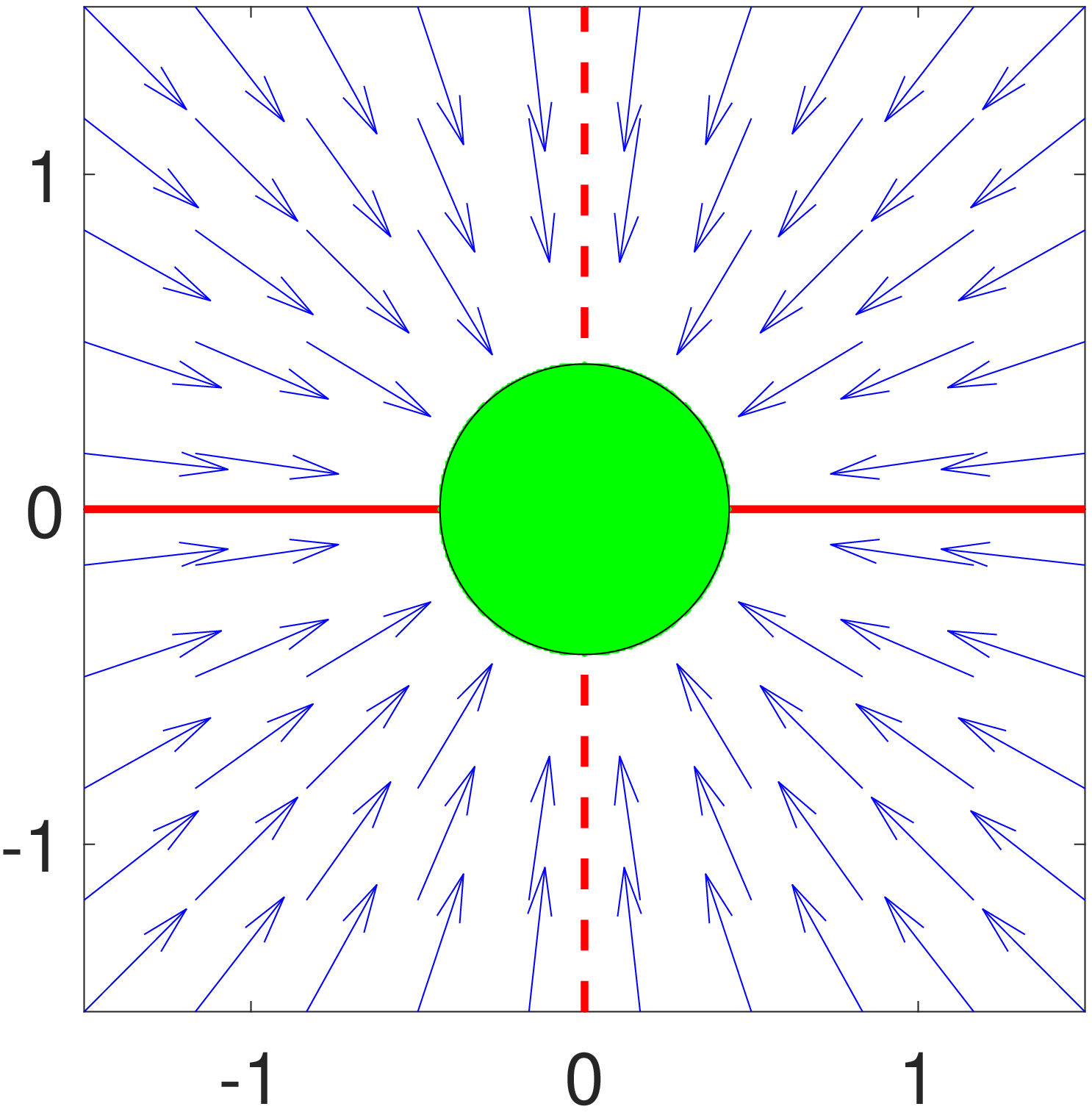}&
		\includegraphics[height=4cm]{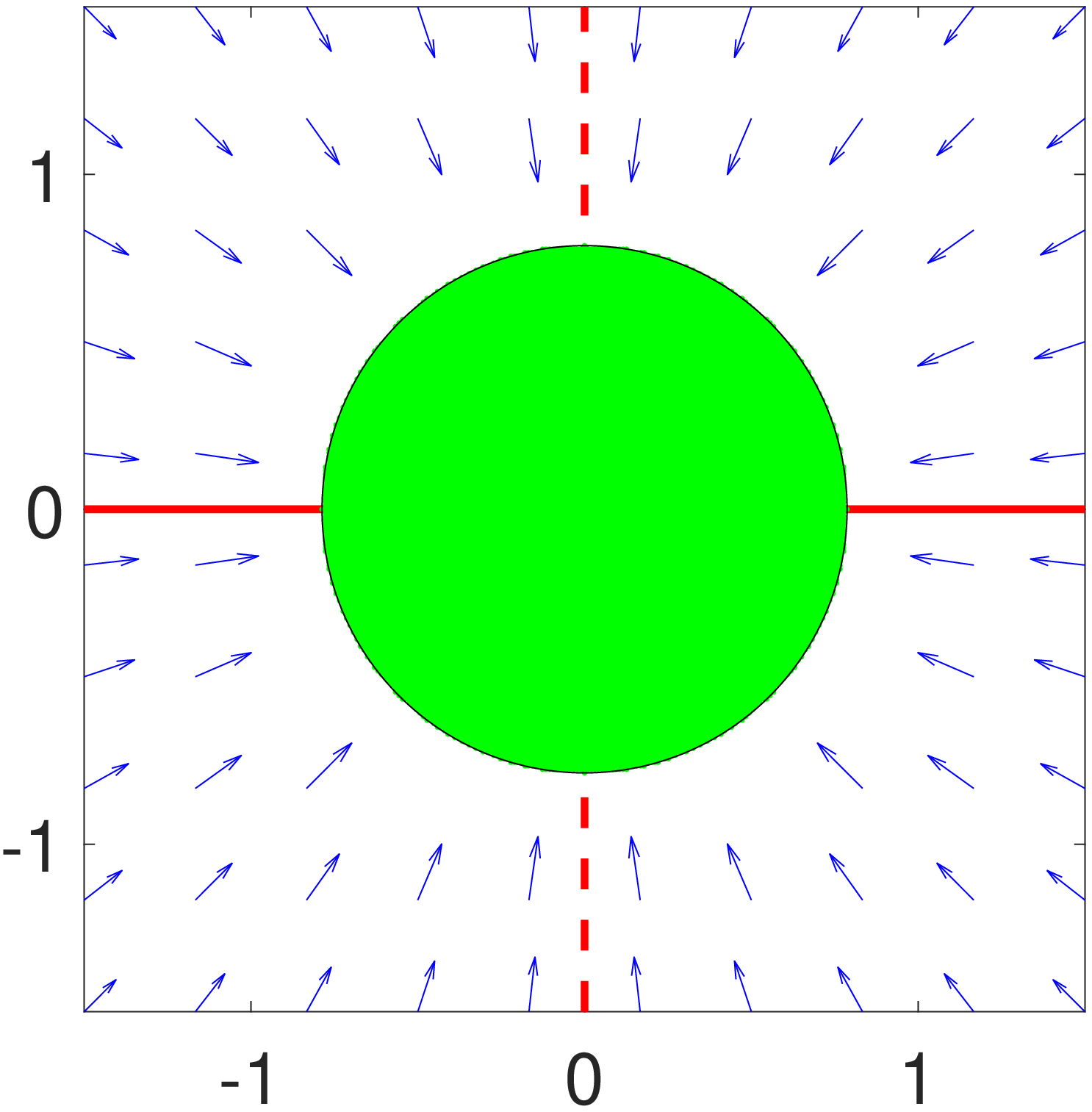}&
		\includegraphics[height=4cm]{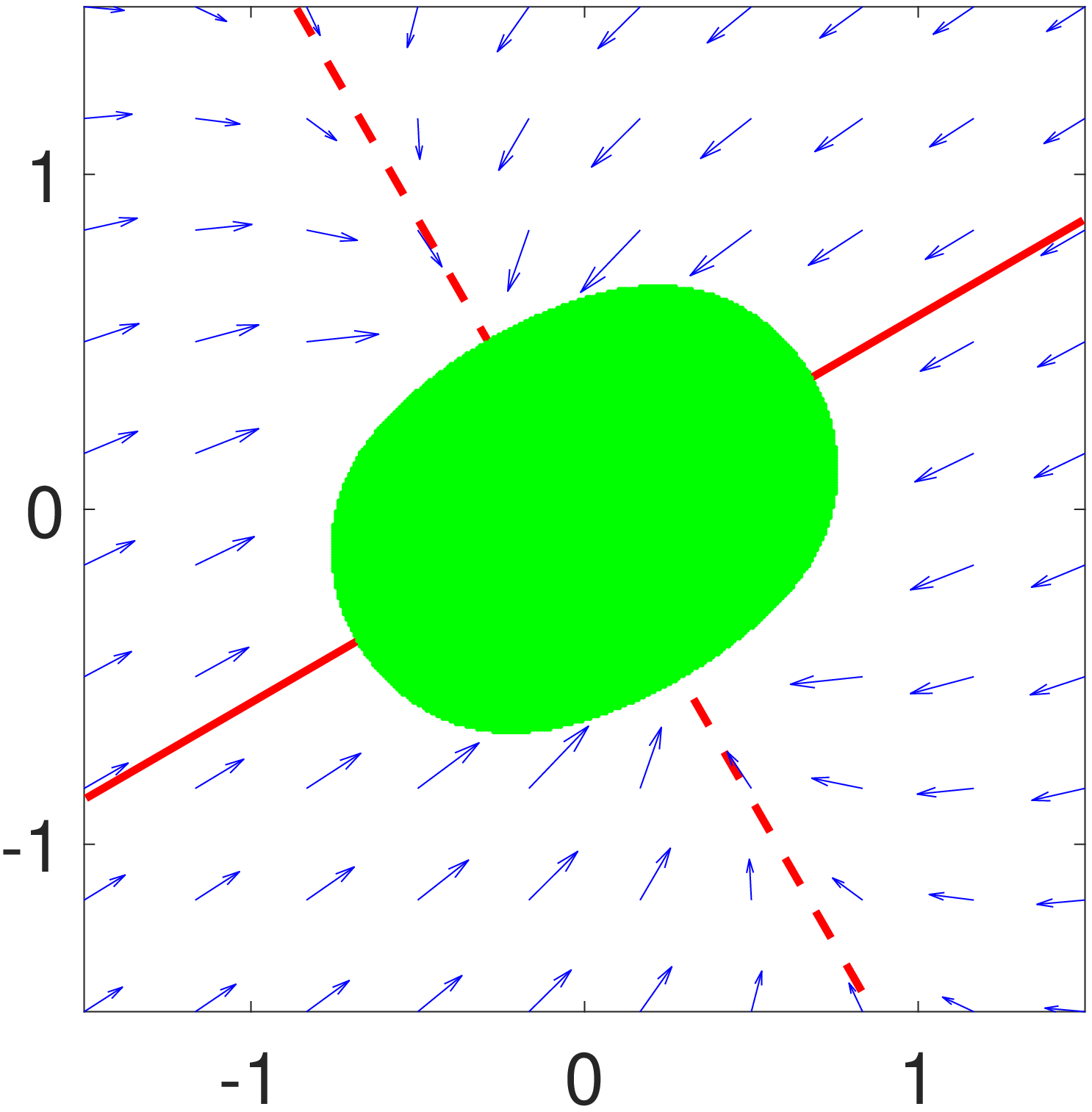}
		\vspace{0.1cm}\\
		\!\!\!\!\!\!\!\!&\!\!\!\!& \footnotesize{ (d) $\alpha_k {=} 1.3$} &\footnotesize{
			(e) $(\alpha,p)_k{=}(1.3,0.5)$} &\footnotesize{
			(f) \!\!$(\alpha,p,\theta,a)_k{=}(1.3,0.5,\frac{\pi}{6},0.4)$}
		\vspace{0.3cm}\\
		%\!\!\!\!\!\!\!\!&\!\!\!\!& $\;$ (d) $\,\alpha_k \;{=}\; 1.3$ &
		%$\;$ (e) $\,p_k \;{=}\; 0.5$ &
		%$\;\:$ (f) $\,a_k \;{=}\; 0.4$, $\theta_k \;{=}\; 30^{\circ}$
		%\vspace{0.3cm}\\
		\!\!\!\!\!\!\!\!{\rotatebox{90}{$\qquad\qquad\;\;$ pixel $l$}}&\!\!\!\!&
		%$k$&
		\includegraphics[height=4cm]{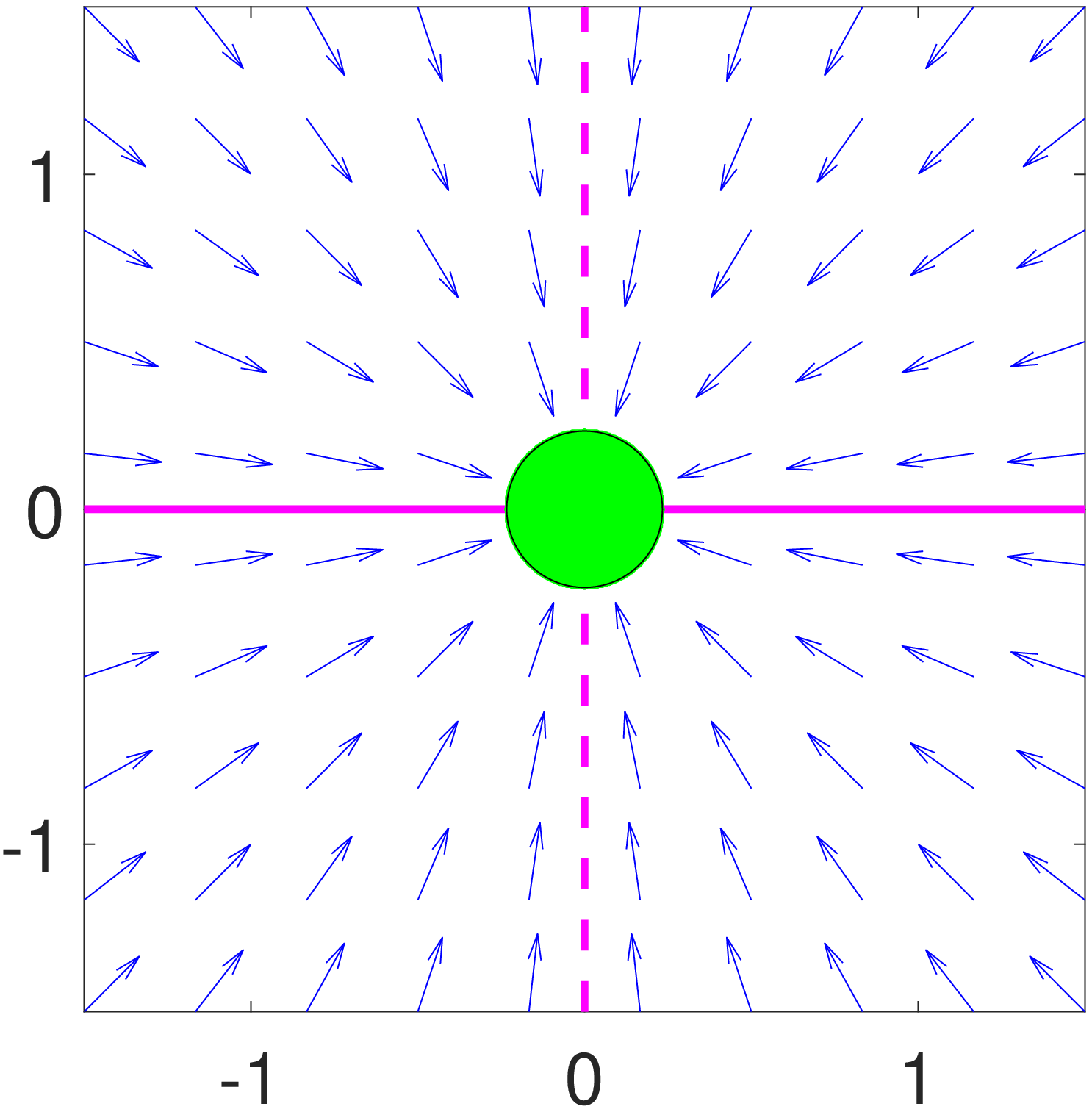}&
		\includegraphics[height=4cm]{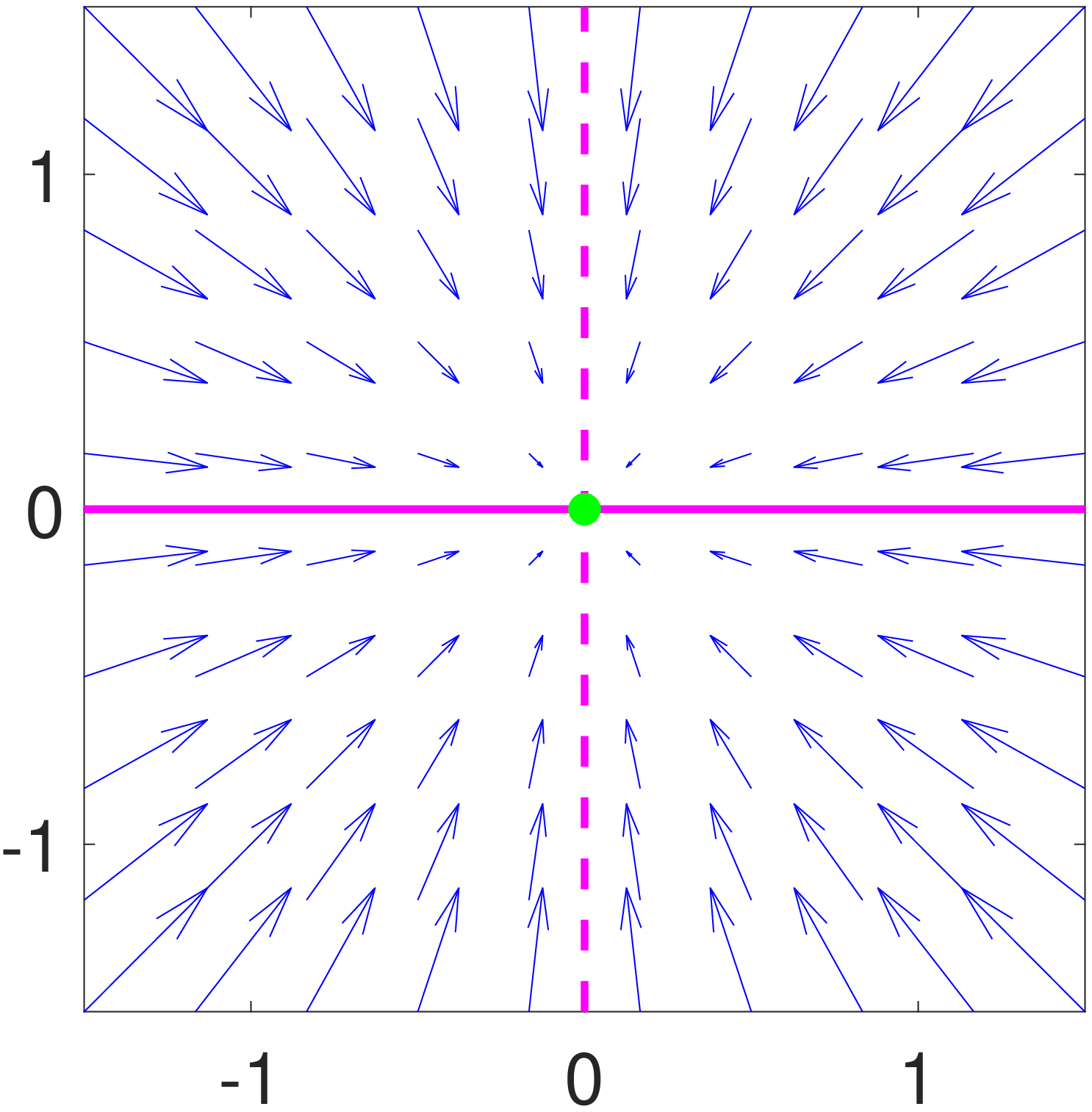}&
		\includegraphics[trim={1.07cm 1.08cm 1.37cm 1.57cm},clip,height=4cm]{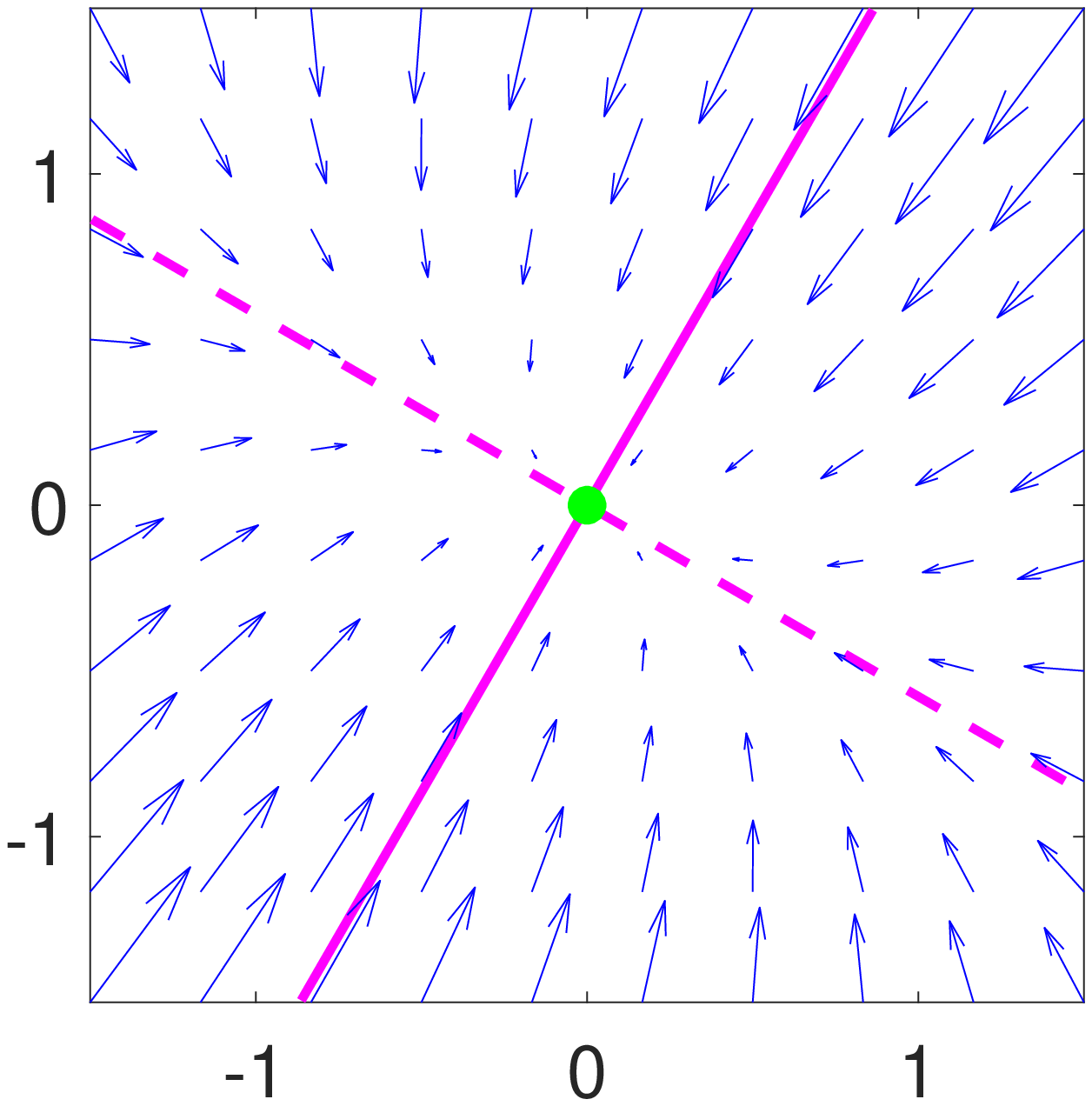}
		\vspace{0.1cm}\\
		\!\!\!\!\!\!\!\!&\!\!\!\!& \footnotesize{(g) $\alpha_l {=} 0.7$} &\footnotesize{(h) $(\alpha,p)_l{=}(0.7,2)$ }& \footnotesize{
			(i)\!\! $(\alpha,p,\theta,a)_l{=}(0.7,2,\frac{\pi}{3},0.6)$}
	\end{tabular}
	\caption{2D vector fields $\bm{e}_f^{\beta}: \R^2 \rightrightarrows \R^2$ in (\ref{eq:prox_ef}), representing the effect of the proximal operator $\mathrm{prox}_f^{\beta}$ on a (gradient) vector $\bm{w} = (w_1,w_2)$, for the same gradient penalty functions $f$ shown in Figures \ref{fig:regs_h}(a)-(i) and for a fixed proximity parameter $\beta = 3$.}
	\label{fig:regs_prox}
\end{figure}

We first remark that for all the considered gradient penalty functions $f$ and parameters $\beta>0$ we have $\mathrm{prox}_f^{\beta}(\bm{0}) = \bm{0} \,\;{\Longrightarrow}\;\, \bm{e}_f^{\beta}(\bm{0}) = \bm{0}$ and that for all penalties with shape parameter $p_i \leq 1$ - namely, the penalties in Figures \ref{fig:regs_h}(a)-(g) - there exists a region in the $\bm{w}$ domain (with centre the origin, size depending on $\beta$ and shape depending on the penalty itself) for which we have $\mathrm{prox}_f^{\beta}(\bm{w}) = \bm{0} \,\;{\Longrightarrow}\;\, \bm{e}_f^{\beta}(\bm{w}) = -\bm{w}$. This means that any input gradient vector $\bm{w}$ belonging to this region is ``completely'' regularised by the gradient penalty function, in the sense that it is proximal-mapped to the null gradient vector. 
%This property is related to the well-known sparsity-promoting capability of the \mbox{$\ell_p$-norm} regularisers with $p \leq 1$. 
For visualisation purposes, such sparsity-promoting regions are depicted in green (without showing the arrows pointing towards the origin) in the vector field representations of Figures  \ref{fig:regs_prox}(a)-(g).

As expected, for isotropic penalties - namely, the WTV and WTV$_p^{\mathrm{sv}}$ penalties shown in the first two columns of Figure \ref{fig:regs_h} and the WDTV$_p^{\mathrm{sv}}$ penalty with unitary anisotropy parameter depicted in Figure \ref{fig:regs_h}(c) - the associated vector fields $\bm{e}_f^{\beta}$ are radial with vectors pointing towards the origin and the sparsity-promotion regions are circularly shaped - see  Figures  \ref{fig:regs_prox}(a)-(e),(g). This means that the regularisation effect yielded by the isotropic penalties depicted in Figures \ref{fig:regs_h}(a)-(e),(g),(h) on $\bm{w}$ is only a shrinkage of its norm $\|\bm{w}\|_2$, namely
\begin{equation}\nonumber
\,\mathrm{prox}_f^{\beta}(\bm{w}) = \xi \, \bm{w} \;{\Longrightarrow}\; \bm{e}_f^{\beta}(\bm{w}) = -(1-\xi) \, \bm{w},
\end{equation}
with a shrinkage coefficient $\xi \in [0,1)$ only depending on the norm itself. This result has been proved, e.g., in \cite{tvpl2}(Proposition 1), where analytical expressions for $\xi$ as a function of $\|\bm{w}\|_2$ as well as of the shape and proximal parameters have been given for the proximal operator of a TV$_p$ penalty of the form $f(\bm{g};p) = \| \bm{g} \|_2^p$. Since it follows immediately from Definition \ref{def:proxx} that 
\begin{equation}\nonumber
\mathrm{prox}_f^{\beta}(\bm{w}) = \mathrm{prox}_{\widetilde{f}}^{\widetilde{\beta}}(\bm{w}),\;\quad\forall \bm{w} \in \R^2\,
\end{equation}
with $\widetilde{f}(\bm{g};\alpha,p) = \alpha^p \| \bm{w} \|_2^p$ and $\widetilde{\beta} = \alpha^p \beta,$, 
then the results in \cite{tvpl2} can be straightforwardly extended to cover the more general case of a WTV$_p$ penalty. These results provide an analytical interpretation of the visual results reported in Figures \ref{fig:regs_h}(a)-(e),(g),(h). In particular, by observing the vector fields depicted in these Figures, it is clear how larger scale parameter values $\alpha_i$ in the WTV penalty yield stronger gradient shrinkage effects as well as sparsity-promoting regions of larger radii. Then, by comparing the vector fields in Figures \ref{fig:regs_h}(d),(e), one can notice that, for a fixed scale parameter $\alpha_i$, decreasing the shape parameter $p_i$ (starting from $p_i = 1$) in the WTV$_p^{\mathrm{sv}}$ penalty yields weaker shrinkage effects on gradients outside the sparsity-promoting regions but larger radii of these regions. Finally, Figures \ref{fig:regs_h}(g),(h) show that increasing $p_i$ (for a fixed $\alpha_i$ and starting from $p_i = 1$) in the WTV$_p^{\mathrm{sv}}$ penalty yields stronger gradient shrinkage effects and, for any $p_i > 1$, the sparsity-promotion regions reduce to the pont $\bm{w} = \bm{0}_2$.

Clearly, the vector fields in Figures \ref{fig:regs_prox}(f),(i), associated to the WDTV$_p^{\mathrm{sv}}$ anisotropic penalties are not radial. To be more precise, they are radial only when restricted to input vectors $\bm{w}$ lying on the two straight lines having direction defined by angles $\theta_i$ and $\theta_i + \pi/2$ (solid/dashed red and magenta lines in Figures \ref{fig:regs_prox}(f),(i)). In general, the regularisation effect of the WDTV$_p^{\mathrm{sv}}$ penalties on input vectors $\bm{w}$ is stronger along the direction $\theta_i$. 
Finally, the  sparsity-promotion regions 
%- which, we recall, are larger than the origin $\bm{w} = \bm{0}$ for $p_i \leq 1$, hence only in the case of Fig. \ref{fig:regs_h}(f) - 
are elongated in the direction defined by $\theta_i$ and their elongation is negatively correlated with the value of the local anisotropy parameter $a_i$.

\begin{comment}
\begin{equation}\label{eq:DTVp_joint}\tag{\text{$\WDTV^{sv}_{\bm{p}}$}-\text{$\mathrm{L}_q$}}
\begin{split}
\left\{\bm{u}^*,\bm{\Theta}^*\right\}\;{\in}\;&\argmin{\bm{u},\bm{\Theta}}\Bigg\{ \WDTV^{sv}_{\bm{p}}(\bm{u}) - N\ln\prod_{i=1}^{N}\Bigg(\frac{1}{2 \pi |\bm{\Sigma}_i|^{1/2}}\times\\
&\phantom{XXXX}\times\frac{p_i}{\Gamma(1/p_i) \, 2^{\:\!1/p_i}}\Bigg)
+\mu\,\mathrm{L}_q(\bm{\A u};\bm{b})\Bigg\} \\
&\text{   with } \bm{\Theta }= (\bm{\lambda^{(1)}},\bm{\lambda^{(2)}},\bm{p},\bm{\theta})\in\R_{++}^{N\times 3}\times[0,\pi)^{N}\,.
\end{split}
\end{equation}
\end{comment}

\section{Geometrical interpretation}
\label{sec:geom}

In this section, we enrich the statistical and analytical study of the space-adaptive regularisers introduced in the previous sections by providing some insights helpful to understand their local behaviour from a geometrical point of view. To do so, we unify and expand some considerations from \cite{BayramDTV2012,KonDonKnuDTGV19,ParMasSch18applied,Demircan2020} and start recalling the dual definition of TV:
%on $\bm{\Omega} = \left\{ (i,j): i=1,\ldots,n_1,~ j=1,\ldots, n_2 \right\}\subset\mathbb{R}^2$ where the non-vectorised representation of $\bm{u}:\bm{\Omega}\to\mathbb{R}$ is preferred:
\begin{align}  \label{eq:TV}
\mathrm{TV}(\bm{u}) 
%& = \| \bm{ \D u} \|_{2,1} = \sum_{i=1,~j=1}^{n_1, n_2} \| (\bm{ \D u})_{i,j} \|_2   
= \sum_{i=1}^{N}  \max_{\bm{w}_i\in \mathcal{B}_1(\bm{0})} \langle (\bm{ \D u})_{i}, \bm{w}_{i}\rangle,
\end{align}
where $\mathcal{B}_1(\bm{0})$ denotes the two-dimensional Euclidean unit ball centred in the origin. %It corresponds to the discrete analogue of the infinite-dimensional definition of TV whose study has a long tradition in the field of geometric measure theory and free discontinuity problems, see, e.g.~\cite{AmbrosioBV}. 
%Using such definition, for every $(i,j)\in\bm{\Omega}$, the dual variable $\bm{w}_{i,j}$ is thus constrained to belong to $\mathcal{B}_1(\bm{0})$. 
Such constraint can be equivalently expressedby requiring $\| \bm{w}_{i}||_2 \leq 1$ for all $i=1,\ldots,N$.
%Whenever the function $u$ is smooth enough, i.e.~at least of $W^{1,1}(\bm{\Omega})$ regularity, one can see that a standard application of the divergence theorem entails:
%\begin{equation}  \label{eq:TV2}
%    \mathrm{TV}(u)= \sup\left\{ -\int_\bm{\Omega} \nabla u(x)\cdot \bm{\phi}(x)~dx: \bm{\phi}\in C^1_c(\bm{\Omega};\R^2), \bm{\phi}(x) \in \mathcal{B}_1(\bm{0})~ \forall x\in\bm{\Omega}  \right\},
%\end{equation}
%and that the supremum is achieved by $\mathbf{\phi}^* = \frac{\nabla u}{|\nabla u|}$ so that $\mathrm{TV}(u)= \| \nabla u \|_{L^1(\bm{\Omega})}$.

Following \cite{BayramDTV2012}, we can now replace the set $\mathcal{B}_1(\bm{0})$ in \eqref{eq:TV} with a \emph{fixed} two-dimensional elliptical region $\mathcal{E}_{a,\theta}(\bm{0})$ centred in the origin and defined in terms of its orientation $\theta\in [-\pi/2,\pi/2)$ with respect to the horizontal $x$-axis  and eccentricity $a\in (0,1]$, that is:
\begin{equation}  \label{eq:super-ellipse}
\mathcal{E}_{a,\theta}(\bm{0}) {: =} \!\left\{ (x_1,x_2)\in\R^2: |x_1\cos\theta +x_2\sin\theta|^{2} + \left|\frac{-x_1\sin\theta+x_2\cos\theta}{a}\right|^{2}\leq 1\right\}.
\end{equation}
Note that as $a\to 0$, the set $\mathcal{E}_{a,\theta}(\bm{0}) $ degenerates to the line $x_2 = \tan \theta~ x_1$. 

Recalling definitions \eqref{siginvdec} and \eqref{eq:comp2} of the matrices $\bm{\Lambda}_a$ and $\bm{\mathrm{R}}_{-\theta}$ and denoting (formally, given the purely discrete setting we are working on) by $D_{\theta} u_{i}= (\bm{\D u})_{i}\cdot \bm{v}$ and $D_{\theta^\perp} u_{i}= (\bm{\D u})_{i}\cdot \bm{v}^\perp$ the directional derivatives along the direction $\bm{v}=(\cos\theta,\sin\theta)$ and its orthogonal $\bm{v}^\perp=(-\sin\theta,\cos\theta)$, we define element-wise the  \emph{directional gradient} $\widetilde{\bm{\D}}_{a,\theta}\bm{u} \in(\mathbb{R}^2)^N$ of $\bm{u}$ as
\begin{equation}
\widetilde{\bm{\D}}_{a,\theta}\bm{u}:=(\bm{\Lambda}_{a} \bm{\mathrm{R}}_{-\theta}  (\bm{\D u})_{i})_{i} = \begin{pmatrix}
D_\theta u_{i} \\
a D_{\theta^\perp} u_{i}
\end{pmatrix}_{i}.
\end{equation}
% we get the following compact expression for $\mathrm{DTV}(\bm{u})$:
% \begin{equation}  \label{eq:DTV_formulation}
%   \mathrm{DTV}(\bm{u}) = \sum_{i=1,~j=1}^{n_1,n_2} \| (\widetilde{\bm{\D}}_{a,\theta}\bm{u})_{i,j} \|_2 = \| \widetilde{\bm{\D}}_{a,\theta}\bm{u} \|_{2,1},
% \end{equation}
By this definition we can thus write the \emph{directional} formulation of TV firstly used in \cite{BayramDTV2012} and later applied in several other works (see, e.g., \cite{Zhang2013,KonDonKnuDTGV19,Demircan2020}) for promoting TV smoothness along $\bm{v}$. Note that the \eqref{eq:WDTVp_reg} reduces to this definition by choosing $\alpha_i=p_i= 1$, $\theta_i=\theta\in[-\pi/2,\pi/2)$ and $a_i=a\in(0,1]$, $\forall i$. It reads
\begin{align} 
\mathrm{DTV}(\bm{u}) 
%&= \| \widetilde{\bm{\D}}_{a,\theta}\bm{u} \|_{2,1} 
=
\sum_{i=1}^{N} \| (\widetilde{\bm{\D}}_{a,\theta}\bm{u})_{i} \|_2 
%= 
% \sum_{i=1,~j=1}^{n_1,n_2} \left\lVert \begin{pmatrix}
%D_\theta u_{i,j} \\
%a D_{\theta^\perp} u_{i,j}
%\end{pmatrix} \right\rVert_2  
= 
\sum_{i=1}^{N}  \max_{\bm{w}_{i}\in \mathcal{E}_{a,\theta}(\bm{0}) }  \langle (\bm{ \D u})_{i}, \bm{w}_{i}\rangle,
\label{eq:DTV_formulation} 
\end{align}
where here and in what follows we omit the explicit dependence of the regularisers on the hyperparameters to facilitate the overall readability.
%where $\bm{\Lambda}_{a}:=\text{diag}(1,a)$, $\bm{R}_{-\theta}$ is the unitary rotation matrix of angle $-\theta$, while $D_\theta u_{i,j}$ and $D_{\theta^\perp}u_{i,j}$ denote the directional derivatives along the direction $\bm{v}=(\cos\theta,\sin\theta)$ and its orthogonal $\bm{v}^\perp=(-\sin\theta,\cos\theta)$, i.e. $D_\theta u_{i,j}= (\bm{\D u})_{i,j}\cdot \bm{v}$ and $D_{\theta^\perp} u_{i,j}= (\bm{\D u})_{i,j}\cdot \bm{v}^\perp$.
Note that, differently from TV, DTV is computed as the sum of maximum values of scalar products in which the dual functions $\bm{w}_{i}$ are forced to belong to $\mathcal{E}_{a,\theta}(\bm{0})$ at any point. 
Following \cite{KonDonKnuDTGV19}, we can now observe that for $\bm{w}_{i}\in \mathcal{E}_{a,\theta}(\bm{0})$, we have
\begin{equation}\nonumber
\langle (\bm{\D u})_{i}, \bm{w}_{i} \rangle = \langle (\bm{\D u})_{i}, \bm{\mathrm{R}}_{\theta}\bm{\Lambda}_{a}\widetilde{\bm{w}}_{i} \rangle = \langle \bm{\Lambda}_{a}\bm{\mathrm{R}}_{-\theta} (\bm{\D u})_{i}, \widetilde{\bm{w}}_{i}  \rangle = \langle (\widetilde{\bm{\D}}_{a,\theta}\bm{u})_{i},\widetilde{\bm{w}}_{i}\rangle 
\end{equation}
where $\widetilde{\bm{w}}_{i}\in\mathcal{B}_{1}(\bm{0})$ for all $i=1,\ldots,N$.
Thus, we deduce:
\begin{equation}    \label{eq:directional_formulation} 
\max_{\bm{w}_{i} \in \mathcal{E}_{a,\theta}(\bm{0})}   \langle (\bm{ \D u})_{i}, \bm{w}_{i}\rangle  = 
\max_{\widetilde{\bm{w}}_{i}\in \mathcal{B}_1(\bm{0})} \langle (\widetilde{\bm{\D}}_{a,\theta}\bm{u})_{i}, \widetilde{\bm{w}}_{i} \rangle,
\end{equation} 
which can be used to show via the standard Cauchy-Schwarz inequality that at any point $i=1,\ldots,N$ the maximum is achieved by the normalised directional gradient vector, i.e. by the vector $\widetilde{\bm{w}}_{i}=(\widetilde{\bm{\D}}_{a,\theta}\bm{u})_{i}/\|(\widetilde{\bm{\D}}_{a,\theta}\bm{u})_{i}\|_2$, so that
\begin{align} 
&  \langle (\widetilde{\bm{\D}}_{a,\theta}\bm{u})_{i} , \widetilde{\bm{w}}_{i} \rangle =  \langle (\widetilde{\bm{\D}}_{a,\theta}\bm{u})_{i},  \frac{(\widetilde{\bm{\D}}_{a,\theta}\bm{u})_{i}}{\|(\widetilde{\bm{\D}}_{a,\theta}\bm{u})_{i}\|_2} \rangle \nonumber \\
& = \|(\widetilde{\bm{\D}}_{a,\theta}\bm{u})_{i} \|_2= \| \bm{\Lambda}_{a} \bm{\mathrm{R}}_{-\theta}  (\bm{\D u})_{i} \|_2,  \label{eq:max_vector}
\end{align}
which justifies \eqref{eq:DTV_formulation}. 

%The dependence on the angular parameter $\theta$ allows for non-null inclination of the curve with respect to the coordinate axes and the parameter $a\in (0,1]$ determines the eccentricity of the curve.
Inspired by \cite{ParMasSch18analysis}, we report in Figure \ref{fig:dualFunctions} a graphical representation of the considerations above.  There, we denote in red a fixed non-zero gradient vector $(\bm{\D u})_i \in \R^2$ evaluated at a certain point $i=1,\ldots,N$, in blue the direction $\bm{v}=(\cos\theta,\sin\theta)\in\R^2$ drawing an angle $\theta$ with the $x$-axis and in green the projection of $(\bm{\D u})_i$ along $\bm{v}$, i.e~the directional gradient $(\widetilde{\bm{\D}}_{a,\theta}\bm{u})_{i}$. The unitary ball $\mathcal{B}_1(\bm{0})$ is coloured black while the ellipses $\mathcal{E}_{a,\theta}(\bm{0}) $ for different values of $a\in(0,1] $ are coloured magenta. For each plot, the unitary vector $\widetilde{\bm{w}}_{i}$ realising the maximum in \eqref{eq:DTV_formulation} is drawn (magenta). Note that for $a=1$ we retrieve that the vector $\widetilde{\bm{w}}_{i}$  maximising the scalar product is the one parallel to $(\bm{\D u})_{i}$  (note that in such case the directionality does not affect the value computed, as $\bm{\mathrm{R}}_{-\theta}$ is unitary). However, as $a\to 0$ we observe that $\widetilde{\bm{w}}_{i}$ progressively aligns with $\bm{v}=(\cos\theta,\sin\theta)$, thus promoting directional regularisation.

\begin{figure}[t!] 
	\begin{subfigure}[t]{0.45\textwidth}
		\includegraphics[width=\textwidth]{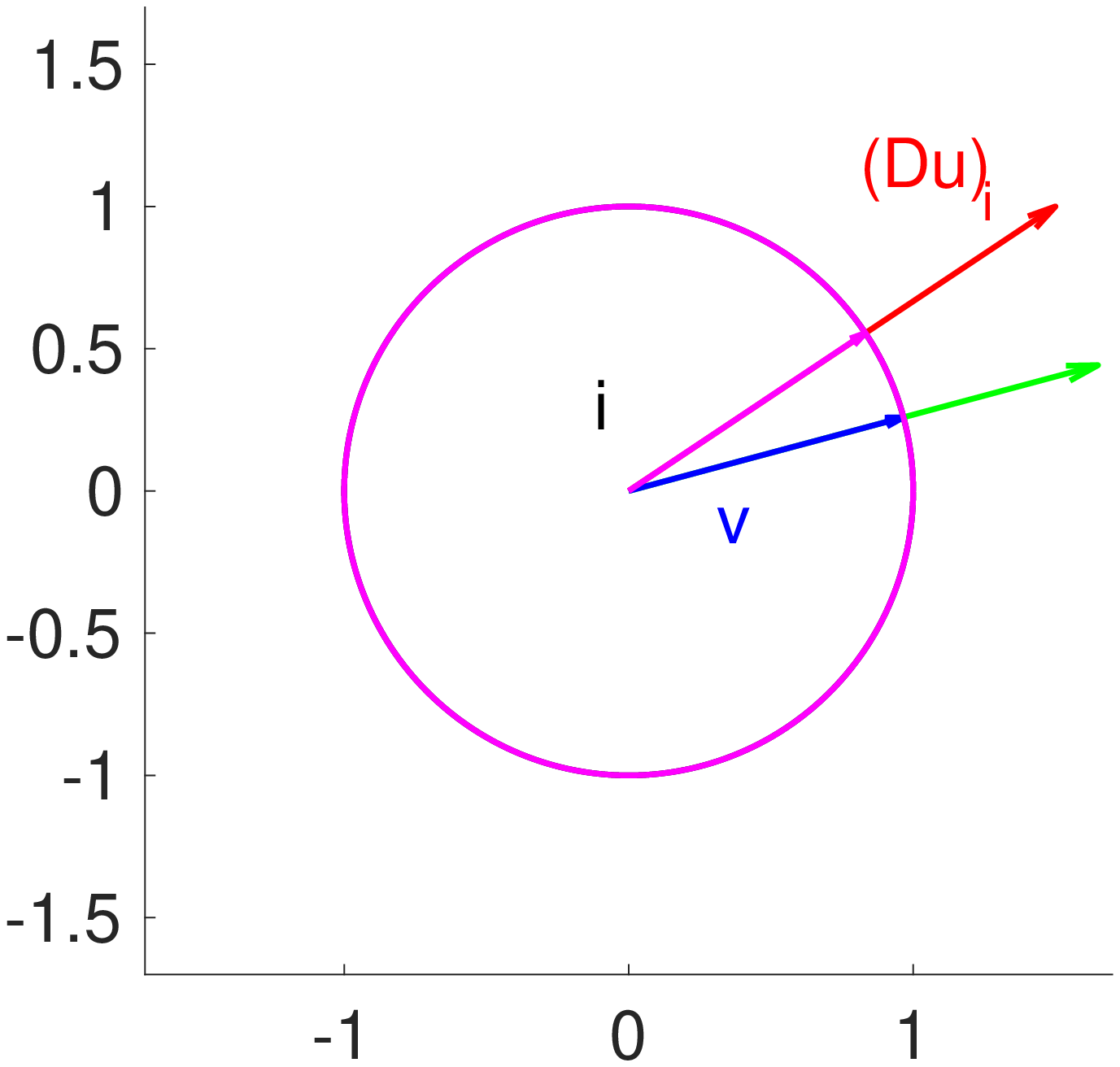}
		\caption{$a=1$}
	\end{subfigure}\ 
	\begin{subfigure}[t]{0.45\textwidth}
		\includegraphics[width=\textwidth]{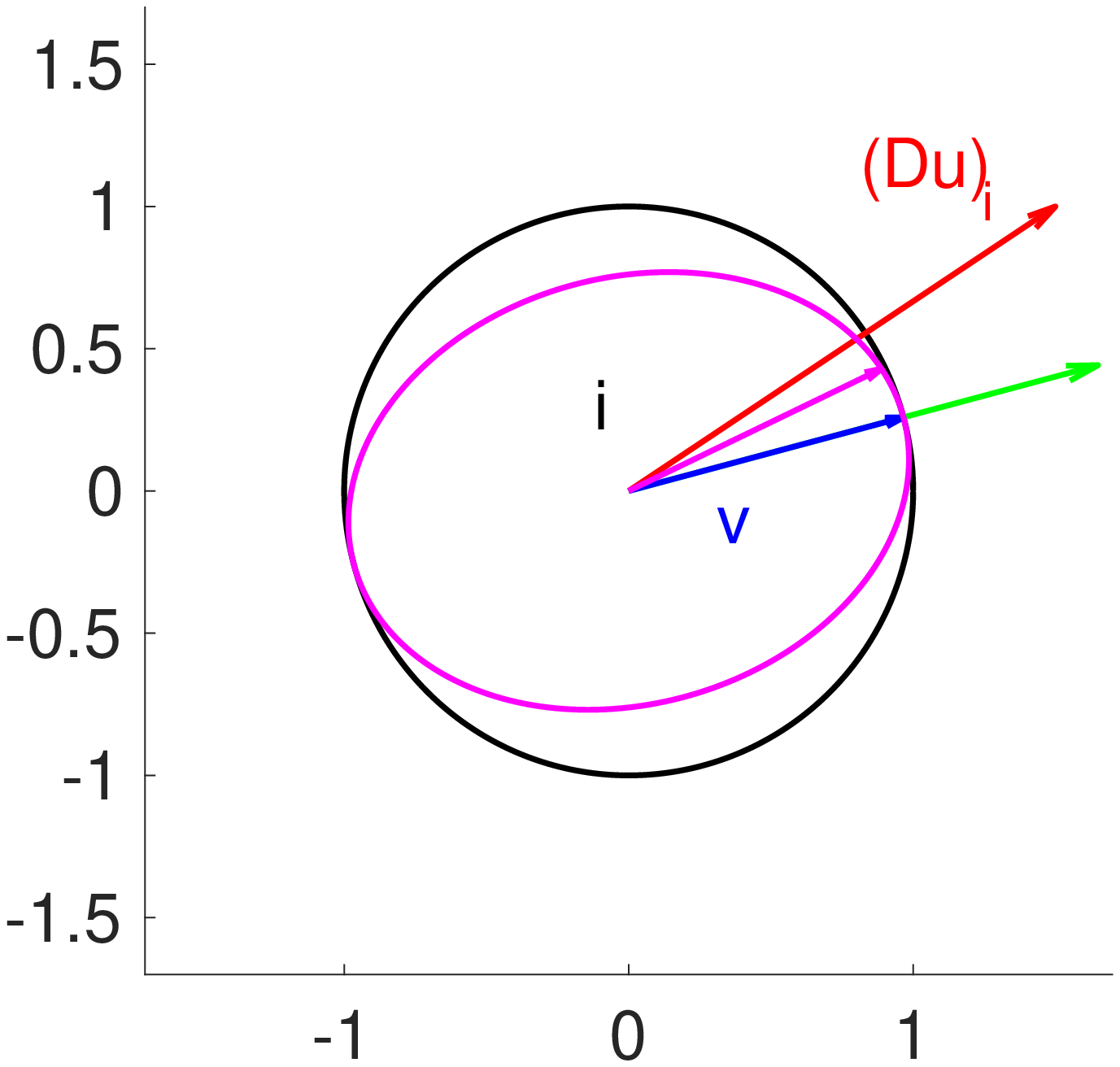}
		\caption{$a=0.5$}
	\end{subfigure}\\
	\begin{subfigure}[t]{0.45\textwidth}
		\includegraphics[width=\textwidth]{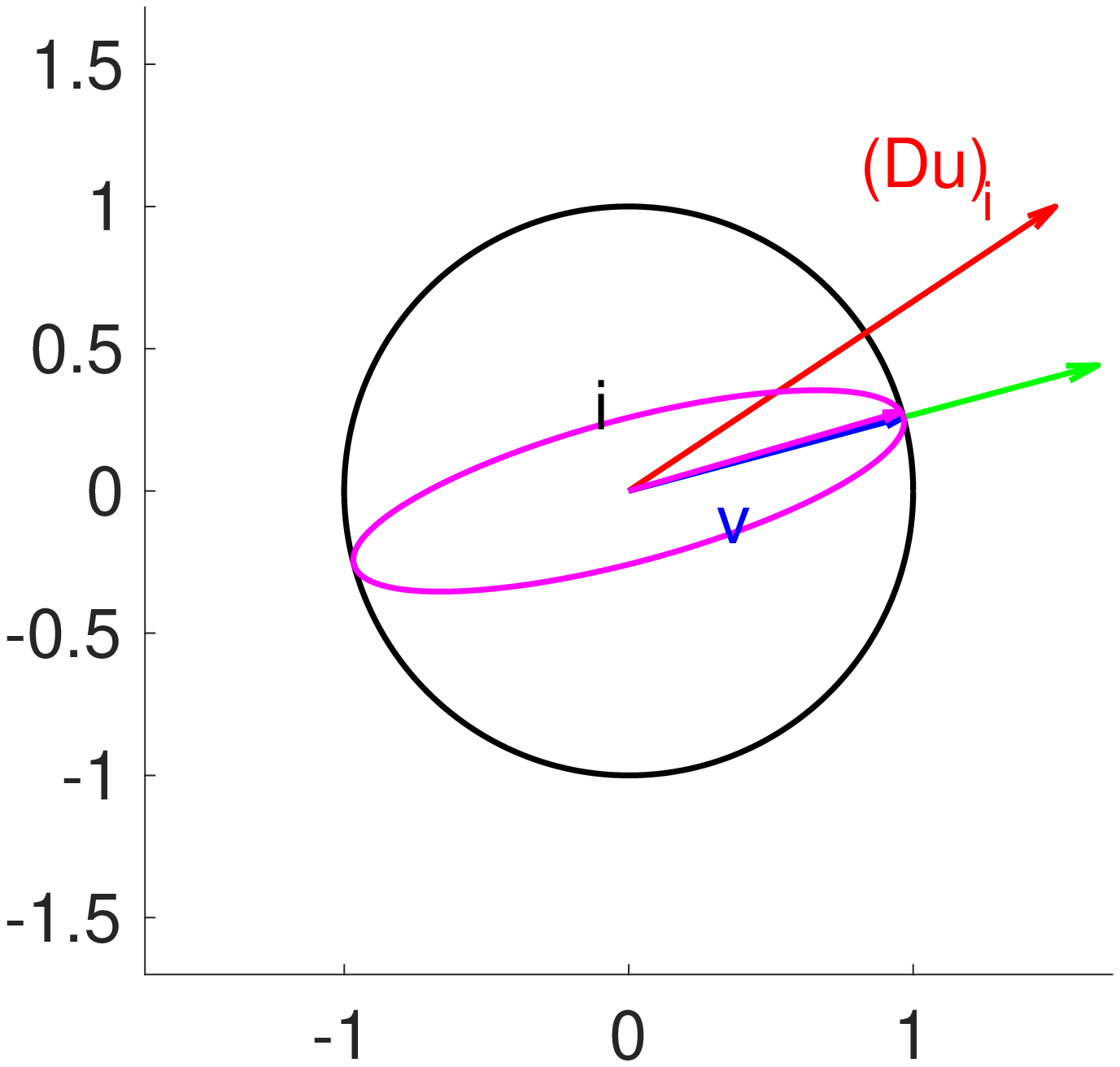}
		\caption{$a=0.25$}
	\end{subfigure}
	\begin{subfigure}[t]{0.45\textwidth}
		\includegraphics[width=\textwidth]{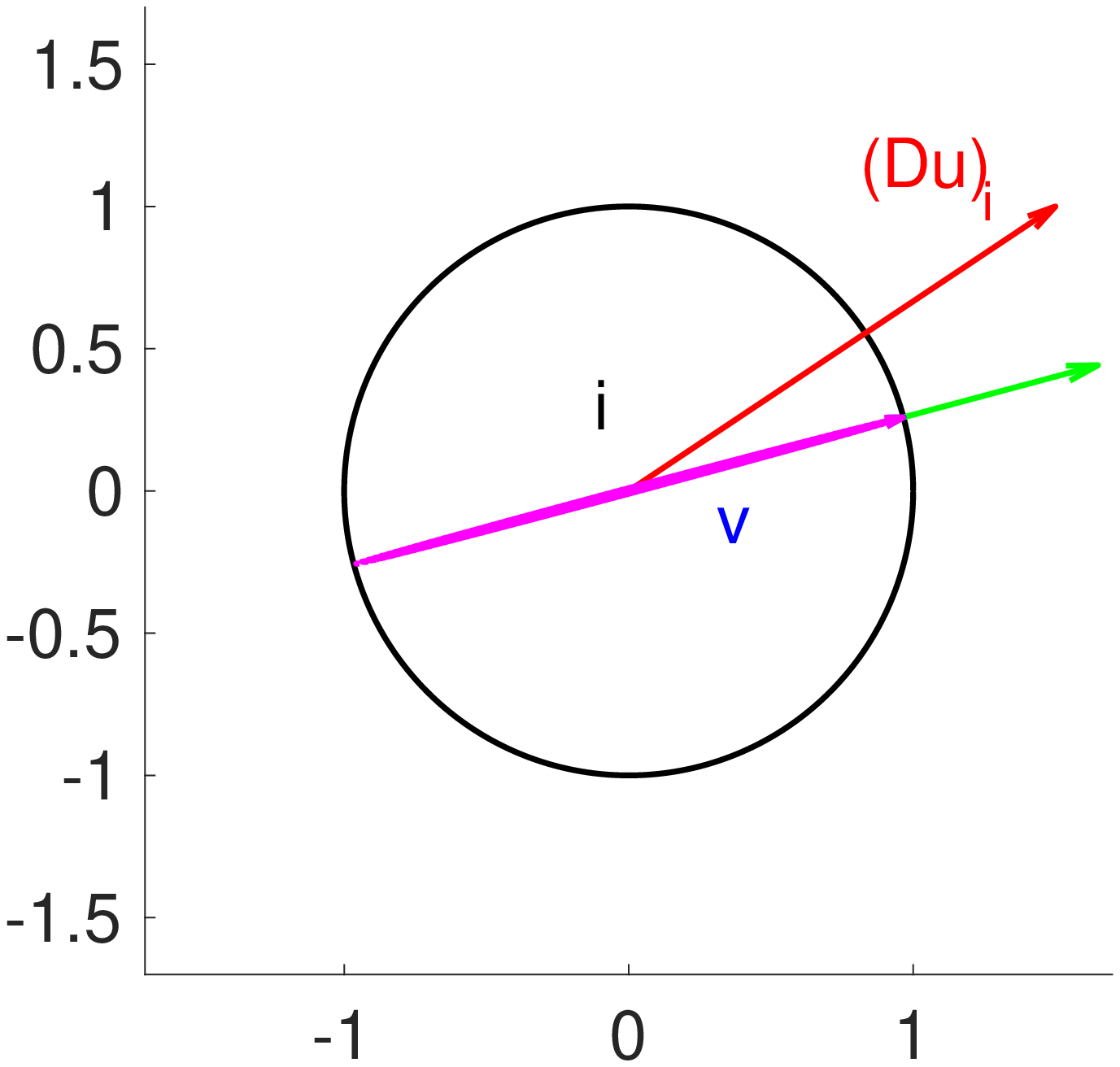}
		\caption{$a=0.01$}
	\end{subfigure}
	\caption{Directional behaviour of $\mathrm{DTV}$ regularisation \eqref{eq:DTV_formulation}. }
	\label{fig:dualFunctions}
\end{figure}

As previously remarked, the unit vector $\bm{v}$  defining the orientation of the ellipse $\mathcal{E}_{a,\theta}(\bm{0})$ is defined in terms of the angle $\theta$, which makes the use of the DTV regulariser useful in practice only when $\theta$ can be easily estimated. This is the case, for instance, of geometric textured images or of images of very specific scenes (see Figure \ref{fig:teas1}), which limits significantly the application of DTV in practice.

Such limitation can be overcome by considering the following natural space-variant extension of the DTV regulariser \eqref{eq:DTV_formulation}, which comes from \eqref{eq:WDTVp_reg} regulariser with $\alpha_i=p_i= 1$, $\forall i$, and space-variant $\theta_i\in[-\pi/2,\pi/2)$ and $a_i\in(0,1]$:
\begin{align}  \label{eq:DTV_sv}
\mathrm{DTV}^{sv}(\bm{u}) 
%& 
%= \| \widetilde{\bm{\D}}_{\bm{a},\bm{\theta}}\bm{u} \|_{2,1} 
=  \sum_{i=1}^{N} \| (\widetilde{\bm{\D}}_{\bm{a},\bm{\theta}}\bm{u})_{i} \|_2  = \sum_{i=1}^{N} \| \bm{\Lambda}_{a_{i}} \bm{\mathrm{R}}_{-\theta_{i}} (\bm{\D u})_{i}\|_2  =  \sum_{i=1}^{N} 
\max_{\bm{w}_{i}\in \mathcal{E}_{a_{i},\theta_{i}}(\bm{0})}  \langle (\bm{ \D u})_{i}, \bm{w}_{i}\rangle.
\end{align}
where now $\bm{a}=(a_{i})_{i}\in (0,1]^{N}, \bm{\theta}=(\theta_{i})_{i} \in [0,\pi)^{N}$ and where we have used the simplified notation $(\widetilde{\bm{\D}}_{\bm{a},\bm{\theta}}\bm{u})_{i} =  \bm{\Lambda}_{a_{i}} \bm{\mathrm{R}}_{-\theta_{i}} (\bm{\D u})_{i}\in\mathbb{R}^2$.  In this case, a space-variant adjustment of the directional smoothing (from strongly anisotropic along the direction $\bm{v}_{i}=(\cos\theta_{i},\sin\theta_{i})$ with $a_{i}=0$ to fully isotropic with $a_{i}=1$) is allowed at any point. For such regulariser, the same geometrical considerations as before hold, the difference being that the orientations $\theta_{j}$ may change from one point to another. Few choices can be made here. Following the Edge Adaptive Total Variation (EATV) approach proposed in \cite{Zhang2013}, one possibility consists in estimating the local directions $\bm{v}_{i}$ by imposing that $\bm{v}_{i} \perp (\bm{\D b}_\sigma)_{i}$, where, for $\sigma>0$,  $\bm{b}_\sigma$  denotes a smoothed version of the given image $\bm{b}$. This choice, however, is very sensitive to  noise oscillations and it may misguide the local directional behaviour if these are too large. Alternatively, as considered in \cite{Grasmair2010,Estellers2015,Ehrhardt2016,Lefkimmiatis2015} and more recently in \cite{PANG2020,Demircan2020}, the dependence on the image to retrieve can be encoded explicitly in the definition of the regularisation by allowing $\theta_{i}$ to be a function of the target image $\bm{u}$ (i.e.~$\theta_{i} = \theta_{i}(\bm{u})$) using, for instance, information coming from the structure tensor. This procedure is much more robust, but the nonlinear dependence on $\bm{u}$ in the definition of $\theta_{i}$ may significantly complicate the problem from an optimisation viewpoint. For further estimation strategies based on maximum likelihood approaches, we refer the reader to \cite{HWTV,CMBBE,DTVp_siam} and to the following discussion in Section \ref{sec:parameter_estimation_results}. Whatever the approach considered, it is worth remarking that an accurate and robust estimation of the space-variant parameters $\bm{a}$ and $\bm{\theta}$ is a very challenging problem
%Provided that  the DTV$^{sv}$ regulariser \eqref{eq:DTV_sv} can be thus used as an locally-adapted edge-preserving regularisation favouring DTV smoothing at each pixel $(i,j)$. 

\begin{remark}
	The values $a_{i}\in(0,1]$ for all $i=1,\ldots,N$ have to be interpreted as `confidence' parameters enforcing a strong anisotropic TV smoothing ($a_{i}\approx 0$) whenever a good local estimation of $\theta_{i}$ is available, while leaving the behaviour to be close-to-isotropic ($a_{i}\approx 1$) whenever the estimation of $\theta_{i}$ is unreliable.
\end{remark}

%We remark that in the considerations above we have informally looked at the local maximising vectors $\bm{w}_{i,j}\in\mathbb{R}^2$ to deduce the global  expression of the DTV and DTV$^{sv}$ regularisers \eqref{eq:DTV_formulation2} and \eqref{eq:DTV_sv}, respectively. This, informally, corresponds to ``exchange" the sum and the maximum operator which of course is not possible in general. Hover, from a geometrical point of view, such considerations helps a lot in the understanding of the local geometrical differences w.r.t. to the standard TV definition \eqref{eq:TV}.

As noted in Sections \ref{sec:Bayes} and \ref{sec:map}, we can further incorporate in \eqref{eq:DTV_sv} an additional shape/sharpness parameter vector $\bm{p}=(p_i)_i\in\R_{++}^N$, thus considering the regulariser 
% thus considering either
% \begin{align} 
%     \mathrm{DTV}^{sv}_p(\bm{u}) & = 
%     \| \widetilde{\bm{\D}}_{\bm{a},\bm{\theta}}\bm{u} \|^p_{2,p} = 
%       \sum_{i=1,~j=1}^{n_1,n_2} 
%     \| (\widetilde{\bm{\D}}_{\bm{a},\bm{\theta}}\bm{u})_{i,j} \|_2^p=
%     \sum_{i=1,~j=1}^{n_1,n_2} \| \bm{\Lambda}_{a_{i,j}} \bm{\mathrm{R}}_{-\theta_{i,j}} (\bm{\D u})_{i,j}\|^{p}_2 \nonumber\\
%     & = \sum_{i=1,~j=1}^{n_1,n_2} \left( \max_{\bm{w}_{i,j}\in \mathcal{E}_{a_{i,j},\theta_{i,j}}(\bm{0})}  \langle (\bm{ \D u})_{i,j}, \bm{w}_{i,j}\rangle \right)^p \label{eq:DTVp_sv1}
% \end{align}
% as well as its space-variant counterpart defined in terms of $\bm{p}=(p_{i,j})_{i,j} \in \mathbb{R}^{n_1\times n_2}_{+}$:
\begin{align}  
\mathrm{DTV}_{\bm{p}}^{sv}(\bm{u})
% = \sum_{i=1}^{N} \|(\widetilde{\bm{\D}}_{\bm{a},\bm{\theta}}\bm{u})_{i}  \|_2^{p_{i}}
=  \sum_{i=1}^{N} \| \bm{\Lambda}_{a_{i}} \bm{\mathrm{R}}_{-\theta_{i}} (\bm{\D u})_{i}\|^{p_{i}}_2
= \sum_{i=1}^{N} \left( \max_{\bm{w}_{i}\in \mathcal{E}_{a_{i},\theta_{i}}(\bm{0})}  \langle (\bm{ \D u})_{i}, \bm{w}_{i}\rangle \right)^{p_{i}}, \label{eq:DTVp_sv2}
\end{align}
which is a particular instance of \eqref{eq:WDTVp_reg} taking $\alpha_i= 1$, $\forall i$.
The presence of the parameters $p_i$ in \eqref{eq:DTVp_sv2} does not alter the directional behaviour of such regulariser in comparison with the one observed for DTV$^{sv}$. However, as thoroughly discussed in Sections \ref{sec:sharpness} and \ref{sec:prox},  such behaviour is made sharper for $p_{i}<1$ and smoother for $p_{i}>1$. Note, in particular, that when $p_{i}= 2$ the DTV$_{2}^{sv}$ regulariser acts locally as a Tikhonov-type squared $\ell_2$-norm of the directional gradient $\widetilde{\bm{\D}}_{a_i,\theta_i}\bm{u}$.
%We remark that dealing with the dual form of the terms above is made challenging due to the presence of the (space-variant) exponent $p$. Considering $p$-powers of the scalar products may result in fact in the computation of complex values. We thus avoid to speculate on this matter, leaving this for future research. 

We conclude this section with some considerations regarding weighted models. Recalling \eqref{eq:TV}, we notice that  introducing a space-variant parameter vector $(\alpha_{i})_{i}\in\mathbb{R}^{N}_{++}$ corresponds simply to inflate/deflate the Euclidean ball $\mathcal{B}_1(\bm{0})$ and to look for maxima therein, which corresponds to the choice
\begin{align}  \label{eq:WTV}
\mathrm{WTV}(\bm{u}) & = \sum_{i=1}^{N} \alpha_{i}\| (\bm{ \D u})_{i} \|_2  =
\sum_{i=1}^{N}
\max_{\bm{w}_{i}\in \mathcal{B}_{\alpha_{i}}(\bm{0})}  \langle (\bm{ \D u})_{i}, \bm{w}_{i} \rangle
\end{align}
where for $\alpha_i>0$, $\mathcal{B}_{\alpha_{i}}(\bm{0}):=\left\{ \bm{z}\in\mathbb{R}^2: \|\bm{z}\|_2 \leq \alpha_{i}\right\}$ and where the vector $\widetilde{\bm{w}}_{i} =  \alpha_{i} \frac{ (\bm{\D u})_{i}}{\|(\bm{\D u})_{i}\|_2}$ maximises the scalar products at any point. 
%This formulation has been thoroughly studied in infinite-dimensional settings in \cite{Hintermueller2016,Hintermuller2017a} where, additionally, relations with plain TV regularisation models combined with weighted data fidelities are also investigated. 
By analogous considerations as above, we can finally draw a connection with the regulariser defined in \eqref{eq:WTDVp_reg2}, which, recalling the discussion above, can be written as:
\begin{align}  \label{eq:WTVp}
\mathrm{WDTV}^{sv}_{\bm{p}}(\bm{u})  = \sum_{i=1}^{N} \alpha_{i}^{p_{i}}\| (\widetilde{\bm{\D}}_{\bm{a},\bm{\theta}}\bm{u})_{i} \|^{p_{i}}_2 = \sum_{i=1}^{N} \alpha_{i}^{p_{i}}
\left( \max_{\bm{w}_{i}\in \mathcal{E}_{a_{i},\theta_{i}}(\bm{0})}  \langle (\bm{ \D u})_{i}, \bm{w}_{i}\rangle \right)^{p_{i}}.
\end{align}
% \begin{align}
%     \mathrm{WTV}_{\bm{p}}^{sv}(\bm{u})& = \sum_{i=1,~j=1}^{n_1, n_2}
%     \left( \max_{\bm{w}_{i,j}\in \mathcal{B}_{\alpha_{i,j}}(\bm{0})}  \langle (\bm{ \D u})_{i,j}, \bm{w}_{i,j} \rangle \right)^{p_{i,j}} = \sum_{i=1,~j=1}^{n_1, n_2}
%     \left(
%     \alpha_{i,j} \| (\bm{ \D u})_{i,j} \|_2
%     \right)^{p_{i,j}}
%     \\
%     & =\sum_{i=1,~j=1}^{n_1, n_2} \alpha_{i,j}^{p_{i,j}}\| (\bm{ \D u})_{i,j} \|^{p_{i,j}}_2, 
% \end{align}
% which can be used to favour sharper/smoother  weighted regularisations.

% \medskip

% Interesting generalisations motivating further research and based on the geometrical arguments discussed above are, for instance:
% \begin{itemize}
% \item The very general definition of a space-variant weighted and directional TV-type regulariser combining the benefits of the directional modelling \eqref{eq:DTVp_sv2} with the weighting in \eqref{eq:WTVp}, which would read
% \begin{align}  \label{eq:WTVp}
%     \mathrm{WDTV}^{sv}_{\bm{p}}(\bm{u}) & = \sum_{i=1,~j=1}^{n_1, n_2} \alpha_{i,j}^{p_{i,j}}\| (\widetilde{\bm{\D}}_{\bm{a},\bm{\theta}}\bm{u})_{i,j} \|^{p_{i,j}}_2. 
% \end{align}

\begin{remark}
	If $p_{i}= 1$ for all $i=1,\ldots,N$, the regulariser \eqref{eq:WTVp} takes the form
	\begin{align}  \label{eq:WDTV}
	\mathrm{WDTV}^{sv}(\bm{u}) 
	=  \sum_{i=1}^{N} 
	\max_{\bm{w}_{i}\in \mathcal{E}_{a_{i},\theta_{i},\alpha_{i}}(\bm{0})} \langle (\bm{ \D u})_{i}, \bm{w}_{i}\rangle 
	\end{align}
	where $\mathcal{E}_{a_{i},\theta_{i},\alpha_{i}}(\bm{0})$ denotes the 2D ellipse centred in the origin with eccentricity $a_{i}\in(0,1]$, orientation $\theta_{i}\in[-\pi/2,\pi/2)$ w.r.t.~to the $x-$axis and width/height equal to $2\alpha_{i}$ and $2\alpha_{i}a_{i}$, respectively. 
	% As we can observe from Figure \ref{fig:weightedDTV}, the use of a weighted modelling (dotted plots) does not alter the directional behaviour described above, but, simply, weights more ($\alpha_{i}>1$) or less ($\alpha_{i}<1$) the directional contributions (magenta arrows).
	% \begin{figure}[h!] 
	% \centering
	% 	\begin{subfigure}[t]{0.45\textwidth}
	% 	\centering
	% 		\includegraphics[width=\textwidth]{images/ellipses/constraintSetDual_weighted_a07.eps}
	% 		\caption{$a=0.7$}
	%     \end{subfigure}\ 
	%     \begin{subfigure}[t]{0.45\textwidth}
	%     \centering
	% 		\includegraphics[width=\textwidth]{images/ellipses/constraintSetDual_weighted_a04.eps}
	% 		\caption{$a=0.4$}
	%     \end{subfigure}
	% \caption{Directional behaviour of $\mathrm{WDTV}$ regularisation \eqref{eq:WDTV}, $\alpha=1.7$.}
	% \label{fig:weightedDTV}
	% \end{figure}
\end{remark}

The authors believe that an interesting generalisation of the discussion above shall address situations where the constraint set is non-convex and, in particular, it is defined in terms of  \emph{Lam\'e curves} centred in $\bm{0}$, i.e. defined by
\begin{equation}  \label{eq:super-ellipse}
\!\!\!\mathcal{L}_{\beta,a,\theta}(\bm{0}) {: = }\left\{ (x_1,x_2)\in\R^2:~ |x_1\cos\theta {+}x_2\sin\theta|^{\beta} {+} \left|\frac{{-}x_1\sin\theta{+}x_2\cos\theta}{a}\right|^{\beta}{\leq} 1\right\}.
\end{equation}
Such set is non-convex as soon as $\beta<1$, see Figure \ref{fig:lame}. The use of such general shapes may lead to consider new gradient-based regularisations where the underlying geometry constraining the dual functions favours smoothing in  different ways.

\begin{figure}[t!] 
	\centering
	\begin{subfigure}[t]{0.45\textwidth}
		\centering
		\includegraphics[height=4.5cm]{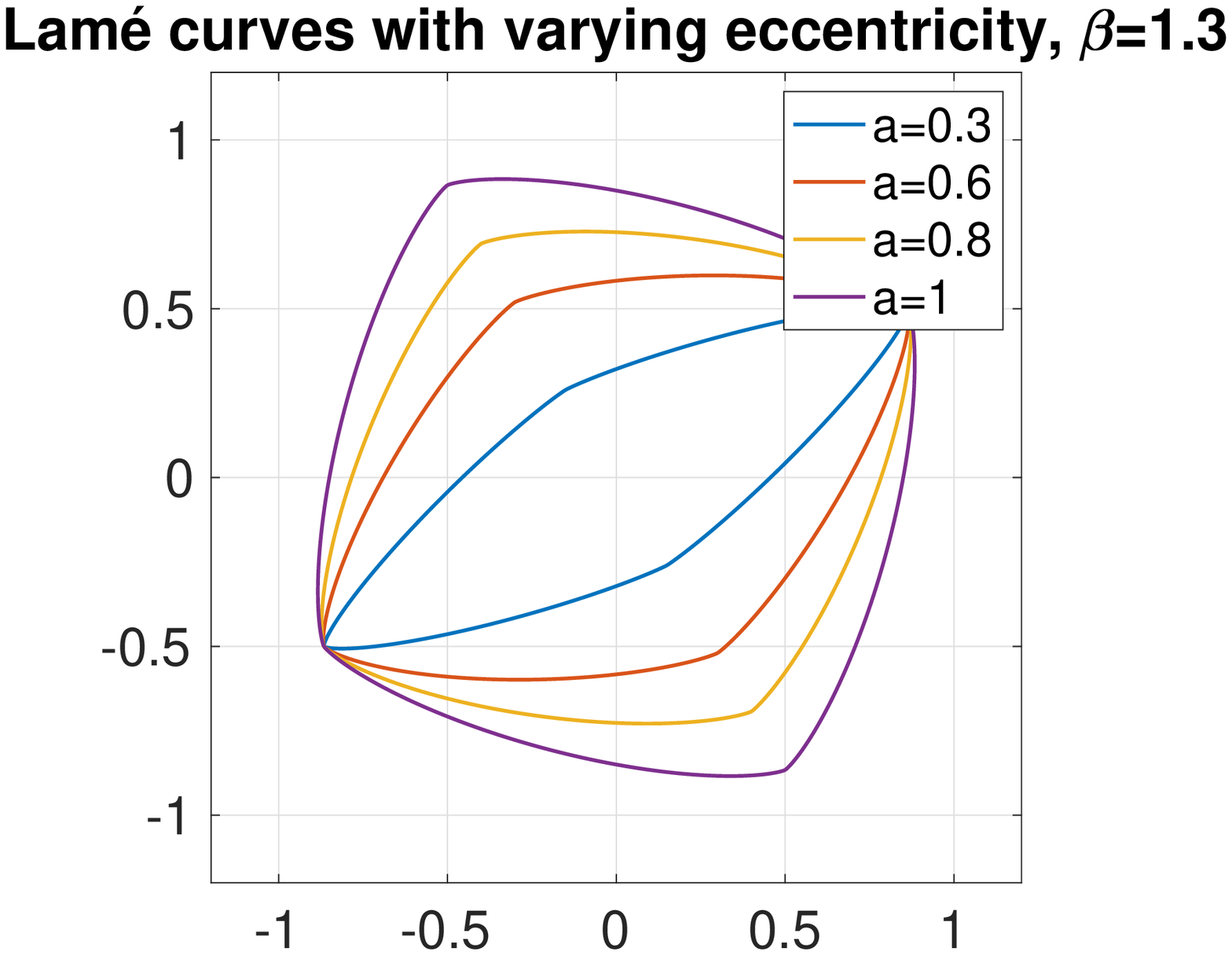}
	\end{subfigure}\ 
	\begin{subfigure}[t]{0.48\textwidth}
		\centering
		\includegraphics[height=4.5cm]{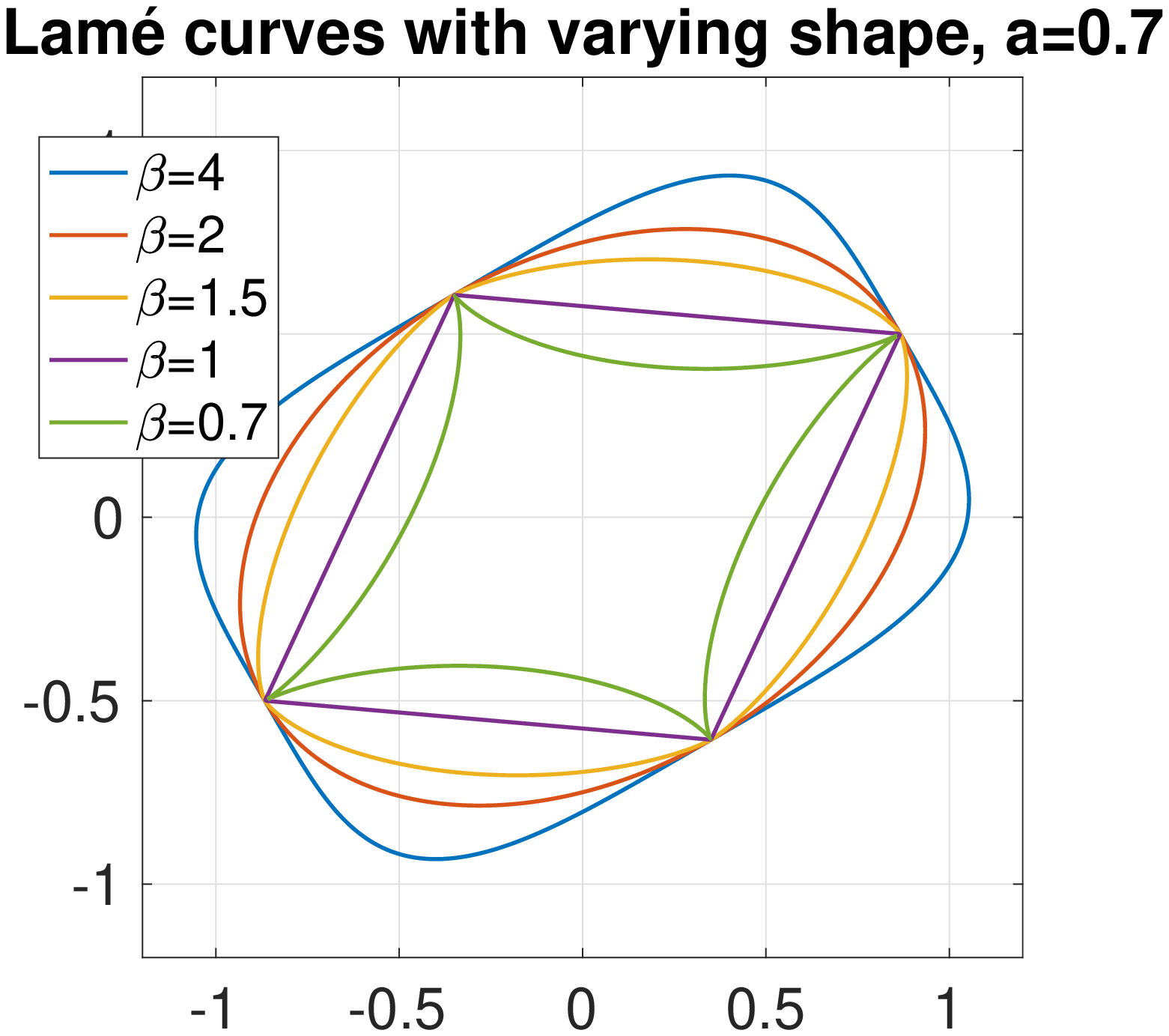}
	\end{subfigure}
	\caption{Lamé curves \eqref{eq:super-ellipse} with varying eccentricity and shape parameters, $\theta=\pi/6$.}
	\label{fig:lame}
\end{figure}

% The formulation and the rigorous analysis of \eqref{eq:DTVp_sv1} and \eqref{eq:DTVp_sv2} in an infinite-dimensional setting, which is made very challenging even in smooth scenarios by the non-trivial underlying dual formulation. Some work in this direction has been done in \cite{Hintermuller2013,Hintermuller2015} where discrete-to-continuum results are provided. However, a rigorous study in this framework is still missing.
% Inspired by \cite{Condat2017}, the new definition of the space-variant DTV models above in terms of suitable non-linear vector fields, which may be particularly useful for the definition of tailored  discretisations of directional gradients avoiding aliasing and anisotropic artefacts.

\section{Joint hypermodelling}
\label{sec:joint}

In this section, we provide explicit expressions of the negative log-hyperprior $-\ln\mathbb{P}(\bm{\Theta})$ and of the negative log-likelihood $-\ln\mathbb{P}(\bm{b}\mid \bm{\A u})$ in \eqref{eq:map_last}, which allows to  derive the final joint hypermodel.

\subsection{Non-informative hyperprior}
When no a priori knowledge or intuition about the value of the unknown prior hyperparameters is available, a uniform distribution for the random vector
$\bm{\Theta}$ can be set, thus considering a (possibly improper) non-informative hyperprior. In formulas this corresponds to set
\begin{equation}
\label{eq:hyper}
\mathbb{P}(\bm{\Theta}) = \varrho \, \chi_{\mathcal{D}_{\bm{\Theta}}}(\bm{\Theta})\,,\;\text{with }\varrho\in\R_{++}\,,
%\;\forall \bm{\Theta} \in {C_{\bm{\Theta}}}\,,
\end{equation}
%where $C_{\bm{\Theta}}$ denotes the domain of existence of $\bm{\Theta}$ and $\mu(C_{\bm{\Theta}})$ its measure.
%
from which it follows
\begin{equation}
-\ln\mathbb{P}(\bm{\Theta})
\;{=}\;
-\ln \varrho + \iota_{\mathcal{D}_{\bm{\Theta}}}(\bm{\Theta}) \, .
\label{eq:loghyp}
\end{equation}

%
\begin{comment}
Plugging \eqref{eq:loghyp} into \eqref{eq:map}, and considering that the term $-\ln \varrho$ is constant, our joint model reduces to
%
\begin{equation}
\label{eq:mapp}
\left\{\bm{u}^*,\bm{\Theta}^*\right\}\in\argmin{\bm{u},\bm{\Theta}}\left\{-\ln\mathbb{P}(\bm{u}\mid\bm{ \Theta})-\ln\mathbb{P}(\bm{\Theta})-\ln\mathbb{P}(\bm{b}\mid \bm{\A u})\right\}\,,
\end{equation}
\end{comment}

\subsection{GG likelihood leads to L$_q$ fidelity term}
\label{subsec:loglik}
First, based on the expression of the considered GG likelihoods in \eqref{eq:likeli}-\eqref{eq:likeli_inf}, the negative log-likelihood term $-\ln\mathbb{P}(\bm{b}\mid \bm{\A u})$ in  \eqref{eq:map} takes the form
\begin{align}
\label{eq:lnlkh1}
q<+\infty : &\left\{
\begin{array}{cl}
-\ln\mathbb{P}(\bm{b}\mid \bm{\A u})   \,\;{=}\;\,&
\omega^q \,\left\|\bm{\A u}-\bm{b}\right\|_q^q 
\;{-}\;\displaystyle{
	M\ln\left(\frac{\omega}{2}\frac{q}{\Gamma(1/q)}\right)} \,.   \\
\phantom{ -\ln\mathbb{P}(\bm{b}\mid \bm{\A u}) } \;{=}\;\,&
\omega^q \,\left\|\bm{\A u}-\bm{b}\right\|_q^q  
\;{+}\; C_q \, , 
\end{array}\right.\\
&\\
\label{eq:lnlkh2}
q=+\infty : &\left\{
\begin{array}{cl}
-\ln\mathbb{P}(\bm{b}\mid \bm{\A u})   \,\;{=}\;\,& \iota_{[0,1/\omega]}\left(\|\bm{\A u}-\bm{b}\|_{\infty}\right) -\displaystyle{ M\ln\frac{\omega}{2}}\phantom{\frac{1}{X}} \\
\phantom{ -\ln\mathbb{P}(\bm{b}\mid \bm{\A u}) } \;{=}\;\,&
\iota_{[0,1/\omega]}\left(\|\bm{\A u}-\bm{b}\|_{\infty}\right)
\;{+}\; C_\infty \, , 
\end{array}\right.
\end{align}
%
%\begin{eqnarray}
%-\ln\mathbb{P}(\bm{b}\mid \bm{\A u})
%&\,\;{=}\;\,& 
%\beta^q \left\|\bm{\A u}-\bm{b}\right\|_q^q 
%\;{-}\;
%N\ln\left(\frac{\beta}{2}\frac{q}{\Gamma(1/q)}\right) \,. 
%\label{eq:lnlkh0} \\
%&\,\;{=}\;\,&
%\mu \, \mathrm{L}_q(\bm{\A u};\bm{b}) 
%\;{+}\; \mathrm{constant} \, ,
%\label{eq:lnlkh}
%\end{eqnarray}
%
where the quantities $C_q, C_\infty>0$ appearing in \eqref{eq:lnlkh1}-\eqref{eq:lnlkh2} do not depend on the optimisation variable $\bm{u}$, so they can be dropped - in \eqref{eq:map}.

We now introduce the functional $F_q(\bm{\A u};\bm{b}):\R^M\to\R_{+}$ which is defined as
\begin{eqnarray}
\label{eq:fid1}
q<+\infty : & \qquad F_q(\bm{\A u};\bm{b}) :=& q\,\omega^q\mathrm{L}_q(\bm{\A u};\bm{b}),\\
\label{eq:fid2}
q=+\infty : & \qquad F_{\infty}(\bm{\A u};\bm{b}) :=& \iota_{[0,1/\omega]}(\|\bm{\A u}-\bm{b}\|_{\infty})
\end{eqnarray}
with
\begin{equation}
\label{eq:Lqfid}
\mathrm{L}_q(\bm{\A u};\bm{b}) 
\,\;{=}\;\, 
\frac{1}{q}\|\bm{\A u}-\bm{b}\|_q^q\,.
\end{equation}
%
%As one may expect, in the sequel $\mu$ will play the usual role of the regularisation parameter - namely, balancing the effects of regularisation and data-fidelity - in the considered variational Bayes model \eqref{eq:map}.
%Actually, here both the scale parameter $\beta$ - or, equivalently, the standard deviation $\sigma$ - and the shape parameter $q$ of the AWGG noise corruption are assumed to be known. Hence, according to the definition in \eqref{eq:Lqfid}, the value of the regularisation parameter $\mu$ can also be known. However, it is well-established that setting a priori the $\mu$ value based on the true noise parameters makes the resulting variational model highly unstable. This is even more true in case of space-variant regularisers with automatically-selected local parameters. As a consequence, $\mu$ is regarded as an unknown parameter that will also be automatically selected based on some a posteriori criterion, namely the popular discrepancy principle {\color{red}[xxx]}. 
%Details will be given in Section~\ref{sec:admm}.

\subsection{Joint variational Bayesian hypermodels}
\label{subsec:joint}

We are now ready to derive the explicit instances of the hypermodel \eqref{eq:map_last} in terms of the selected priors, hyperpriors and fidelity functionals discussed above. To improve readability, we will consider in the following data terms $F_q$ with $q<+\infty$. However, as it will be remarked at the end of the section, analogous derivations can be easily extended to the case $q=+\infty$. 

\medskip

Problem \eqref{eq:map_last} can be reformulated as
%
%Pluggig...we obtain
%
\begin{equation}
\left\{\bm{u}^*,\bm{\Theta}^*\right\}
%\argmin{
%\bm{u}\in\R^N,
%\bm{\Theta}\in\mathcal{D}_{\bm{\Theta}}}
%\left\{\,
%-\ln\mathbb{P}(\bm{u}\mid\bm{\Theta})
%\;{-}\;
%\ln\mathbb{P}(\bm{\Theta})
%\;{-}\;
%\ln\mathbb{P}(\bm{b}\mid \bm{\A u})
%\,\right\}
%\nonumber \\
\;\;{\in}\;\;
\argmin{
	\bm{u}\in\R^N,
	\bm{\Theta}\in\mathcal{D}_{\bm{\Theta}}}
\left\{\,
\mathcal{R}(\bm{u},\bm{\Theta})\;{+}\; 
\mathcal{H}(\bm{\Theta})
\;{+}\;
\mu \, \mathrm{L}_q(\bm{\A u};\bm{b})
\,\right\}\,,
\label{eq:fin_model}
\end{equation}
where $\mathcal{R}(\bm{u},\bm{\Theta})$ denotes one of the regularisation terms previously discussed in Section \ref{subsec:logpr}, while $\mathcal{H}(\bm{\Theta})$ accounts for possibly multiple terms depending  only on the hyperparameter vector $\bm{\Theta}\in\mathcal{D}_{\bm{\Theta}}$. The uniform hyperprior \eqref{eq:hyper} acts here simply by enforcing optimisation on  $\mathcal{D}_{\bm{\Theta}}$ only. The parameter $\mu>0$ is a  regularisation parameter whose choice will be specified for each hypermodel in the following.

Similarly as what discussed in  Section \ref{subsec:logpr},
%has already clarified how all the considered regularisation terms $\mathcal{R}(\bm{u},\bm{\Theta})$ exhibit the same abstract form; analogous considerations hold for functions
we have that $\mathcal{H}(\bm{\Theta})$ can be expressed in general form as
%It is worth noticing - this will be useful in Sections \ref{sec:parameter_estimation}-\ref{sec:admm} where algorithmic optimisation will be discussed - that all considered functions $\mathcal{R}(\bm{u},\bm{\Theta})$ and $\mathcal{H}(\bm{\Theta})$ exhibit the same abstract form, namely
%
\begin{equation}
\label{eq:gs}
%\mathcal{R}(\bm{u};\bm{\Theta}) = \sum_{i=1}^{N}f(\bm{g}_i;\bm{\Theta}_i)\,,\quad 
\mathcal{H}(\bm{\Theta}) = \sum_{i=1}^{N}h(\bm{\Theta}_i)   \,, \quad h:\mathcal{D}_{\bm{\Theta}_i}\to\R\,,
\end{equation}
where the function $h$ is a \emph{parameter penalty function} whose form will be specified for each regulariser.

We start our considerations from the TV prior \eqref{eq:TV_prior}. By plugging \eqref{eq:hyper} and  \eqref{eq:likeli} into \eqref{eq:map_last}, we get
\begin{align}
\left\{\bm{u}^*,\alpha^*\right\}
%\;{=}\;\left\{\bm{u}^*,\alpha^*\right\}
\;{\in}\;&\argmin{\bm{u}\in\R^N,\,\alpha\in\R_{++}}\left\{\alpha\sum_{i=1}^{N}\|(\bm{\D u})_i\|_2 \;{-}\; N\ln\alpha \;{+}\; q\,\omega^q\,\mathrm{L}_{q}(\bm{\A u};\bm{b})\right\}\\
\label{eq:TV_hyper}
\;{=}\;&\argmin{\bm{u}\in\R^N,\,\alpha\in\R_{++}}\left\{\mathrm{TV}(\bm{u}) \;{-}\; \frac{N}{\alpha}\ln\alpha \;{+}\; \mu\,\mathrm{L}_{q}(\bm{\A u};\bm{b})\right\}\,,\quad\text{with}\quad\mu\;{:=}\;\frac{q\,\omega^q}{\alpha}\,,
\end{align}
%where $\mu$ plays the role of the regularisation parameter, as it balances the contribution of the regulariser and of the fidelity term in the overall functional, 
where, we recall that in this case $\bm{\Theta}=\alpha\in\mathbb{R}_{++}$, whence parameter penalty function $h_{\mathrm{TV}}$ reads
\begin{equation}  \label{eq:TV_h2}
%h_{\mathrm{TV}}(\bm{\Theta}_i) =
h_{\mathrm{TV}}(\alpha) = -\frac{1}{\alpha}\ln\alpha\,,\quad i=1,\ldots,N\,. 
\end{equation}

%The explicit expressions of the functions $f,h$ for the TV-L$_q$ hypermodel, as well as the existence domain for the scalar parameter $\alpha$, are reported in Table~\ref{tab:models}.

For the TV$_p$ and the DTV regularisers in \eqref{eq:TVp} and \eqref{eq:DTV}, respectively,
%that - we recall - can be thought of as deriving from special cases of the space-invariant prior pdfs introduced in Sec.~\ref{subsec:prior},
hypermodels with similar form as in \eqref{eq:TV_hyper} can be derived. In particular, as far as the TV$_p$ regularisation term is concerned, we have that $\bm{\Theta}=(\alpha,p)\in\mathbb{R}^2_{++}$ and that the parameter penalty function takes the form
%- see Table~\ref{tab:models}.
\begin{equation}
% h_{\mathrm{TV}_p}(\bm{\Theta}_i) \;{=}\;
h_{\mathrm{TV}_p}(\alpha,p) \;{=}\; -\frac{1}{\alpha^p}\ln\frac{\alpha\,p}{\Gamma(1/p)}\,,\quad i=1,\ldots,N\,,
\end{equation}
while for the DTV$_p$ regulariser we have $\bm{\Theta}=(\alpha,p,\theta,a)\in\mathbb{R}^2_{++}\times [-\pi/2,\pi/2)\times (0,1]$ and the parameter penalty function reads
\begin{equation}
%h_{\mathrm{DTV}_p}(\bm{\Theta}_i) \;{=}\;
h_{\mathrm{DTV}_p}(\alpha,\theta,a) \;{=}\; -\frac{1}{\alpha}\ln\left(\frac{a}{2\pi}\frac{\alpha^2}{4}\right)\,,\quad i=1,\ldots,N\,.
\end{equation}

As far as space-variant hypermodels are concerned, we start considering the WTV regulariser for which $\bm{\Theta}=\bm{\alpha}\in \mathbb{R}_{+}^N$. Model \eqref{eq:fin_model} thus turns into
\begin{align}
\label{eq:WTV_hyp}
\begin{split}
% \left\{\bm{u}^*,\bm{\Theta}^*\right\} =
\left\{\bm{u}^*,\bm{\alpha}^*\right\} %\left\{\bm{u}^*,\bm{\alpha}^*\right\}
\;{\in}\;&\argmin{\bm{u}\in\R^N,\,\bm{\alpha}\in\mathbb{R}_+^N}\Bigg\{\sum_{i=1}^{N}\alpha_i\,\|(\bm{\D u})_i\|_2 \;{-}\; \sum_{i=1}^N\ln\alpha_i +\mu\,\mathrm{L}_{q}(\bm{\A u};\bm{b})\Bigg\}\nonumber\\
& \text{with}\quad\mu\;{:=}\;q\,\omega^q\,,
\end{split}
\end{align}
where the function $h_{\mathrm{WTV}}$ is defined by
\begin{equation}
%h_{\mathrm{WTV}}(\bm{\Theta}_i) =
h_{\mathrm{WTV}}(\alpha_i)  = -\ln\alpha_i\,,\quad i=1,\ldots,N\,.
\end{equation}
For the WTV$_{\bm{p}}^{sv}$, we have that $\bm{\Theta}=(\bm{\alpha},\bm{p})\in \mathbb{R}_{++}^N\times\mathbb{R}_{++}^N $. The hypermodel \eqref{eq:fin_model} here specifies into
\begin{align}
\label{eq:WTVp_hyp}
\begin{split}
% \left\{\bm{u}^*,\bm{\Theta}^*\right\}
%= 
\left\{\bm{u}^*,\bm{\alpha}^*,\bm{p}^*\right\}%\left\{\bm{u}^*,\bm{\alpha}^*,\bm{p}^*\right\}
\;{\in}\;&\argmin{\bm{u}\in\R^N,\,\bm{\alpha}\in\mathbb{R}_{++}^N,\, \bm{p}\in\mathbb{R}_{++}^N}\Bigg\{\sum_{i=1}^{N}\alpha_i^{p_i}\,\|(\bm{\D u})_i\|_2^{p_i} \;{-}\; \sum_{i=1}^N\ln\frac{\alpha_i\,p_i}{\Gamma(1\,/\,p_i)}+ \mu\,\mathrm{L}_{q}(\bm{\A u};\bm{b})\Bigg\}\nonumber\\
& \text{with}\quad\mu\;{:=}\;q\,\omega^q\,,
\end{split}
\end{align}
with penalty function $h_{\mathrm{WTV}_{\bm{p}}^{sv}}$ defined by
\begin{equation} \label{eq:WTVp_h2}
%h_{\mathrm{WTV}_\bm{p}^{sv}}(\bm{\Theta}_i) =
h_{\mathrm{WTV}_{\bm{p}}^{sv}}(\alpha_i,p_i) = -\ln\frac{\alpha_i\,p_i}{\Gamma(1/p_i)}\,,\quad i = 1,\dots,N\,.
\end{equation}
Finally, for WDTV$_{\bm{p}}^{sv}$, we have $\bm{\Theta}=(\bm{\alpha},\bm{p},\bm{\theta},\bm{a})\in \mathbb{R}_{++}^N\times\mathbb{R}_{++}^N \times [-\pi/2,\pi/2)^N\times (0,1]^N $, whence the final hypermodel reads
\begin{align}
\label{eq:WDTVp_hyp}
\begin{split}
%\left\{\bm{u}^*,\bm{\Theta}^*\right\} %
\left\{\bm{u}^*,\bm{\alpha}^*,\bm{p}^*,\bm{\theta}^*,\bm{a}^*\right\}
& \;{\in}\; \argmin{\bm{u}\in\R^N, 
	\bm{\alpha}\in\mathbb{R}_+^N,\, \bm{p}\in\mathbb{R}_{++}^N,\,
	\bm{\theta}\in[-\pi/2,\pi/2)^N,\,
	\bm{a}\in(0,1]^N
}\Bigg\{\sum_{i=1}^{N}\alpha_i^{p_i}\,\|\bm{\Lambda}_{a_i}\bm{\mathbf{\R}}_{-\theta_i}(\bm{\D u})_i\|_2^{p_i}\\
&\;{-}\;\sum_{i=1}^N\ln\left(\frac{a_i}{2\pi}\frac{p_i\alpha_i^2}{\Gamma(2\,/\,p_i)2^{2/p_i}}\right) + \mu\,\mathrm{L}_{q}(\bm{\A u};\bm{b})\Bigg\}\,,\quad\text{with}\quad\mu\;{:=}\;q\,\omega^q\,,
\end{split}
\end{align}
with parameter penalty function takes the form
\begin{equation}
%h(\bm{\Theta}_i) =
h_{\mathrm{WDTV}_{\bm{p}}^{sv}}(\alpha_i,p_i,\theta_i,a_i) = -\ln\left(\frac{a_i}{2\pi}\frac{p_i\,\alpha_i^2}{\Gamma(2/p_i)2^{2/p_i}}\right)\,,\quad i = 1,\ldots,N\,.
\label{eq:WDTVp_h2}
\end{equation}

For all considered space-invariant and space-variant hypermodels, we summarise in Table~\ref{tab:models} the gradient and parameter penalty functions $f$ and $h$, respectively, as well as the parameter domains $\mathcal{D}_{\bm{\Theta}_i}$ and the regularisation parameters $\mu$. In the table, we also report the reference papers in which the aforementioned regularisers have been first introduced and/or analysed in probabilistic terms.

\begin{table}[!t]
	\centering
	\resizebox{\textwidth}{!}
	{
		\renewcommand{\arraystretch}{3}
		\renewcommand{\tabcolsep}{0.1cm}
		\begin{tabular}{c|c|c|c|c|c|c|}\hline\hline
			&$\mathcal{R}(\bm{u},\bm{\Theta})$&$f(\bm{g}_i;\bm{\Theta}_i)$&$h(\bm{\Theta}_i)$&$\mathcal{D}_{\bm{\Theta}_i}$&$\mu$&Ref\\
			\hline\hline
			\rowcolor{azure}
			&TV&$\|\bm{g}_i\|_2$&$\displaystyle{-\frac{1}{\alpha}\ln\alpha}$&$\alpha\in\R_{++}$&$\displaystyle{\frac{q\,\omega^q}{\alpha}}$&\cite{ROF}\\
			%\hline
			%\rowcolor{Gray}
			\rowcolor{azure}
			&TV$_p$&$\|\bm{g}_i\|_2^p$&$\displaystyle{-\frac{1}{\alpha^p}\ln\frac{\alpha\,p}{\Gamma(1/p)}}$&$(\alpha,p)\in\R_{++}^2$&$\displaystyle{\frac{q\,\omega^q}{\alpha^p}}$&\cite{tvpl2}\\
			%\hline
			\rowcolor{azure}
			\multirow{-3}{*}{\STAB{\rotatebox[origin=c]{90}{space-invariant}}}&DTV&$\|\bm{\Lambda}_a\bm{\mathrm{R}}_{-\theta}\bm{g}_i\|_2$&$\displaystyle{-\frac{1}{\alpha}\ln\left(\frac{a}{2\pi}\frac{\alpha^2}{4}\right)}$&$(\alpha,\theta,a)\in\R_{++}\times[-\pi/2,\pi/2)\times(0,1]$&$\displaystyle{\frac{q\,\omega^q}{\alpha}}$&\cite{BayramDTV2012}\\
			\hline
			\rowcolor{cosmiclatte}%cosmiclatte
			&WTV&$\alpha_i\|\bm{g}_i\|_2$&$\displaystyle{-\ln\alpha_i}$&$\alpha_i\in\R_{++}$&$\displaystyle{q\,\omega^q}$&\cite{CLPS}\\
			%\hline
			\rowcolor{cosmiclatte}
			&WTV$^{sv}_{\bm{p}}$&$\alpha_i^{p_i}\|\bm{g}_i\|_2^{p_i}$&$\displaystyle{-\ln\frac{\alpha_i\,p_i}{\Gamma(1/p_i)}}$&$(\alpha_i,p_i)\in\R_{++}^2$&$\displaystyle{q\,\omega^q}$&\cite{CMBBE}\\
			%\hline
			%\rowcolor{Gray}
			\rowcolor{cosmiclatte}
			\multirow{-3}{*}{\STAB{\rotatebox[origin=c]{90}{space-variant}}}&WDTV$^{sv}_{\bm{p}}$&$\alpha_i^{p_i}\|\bm{\Lambda}_{a_i}\bm{\mathrm{R}}_{-\theta_i}\bm{g}_i\|_2^{p_i}$&$\displaystyle{-\ln\left(\frac{a_i}{2\pi}\frac{p_i\,\alpha_i^2}{\Gamma(2/p_i)\,2^{2/p_i}}\right)}$&$(\alpha_i,p_i,\theta_i,a_i)\in\R_{++}^2\times[-\pi/2,\pi/2)\times(0,1]$&$\displaystyle{q\,\omega^q}$&\cite{DTVp_siam}\\
			\hline
		\end{tabular}
	}
	\caption{Gradient penalty function $f$, parameter penalty function $h$, parameters domain $\mathcal{D}_{\bm{\Theta}_i}$ and regularisation parameter $\mu$ for the space-invariant and space-variant regularisation hypermodels considered with L$_q$ data fidelity, $q<+\infty$.}
	\label{tab:models}
\end{table}

\begin{remark}[Additive i.i.d. uniform noise]
	When the corrupting noise is AIU, i.e. $q=+\infty$ and the data term is written as in \eqref{eq:fid2}, the regularisation parameter $\mu$ does not appear explicitly in the final hypermodels. However, one can clearly observe that functions $f,h$ and the parameter domain $\mathcal{D}_{\bm{\Theta}_i}$ have the same expressions as the ones listed in Table~\ref{tab:models}.
\end{remark}

As far as the value of the regularisation parameter $\mu$ is concerned, we remark  that when both the scale parameter $\omega$  and the shape parameter $q$ of the AIGG noise distribution are assumed to be known, the parameter $\mu>0$ in all models above is also known.
%
% Analogous considerations hold for the space-invariant hypermodels where $\mu$ can be considered fixed once that also an estimate for $\alpha$ (and $p$) is available.
%As one may expect, in the sequel $\mu$ will play the usual role of the regularisation parameter - namely, balancing the effects of regularisation and data-fidelity - in the considered variational Bayes model \eqref{eq:map}.
%Actually, here both the scale parameter $\beta$ - or, equivalently, the standard deviation $\sigma$ - and the shape parameter $q$ of the AWGG noise corruption are assumed to be known. Hence, according to the definition in \eqref{eq:Lqfid}, the value of the regularisation parameter $\mu$ can also be known. 
%
However, it is quite well known that setting a priori the $\mu$ value based on the true noise parameters does not guarantee that the empirical noise level calculated starting from the output residual image $\bm{r}^*(\mu)=\bm{\A u}^*(\mu)-\bm{b}$ coincides with the true underlying noise level. 
%makes the resulting variational hypermodel highly unstable {\color{red} provide more details ??}. This holds also in the case of space-variant regularisers. 

Hence, $\mu$ will be regarded in the following as a further unknown parameter to be estimated based on the GDP strategy presented in Section \ref{sec:gdp} and 
further detailed in Section~\ref{sec:admm}.

\section{Coupling image statistics with variational modelling: parameter selection}   \label{sec:parameter_estimation}

In this section, we address the estimation of the parameters $\bm{\Theta}$ arising in the final joint hypermodel \eqref{eq:fin_model}. As pointed out in  Section \ref{sec:map}, a key step considered in the following for tackling the $\bm{\Theta}$-update step in the alternating scheme \eqref{eq:sub_th}-\eqref{eq:sub_u} consists in neglecting the normalisation constant $c(\bm{\Theta})$. Although this approximation causes of course a lack of consistency with the original model, the estimation results reported in this section will support the rationale of our choice. An extensive analysis of the good statistical properties of the estimator considered in the sequel has been provided in \cite[Section 7]{DTVp_siam}. There, the authors showed that the considered estimator is unbiased, with empirical variance and root mean square error decaying to zero.

% \red{Explain the ML-parameter selection procedure and the dependence on neighbourhood radius appearing there. Given the averaging-type procedure here, shall we introduce a pointer towards hyper-prior ideas?}

\subsection{Inspecting space-variance}  \label{sec:parameter_estimation_results}

% \textcolor{blue}{
% Here showing estimated 1D and 2D PDFs on synthetic images showing local adaptation to image structures, similar to SIAM paper.
% The proposed ML-based parameter estimation considered adapts very well to the local estimation of the distribution of image gradients or their norms. 
% It is therefore natural wondering how to embed this strategy into a numerical solver for computing the solution of some exemplar image restoration problem. Compared to other non-space variant models, this has the great advantage of combining model-driven information (that is, the a-priori GGD modelling on image gradients) with data-driven ones (the specificity of smoothness and anisotropy information estimated directly from the data at any pixel). The price to pay for this is the large computational cost required to perform the parameter estimation, which in principle one would like to iterate along the iterations, but which, in order to reduce the computational efforts, is in practice performed once-for-all as a initialisation step. }
According to \eqref{eq:fin_model}, the general form of the $\bm{\Theta}$-update \eqref{eq:sub_th}
%In this section, we address the solution of the , whose general form, according to \eqref{eq:fin_model}, 
reads
\begin{align}\nonumber
\bm{\Theta}^{(k+1)} \;{\in}\;&\argmin{\bm{\Theta}\in \mathcal{D}_{\bm{\Theta}}}\left\{-\ln\mathbb{P}(\bm{z}(\bm{u}^{(k)})\mid\bm{\Theta}) \right\}%\;{=}\; \argmin{\bm{\Theta}\in C_{\bm{\Theta}}}\left\{\sum_{i=1}^{N}g((\bm{\D u^{(k)}})_i;\bm{\Theta}_i)\right\}
\\
\label{eq:marg_min}
\;{=}\;&\argmin{\bm{\Theta}\in \mathcal{D}_{\bm{\Theta}}}\left\{\sum_{i=1}^{N}\left(f((\bm{\D u^{(k)}})_i;\bm{\Theta}_i) + h(\bm{\Theta}_i)\right)\right\}\,,
\end{align}
where $f$ and $h$ denote the general gradient and parameter penalty functions, respectively summarised in Table~\ref{tab:models} for the hypermodels of interest.
In light of the separability induced by the summation, problem \eqref{eq:marg_min} can be addressed by solving $N$ minimisation problems of the form
%First, observe that when selecting the prior among the ones reviewed in Section \ref{subsec:prior}, the objective function in \eqref{eq:marg_min} can be always expressed as a separable sum
%\begin{equation}
%\label{eq:sep_sum}
%-\ln\mathbb{P}(\bm{u}^{(k)}\mid \bm{\Theta})\;{=}\; \sum_{i=1}^{N}g((\bm{\D u}^{(k)})_i;\Theta_i)\,,
%\end{equation}
%where the analytic form of function $g$ changes according to the chosen prior.
%Problem \eqref{eq:marg_min} can be thus equivalently recast as follows
\begin{equation}
\label{eq:marg_sep_min}
\bm{\Theta_i}^{(k+1)} \in \argmin{\bm{\Theta_i}\in \mathcal{D}_{\bm{\Theta}_i}}\left\{f((\bm{\D u}^{(k)})_i;\bm{\Theta_i})+h(\bm{\Theta}_i)\right\}\,, \quad i\;{=}\;1 \ldots,N,
\end{equation}
where, notice, the information on $\bm{u}^{(k)}$ required to perform the update of $\bm{\Theta}_i$, is synthesised in the sole value $(\bm{\D u}^{(k)})_i$.
However, if the true underlying value $(\bm{\D u})_i$ is highly damaged by noise and blur, the local estimate is expected to be particularly poor and unreliable.

As a way to overcome this limitation, the estimation problem \eqref{eq:marg_sep_min} can be recast so as take into account the information encoded in a set of pixels close to pixel $i$. More specifically, for any $i=1,\ldots, N$, we consider the square neighbourhood $\mathcal{J}_i^r$ centred at pixel $i$ with side $2r+1$ and dimension $card(\mathcal{J}_i^r)=(2r+1)^2=:m$ and compute the discrete gradients points in $\mathcal{J}_i^r$. These quantities will be then used for the estimation of the $i$-th unknown parameter $\bm{\Theta}_i$. Statistically, the selected strategy relies on the assumption that in each of the considered neighbourhoods the gradients (or their magnitudes) are independently sampled from the same distribution.
%\textcolor{red}{ As a consequence, the maps of parameters can be thought as being, in expectation, piece-wise affine.}

We thus introduce the following sets of samples drawn around $i$, for $i=1,\ldots,N$:
\begin{equation}\label{eq:sampl}
\mathcal{S}_i := \left\{
\begin{array}{ll}
\left\{\|(\bm{\D u}^{(k)})_j\|_2\, :\, j\in\mathcal{J}_i^r\right\}&\text{   for the  }\WTV\text{  and  }\WTV^{sv}_{\bm{p}}\text{  regularisers }  \\
\left\{(\bm{\D u}^{(k)})_j\, :\, j\in\mathcal{J}_i^r\right\}& \text{   for the  }\WDTV^{sv}_{\bm{p}}\text{  regulariser }  
\end{array}\right..
\end{equation}
% and the set of indices $\mathcal{J}_i^r = \left\{j\,:\, u_j \in \mathcal{C}_i^r\right\}$, with $card(\mathcal{J}_i^r)=m$.

For each $i=1,\ldots,N$, by exploiting the mutual independence of the gradients, problem \eqref{eq:marg_sep_min} can thus be formulated as follows:
\begin{align}
\nonumber
\bm{\Theta}_i^{(k+1)} \;{\in}\;& \argmin{\bm{\Theta}_i\in \mathcal{D}_{\bm{\Theta}_i}}\left\{-\ln\prod_{j\in\mathcal{J}_i^r}\mathbb{P}(\mathcal{S}\mid \bm{\Theta}_i)\;{=}-\ln\mathbb{P}((\bm{\D u}^{(k)})_j;\bm{\Theta}_j)\right\}\\
\label{eq:th_sep_ens}
\;{=}\;&\argmin{\bm{\Theta}_i\in \mathcal{D}_{\bm{\Theta}_i}}
\left\{\sum_{j\in\mathcal{J}_i^r}\left(f((\bm{\D u}^{(k)})_j;\bm{\Theta}_j)+h(\bm{\Theta}_j)\right)\right\}\,.
\end{align}

We now specify the formulation of the minimisation problem \eqref{eq:th_sep_ens} in correspondence with the regularisation terms considered in this review. For the sake of better readability, in this section the outer iteration superscript $k$ will be neglected and the discrete gradient $(\bm{\D u})_j$ at pixel $j$ will be simply denoted by $\bm{g}_j$.

\subsection{Parameter estimation for the WTV regulariser} \label{sec:thWTV}
%We denote by $\mathcal{S}_i$ the set of gradient magnitudes computed at pixels belonging to $\mathcal{C}_{i}^{r}$; in formula, $\mathcal{S}_i=\left\{\|(\bm{\D u})_j\|_2\right\}_{j=1}^{m}$.
%the set $\mathcal{S}_i:=\{\left(\D u\right)_j\}_{j=1}^m$, with $u=u^{(k)}$. The gradients $\left(\D u\right)_j$ are computed at pixels belonging to the square neighbourhood $\mathcal{C}_{i}^{r}$ centred at pixel $i$ with side $2r+1$. 
%We assume the samples in $\mathcal{S}_i$ to be independent. One can notice that the former assumption is more restrictive than the one which allows to model $u$ as a MRF. \red{DARE QUALCHE DETTAGLIO SULLA DISCRETIZZAZIONE}.
%We also recall that in the WTV Gibbs' prior, the vector of unknown parameters reduces to $\bm{\Theta }= \bm{\alpha}\in\R_{++}^{N}$. Therefore, due to the mutual independence of the samples, the conditioned pdf $\mathbb{P}(\mathcal{S}_i\mid \alpha_i)$ can be written as
%\begin{equation}
%   \mathbb{P}(\mathcal{S}_i\mid \alpha_i) \;{=}\; \prod_{j=1}^{m} \mathbb{P}(\|(\bm{\D u} )_j\|_2\mid \alpha_i) = \prod_{j=1}^{m}\exp\left(-g((\bm{\D u})
%  ^{(k)})_j;\alpha_i)\right)\,,
%\end{equation}
%where
We start considering $\WTV$ regularisation.  Recalling the definition of the gradient and parameter penalty functions $f_{\mathrm{TV}}$, $h_{\mathrm{TV}}$ in \eqref{eq:TV_h},\eqref{eq:TV_h2} and the hyperparameter domain $\mathcal{D}_{\bm{\Theta}_i}$  specified in Table~\ref{tab:models}, the problem of interest turns into
%\begin{equation}
%g((\bm{\D u})_j;\alpha_i) = -\ln\alpha_i + \alpha_i \|(\bm{\D u} )_j\|_2\,.
%\end{equation}
%Therefore, problem \eqref{eq:marg_sep_min} reduces to
\begin{equation}
\label{eq:min_par_wtv}
\alpha_i^{*}\in\argmin{\alpha_i\in \mathbb{R}_{++}}\left\{\mathcal{G}(\alpha_i)\;{:=}\;-\ln\mathbb{P}(\mathcal{S}_i\mid \alpha_i)\;{=}\; -m\ln\alpha_i + \sum_{j\in\mathcal{J}_i^r}\alpha_i\|\bm{g}_j\|_2\right\}\,.
\end{equation}
%Therefore, instead of solving \eqref{eq:marg_sep_min} we are rather interested in addressing the following minimisation problem
%\begin{equation}
%\label{eq:marg_loc_min}
%\Theta_i^* \in \argmin{\Theta_i}\left\{\sum_{j=1}^{m}h((\D u)_j;\Theta_i)\right\}\,, \quad u = u^{(k)}\,,\quad i\;{=}\;1 \ldots,n\,.
%\end{equation}
%In the following, we will detail the specific forms of \eqref{eq:marg_loc_min} when the regularisation terms considered in this review are adopted.
The following result holds true.
\begin{proposition}\label{prop:ex_scale}
	The function $\mathcal{G}:\mathbb{R}_{++}\to\R$ in \eqref{eq:min_par_wtv} is smooth and convex, hence it admits a unique global minimiser.
\end{proposition}

In particular, since $\mathcal{G}$ is differentiable on $\mathbb{R}_{++}$, the solution of the $i$-th minimisation problem \eqref{eq:min_par_wtv} can be simply found by imposing a first-order optimality condition:
\begin{equation}
\label{eq:upd_par_wtv}
\mathcal{G}'(\alpha_i) = -\frac{m}{\alpha_i} + \sum_{j\in\mathcal{J}_i^r}\|\bm{g}_j\|_2 = 0\,,\text{ whence }\alpha_i^{*} = \left(\frac{1}{m}\sum_{j\in\mathcal{J}_i^r}\|\bm{g}_j\|_2\right)^{-1}\,.
\end{equation}
Notice that in order to avoid degenerate configurations arising when considering neighbourhoods with null gradients, a small regularisation parameter $0<\varepsilon\ll 1$ can be added to the local mean in \eqref{eq:upd_par_wtv}. The selection of pixels involved in \eqref{eq:upd_par_wtv} can be efficiently carried out based on fast 2D convolution operators (realised by a fast 2D discrete transform) of the map of gradient norms with a square $(2r+1) \times (2r+1)$ averaging kernel.

In Figure \ref{fig:sky_wtv}, we analyse the performance of the parameter estimation strategy outlined above on selected sub-regions of the image in Figure \ref{fig:sky_im} as well as on the image itself.  %We remark that the small $\epsilon>0$ which allows the formula in \eqref{eq:upd_par_wtv} to be always defined, is such the estimated parameters can at most be equal 255.
The local neighbourhoods shown here consist of an almost constant red-bordered region and two textured regions - see Figure \ref{fig:wtv_sub1} and Figures \ref{fig:wtv_sub2},\ref{fig:wtv_sub3}, respectively - the last two differing in terms of directional features; in fact, the magenta-bordered region presents horizontally oriented features, while the texture in the cyan-bordered neighbourhood does not present a dominant directionality.

We compute the hL pdfs returning the best fitting both of the  global and of the local histograms of the gradient magnitudes. More specifically, we first calculate \eqref{eq:upd_par_wtv} for the whole image, i.e. when the summation index $j$ goes from $1$ to $N$, that will return the global scale parameter. Then, the same formula is applied when the set of samples is restricted to the gradient magnitudes of the three sub-regions, so as to obtain local scale parameters. The estimated parameters are reported in the caption.
In Figure \ref{fig:wtv_glob}, and in the close-up in Figure \ref{fig:wtv_glob_zoom}, we show the histogram of the gradient magnitudes of the whole image. The superimposed solid green line represents the global hL distribution. The histogram of the gradient magnitudes in the selected neighbourhoods together with the corresponding estimated pdfs are shown in Figures \ref{fig:wtv_sub1_}-\ref{fig:wtv_sub1_zoom}, for the constant region, and in Figures \ref{fig:wtv_sub2_}-\ref{fig:wtv_sub2_zoom} and \ref{fig:wtv_sub3_}-\ref{fig:wtv_sub3_zoom} for the textured regions. In the local histograms, we also report the global pdf. The comparison immediately reveals how the space-variant approach guarantees a more accurate modelling of local features; this is also reflected into the values of the estimated global and local scale parameters, which appear to be very different from each other, except for the case of the two textured regions. In fact, as discussed before, directional dissimilarities can not be detected when adopting a (univariate) hL prior.

\begin{figure}[]
	\centering
	\begin{subfigure}{0.32\textwidth}
		\centering
		\includegraphics[width=1.2in]{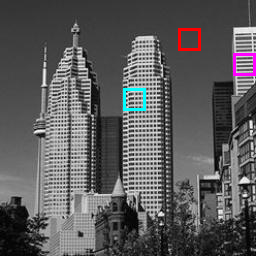} 
		\caption{Test image}
		\label{fig:sky_im}
	\end{subfigure}
	\begin{subfigure}{0.32\textwidth}
		\centering
		\includegraphics[width=1.6in]{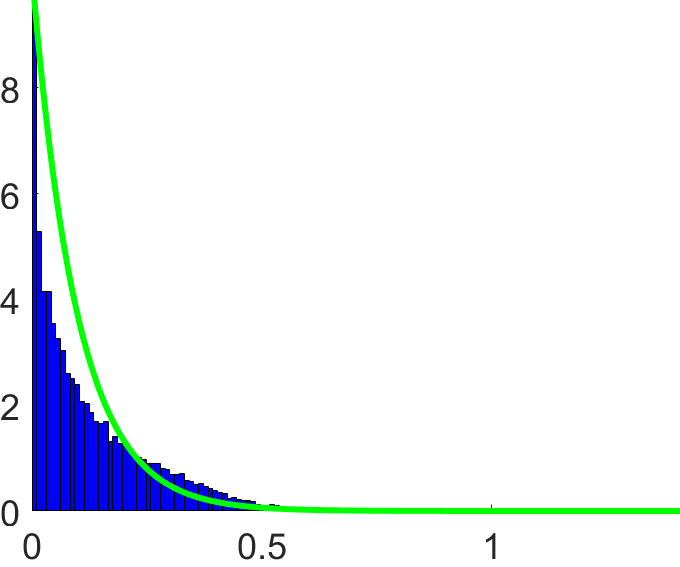} 
		\caption{Global histogram}
		\label{fig:wtv_glob}
	\end{subfigure}
	\begin{subfigure}{0.32\textwidth}
		\centering
		\includegraphics[width=1.6in]{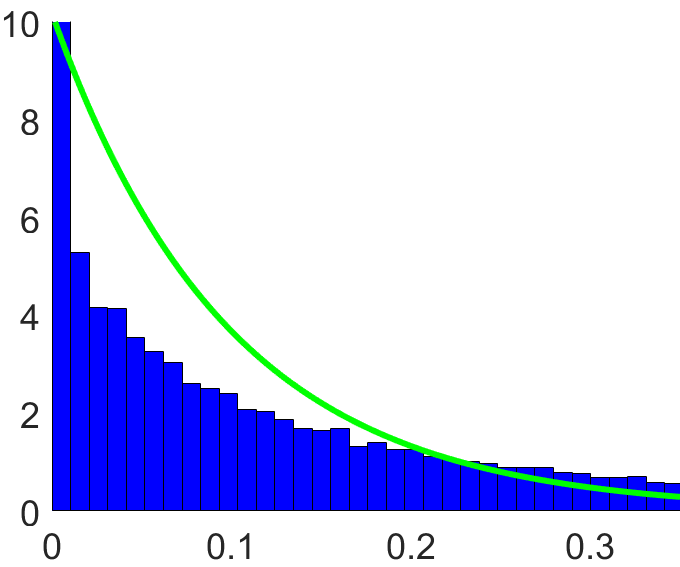}
		\caption{close-up}
		\label{fig:wtv_glob_zoom}
	\end{subfigure}\\
	\begin{subfigure}{0.32\textwidth}
		\centering
		\includegraphics[width=1.2in]{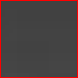} 
		\caption{local histogram}
		\label{fig:wtv_sub1}
	\end{subfigure}
	\begin{subfigure}{0.32\textwidth}
		\centering
		\includegraphics[width=1.6in]{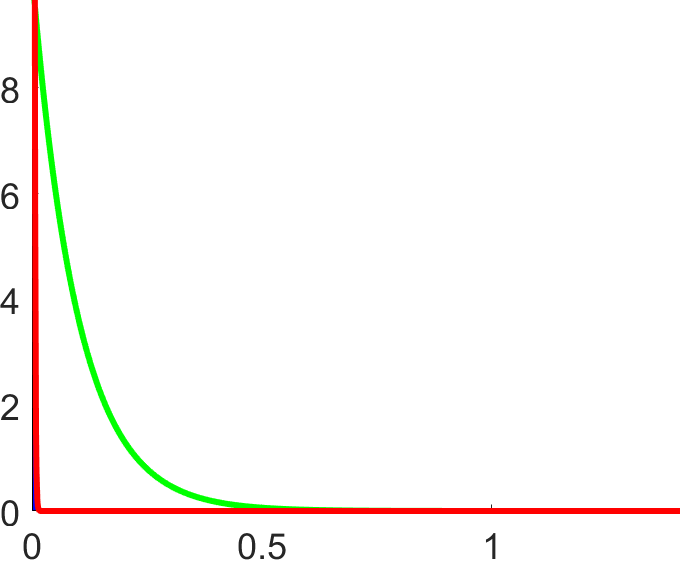} 
		\caption{local histogram}
		\label{fig:wtv_sub1_}
	\end{subfigure}
	\begin{subfigure}{0.32\textwidth}
		\centering
		\includegraphics[width=1.6in]{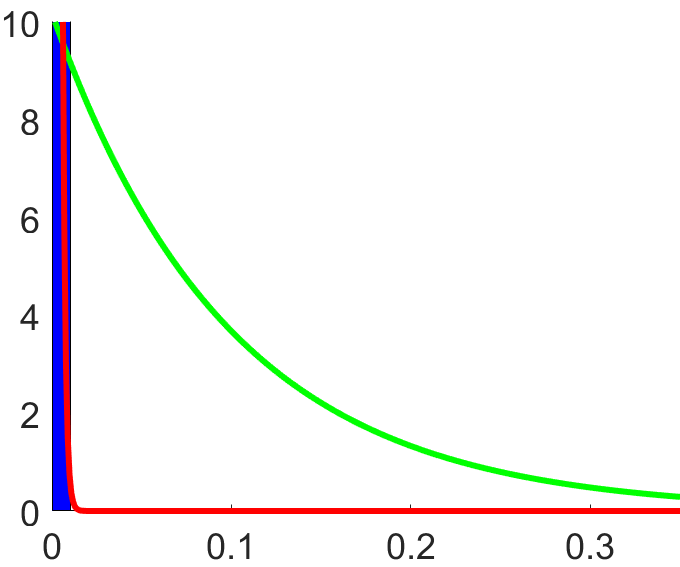}
		\caption{close-up}
		\label{fig:wtv_sub1_zoom}
	\end{subfigure}\\
	\begin{subfigure}{0.32\textwidth}
		\centering
		\includegraphics[width=1.2in]{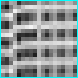} 
		\caption{local histogram}
		\label{fig:wtv_sub2}
	\end{subfigure}
	\begin{subfigure}{0.32\textwidth}
		\centering
		\includegraphics[width=1.6in]{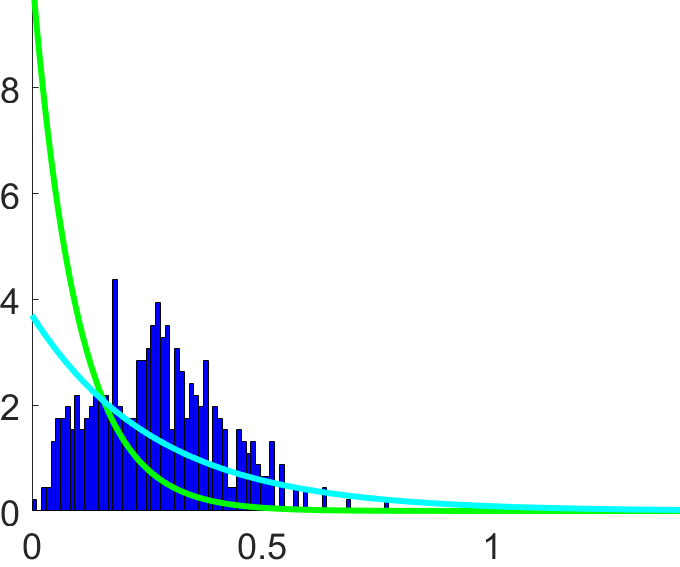} 
		\caption{local histogram}
		\label{fig:wtv_sub2_}
	\end{subfigure}
	\begin{subfigure}{0.32\textwidth}
		\centering
		\includegraphics[width=1.6in]{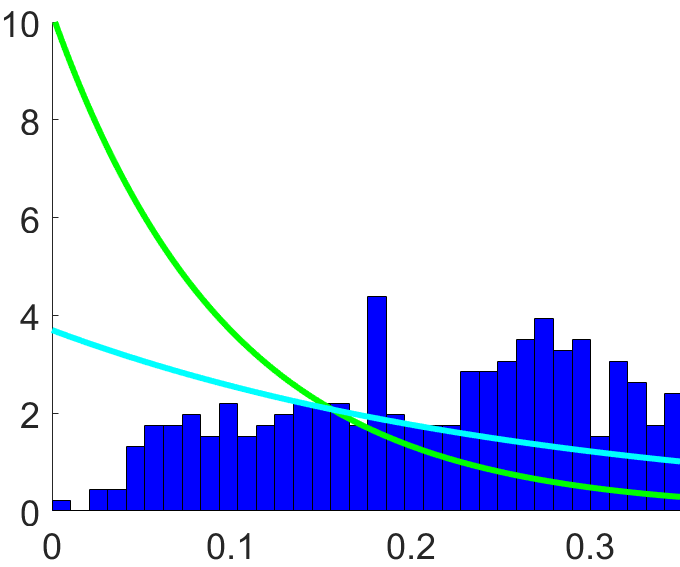}
		\caption{close-up}
		\label{fig:wtv_sub2_zoom}
	\end{subfigure}\\
	\begin{subfigure}{0.32\textwidth}
		\centering
		\includegraphics[width=1.2in]{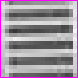} 
		\caption{local histogram}
		\label{fig:wtv_sub3}
	\end{subfigure}
	\begin{subfigure}{0.32\textwidth}
		\centering
		\includegraphics[width=1.6in]{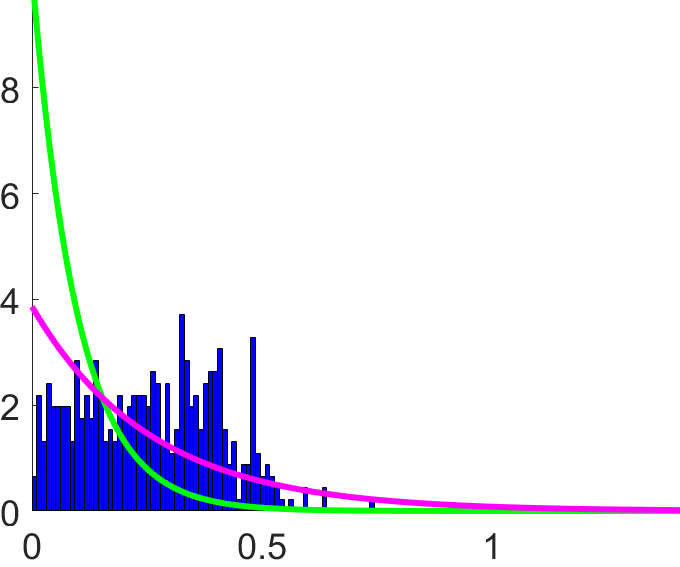} 
		\caption{local histogram}
		\label{fig:wtv_sub3_}
	\end{subfigure}
	\begin{subfigure}{0.32\textwidth}
		\centering
		\includegraphics[width=1.6in]{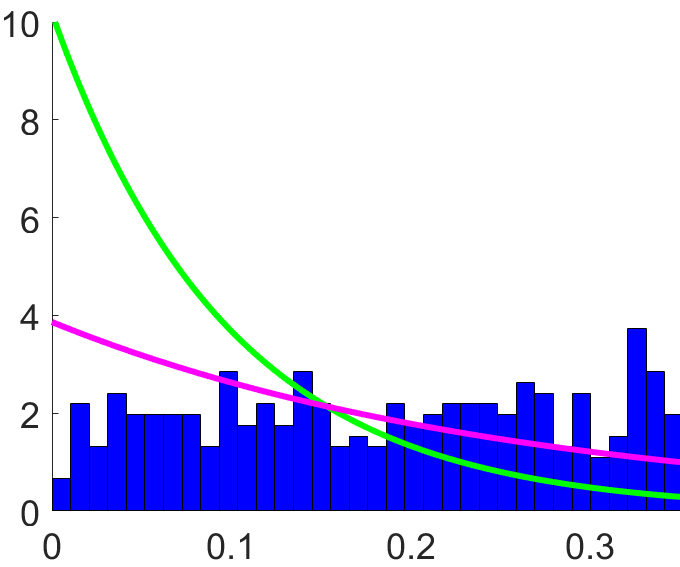}
		\caption{close-up}
		\label{fig:wtv_sub3_zoom}
	\end{subfigure}
	\caption{\emph{Parameters estimate for the $\WTV$ regulariser.} From top to bottom: histogram of the gradient magnitudes on the whole test image ($\alpha^*=10.17$), on a constant region ($\alpha^*=200.87$) and on two different texture regions ($\alpha^*=3.70$ and $\alpha^*=3.86$), with the corresponding close-up(s).}
	\label{fig:sky_wtv}
\end{figure}

\medskip

%The analysis carried out in Fig.~\ref{fig:sky_wtv} motivates the need of a space variant approach from the point of view of the probabilistic modelling. 
In order to analyse in more detail the connection between the estimated scale parameters and the local regularisation strength, in Figure \ref{fig:scale_} we show the $\bm{\alpha}$-map corresponding to different test images.
We observe that the scale parameters assume higher values on smooth or piece-wise constant regions, whereas lower values are obtained in correspondence of edges and texture. In those areas, a weaker regularisation is indeed preferable in order to preserve details.
Note also that the $\bm{\alpha}$-maps are sensitive to the choice of radius $r$. When considering small values of $r$ - see, for instance, the map on the \texttt{barbara} image with $r=2$ - possibly small artefacts due to image compression or resolution may appear. A similar effect is expected in the presence of noise. On the other hand, setting a large radius $r$, could make some details or finer structures in the image less detectable, as in the case of the map for the \texttt{geometric} image with $r=7$, where inner edges are not visible in the final map.

\begin{figure}
	\centering
	\resizebox{\textwidth}{!}{
		\begin{tabular}{ccccc}
			&&$r=2$&$r=5$&$r=7$\\
			\raisebox{0.1\height}{\rotatebox{90}{\texttt{\phantom{cc}}\texttt{geometric}}}&{ \includegraphics[height=0.95in]{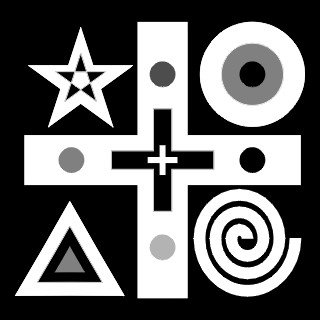}}&{ \includegraphics[height=0.95in]{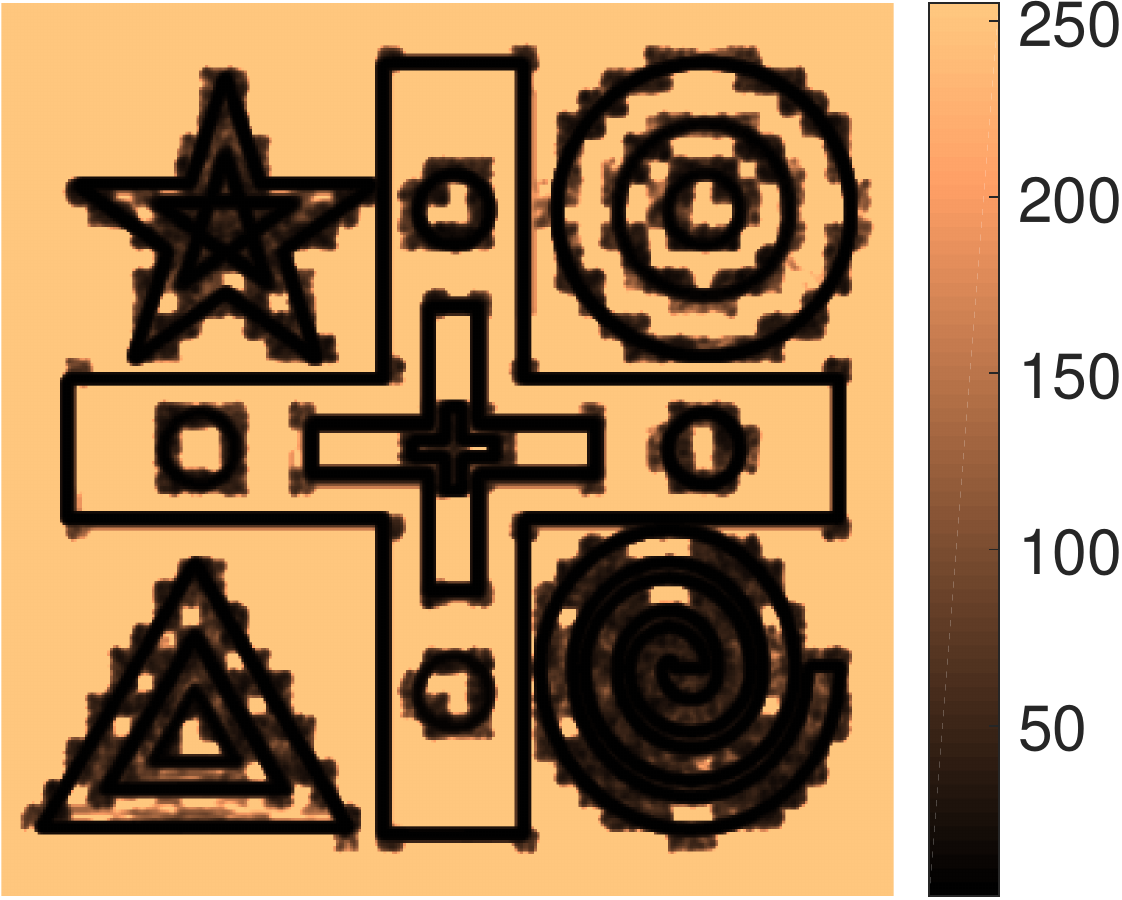} }&{
				\includegraphics[height=0.95in]{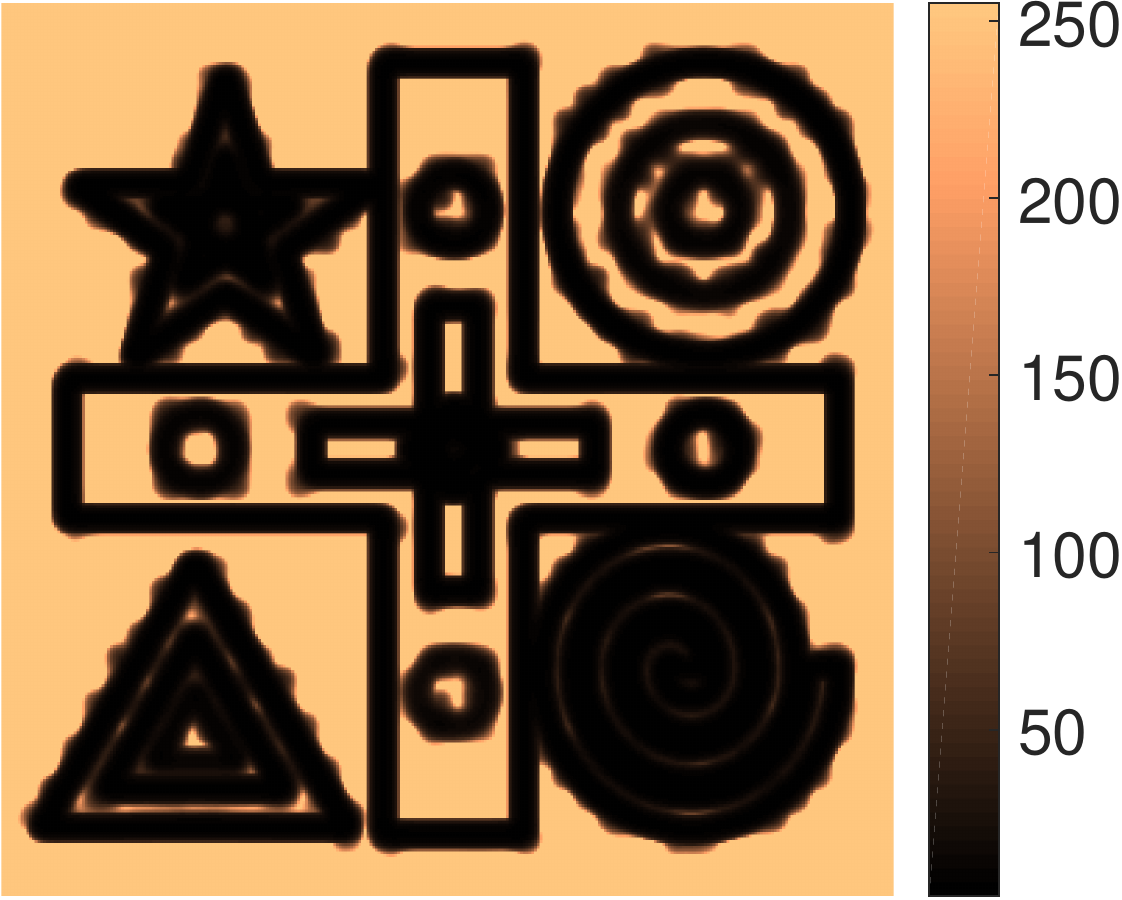}} &{
				\includegraphics[height=0.95in]{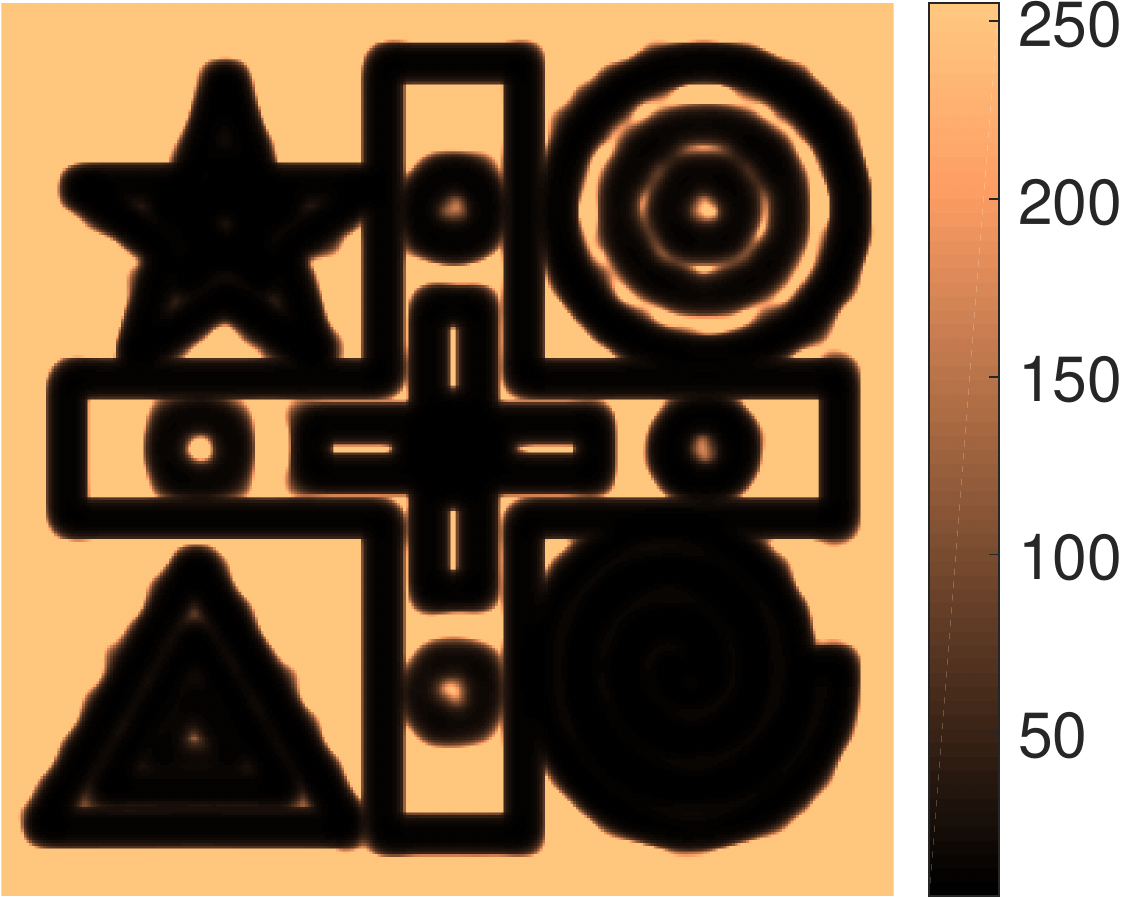}}\\
			\raisebox{0.1\height}{{\rotatebox{90}{\texttt{\phantom{cc}}\texttt{barbara}}}}&{\includegraphics[height=0.95in]{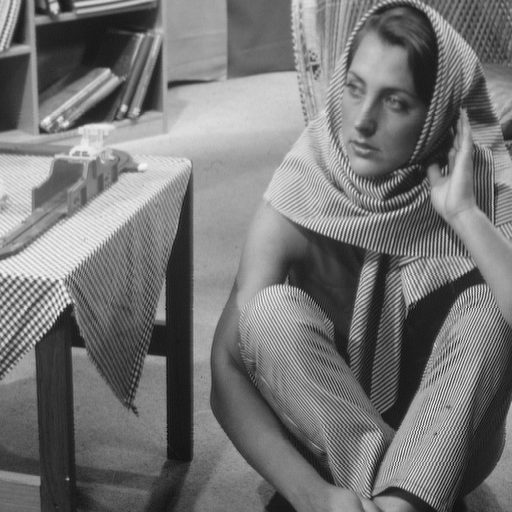} }&
			{\includegraphics[height=0.95in]{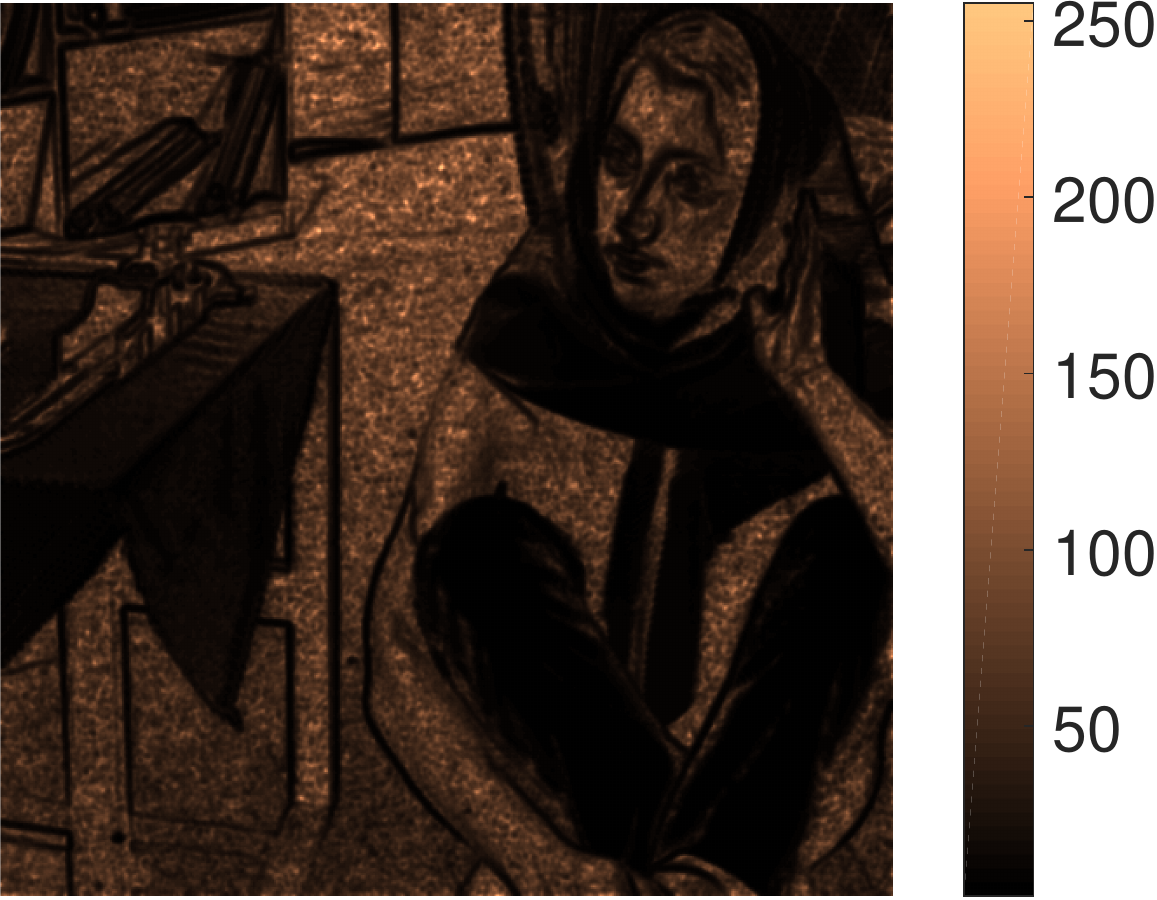} }&
			{	\includegraphics[height=0.95in]{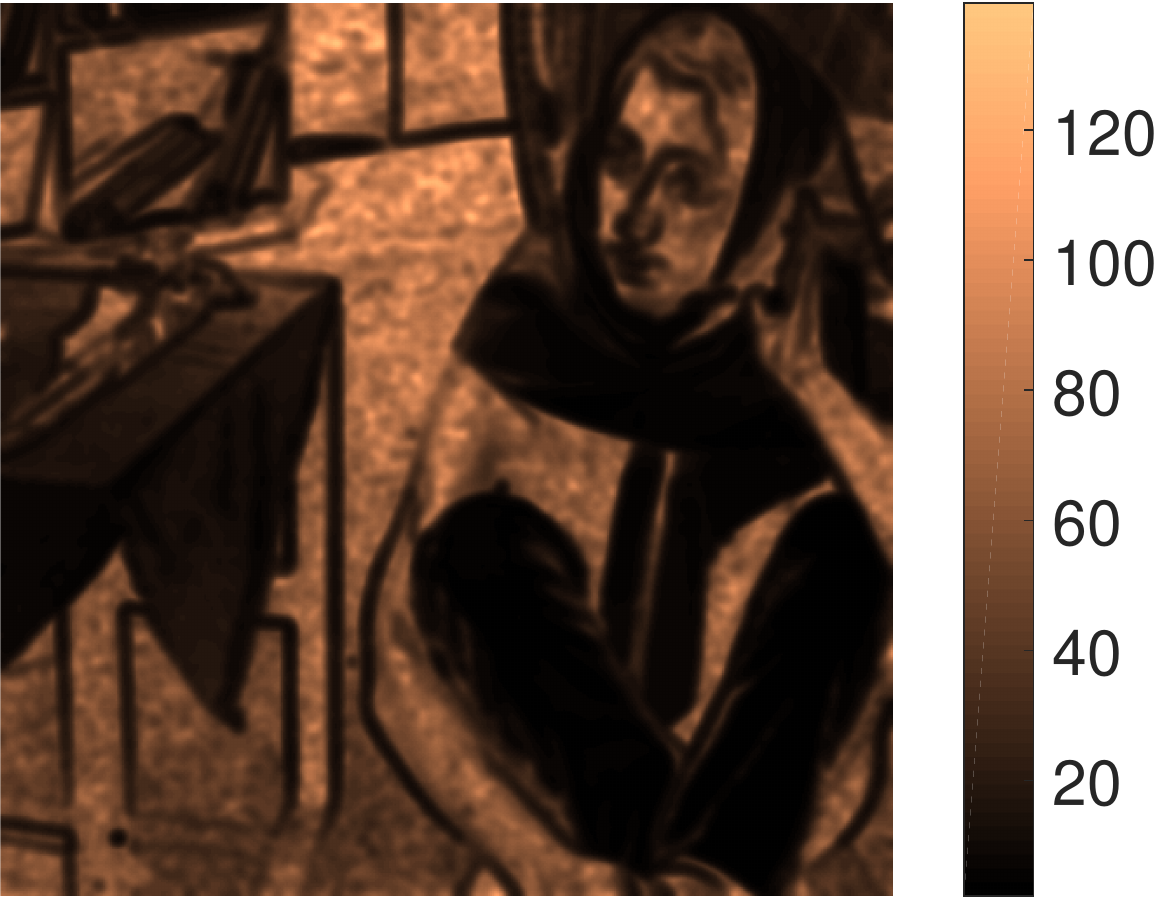}}&
			{	\includegraphics[height=0.95in]{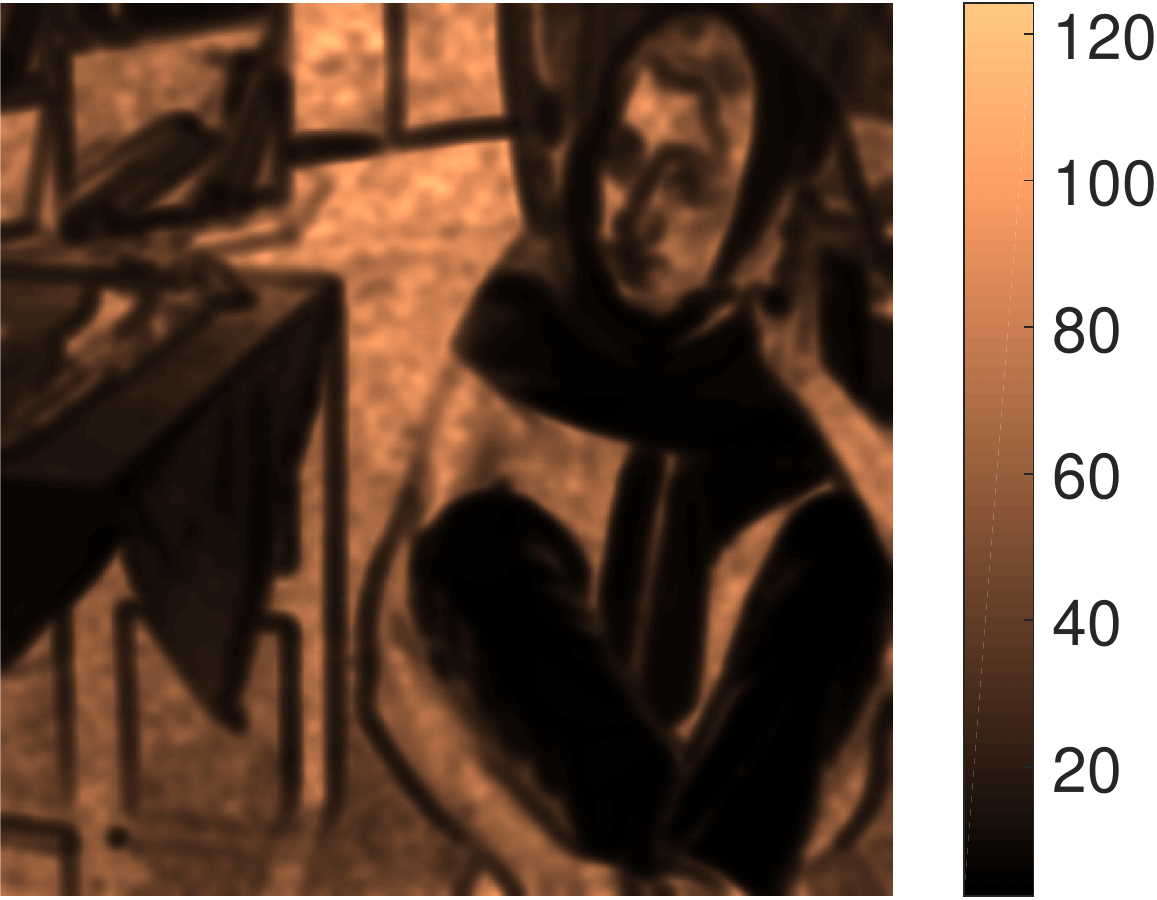} }\\
			\raisebox{0.1\height}{{\rotatebox{90}{\texttt{\phantom{cc}}\texttt{aneurism}}}}&	{\includegraphics[height=0.95in]{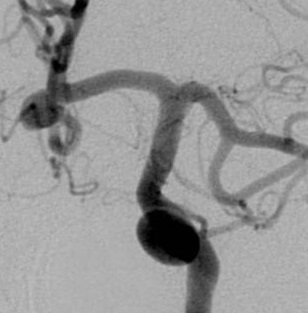} }&
			{	\includegraphics[height=0.95in]{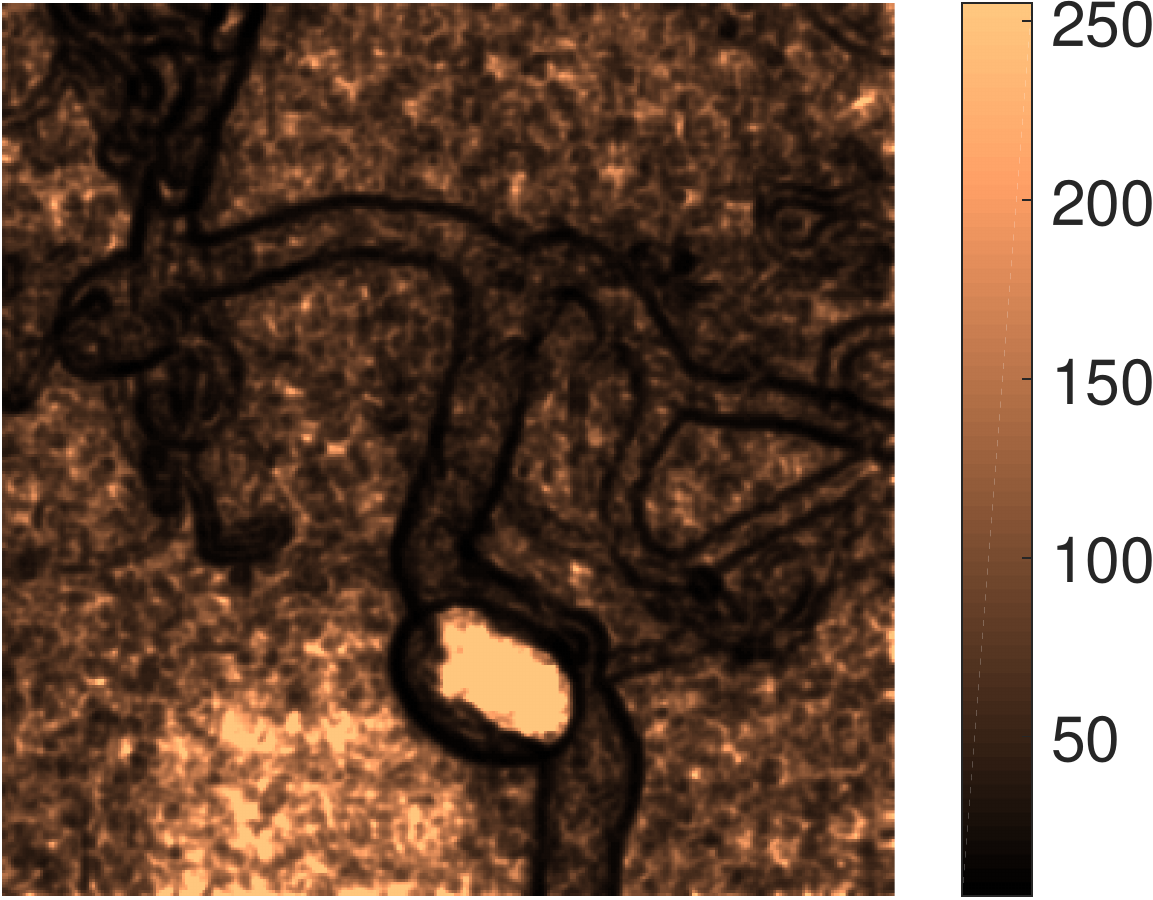}} &
			{	\includegraphics[height=0.95in]{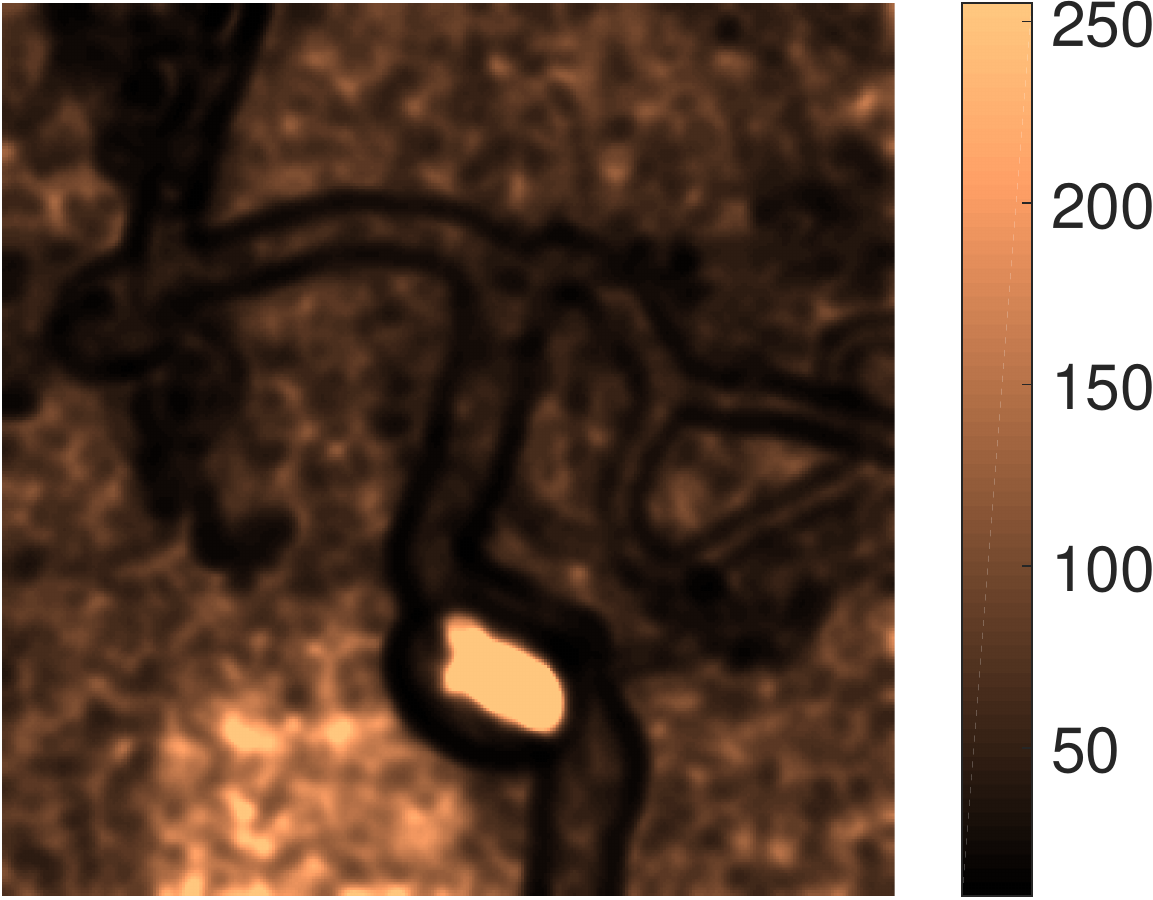}} &
			{\includegraphics[height=0.95in]{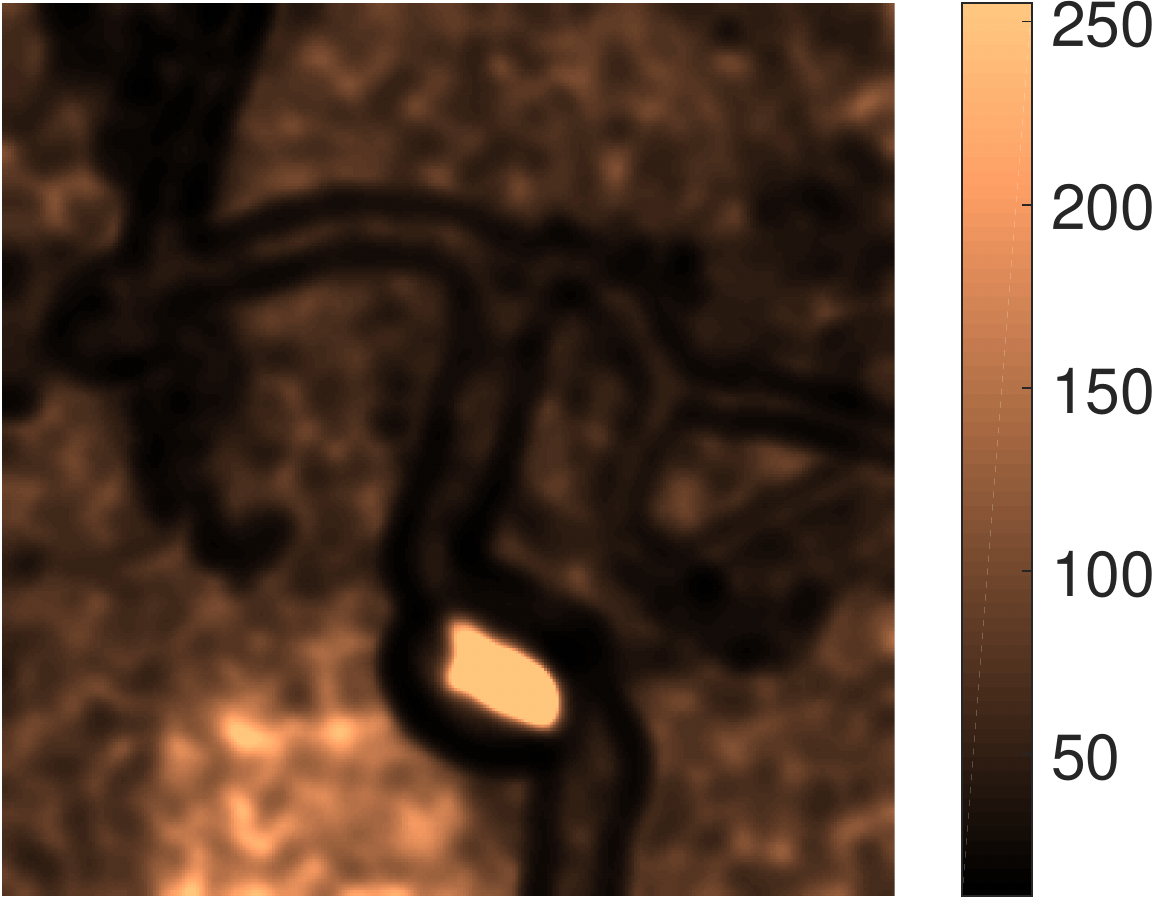} }
	\end{tabular}}
	\caption{Test images with the corresponding $\bm{\alpha}$-map for different values of radius $r$.}
	\label{fig:scale_}
\end{figure}

\subsection{Parameter estimation for the $\mathrm{WTV}_p^{sv}$ regulariser}\label{sec:thWTVp}

We now consider the $\mathrm{WTV}^{sv}_{\bm{p}}$ regulariser. Recalling the definitions for $f_{\mathrm{WTV}_{\bm{p}}^{sv}}$ and $h_{\mathrm{WTV}_{\bm{p}}^{sv}}$ given in \eqref{eq:WTVp_h} and \eqref{eq:WTVp_h2},  respectively, and the hyperparameter space $\mathcal{D}_{\bm{\Theta}_i}$ in Table~\ref{tab:models}, we have that the general problem \eqref{eq:th_sep_ens} reduces to 
\begin{align}
\label{eq:min_par_wtv_p}
\left\{\alpha_i^*,p_i^*\right\}\in&\argmin{(\alpha_i,p_i)\in \mathbb{R}_{++}^2 }\bigg\{\mathcal{G}(\alpha_i,p_i)\;{:=}\;-\ln\mathbb{P}(\mathcal{S}_i\mid \alpha_i,p_i) \;{=}\; -m\ln\alpha_i\\
&\phantom{XXXXXX}\;{+}\;m\ln\Gamma\left(1+\frac{1}{p_i}\right)+\sum_{j\in\mathcal{J}_i^r}\alpha_i^{p_i}\|\bm{g}_j\|_2^{p_i}\bigg\}\,,
\end{align}
Proceeding analogously as before, we have that by imposing a first order optimality condition on $\mathcal{G}(\alpha_i,p_i)$ with respect to $\alpha_i$, we get
\begin{equation}
\frac{\partial}{\partial \alpha_i}G(\alpha_i,p_i)  = -\frac{m}{\alpha_i} + p_i\alpha_i^{p_i - 1}\sum_{j\in\mathcal{J}_i^r}\|\bm{g}_j\|_2^{p_i} = 0\,,
\end{equation}
which yields the following closed-form formula for the estimation of $\alpha_i$:
\begin{equation}
\label{eq:est_scale_shape}
\alpha_i^*(p_i) = \left(\frac{p_i}{m}\sum_{j\in\mathcal{J}_i^r}\|\bm{g}_j\|_2^{p_i}\right)^{-\displaystyle{\frac{1}{p_i}}}\,.
\end{equation}
It is easy to verify that the second derivative of $\mathcal{G}$ with respect to $\alpha_i$ computed at $\alpha_i^*(p_i)$ is strictly positive, hence the stationary point in \eqref{eq:est_scale_shape} is a minimum. Similarly as for \eqref{eq:upd_par_wtv}, also in this case a parameter $0<\varepsilon\ll 1$ shall be added to the summation  \eqref{eq:est_scale_shape} so as to avoid degenerate configurations of gradient magnitudes. Plugging  \eqref{eq:est_scale_shape}, we have
\begin{align}
\label{eq:est_shape_shape}
p_i^*\;{\in}\;& \argmin{p_i\in \R_{++}}\bigg\{G(p_i)\;{:=}\; \mathcal{G}(\alpha_i(p_i),p_i)\;{=}\;\frac{m}{p_i}\log\left(\frac{p_i}{m}\sum_{j\in\mathcal{J}_i^r} \|\bm{g}_j\|_2^{p_i}\right) \\
\notag
&\phantom{XXXX}+m\ln\Gamma\left(1+\frac{1}{p_i}\right)+
\frac{m}{p_i}\bigg\}\,.
\end{align}
When addressing the study of $G$ on $\R_{++}$, one can immediately notice that its behaviour is related to the local configurations of gradient magnitudes. As a result, drawing any conclusion on the existence of minima is in general not trivial.  However, looking at the problem from a computational viewpoint, it appears reasonable to restrict the $p_i$ feasibility set to a bounded interval $[\epsilon,R]$, with $0<\epsilon<R$ and $R>1$. In this case, the following result holds.

%According to the definition of hGGD, the local shape parameter $p_i$ is only required to be strictly positive. However, it is reasonbale to set a lower and an upper bound for the $p_i$ interval research so thatp roblem \eqref{eq:est_shape_shape} can be reformulated as a minimisation problem over a compact constrain set
%\begin{equation}
%\label{eq:est_shape_shape_comp}
%p_i^*\in\arg\min_{p_i\in[\epsilon,M]}h(p_i)\,,
%\end{equation}
%with $h(p_i)$ defined in \eqref{eq:est_shape_shape}.
%\red{Convenient values for $\epsilon$ and $M$}
%When considering the restriction of $G$ to $[\epsilon,M]$, we have:
\begin{proposition}\label{prop:prop_p1}
	The function $G:[\epsilon,R]\to \R$ defined in \eqref{eq:est_shape_shape} is continuous, hence it admits a minimum in its compact domain.
\end{proposition}
In Figure \ref{fig:sky_wtvp}, the estimation of the global and local shape parameters for the $\WTV^{sv}_{\bm{p}}$ regulariser is performed by setting $\epsilon=0.1$ and $R=10$. The dashed green line in Figures \ref{fig:wtvp_glob}-\ref{fig:wtvp_glob_zoom} represents the hGG pdf that best fits the global histogram of the gradient magnitudes where parameters have been estimated as above. One can already observe how the introduction of a further global parameter allows for a better modelling of the global histogram when compared to the solid green line, representing the hL pdf shown in Figure \ref{fig:sky_wtv}. In Figures \ref{fig:wtvp_sub1_},\ref{fig:wtvp_sub2_},\ref{fig:wtvp_sub3_}, we report coloured dashed lines corresponding to the estimated local hGG pdfs; in addition, we superimpose the global hGG pdf  together with the local hL pdfs plotted in Figure \ref{fig:sky_wtv} as solid lines. To facilitate the inspection, we also show close-up(s) of the local histograms in Figures \ref{fig:wtvp_sub1_zoom},\ref{fig:wtvp_sub2},\ref{fig:wtvp_sub3_zoom}.
The benefits associated to the use of a second space-variant parameter are here even more significant. The differences between the selected patches, and between the patches and the global image, is accurately highlighted by the estimated global and local parameters reported in the caption. Note, however that also in this case the selected hGG prior is not capable of detecting directional differences between the two textured sub-regions, due once again to its univariate behaviour.

%-\ref{fig:wtvp_sub1_zoom}, \ref{fig:wtvp_sub2_}-\ref{fig:wtvp_sub2_zoom} and \ref{fig:wtvp_sub3_}-\ref{fig:wtvp_sub3_zoom} the local hGG pdf is reported as a dashed line, together with the local hL pdf of Fig.~\ref{fig:sky_wtv} and the global hGG pdf. 
%Also in this case, the adoption of a space variant approach returns a better modelling of the local image features.

\begin{figure}[!t]
	\centering
	\begin{subfigure}{0.32\textwidth}
		\centering
		\includegraphics[width=1.2in]{images/hist/sky_im.png} 
		\caption{Test image}
		\label{fig:wtvp_sky}
	\end{subfigure}
	\begin{subfigure}{0.32\textwidth}
		\centering
		\includegraphics[width=1.6in]{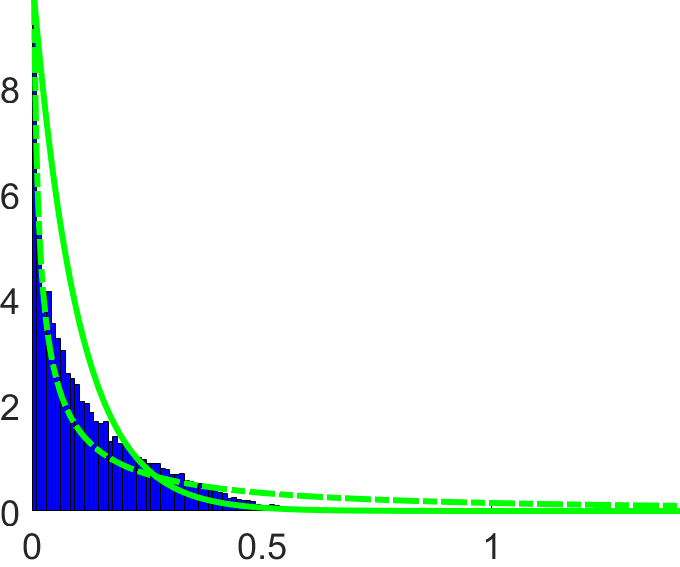} 
		\caption{Global histogram}
		\label{fig:wtvp_glob}
	\end{subfigure}
	\begin{subfigure}{0.32\textwidth}
		\centering
		\includegraphics[width=1.6in]{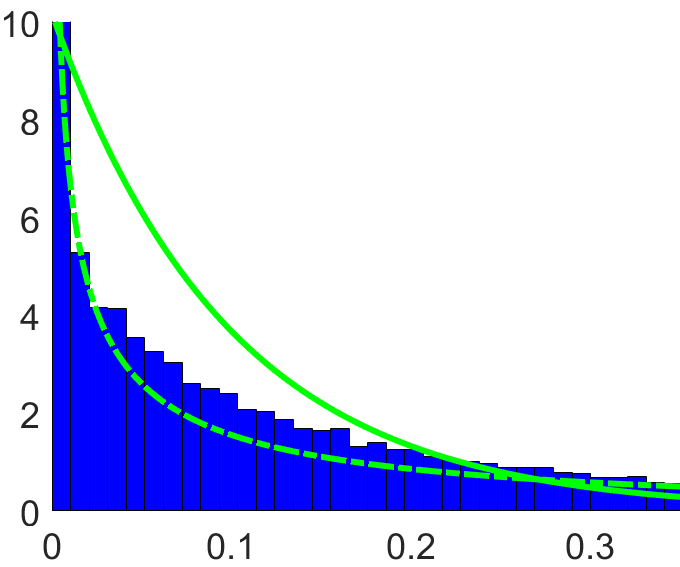}
		\caption{close-up}
		\label{fig:wtvp_glob_zoom}
	\end{subfigure}\\
	\begin{subfigure}{0.32\textwidth}
		\centering
		\includegraphics[width=1.2in]{images/hist/sky_zoom1.png} 
		\caption{local histogram}
		\label{fig:wtvp_sub1}
	\end{subfigure}
	\begin{subfigure}{0.32\textwidth}
		\centering
		\includegraphics[width=1.6in]{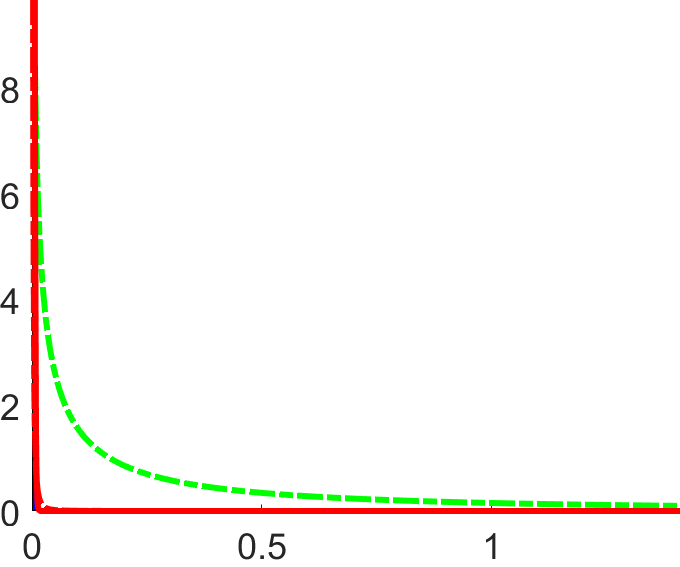} 
		\caption{local histogram}
		\label{fig:wtvp_sub1_}
	\end{subfigure}
	\begin{subfigure}{0.32\textwidth}
		\centering
		\includegraphics[width=1.6in]{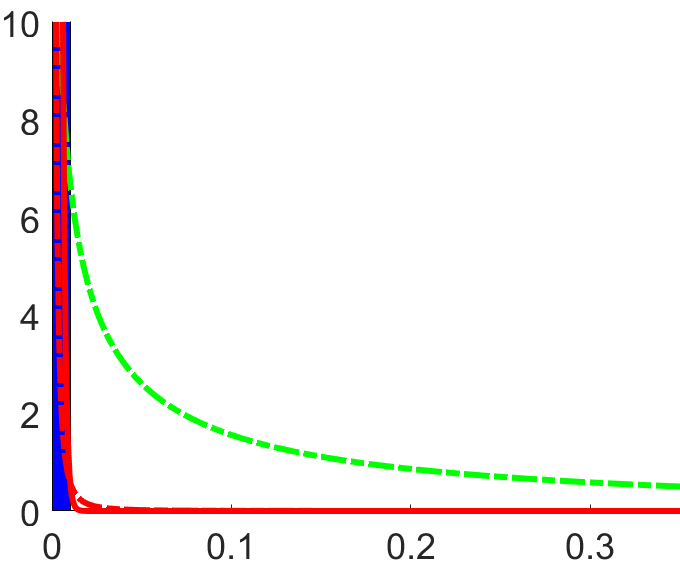}
		\caption{close-up}
		\label{fig:wtvp_sub1_zoom}
	\end{subfigure}\\
	\begin{subfigure}{0.32\textwidth}
		\centering
		\includegraphics[width=1.2in]{images/hist/sky_zoom2.png} 
		\caption{local histogram}
		\label{fig:wtvp_sub2}
	\end{subfigure}
	\begin{subfigure}{0.32\textwidth}
		\centering
		\includegraphics[width=1.6in]{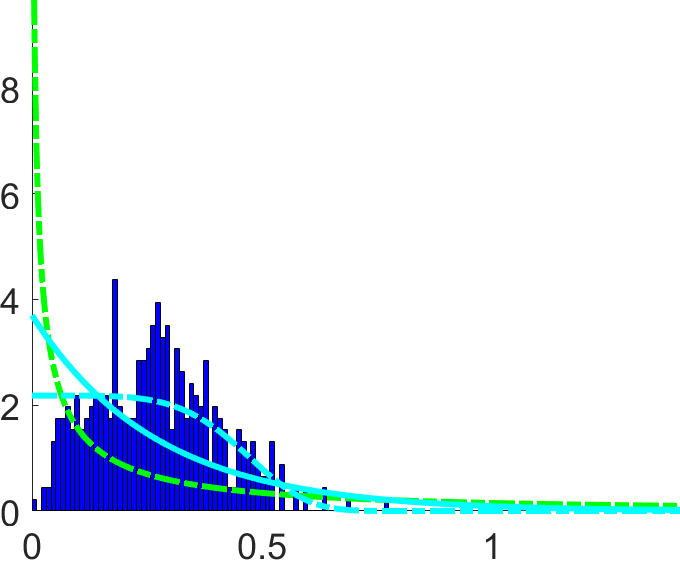}
		\caption{local histogram}
		\label{fig:wtvp_sub2_}
	\end{subfigure}
	\begin{subfigure}{0.32\textwidth}
		\centering
		\includegraphics[width=1.6in]{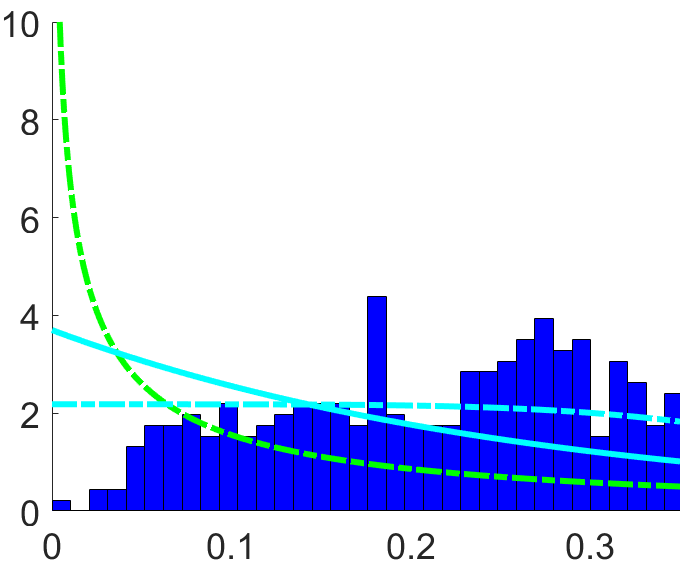}
		\caption{close-up}
		\label{fig:wtvp_sub2_zoom}
	\end{subfigure}\\
	\begin{subfigure}{0.32\textwidth}
		\centering
		\includegraphics[width=1.2in]{images/hist/sky_zoom3.png}
		\caption{local histogram}
		\label{fig:wtvp_sub3}
	\end{subfigure}
	\begin{subfigure}{0.32\textwidth}
		\centering
		\includegraphics[width=1.6in]{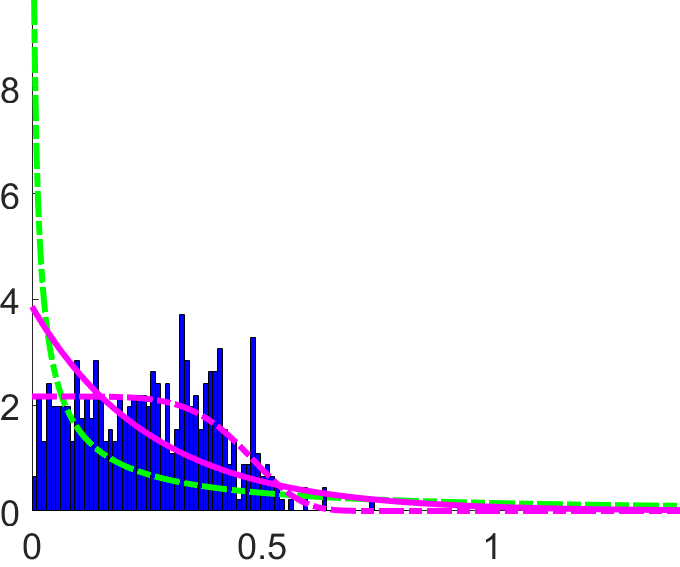}
		\caption{local histogram}
		\label{fig:wtvp_sub3_}
	\end{subfigure}
	\begin{subfigure}{0.32\textwidth}
		\centering
		\includegraphics[width=1.6in]{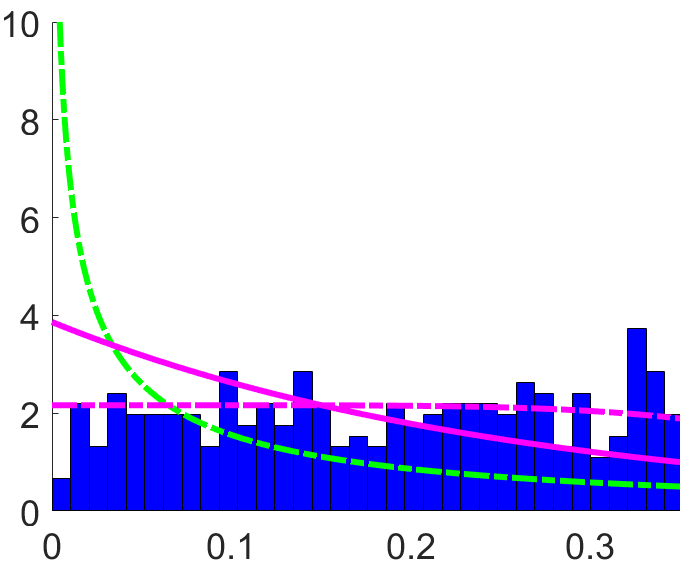}
		\caption{close-up}
		\label{fig:wtvp_sub3_zoom}
	\end{subfigure}
	\caption{\emph{Parameters estimate for the $\WTV^{sv}_{\bm{p}}$ regulariser.} From top to bottom: histogram of the gradient magnitudes on the whole test image ($p^*=0.2, \alpha^*=465.67$), on a constant region ($p^*=0.1, \alpha^*=687.34$) and on two different texture regions ($p^*=4.85, \alpha^*=2.56$ and $p^*=5.65, \alpha^*=2.78$), with the corresponding close-up(s).}
	\label{fig:sky_wtvp}
\end{figure}

Finally, in Figure \ref{fig:sky_maps}  we show the $\bm{\alpha}$- and $\bm{p}$-maps, obtained by considering neighbourhoods of different sizes ($r$) for the image in Figure \ref{fig:sky_im}. In all three cases, the method associates very low $\bm{p}$ values with flat regions (thus promoting enforced sparsity) and higher 
values with texture (where gradients show oscillations). Similarly as what observed for WTV,  the scale parameters $\bm{\alpha}$ are again smaller on regions characterised by finer details, as expected.

\begin{comment}
\begin{figure}[!t]
\centering
\centerline{$r=2$\qquad\qquad\qquad\qquad$r=5$\qquad\qquad\qquad\qquad\qquad$r=7$}
\begin{subfigure}{0.3\textwidth}
\centering
\includegraphics[width=1.4in]{images/maps/sky_p_5.pdf} 
\caption{$\bm{p}$}
\label{fig:shape_1}
\end{subfigure}
\begin{subfigure}{0.3\textwidth}
\centering
\includegraphics[width=1.4in]{images/maps/sky_p_7.pdf} 
\caption{$\bm{p}$}
\label{fig:shape_5}
\end{subfigure}
\begin{subfigure}{0.3\textwidth}
\centering
\includegraphics[width=1.4in]{images/maps/sky_alpha_7_new.pdf}
\caption{$\bm{p}$}
\label{fig:shape_10}
\end{subfigure}\\
\begin{subfigure}{0.3\textwidth}
\centering
\includegraphics[width=1.4in]{images/maps/sky_alpha_2.pdf} 
\caption{$\bm{\alpha}$}
\label{fig:scale_1}
\end{subfigure}
\begin{subfigure}{0.3\textwidth}
\centering
\includegraphics[width=1.4in]{images/maps/sky_alpha_5.pdf} 
\caption{$\bm{\alpha}$}
\label{fig:scale_5}
\end{subfigure}
\begin{subfigure}{0.3\textwidth}
\centering
\includegraphics[width=1.4in]{images/maps/sky_alpha_7.pdf}
\caption{$\bm{\alpha}$}
\label{fig:scale_10}
\end{subfigure}
\caption{Maps of shape parameter $p$ and scale parameter $\alpha$ computed on test image \texttt{skyscraper} for different values of $r$.}
\label{fig:sky_maps}
\end{figure}
\end{comment}

\begin{figure}
	\centering
	\resizebox{\textwidth}{!}{
		\begin{tabular}{cccc}
			&$r=2$&$r=5$&$r=7$\\
			\rotatebox{90}{$\bm{\alpha}$}&\raisebox{-0.5\height}{
				\includegraphics[width=1.4in]{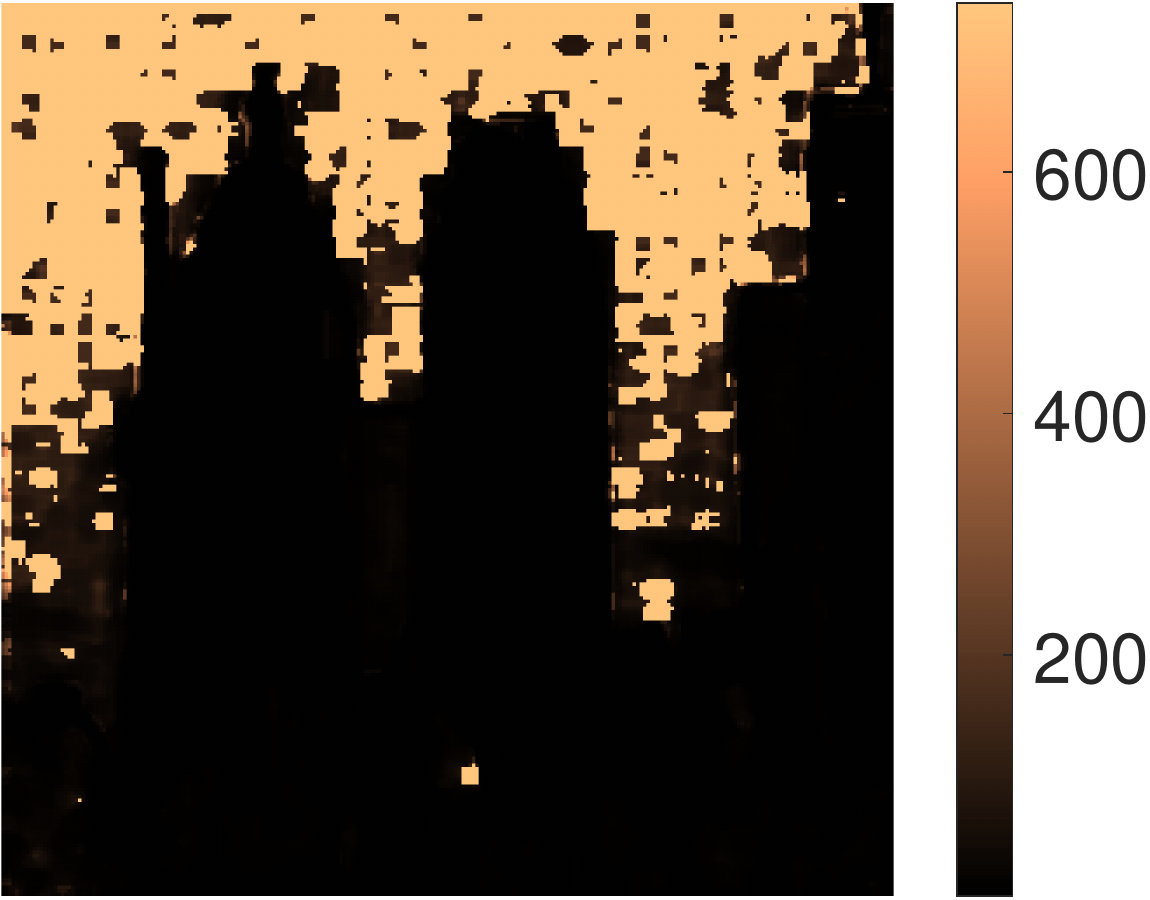}
			} & \raisebox{-0.5\height}{
				\includegraphics[width=1.4in]{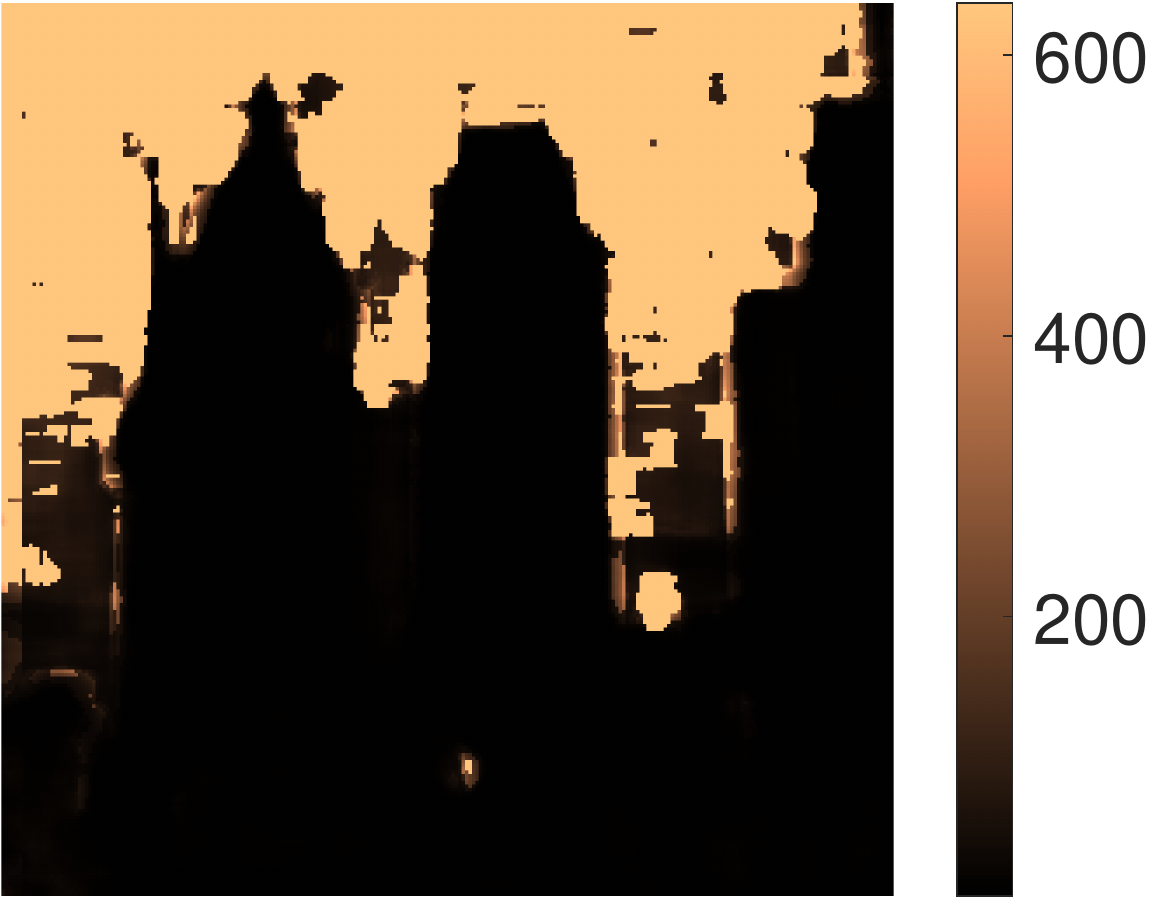}
			}  &\raisebox{-0.5\height}{
				\includegraphics[width=1.4in]{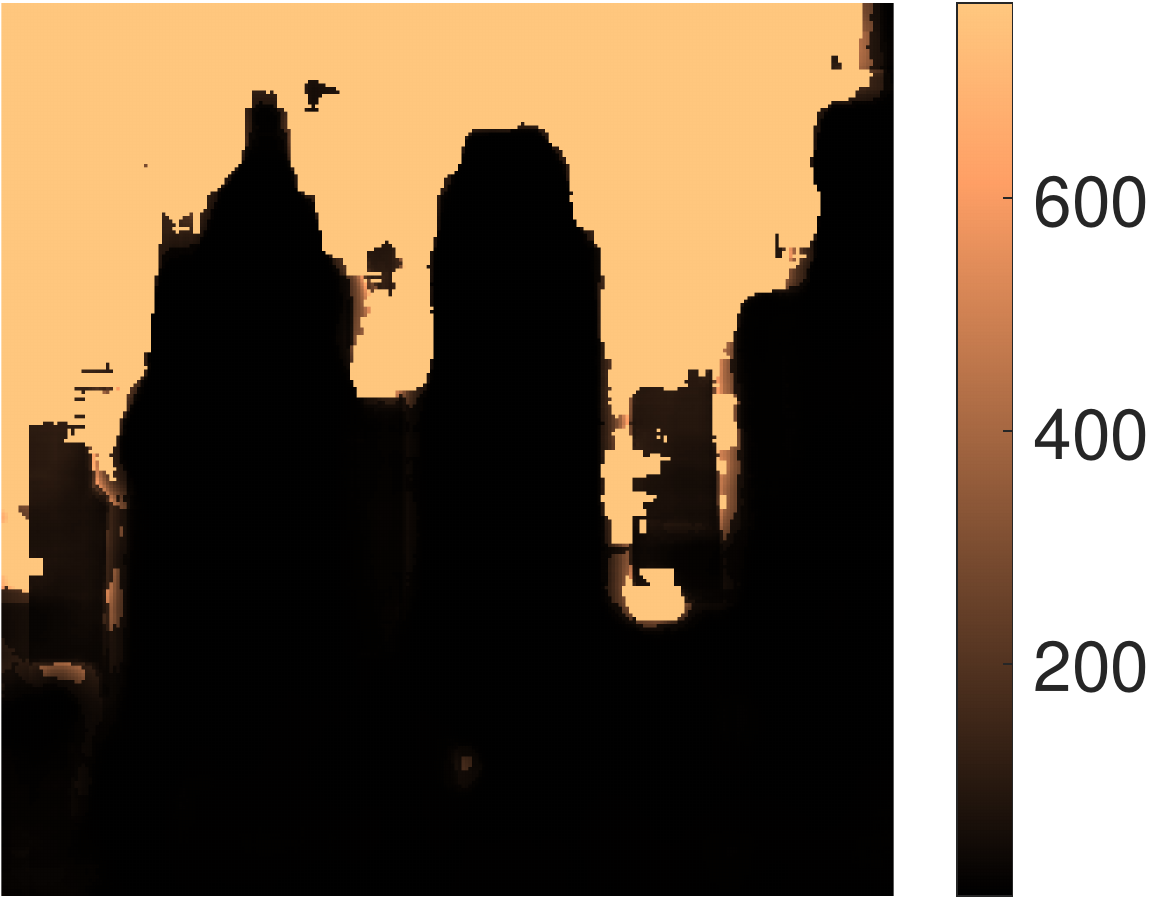} }\\
			{\rotatebox{90}{$\bm{p}$}}&\raisebox{-0.5\height}{\includegraphics[width=1.4in]{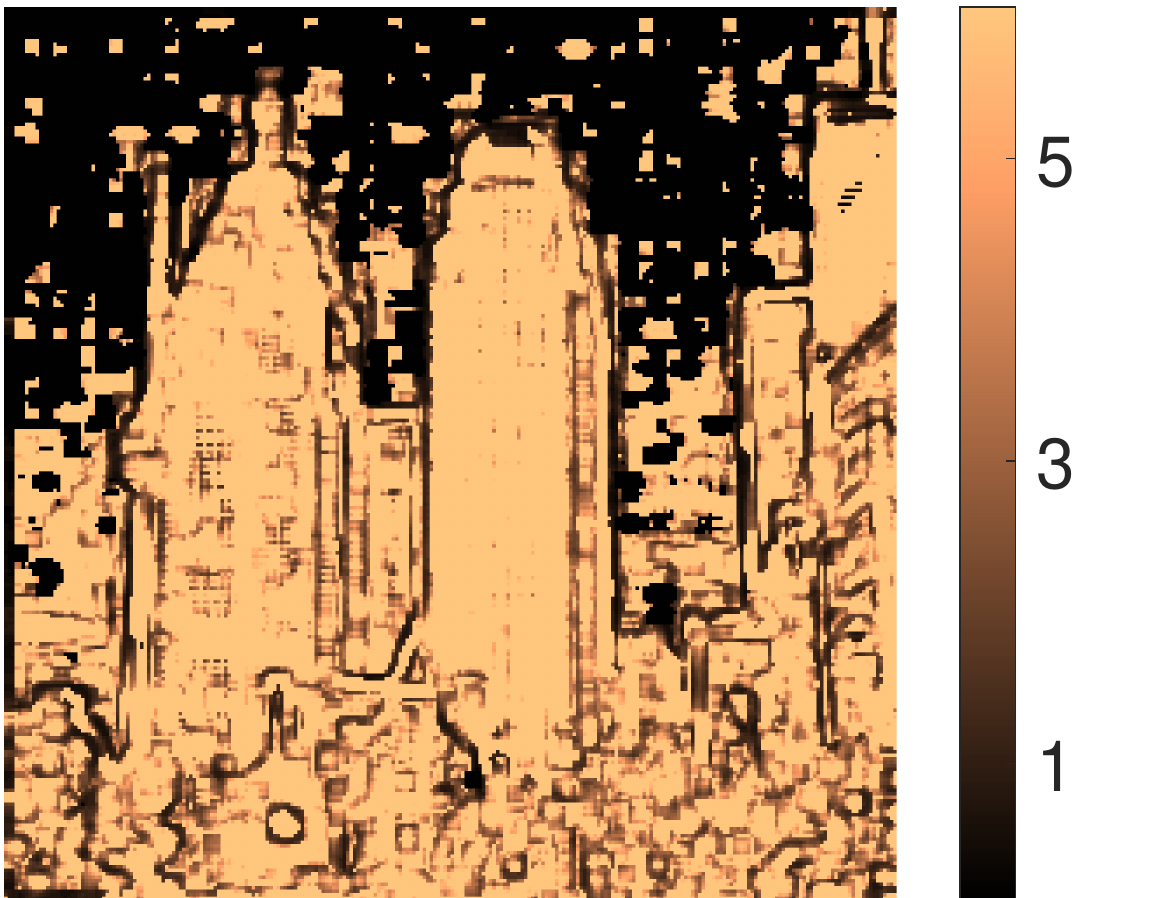} } &\raisebox{-0.5\height}{\includegraphics[width=1.4in]{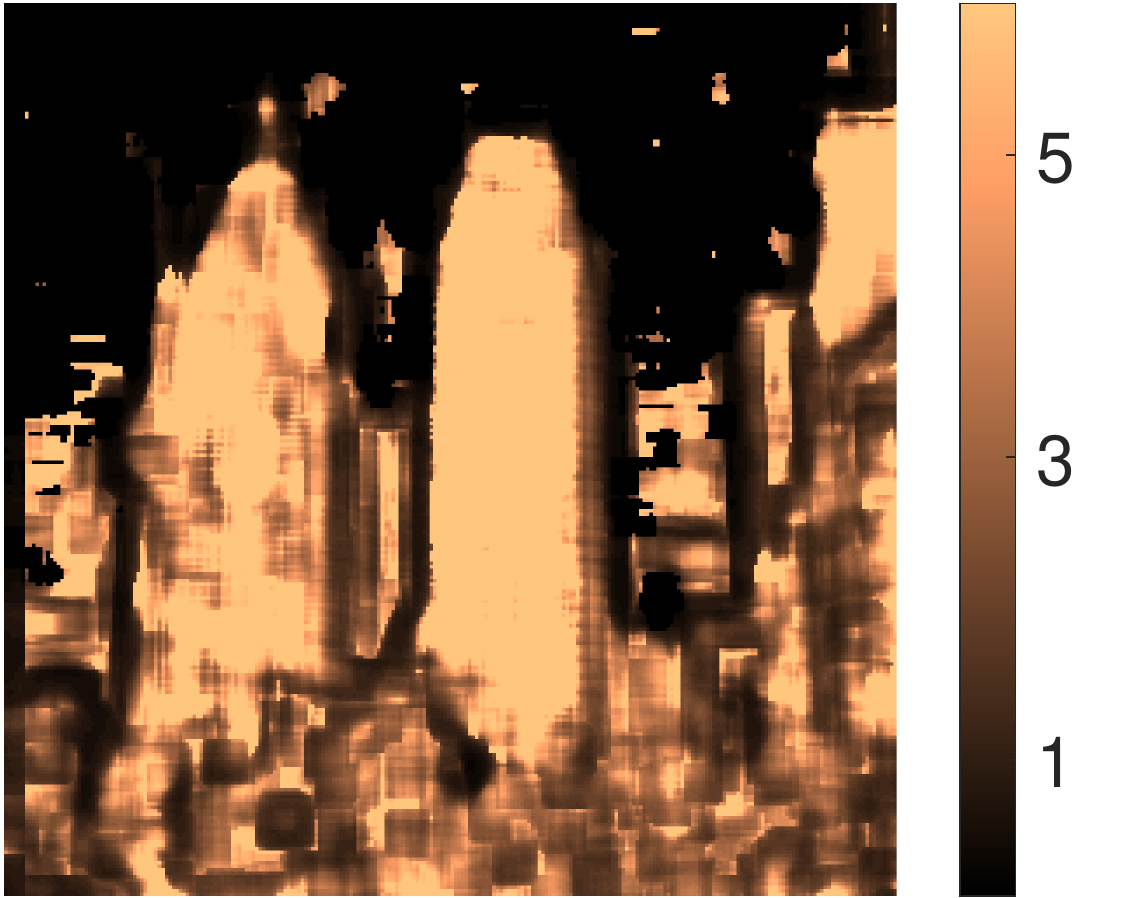}}  &
			\raisebox{-0.5\height}{\includegraphics[width=1.38in]{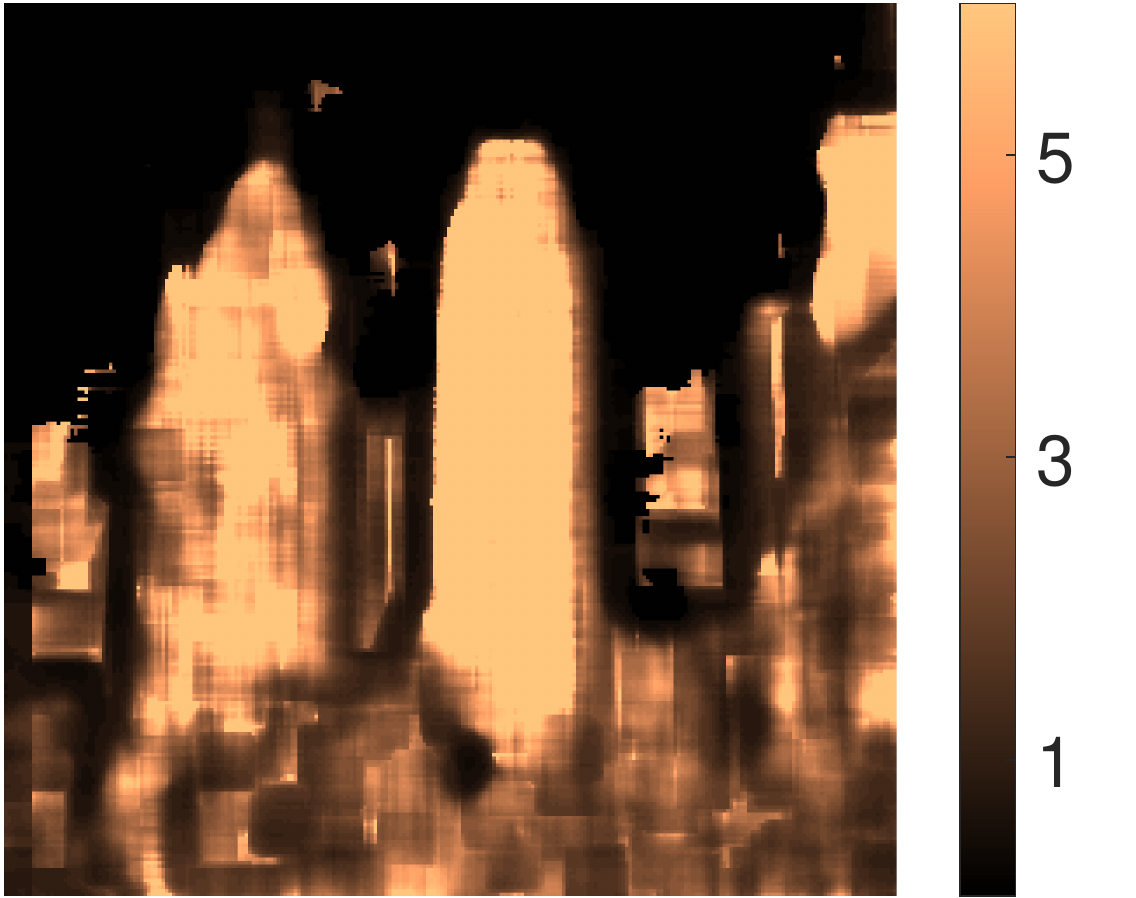}}
	\end{tabular}}
	\caption{The $\bm{\alpha}$ and $\bm{p}$ maps for different values of radius $r$ on the image \texttt{skyscraper} in Figure~\ref{fig:sky_im}.}
	\label{fig:sky_maps}
\end{figure}

\subsection{Parameter estimation for the $\mathrm{WDTV}^{sv}_p$ regulariser}\label{sec:thWDTV}
For the $\mathrm{WDTV}^{sv}_{\bm{p}}$ regularisation term, after selecting the functions $f_{\mathrm{WDTV}^{sv}_{\bm{p}}},h_{\mathrm{WDTV}^{sv}_{\bm{p}}}$ as in \eqref{eq:WDTVp_h},\eqref{eq:WDTVp_h2} and the domain $\mathcal{D}_{\bm{\Theta}_i}$ as specified in Table~\ref{tab:models},  we get that the problem of interest takes the form
%We are now going to detail the parameter estimation problem \eqref{eq:th_sep_ens} when the $\WDTV^{sv}_{\bm{p}}$ prior is adopted, so that the ensamble of samples is set as in \eqref{eq:sampl}, while the functions $f$ and $h$ are defined in \eqref{eq:fg4}.
%As in Secs.~\ref{sec:thWTV},\ref{sec:thWTVp}, the estimate of the parameters $\alpha_i$, $p_i$, $a_i$, $\theta_i$ is performed by exploiting local information. In particular, the ensemble of samples $\mathcal{S}_i$ will now consist of the local image gradients; in formula, $\mathcal{S}_i=\left\{(\bm{\D u})_j\right\}_{j\in\mathcal{J}_i^r}$. In the following the $j$th gradient in $\mathcal{S}_i$ will be denoted by $\bm{t}_{j}$. 

%For each $i$, in light of the mutual independence of the samples, we are interested in solving a problem of the form
\begin{equation}
\left\{\alpha^{*}_i,p^{*}_i,\theta_i^*,a^{*}_i\right\}\in 
\displaystyle{\argmin{(\alpha_i,p_i,\theta_i,a_i)\in \R_{++}^2 \times [-\pi/2,\pi/2)\times (0,1]}} \mathcal{G}(\alpha_i,p_i,\theta_i,a_i)
\label{pb}
\end{equation}
where
\begin{align}\label{eq:obj_DTV}
\begin{split}
\mathcal{G}(\alpha_i,p_i,\theta_i,a_i)  %- \Bigg[ m\log \bigg(\frac{1}{|\bm{\Sigma}|^{1/2}}\frac{1}{\pi \Gamma\big(\frac{2}{p} + 1\big)2^{2/p}\alpha}\bigg) -\;\frac{1}{2\alpha^{p/2}}\sum_{j=1}^{m}(\,\bm{x_{j}}^{T}\,\bm{\Sigma}^{-1}\,\bm{x_{j}}\,)^{p/2}\Bigg] \notag \\
%\notag & = - \Bigg[ -  N\log \bigg(|\Sigma|^{1/2}\pi \Gamma\bigg(\frac{2}{p} + 1\bigg)2^{2/p}m\bigg) -\;\frac{1}{2m^{p/2}}\sum_{j=1}^{N}(\,x_{j}^{T}\,\Sigma^{-1}\,x_{j}\,)^{p/2}\Bigg]\\
& \;{:=}\;  -m\ln a_i + m\ln\Gamma\left(\frac{2}{p_i} \,+\, 1\right)-2m\ln\alpha_i\\
&\phantom{X}\;{+}\;m\left(\frac{2}{p_i}-1\right)\ln 2+ \alpha_i^{p_i} \sum_{j\in\mathcal{J}_i^r}(\,\bm{g}_{j}^{T}\,\bm{\mathrm{R}}_{\theta_i}\bm{\Lambda}_{a_i}^2\bm{\mathrm{R}}_{-\theta_i}\,\bm{g}_{j}\,)^{p_i/2}.
\end{split}
\end{align}
%
% while the last term in \eqref{eq:obj_DTV} is given by 
% %
% \begin{equation}
%     \|\bm{\Lambda_{a_i}}\bm{\mathrm{R}}_{-\theta_i}\bm{g}_j\|_2 = \bm{\mathrm{R}}_{-\theta_i}^T\bm{\Lambda_{a_i}}\bm{\Lambda_{a_i}}\bm{\mathrm{R}}_{-\theta_i}=\bm{\mathrm{R}}_{\theta_i}\bm{\Lambda_{a_i}}^2\bm{\mathrm{R}}_{-\theta_i}\,.
% \end{equation}
%
Note that $\mathcal{G}$ is differentiable on $\mathbb{R}_{++}^2 \times [-\pi/2,\pi/2)\times (0,1]$. By simply imposing a first-order optimality condition on $\alpha_i$, we get the following closed formula:
\begin{equation}
\frac{\partial \mathcal{G} }{\partial \alpha_i}=-2m\frac{1}{\alpha_i}+p_i\alpha_i^{p_i-1}\sum_{j\in\mathcal{J}_i^r}(\bm{g}_{j}^{T}\bm{\mathrm{R}}_{\theta_i}\bm{\Lambda}_{a_i}^2\bm{\mathrm{R}}_{-\theta_i}\bm{g}_{j})^{p_i/2}\end{equation}
which yields
\begin{equation}  \label{eq:expr_m}
\alpha_i^*(p_i,\theta_i,a_i)=\bigg(\frac{p_i}{2m}\sum_{j=1}^{m}(\bm{g}_{j}^{T}\bm{\mathrm{R}}_{\theta_i}\bm{\Lambda}_{a_i}^2\bm{\mathrm{R}}_{-\theta_i}\bm{g}_{j})^{p_i/2}\bigg)^{-\frac{1}{p_i}},
\end{equation}
and which can be regularised depending on $0<\varepsilon\ll 1$ as above.
The stationary point in \eqref{eq:expr_m} can be proved to be a minimum as the second derivative of $\mathcal{G}$ with respect to $\alpha_i$ at $\alpha_i^*$ is strictly positive. Plugging \eqref{eq:expr_m} into \eqref{eq:obj_DTV}, we thus get:
%By introducing $H(p,\sigma_1,\sigma_3;\mathcal{S});\mathcal{A}) {:=}G(p,\sigma_1,\sigma_3,\alpha(p,\sigma_1,\sigma_3;\mathcal{S});\mathcal{A})$
%We now substitute this formula in the expression of $G$, thus getting:
%\begin{eqnarray}
%\mathcal{F}(p,\sigma_{1},\sigma_{2},\sigma_{3})&=&N\log \Gamma \bigg(\frac{2}{p}+1\bigg)-\frac{2N}{p} \log \frac{2N}{p} -\frac{N}{2} \log |\Sigma|+\frac{2N}{p} +\\
%\nonumber
%&+& N\log \pi+\frac{2N}{p}\log\bigg(\sum_{i}(\sigma_{2}x_{i,1}^{2}+\sigma_{1}x_{i,2}^{2}-2\sigma_{3}x_{i,1}x_{i,2})^{p/2}\bigg).
%\label{fun_fin}
%\end{eqnarray}
\begin{align}\label{fun_semifin}
\begin{split}
G(p_i,\theta_i,a_i) \;{:=}\;& \mathcal{G}(\alpha_i^*(p_i,\theta_i,a_i),p_i,\theta_i,a_i)\\
\;{=}\;&m\ln\left[\Gamma\left(\frac{2}{p_i}+1\right)\,\frac{1}{2\,a_i}\right]+\frac{2\,m}{p_i}\left(\ln\frac{p_i}{m}+1\right)\\
&\;{+}\; \frac{2\,m}{p_i} \ln \bigg(\sum_{j=1}^{m}(\bm{g}_{j}^{T}
\bm{\mathrm{R}}_{\theta_i}\bm{\Lambda}_{a_i}^2\bm{\mathrm{R}}_{-\theta_i}
\bm{g}_{j})^{p_i/2}\bigg)\,.
\end{split}
\end{align}
By making now explicit the dependence of $G$ on the entries of $(\bm{\mathrm{R}}_{\theta_i}\bm{\Lambda}_{a_i}^2\bm{\mathrm{R}}_{-\theta_i})$, we have that \eqref{fun_semifin} turns into:
\begin{align}
\label{eq:fun_fin}
\begin{split}
G(p_i,\theta_i,a_i)&=
m\ln\left[\Gamma\left(\frac{2}{p_i}+1\right)\,\frac{1}{2\,a_i}\right]+\frac{2\,m}{p_i}\left(\ln\frac{p_i}{m}+1\right)\\
&+ \frac{2m}{p_i}\ln
\Bigg(\sum_{j\in\mathcal{J}_i^r}
((\cos^2\theta_i+a_i^2\sin^2\theta_i)g_{j,1}^2+(\sin^2\theta_i+a_i^2\cos^2\theta_i)g_{j,2}^2\\
&+2(1-a_i^2)\cos\theta_i\sin\theta_i g_{j,1}g_{j,2})^{p_i/2}
\Bigg)\,.
\end{split}
\end{align}
%As a conclusion, by making explicit the required constraint on each local parameter, 
Problem \eqref{pb}-\eqref{eq:obj_DTV} thus takes the form:
%\begin{align}
%\left\{p_i^*,\phi_i^*,\varrho_i^*\right\}\in \displaystyle{\arg\min_{\substack{p_i\in(0,\infty),\\
%		\phi_i\in[0,2\pi),\\
%		\varrho_i\in[0,1)}}}~ G(p_i,\phi_i,\varrho_i;\bm{x}).   
\begin{equation}
\label{eq:min_prob_unbb}
\left\{p_{i}^{*},\theta^{*}_{i},a_{i}^{*}\right\}\in
\displaystyle{\argmin{(p_i,\theta_i,a_i)\in \R_{++} \times [-\pi/2,\pi/2)\times (0,1]}} 
G(p_i,\theta_i,a_i)
\end{equation}
%\end{align}

We now study the behaviour of $G$ as the triplet $(p_i,\theta_i,a_i)$ approaches the boundary of the set $\widehat{\mathcal{D}}_{\bm{{\Theta}}_i}:=\R_{++}\times [-\pi/2,\pi/2)\times (0,1]$.
%, namely the subset of $\R^3$ to which $\mathcal{D}_{\bm{\Theta}_i}$ reduces after that $\alpha_i$ has been written in terms of $p_i,\theta_i,a_i$.
%
%
%\textcolor{red}{In fact, once that the closed form for the update of $\alpha_i$ has been derived, the constraint set in \eqref{ct} can be regarded as a subset in $\mathbb{R}^3$ defined by the variables $p_i, a_i$ and $\theta_i$ only.}
%
%
Note that, since problem \eqref{eq:min_prob_unbb} is formulated over a non-compact set of $\R^3$, the existence of a solution is in general not guaranteed. One possible way to overcome the problem of non-compactness consists in characterising explicitly the configurations of the samples $\mathcal{S}_i$ for which the functional $G$ in \eqref{eq:fun_fin} does not attain its minimum in $\widehat{\mathcal{D}}_{\bm{{\Theta}}_i}$. To do so, let us first set:
\begin{align}
A(\theta_i,a_i):=&\frac{2m}{p_i} \log \Bigg[\sum_{j\in\mathcal{J}_i^r}((\cos^2\theta_i+a_i^2\sin^2\theta_i)g_{j,1}^2+(\sin^2\theta_i+a_i^2\cos^2\theta_i)g_{j,2}^2\\
&+2(1-a_i^2)\cos\theta_i\sin\theta_i g_{j,1}g_{j,2})^{p_i/2}\Bigg].
\end{align}
For any $p_i>0$, if $A(\theta_i,a_i)$ is bounded as $a_i\to 0^+$, then the functional $G$ in \eqref{eq:fun_fin} tends to $+\infty$ and the minimum is necessarily attained in the interior of $\widehat{\mathcal{D}}_{\bm{{\Theta}}_i}$.
However, if $A(\theta_i,a_i)$ is unbounded as $a_i\to 0^+$, nothing can be said about the behaviour of $G$ at the boundary and, as a consequence, nothing can be said about its minima. In particular, in this situation there may exist one or multiple configurations of the samples $\bm{g}_1,\ldots,\bm{g}_m\in\mathcal{S}_i$ for which $G$ tends to $-\infty$ at the boundary.
In order to characterise such configurations, note that as $a_i\to 0^+$ we have that by continuity:
\begin{equation}
A(\theta_i,a_i)\to \frac{2m}{p_i} \log \Bigg[\sum_{j=1}^{m} (\cos\theta_ix_{j,1}+\sin\theta_ix_{j,2})^{p_i}\Bigg],
\end{equation}
which tends to $-\infty$ if and only if 
\begin{equation}  \label{eq:configurations}
g_{j,2}=-\frac{\cos\theta_i}{\sin\theta_i}~g_{j,1},\qquad \forall j=1,...,m.
\end{equation}
This situation corresponds to the very particular case when the samples $\bm{g}_{j}$ lie all on the line passing through the origin with slope $-\cos\theta_i/\sin\theta_i$, and they can be thus considered as realisations of a degenerate BGG pdf characterised by a positive semidefinite covariance matrix. This sort of configurations can be avoided by requiring that $a_i$ does not get smaller than a fixed value $0<\delta\ll 1$.

%\textcolor{red}{Although these configurations are easily detectable in a preliminary check phase, from a statistical perspective this sort of samples can be considered as realisations of a degenerate BGG pdf.  }

A possible way to guarantee the existence of solutions of the problem \eqref{eq:min_prob_unbb} is to re-formulate the problem over a compact subset of $\R^3$, in analogy with what has been done in Section \ref{sec:thWTVp}. %Although this may sound a little bit artificial, note that for imaging applications such assumption makes perfect sense for different reasons. 
%{\color{red}  $\varrho?$ tagliare!!!!
%Firstly, as far as the range for the parameter $\varrho_i$ is concerned, the degenerate configurations \eqref{eq:configurations} happening as $a_i$ approaches $0^+$ are easily detectable in a pre-processing step and, in practice, very unlikely for natural images since they would correspond to situations where gradient components are linearly correlated for any sample $j=1,\ldots,m$. Therefore, provided we can perform such preliminary check, from the optimisation point of view, the case $a_i=0$ becomes admissible since no other possible configurations are allowed under this choice. However, we remark that from a modelling perspective, the case $a_i=0$ implies that the elliptical level curves of the BGG pdf reduce to a segment as the local covariance matrix is singular. This can be avoided by requiring that $a_i$ can not get smaller than a fixed value $\delta>0$.}
As noted above on the admissible values for $p_i$, we point out that the more we enforce sparsity (i.e. the closer $p_i$ is to zero), the more the BGGD will tend to a Dirac delta distribution, making the estimation of local anisotropy in a neighbourhood of the point considered almost impossible. Hence, the exponent $p_i$ is thought as confined in the closed interval $[\epsilon,R]$, with $0<\epsilon<R$.

%In the following numerical experiments we will chose $L_p$ so that $L_p>0.1$ and $U_p\geq 2$.
We can thus reformulate problem \eqref{eq:min_prob_unbb} as 
\begin{align}\label{minpbpar}
\left\{p_i^{*},\theta_i^{*},a_i^{*}\right\}&\in\min_{p_i,\theta_i,a_i} G(p_i,\theta_i,a_i)\\
\notag&\\
\notag &\textrm{s.t.} \quad p_i\in[\epsilon,R],\;\;-\pi/2 \leq \theta_i \leq \pi/2\,,\;\;\delta \leq a_i \leq 1.
\end{align}
The following result holds true:
%where now the constraint set is compact, which, combined with the continuity of $G$, guarantees that the minimisation problem admits a minimum.
\begin{proposition}\label{prop:ex_scale}
	The function $G:[\epsilon,R]\times[-\pi/2,\pi/2]\times[\delta,1]\to\R$ in \eqref{eq:fun_fin} is continuous and admits a minimum in its compact domain.
\end{proposition}

In Figure \ref{fig:est_DTVp}, we analyse the performance of the outlined parameter estimation strategy for the WDTV$_{\bm{p}}^{sv}$ regulariser, where again the search interval for the local parameter $p_i$ has been set as $[\epsilon,R]=[0.1,10]$. More specifically, in the left column we display selected neighbourhoods from a synthetic image - i.e. a vertical edge in Figure \ref{fig:sub1}, an horizontal edge in Figure \ref{fig:sub2}, and a circular profile in Figure \ref{fig:sub3} - and the two textured regions already considered in Figures \ref{fig:sky_wtv} and \ref{fig:sky_wtvp}. In the middle column of Figure \ref{fig:est_DTVp}, we report the samples extracted from each neighbourhood, together with the level curves of the estimated local BGG pdfs, while in the last column we show the scatter plot of the samples, by drawing once again the level curves of the underlying distribution to facilitate the analysis. 

The estimated pdfs for the three geometrical profiles lie along the horizontal axis, the vertical axis and the first quadrant bisector of the scatter diagram $\bm{\D}_h \bm{u}$-$\bm{\D}_v \bm{u}$, respectively. This behaviour, as expected, corresponds to  the dominant orientation of the gradients in the neighbourhoods.  Finally, the textured regions in Figures \ref{fig:sub3},\ref{fig:sub4}, statistically indistinguishable from the hLd and hGGd viewpoint, result to be significantly different now; in fact, the samples in the former are spread more homogeneously in the scatter diagram, while the gradients in the latter present a dominant edge orientation which is almost aligned with the horizontal axis of the diagram. Such difference is now reflected into the estimated BGG pdfs.

\begin{figure}
	\centering
	\begin{subfigure}[t]{0.32\textwidth}
		\centering
		\includegraphics[scale=1.1]{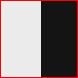}
		\caption{}
		\label{fig:sub1}
	\end{subfigure} 
	\begin{subfigure}[t]{0.32\textwidth}
		\centering
		\includegraphics[scale=0.22]{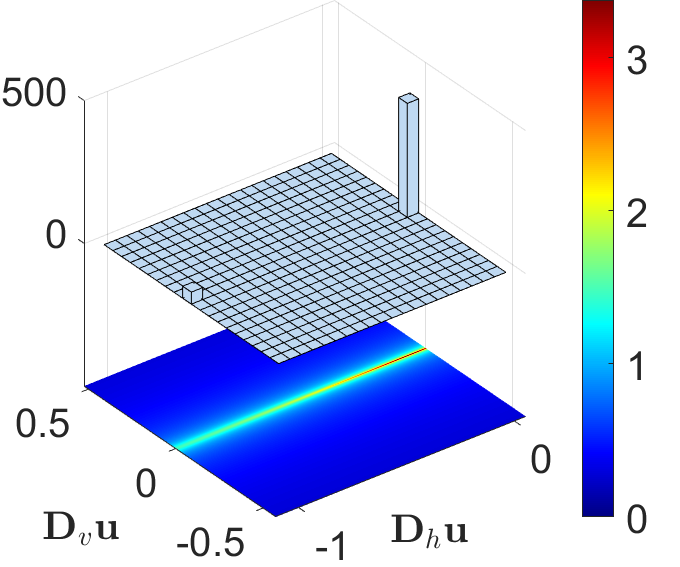}
		\caption{}
	\end{subfigure} 
	\begin{subfigure}[t]{0.32\textwidth}
		\centering
		\includegraphics[scale=0.22]{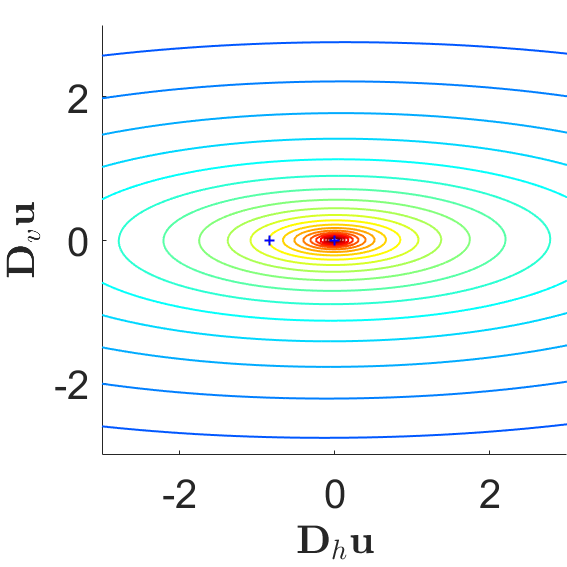}
		\caption{}
	\end{subfigure} \\
	\begin{subfigure}[t]{0.32\textwidth}
		\centering
		\includegraphics[scale=1.1]{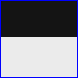}
		\caption{}
		\label{fig:sub2}
	\end{subfigure} 
	\begin{subfigure}[t]{0.32\textwidth}
		\centering
		\includegraphics[scale=0.22]{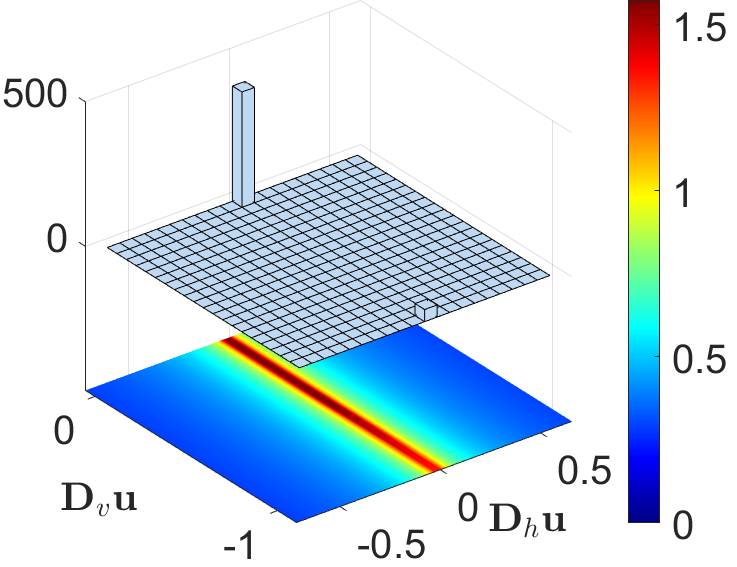}
		\caption{}
	\end{subfigure}
	\begin{subfigure}[t]{0.32\textwidth}
		\centering
		\includegraphics[scale=0.22]{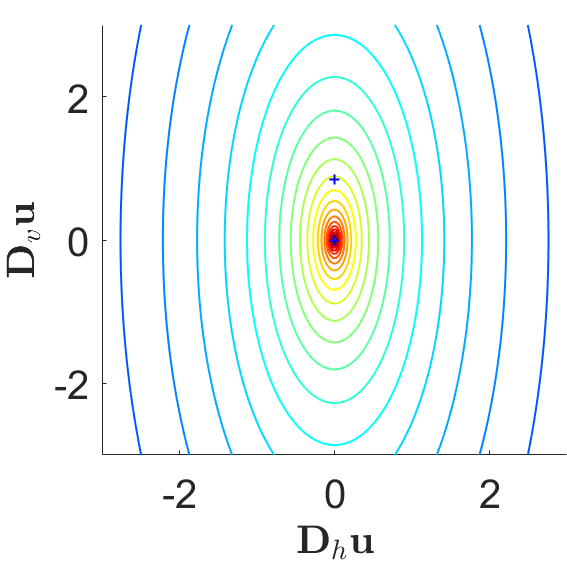}
		\caption{}
	\end{subfigure}\\
	\begin{subfigure}[t]{0.32\textwidth}
		\centering
		\includegraphics[scale=1.1]{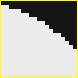}
		\caption{}
		\label{fig:sub3}
	\end{subfigure} 
	\begin{subfigure}[t]{0.32\textwidth}
		\centering
		\includegraphics[scale=0.22]{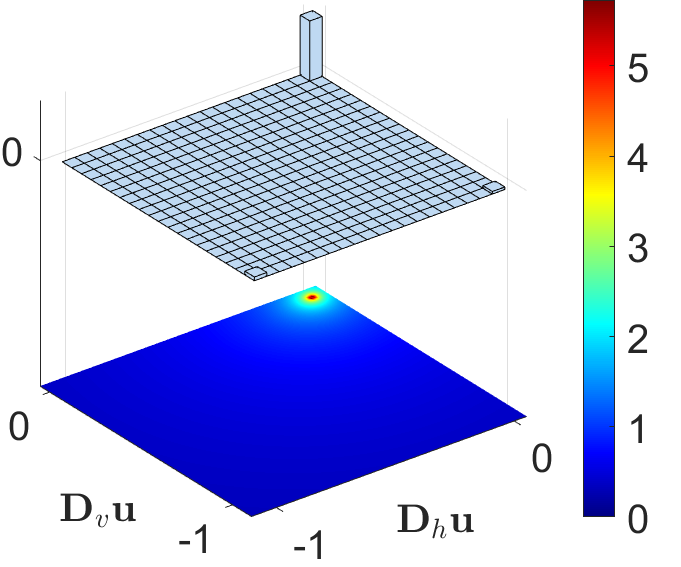} 
		\caption{}
	\end{subfigure} 
	\begin{subfigure}[t]{0.32\textwidth}
		\centering
		\includegraphics[scale=0.22]{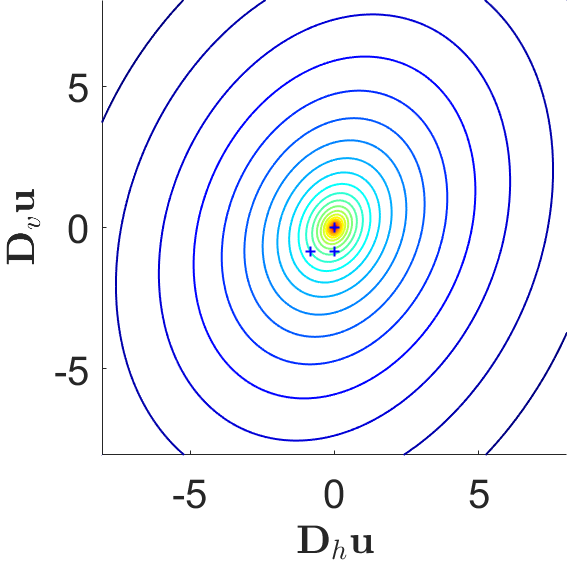}
		\caption{}
	\end{subfigure} \\
	\begin{subfigure}[t]{0.32\textwidth}
		\centering
		\includegraphics[scale=1.1]{images/hist/sky_zoom2.png}
		\caption{}
		\label{fig:sub4}
	\end{subfigure} 
	\begin{subfigure}[t]{0.32\textwidth}
		\centering
		\includegraphics[scale=0.22]{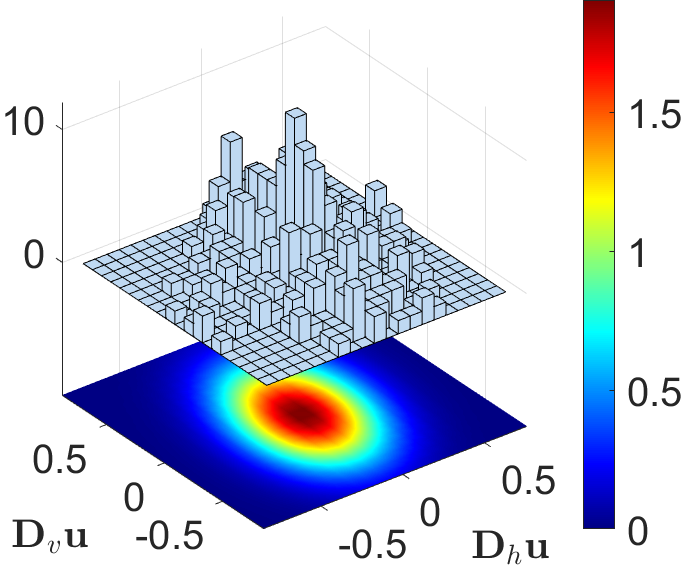}
		\caption{}
	\end{subfigure}
	\begin{subfigure}[t]{0.32\textwidth}
		\centering
		\includegraphics[scale=0.22]{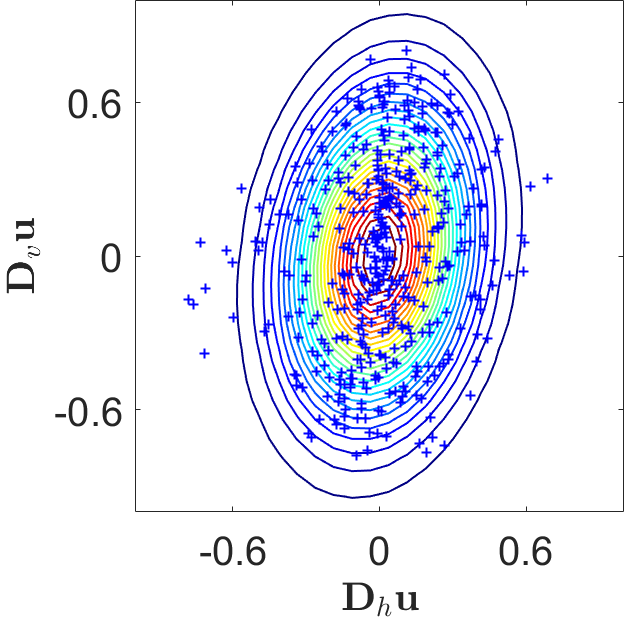}
		\caption{}
	\end{subfigure}\\
	\begin{subfigure}[t]{0.32\textwidth}
		\centering
		\includegraphics[scale=1.1]{images/hist/sky_zoom3.png}
		\caption{}
		\label{fig:sub5}
	\end{subfigure} 
	\begin{subfigure}[t]{0.32\textwidth}
		\centering
		\includegraphics[scale=0.22]{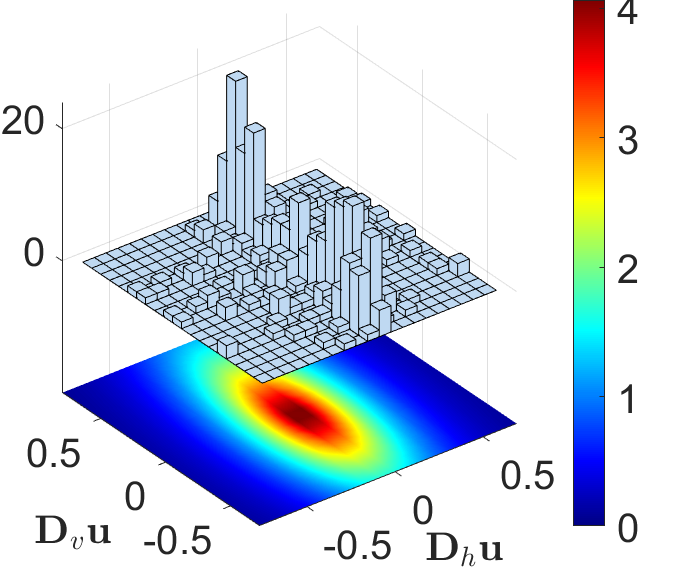}
		\caption{}
	\end{subfigure}
	\begin{subfigure}[t]{0.32\textwidth}
		\centering
		\includegraphics[scale=0.22]{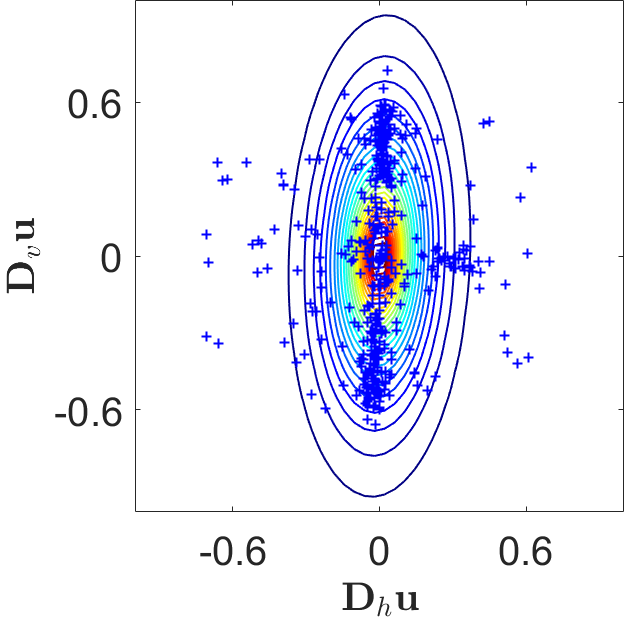}
		\caption{}
	\end{subfigure}
	\caption{\emph{Parameters estimate for the WDTV$_{\bm{p}}^{sv}$ regulariser.} From left to right: selected neighbourhoods, histogram and scatter plot of gradient samples with level curves of the estimated BGGd. From top to bottom, the estimated parameters are: $(\alpha,p,\theta,a)=(2.23,0.1,-89.90,0.33)$, $(\alpha,p,\theta,a)=(2.10,0.1,0,0.37)$, $(\alpha,p,\theta,a)=(1.14,0.1,47.50,0.78)$, $(\alpha,p,\theta,a)=(1.30,2.97,-12.50,0.64)$, $(\alpha,p,\theta,a)=(1.94,1.24,-2.18,0.39)$.}% 
	\label{fig:est_DTVp}
\end{figure}

%$(\alpha=2.10,\,p=0.1,\,\theta=0,\,a = 0.37)$, $(\alpha=1.14,\,p=0.1,\,\theta=47.50,\,a = 0.78)$, $(\alpha=1.30,\,p=2.97,\,\theta=-12.50,\,a = 0.64\,)$, $(\alpha=1.94,\,p=1.24,\,\theta=-2.18,\,a =  0.39)$.}%(\alpha=2.23,\,p=0.1,\,\theta=-89.90,\,a = 0.33)$}   
\section{Algorithmic optimisation}
\label{sec:admm}

From an optimisation point of view, it is not trivial to design a unified optimisation solver for the general $\bm{u}$-estimation problem in the alternating scheme \eqref{eq:sub_th}-\eqref{eq:sub_u}, as it may be either extremely easy (smooth and convex) or extremely difficult (non-convex and non-smooth). We can surely think of specific optimisation algorithms that could be effectively used for solving \eqref{eq:sub_th}-\eqref{eq:sub_u} in specific scenarios such as (l-)BFGS \cite{BFGS} for the smooth and convex case, Nesterov-type proximal schemes \cite{Nesterov-2004,Beck-Teboulle-2009b,Condat13,RFP13} and dual/primal-dual methods \cite{KK04,CGM99,HS06,Chamb04,ChambollePock2011} for the non-smooth convex case and, e.g., \cite{Ochs-etal-2014} for the non-smooth non-convex case). However, in the following we will stick with one single optimisation algorithm for better clarity and consider the Alternating Directional Method of Multipliers (ADMM) \cite{BOYD_ADMM} whose different subproblems can be solved by means of classical tools in the field of proximal calculus \cite{CombPes11}, numerical linear algebra and adaptive discrepancy principle \cite{APE}. Note that albeit proposed and widely applied in convex scenarios, non-convex variants of ADMM have been recently proposed and endowed with global convergence guarantees \cite{Wang2018,Bolte2018}, although not always applicable to the problem at hand due to the (often limiting) assumptions on the operators involved. However, as we will comment in the following, empirical convergence is often observed for general non-convex ADMM algorithms, which makes their use often amenable in practice. For further details on the recent developments of convex and non-convex optimisation algorithms for variational imaging models, we refer the reader to \cite{chambolle2016introduction} and the references therein.

\subsection{ADMM optimisation}

By dropping out the terms in \eqref{eq:fin_model} which do not depend on the unknown image $\bm{u}$, the $\bm{u}$-update step in the alternating scheme \eqref{eq:sub_th}-\eqref{eq:sub_u} reads
\begin{equation}
\bm{u}^{(k+1)}
\in
\argmin{\mathbf{u}\in\R^N}
\bigg\{ 
\sum_{i=1}^N f((\bm{\D u})_i;\bm{\Theta}_i^{(k+1)})
+
F_q(\bm{\A u};\bm{b})
\bigg\},
\label{eq:ADMM_u} 
\end{equation}
where $F_q(\bm{\A u};\bm{b})$ is defined in \eqref{eq:fid1}-\eqref{eq:fid2} and the gradient penalty functions $f$ are summarised in Table~\ref{tab:models} for the different regularisers considered.

By introducing the auxiliary variables $\bm{g}\in\R^{2N}$ and $\bm{r}\in\R^M$, and dropping out the iteration superscript, problem \eqref{eq:ADMM_u}  can be reformulated as:
\begin{equation}
\{ \bm{u}^*,\bm{g}^*,\bm{r}^* \}
\:\;{\in}\;\:
\argmin{\bm{u},\bm{g},\bm{r}}
\bigg\{
\sum_{i = 1}^{N} f(\bm{g}_i;\bm{\Theta}_i)
+
F_q(\bm{r};\bm{b})
\: \bigg\}
\;\;\:\mathrm{s.t.:}\;\,
\left\{\!
\begin{array}{ll}
\bm{g} &\!\!\!\!{=}\; \bm{\D u}   \\
\bm{r} &\!\!\!\!{=}\; \bm{\A u}
\end{array}
\right.
\label{eq:ADMM_func}  
\end{equation}
For every $i=1,\ldots,N$, the quantity $\bm{g}_i=((\bm{\D}_h\bm{ u})_i, (\bm{\D}_v\bm{u})_i)\in\R^2$ stands for the local image gradient at pixel $i$. By means of this change of variable, we can avoid considering the dependence on the linear operator $\bm{\D}$ of the (in general) non-differentiable and possibly non-convex function $f$, while the use of $\bm{r}$ is helpful for the GDP strategy introduced in Section \ref{sec:gdp}. 

We define the augmented Lagrangian functional of problem \eqref{eq:ADMM_func}  as follows:
\begin{align}\label{eq:ADMM_AL}
\begin{split}
\mathcal{L}(\bm{u},\bm{g},\bm{r},\bm{\rho}_t,\bm{\rho}_r;\bm{\Theta})
& \:{:=}\;
\sum_{i=1}^N f(\bm{g}_i;\bm{\Theta}_i)
+
F_q(\bm{r};\bm{b})
- \langle \bm{\rho}_t , \bm{g} - \bm{\D u} \rangle \\
& \:{+}\;
\frac{\beta_g}{2}  \| \bm{g} - \bm{\D u} \|_2^2-
\langle  \bm{\rho}_r , \bm{r} - \bm{\A u} \rangle
+
\frac{\beta_r}{2}  \| \, \bm{r} - \bm{\A u} \|_2^2 ,
\end{split}
\end{align}
where $\beta_g, \beta_r \in \R_{++}$ are the ADMM penalty parameters, while $\bm{\rho}_t \in \mathbb{R}^{2N}$, $\bm{\rho}_r \in \mathbb{R}^M$ are the vectors of Lagrange multipliers associated with the linear constraints $\bm{g}=\bm{\D u}$ and $\bm{r} = \bm{\A u}$
in \eqref{eq:ADMM_func}, respectively.

Solving \eqref{eq:ADMM_func} amounts to seek for solutions of the  saddle point problem:
\begin{eqnarray}
&&\quad
\mathrm{Find}\;\;
(\bm{u}^*, \bm{g}^*, \bm{r}^*) \in \R^N {\times}\: \R^{2N} {\times}\: \R^{M}\;\:\mathrm{and}\;\; 
(\bm{\rho}_t^*, \bm{\rho}_r^*) \in \R^{2N} {\times}\: \R^M 
\:\text{such that:}
\nonumber\\
&&\quad
\mathcal{L}(\bm{u}^*,\bm{g}^*,\bm{r}^*,\bm{\rho}_t,\bm{\rho}_r; \bm{\Theta})
\leq 
\mathcal{L}(\bm{u}^*,\bm{g}^*,\bm{r}^*,\bm{\rho}_t^*, \bm{\rho}_r^*;\bm{\Theta})
\leq 
\mathcal{L}(\bm{u},\bm{g},\bm{r},\bm{\rho}_t^*,\bm{\rho}_r^*, \bm{\Theta})
\label{eq:ADMM_saddle}\\
&&\quad
\:\forall \: (\bm{u},\bm{g},\bm{r})\in \R^N {\times}\: \R^{2N} {\times}\: \R^{M}, \;\, \forall \: (\bm{\rho}_g, \bm{\rho}_r)\in \R^{2N} {\times}\: \R^M.
\nonumber
\end{eqnarray}
%
%with the augmented Lagrangian functional $\mathcal{L}$ defined in \eqref{eq:ADMM_AL}.

\smallskip
Upon suitable initialisation, and for any $j\geq 0$, the $j$-th iteration of the ADMM algorithm applied to solve the saddle-point problem \eqref{eq:ADMM_saddle} thus reads:
\begin{eqnarray}
\bm{u}^{(j+1)} && \;{\in}\;\,
\argmin{\bm{u} \in \mathbb{R}^N}\; 
\mathcal{L}\big(\bm{u},\bm{g}^{(j)},\bm{r}^{(j)},
\bm{\rho}_g^{(j)},\bm{\rho}_r^{(j)};\bm{\Theta}\big),
\label{eq:ADMM_sub_u} \\
\bm{g}^{(j+1)} && \;{\in}\;\,
\argmin{\bm{g} \in \mathbb{R}^{2N}}\;
\mathcal{L}\big(\bm{u}^{(j+1)},\bm{g},\bm{r}^{(j)},\bm{\rho}_g^{(j)},\bm{\rho}_r^{(j)};\bm{\Theta}\big),
\label{eq:ADMM_sub_t} \\
\bm{r}^{(j+1)} && \;{\in}\;\,
\argmin{\bm{r} \in \mathbb{R}^M}\;
\mathcal{L}\big(\bm{u}^{(j+1)},\bm{g}^{(j+1)},\bm{r},\bm{\rho}_g^{(j)},\bm{\rho}_r^{(j)};\bm{\Theta}\big),
\label{eq:ADMM_sub_r} \\
\bm{\rho}_g^{(j+1)} && \;{=}\;\,
\bm{\rho}_g^{(j)} - \beta_g \big(  \bm{g}^{(j+1)} - \bm{\D u}^{(j+1)} \big),
\label{eq:ADMM_sub_rho_t} \\
\bm{\rho}_r^{(j+1)} && \;{=}\;\, 
\bm{\rho}_r^{(j)} - \beta_r  \big(  \bm{r}^{(j+1)} - \bm{\A}\bm{u}^{(j+1)} \big).
\label{eq:ADMM_sub_rho_r}
\end{eqnarray}

%\subsection{Sub-problem for primal variable $\bm{u}$}
%\label{subsec:u}

%\subsection{Sub-problem for primal variable $\bm{t}$}
%\label{subsec:t}

%\subsubsection{Sub-problem for primal variable $\bm{r}$}
%\label{subsubsec:r}
In the following, we make precise the solution of the three sub-problems for the primal variables $\bm{u}$, $\bm{g}$, and $\bm{r}$ in (\ref{eq:ADMM_sub_u})-(\ref{eq:ADMM_sub_r}). The automatic estimation of the regularisation parameter $\mu$ will be addressed in Section \ref{sec:r_admm} concerned with the $\bm{r}$-update.
%as it is embedded in the functional $F_q(\bm{r};\bm{b})$ when $q=1,2$. 

\subsection{Subproblem for the primal variable \emph{u}} \label{sec:u_admm}
%After recalling definition (\ref{eq:ADMM_AL}) of the augmented Lagrangian functional $\mathcal{L}$ and dropping the constant terms (i.e., terms not depending on $\bm{u}$), 
Subproblem (\ref{eq:ADMM_sub_u}) reads
\begin{eqnarray}
\bm{u}^{(j+1)} 
&\;{\in}\;&
\argmin{\bm{u} \in \R^N}
\left\{
\langle \bm{\rho}_g^{(j)} , \bm{\D u} \rangle
\;{+}\;
\langle \bm{\rho}_r^{(j)} , \bm{\A u} \rangle
\;{+}\;
\frac{\beta_g}{2}  \| \bm{g}^{(j)} - \bm{\D u} \|_2^2
\;{+}\;
\frac{\beta_r}{2}  \| \bm{r}^{(j)} - \bm{\A u} \|_2^2
\right\},
\nonumber
\end{eqnarray}
which is quadratic with first-order optimality condition given by
\begin{equation}
\left(
\beta_g \bm{\D}^T\bm{\D} + \beta_r \bm{\A}^T\bm{\A}
\right) \bm{u} 
=
\beta_g\bm{\D}^T\left( \bm{g}^{(j)} - \frac{1}{\beta_g} \bm{\rho}_t^{(j)}\right)
+
\beta_r\bm{\A}^T\left( \bm{r}^{(j)} - \frac{1}{\beta_r} \bm{\rho}_r^{(j)}\right) \, .
\label{eq:subu_ls}
\end{equation}
The coefficient matrix of the linear system above is symmetric positive semidefinite and, under the assumption 
%that the null spaces of matrices $\bm{\A}$ and $\bm{\D}$ have trivial intersection,
%
\begin{equation}
\mathrm{null} (\bm{\A}) \cap \mathrm{null} (\bm{\D}) \;{=}\; \{\bm{0}_N\},
\end{equation}
then it is positive definite so that $\bm{u}^{(j+1)}$ is the unique solution of linear system (\ref{eq:subu_ls}). Matrix $\bm{\A}$ is typically sparse, hence (\ref{eq:subu_ls}) can be solved efficiently by means of (preconditioned) Conjugate Gradient methods. When $\bm{\A}$ is a convolution matrix - like in image restoration with space-invariant blur - the linear system can be solved more efficiently by means of fast 2D discrete transforms.
\bigskip

%{\color{red} vogliamo davvero delle sottosezioni per ogni sottoproblema? Potremmo mettere semplicemente dei paragrafi per quelli 'semplici' e invece una sottosezione per quello con il prox?}

\subsection{Subproblem for the primal variable \emph{g}}\label{sec:t_admm}
After dropping all terms not depending on $\bm{g}$ in (\ref{eq:ADMM_AL}), subproblem (\ref{eq:ADMM_sub_t}) reads
\begin{eqnarray}
\bm{g}^{(j+1)}
&\;{\in}\;&
\argmin{\bm{g} \in \R^{2N}}
\left\{
\sum_{i=1}^{N} f(\bm{g}_i;\bm{\Theta}_i)
\:{-}\;
\langle \bm{\rho}_g^{(j)} , \bm{g} - \bm{\D u}^{(j+1)} \rangle
+
\frac{\beta_g}{2} \big\| \, \bm{g} - \bm{\D u}^{(j+1)} \big\|_2^2
\right\}
\nonumber\\
&\;{=}\; &
\argmin{\bm{g} \in \R^{2N}}
\left\{
\sum_{i=1}^{N} f(\bm{g}_i;\bm{\Theta}_i)
\;{+}\;
\frac{\beta_g}{2}
\|\bm{g}-\bm{w}^{(j)}\|_2^2
\right\}, 
\end{eqnarray}
with vector $\bm{w}^{(j)} \in \R^{2N}$ defined by
\begin{equation}
\bm{w}^{(j)}
\;{:=}\;
\bm{\D} \bm{u}^{(j+1)} + \frac{1}{\beta_g} \bm{\rho}_t^{(j)}
\: .
\end{equation}
%
%with $\left( D u^{(k+1)} \right)_i, \left( \lambda_t^{(k)} \right)_i \in \mathbb{R}^2$ denoting the discrete gradient 
%of image $u^{(k+1)}$ and the Lagrange multipliers at pixel $i$, respectively.
Solving the $2N$-dimensional minimisation problem above is thus equivalent to solve the following $N$ independent $2$-dimensional problems:
\begin{align}
\bm{g}^{(j+1)}_i 
\,\;{\in}\;\:&
\argmin{\bm{g}_i \in \R^2}
\left\{ \,
f(\bm{g}_i;\bm{\Theta}_i)
\;{+}\;
\frac{\beta_g}{2} 
\left\| \bm{g}_i - \bm{w}_i^{(j+1)} \right\|_2^2
\,\right\}
\nonumber\\
%\,\;{=}\;\:&\argmin{\bm{g}_i \in \R^2}
%\left\{ \,
% v(\bm{g}_i;\bm{\Theta}_i(\bm{u}))
%\;{+}\;
%\frac{\widetilde{\beta}_g}{2} 
%\left\| \bm{g}_i - \bm{w}_i^{(k+1)} %\right\|_2^2
%\,\right\},
%\quad i = 1, \ldots , N \: , \\
%\,\;{=}\;\:&
%\mathrm{prox}_{
%v(\,\cdot\,;\bm{\Theta}_i(\bm{u}))}^{\wideti%lde{\beta}_g}\left(\bm{w}_i^{(k+1)}\right)\,%,
%\quad i = 1, \ldots , N \: ,
\,\;{=}\;\:&
\mathrm{prox}_{
	f(\,\cdot\,;\bm{\Theta}_i)}^{\beta_g}\left(\bm{w}_i^{(j+1)}\right)\,,
\quad i = 1, \ldots , N \: ,\label{eq:sub_t_i}
\end{align}
%
%where the newly introduced function $f$ are defined as
%\begin{equation}
%   f(\bm{t}_i;\bm{\Theta}_i(\bm{u})) =  \frac{1}{\alpha_i}h(\bm{t}_i;\bm{\Theta}_i(\bm{u}))\,,
%\end{equation}
%\begin{eqnarray}
%\bm{t}^{(k+1)}_i 
%&\,\;{\in}\;\:&
%\argmin{\bm{t}_i \in \R^2}
%\left\{ \,h(\bm{t}_i;\bm{\Theta}_i(\bm{u}))\;{+}\;
%\frac{\beta_t}{2} \left\| \bm{t}_i - \bm{w}_i^{(k+1)} \right\|_2^2\,\right\}
%\nonumber\\
%&\,\;{=}\;\:&
%\mathrm{prox}_{
%h(\,\cdot\,;\bm{\Theta}_i(\bm{u}))}^{\beta_t}\left(\bm{w}_i^{(k+1)}\right)
%,
%\quad i = 1, \ldots , N \: , \label{eq:sub_t_i}
%\end{eqnarray}
%
where $\,\mathrm{prox}_{
	f(\,\cdot\,;\bm{\Theta}_i)}^{\beta_g}: \R^2 \rightrightarrows \R^2$ denotes the proximal operator of the gradient penalty function $f(\,\cdot\,;\bm{\Theta}_i)$ with proximity parameter $\beta_g$ - see Definition \ref{def:proxx} and Section \ref{sec:prox} - 
%defined by
%
%\begin{align}
%\label{eq:prox1}
% v(\bm{g}_i;\bm{\Theta}_i(\bm{u})) \;{:=}\;&\frac{1}{\alpha_i}\,f(\bm{g}_i;\bm{\Theta}_i(\bm{u}))\,,\;\;\widetilde{\beta}_g\;{:=}\;\frac{\beta_g}{\alpha_i}\;\;\text{ for}\;\WTV\,,\\
%\label{eq:prox2}
% v(\bm{g}_i;\bm{\Theta}_i(\bm{u})) \;{:=}\;&\frac{1}{\alpha_i^{p_i}}f(\bm{g}_i;\bm{\Theta}_i(\bm{u}))\,,\;\widetilde{\beta}_g\;{:=}\;\frac{\beta_g}{\alpha_i^{p_i}}\;\;\text{for}\;\WTV_{\bm{p}}^{sv},\,\WDTV_{\bm{p}}^{sv}\,,
%\end{align}
and where the vectors $\bm{w}^{(j+1)}_i \in \R^2$ at any iteration read
\begin{equation}
\bm{w}^{(j+1)}_i \:\;{=}\;\: \left( \bm{\D} \bm{u}^{(j+1)} \right)_i + \frac{1}{\beta_g} \left( \bm{\rho}^{(j)}_t \right)_i \;\: ,
\quad i = 1, \ldots , N \: .
\label{eq:sub_t_i_q}
\end{equation}
%

%The solutions of the $N$ bivariate optimisation problems in \eqref{eq:sub_t_i} require the computation of a special proximal mapping operator. We dedicate the following Section \ref{sec:po} to carefully discuss the solution of this optimisation problem and show that it can be eventually re-written as a one-dimensional optimisation problem and thus solved efficiently.

We start detailing the solving procedure for problem \eqref{eq:sub_t_i} under the adoption of a $\WDTV^{sv}_{\bm{p}}$ regularisation term, which corresponds to consider the gradient penalty function $f_{\WDTV^{sv}_{\bm{p}}}$ defined in \eqref{eq:WDTVp_h}; the proximal maps arising for the $\WTV$ and $\WTV^{sv}_{\bm{p}}$ regularisers will be discussed afterwards as special cases. In \cite{DTVp_siam}, the authors proved a result on the existence of solutions for problem \eqref{eq:sub_t_i}. Before reporting the statement, we recall that in the following, for $\bm{v},\bm{w}\in\R^n$ we denote by $\bm{v} \circ\bm{w}$, $|\bm{v}|$ and $\mathrm{sign}(\bm{v})$ 
the component-wise (or Hadamard) product between $\bm{v}$ and $\bm{w}$ and 
the component-wise absolute value and sign of $\bm{v}$, respectively.

\begin{lemma}\label{lem::prox_sol1}
	Let $f: \R^2 \to \R_+$ be the (parametric and not necessarily convex) function defined by
	\begin{equation}
	f(\bm{g}) \,\;{:=}\;\,
	\alpha^p \left\| 
	\bm{\Lambda}_a \bm{\mathrm{R}}_{-\theta} \, \bm{g} \right\|_2^p \, , \quad \bm{g} \in \R^2 \, ,
	\end{equation}
	with parameters $\alpha,p \in \R_{++}$, $a \in (0,1]$, $\theta \in [-\pi/2,\pi/2)$, $\bm{\Lambda}_a = \diag(1,a)$ and $\bm{\mathrm{R}}_{-\theta}$ the $2 \times 2$ rotation matrix of angle $-\theta$, and let $\,\mathrm{prox}_f^{\beta}: \R^2 \rightrightarrows \R^2$ be the proximal operator of $f$ with proximity parameter $\beta \in \R_{++}$ defined by
	\begin{equation}
	\bm{g}^* \,\;{\in}\;\, \mathrm{prox}_f^{\beta}(\bm{w})
	\,\;{:=}\;\,
	\argmin{\bm{g} \in \R^2}
	\left\{
	F(\bm{g}) \;\;{:=}\;\,
	f(\bm{g})
	%\alpha^p \left\| 
	%\bm{\Lambda}_a \bm{\R}_{-\theta} \, \bm{g} \right\|_2^p
	+ \frac{\beta}{2} \left\| \bm{g} - \bm{w} \right\|_2^2
	\right\}, \quad
	\bm{w} \in \R^2 \, .
	\label{eq:prox_full}
	\end{equation}
	Then, problem \eqref{eq:prox_full} admits at least one solution, which is unique when $p\geq 1$. Moreover, after defining
	\begin{equation}
	\widetilde{\bm{w}} \,\;{:=}\;\, \bm{\mathrm{R}}_{-\theta}\, \bm{w}, \quad\; 
	\bm{s} \,\;{:=}\;\, \mathrm{sign}(\widetilde{\bm{w}}), \quad\;
	\overline{\bm{w}} \,\;{:=}\;\, 
	|\widetilde{\bm{w}}|, \quad\;
	\bar{\beta} \,\;{:=}\;\, \frac{\beta}{\alpha^p},
	\label{eq:poi}
	\end{equation}
	we have that any solution $\bm{g}^*$ of \eqref{eq:prox_full} can be expressed as
	\begin{equation}
	\bm{g}^* 
	\;{=}\;\, 
	\bm{\mathrm{R}}_{\theta} \left( \bm{s} \:{\circ}\: \bm{z}^*\right),
	\quad\;\:
	\bm{z}^* \;{\in}\; \argmin{\bm{z} \in \mathcal{H}_1 \subset \R^2} \, H(\bm{z}) \, , 
	\label{eq:propprox1}
	\end{equation}
	where $H: \R^2 \to \R_+$ and $\mathcal{H}_1 \subset \R^2$ are defined by
	\begin{equation}
	H(\bm{z}) :=
	\left\| 
	\bm{\Lambda}_a \, \bm{z} \right\|_2^p
	{+}\;
	\frac{\bar{\beta}}{2} 
	\left\| \, \bm{z} - \overline{\bm{w}} \, \right\|_2^2, \qquad
	\mathcal{H}_1 := \mathcal{H} \;{\cap}\; \left( \big[ 0,\overline{w}_1 \big] \times \big[0,\overline{w}_2\big] \right),  \label{eq:def_arc_hyperb}
	\end{equation}
	with $\mathcal{H}$ being 
	
	\begin{enumerate}
		\item the rectangular hyperbola defined by
		\begin{equation}
		\mathcal{H} {:=}\!
		\left\{ 
		\bm{z} \in \R^2: \;
		\left(z_1-c_1\right) \left(z_2-c_2\right){=} c_1 c_2,\, c_1 {=} -\frac{a^2\,\overline{w}_1}{1-a^2}, c_2 {=} \frac{ \overline{w}_2}{1-a^2}  \right\}
		\label{eq:Ah}
		\end{equation}
		for $a \in (0,1)$ and $\,\overline{w}_1 \overline{w}_2 \neq 0$;
		\item the line defined by
		\begin{equation}
		\mathcal{H} \;{:=}\;
		\left\{\,
		\bm{z} \in \R^2: \;\, 
		\overline{w}_2 z_1 - \overline{w}_1 z_2 = 0 \,\right\}
		\label{eq:Ah2}
		\end{equation}
		for $a \in (0,1)$ and $\,\overline{w}_1 \overline{w}_2 = 0$, or for  $a = 1$ and any $\,\overline{w}_1,  \overline{w}_2 \in \R_{+}$.
	\end{enumerate}
\end{lemma}

\begin{corollary}\label{corol}
	The minimisers $\bm{z}^* \in \R^2$ in \eqref{eq:propprox1} can be obtained as follows:
	\begin{equation}
	\bm{z}^* = 
	\left( 
	z_1^*,
	\frac{c_2 \, z_1^*}{z_1^* - c_1} 
	\right) \, ,
	\end{equation}
	where  $c_1$, $c_2 \in \R$ are defined in \eqref{eq:Ah} and $z_1^* \in \R$ is the solution(s) of the following $1$-dimensional constrained minimisation problem:
	\begin{equation}
	z_1^* \,\;{\in}\;\, \argmin{\xi \;{\in}\; [0,\overline{w}_1]}
	\, \left\{ \,
	h(\xi) 
	\;{:=}\; \left( h_1(\xi) \right)^{p/2}
	\:{+}\;\,
	\frac{\bar{\beta}}{2} \,
	h_2(\xi)
	%\,\;{-}\;\,
	%\frac{\overline{\beta}}{2} \,
	%h_2(\xi)
	\, \right\} \, ,
	\end{equation}
	%\begin{equation}
	%h_1(\xi) = \xi^2 \left( 1 + \kappa\frac{c_2^2}{(\xi-c_1)^2} \right), \quad
	%\textcolor{red}{h_2(\xi) = \xi \, (\kappa-1) \left( \xi - 2 c_1 + 2 \, %\frac{c_2^2}{\kappa(\xi-c_1)} \right)} \, .
	%\end{equation}
	\begin{equation}
	h_1(\xi) = \xi^2 \left( 1 + \frac{a^2 \, c_2^2}{(\xi-c_1)^2} \right), \, \quad
	h_2(\xi) = 
	\left(
	\xi - \overline{w}_1
	\right)^2 
	+
	\left(
	\frac{c_2\,\xi}{\xi - c_1} - \overline{w}_2
	\right)^2 .
	%c_1\left(\frac{a-1}{a}\right)\right)^2+\left(c_2\left(\frac{\xi}{\xi-c_1}\right) - (1-a)q_2\right)^2 \, .
	\end{equation}
\end{corollary}

We will omit the proof of this lemma, and provide only a brief graphical sketch of the key steps leading to \eqref{eq:propprox1}. First, in order to get some clues about the approximate position of the minimiser $\bm{z}^*$ in the plane $z_1$-$z_2$, we restrict the study of the function $H$ to the one-parameter family of ellipses
\begin{equation}\label{eq:ell_prox}
\mathcal{E}_R(\bm{0}) = \left\{(z_1,z_2)\in\R^2\mid  z_1^2 + a^2 z_2^2 = R\right\}\,.
\end{equation}
%
%As a way to provide a better understanding of Lemma \ref{lem::prox_sol1}, we remark that one of the key points in the proof reported in \cite{DTVp_siam} is to restrict the study of the objective function $H$ to the one-parameter family of ellipses
%\begin{equation}\label{eq:ell_prox}
%   \mathcal{E}_R = \left\{(z_1,z_2)\in\R^2\mid  z_1^2 + a^2 z_2^2 = R\right\}\,.
%\end{equation}
One can prove that $\bm{z}^*$ needs to belong to the hyperbola $\mathcal{H}$ defined in \eqref{eq:Ah}. More specifically, the sought $\bm{z}^*$ has to coincide with one of the two points in $\mathcal{E}_R\cap\mathcal{H}$ belonging to the first quadrant of the plane $z_1$-$z_2$. In Figure \ref{fig:prox_siam}, we show one ellipse $\mathcal{E}_R$, which is depicted with a blue dashed line, and the hyperbola $\mathcal{H}$, plotted with a solid magenta line. We conclude that $\bm{z}^*$ lies on the arc of hyperbola $\mathcal{H}_1$ which is delimited by the origin $\bm{O}$ and $\overline{\bm{w}}$; $\mathcal{H}_1$ is also illustrated in Figure \ref{fig:prox_siam} with a solid red line.

%When $\bm{z}^*$ is sought along $\mathcal{E}_R$, it is easy to prove that it also has to belong to the hyperbola $\mathcal{H}$ defined in \eqref{eq:Ah}. In Figure \label{fig:prox_siam}, we report one ellipse $\mathcal{E}_R$

%One can prove that the minimiser of the cost function $H$ over $\mathcal{H}\cap \mathcal{E}_R$ the mentioned restriction belongs to the first quadrant of the plane $z_1$-$z_2$, which allows to conclude that the minimiser of the original function $H$ also belongs to the first quadrant. In particular, it lies on the arc of hyperbola $\mathcal{H}_1$ delimited by the origin $\bm{O}$ and the point $\overline{\bm{w}}$. 

%We represent the outlined configuration in Fig. \ref{fig:prox_siam}, where the parametric ellipse is depicted in blue, while the hyperbola $\mathcal{H}$ and its subset $\mathcal{H}_1$ are plotted in magenta and red color, respectively. 

\begin{remark}
	
	Upon the adoption of the $\WTV^{sv}_{\bm{p}}$ regulariser, a similar result can be proven - see \cite[Proposition 1]{tvpl2}. More specifically, in isotropic settings there holds $a=1$, which yields that the parametric family of ellipses in \eqref{eq:ell_prox} reduces to a parametric family of circles. Moreover, the hyperbola $\mathcal{H}$ and the arc $\mathcal{H}_1$ turn into a line and a segment, respectively. The simplified configuration is reported in Figure \ref{fig:prox_wtvp}; notice that also in this case $\mathcal{H}_1$ lies between the origin $\bm{O}$ and $\overline{\bm{w}}=|\bm{w}|$.
	
	\begin{figure}
		\centering   
		\begin{subfigure}{0.45\textwidth}
			\centering
			\includegraphics[height=4.3cm]{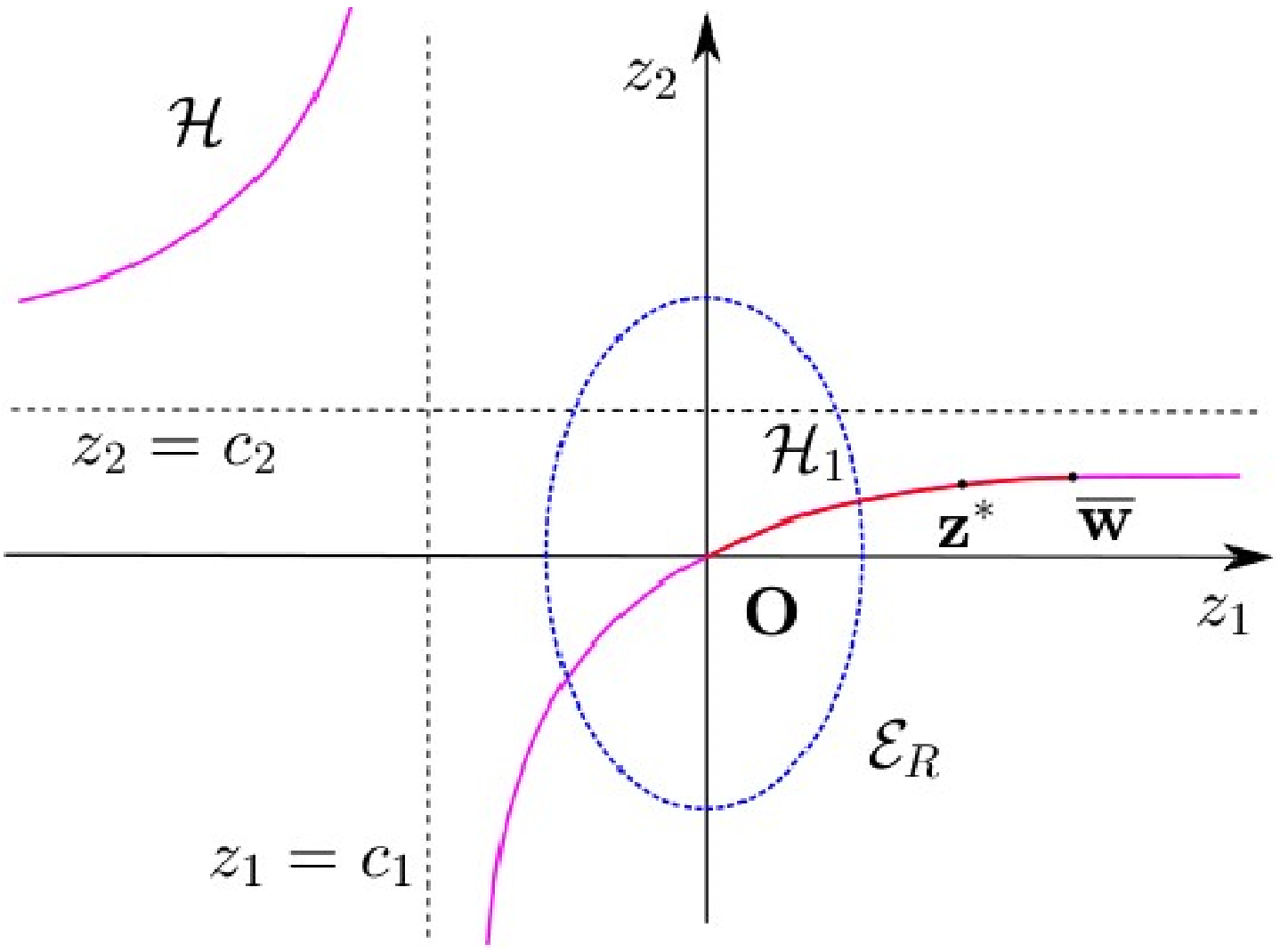}
			\caption{}
			\label{fig:prox_siam}
		\end{subfigure}
		\begin{subfigure}{0.45\textwidth}
			\centering
			\includegraphics[height=4.3cm]{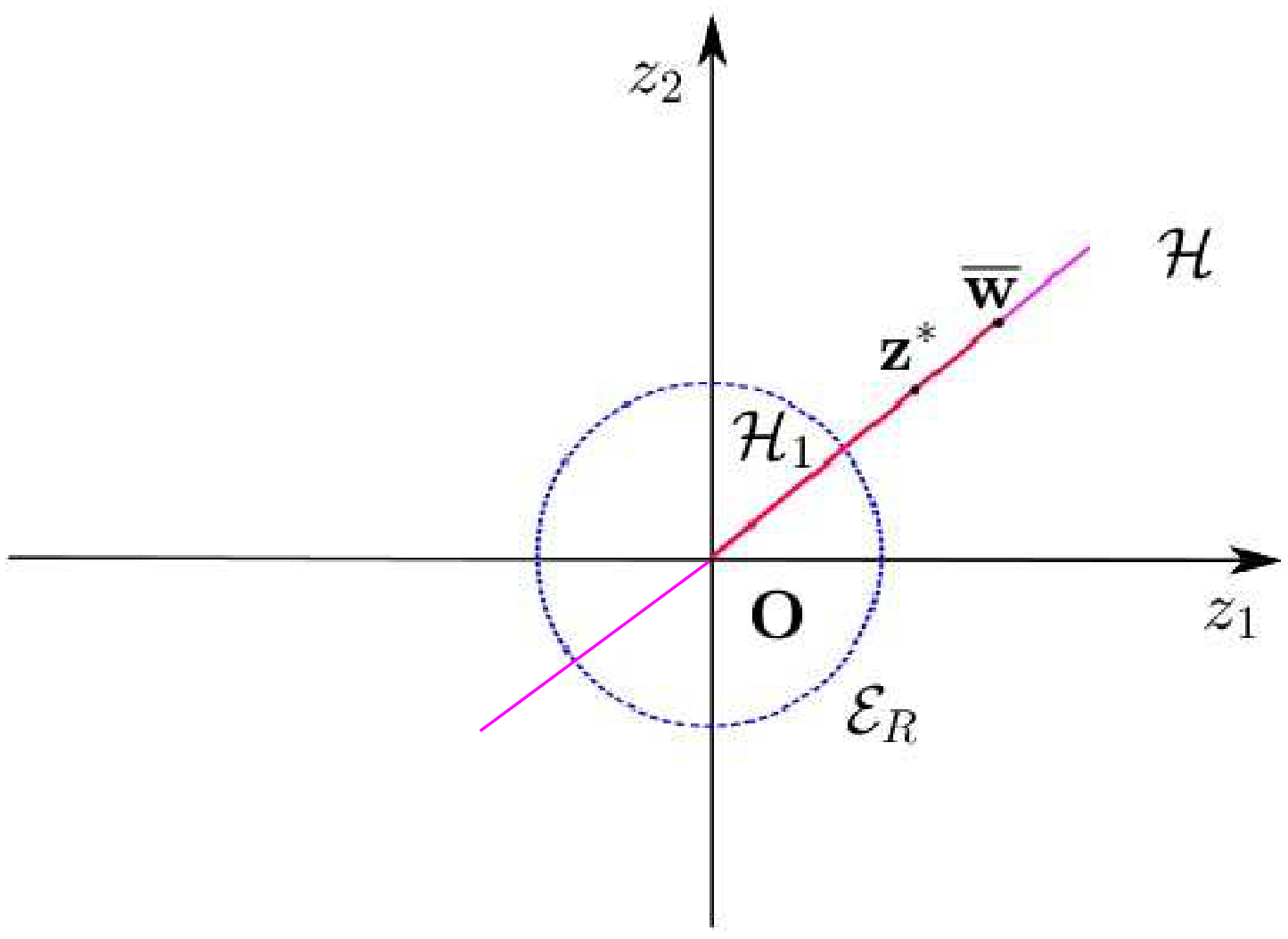}
			\caption{}
			\label{fig:prox_wtvp}
		\end{subfigure}
		\caption{Graphical representation of the bivariate minimisation problems arising for the $\WDTV^{sv}_{\bm{p}}$ (a) and the $\WTV^{sv}_{\bm{p}}$ regularisers (b).}
		\label{fig:prox}
	\end{figure}
	
	For the $\WTV$ regulariser, the solutions of the separable $\bm{g}_i$-subproblems can be written in closed form by means of a soft-thresholding operator, see, e.g. \cite{HWTV}.
	
\end{remark}

\subsection{Subproblem for the primal variable \emph{r}} \label{sec:r_admm}

For $1\leq q < +\infty$ and  $F_q$ as in \eqref{eq:fid1}, after dropping all terms not depending on $\bm{r}$ in (\ref{eq:ADMM_AL}), subproblem (\ref{eq:ADMM_sub_r}) reads
\begin{eqnarray}
\bm{r}^{(j+1)} 
&\;{\in}\; &
\argmin{\bm{r} \in \R^M}
\left\{
\mu^{(j)} \, \mathrm{L}_q(\bm{r};\bm{b}) 
\:{-}\;
\langle \bm{\rho}_r^{(j)} , \bm{r} - \bm{\A u}^{(j+1)} \rangle
+
\frac{\beta_r}{2} \big\| \, \bm{r} - \bm{\A u}^{(j+1)} \big\|_2^2
\right\}
\nonumber\\
&\;{=}\; &
\argmin{\bm{r} \in \R^M}
\left\{
\lambda^{(j)}\,\mathrm{L}_q(\bm{r};\bm{b}) 
\:{+}\; \frac{1}{2} \, \| \bm{r} - \bm{y}^{(j)} \|_2^2
\right\}
\label{eq:subr_l}
%\\
%&\;{=}\;&
%\mathrm{prox}_{\mathrm{L}_q}^{\lambda^{(k)}}
%\left(\bm{q}^{(k)}\right), 
%\label{eq:subr_ll}
\end{eqnarray}
%
%where $\mathrm{prox}_{\mathrm{L}_q}^{\gamma}$ denotes the proximal operator of function $\mathrm{L}_q$ with proximity parameter $\gamma$ and where the newly introduced variables $\gamma \in \R_{++}$ and $\bm{q}^{(k)} \in \R^M$ are defined by
%
where the  variables $\lambda^{(j)} \in \R_{++}$ and $\bm{y}^{(j)} \in \R^M$ are defined by
\begin{equation}\label{eq:qk}
\lambda^{(j)} \;{:=}\; \frac{\mu^{(j)}}{\beta_r}, 
\quad\;\:
\bm{y}^{(j)} \;{:=}\; \bm{\A u}^{(j)} + \frac{1}{\beta_r} \bm{\rho}_r^{(j)} \, .
\end{equation}
%
%\textcolor{red}{
%\begin{definition}[Generalised Discrepancy Principle]
%\begin{equation}\tag{GDP}\label{eq:GDP}
%\text{Select }\mu^{(k+1)}=\mu\text{ such that } \|\bm{r}^{(k+1)}\|_q^q = \delta_{q}\,.
%\end{equation}
%\end{definition}
%}

Note that, as already observed in Section \ref{subsec:joint}, the regularisation parameter $\mu$ is not assumed to be fixed but it is rather estimated along the ADMM iterations (whence the $^{(j)}$ superscript) based on the GDP strategy detailed in Section \ref{sec:gdp}.
In order to update $\mu^{(j)}$, i.e. $\lambda^{(j)}$, so that the GDP is automatically satisfied, we can regard $\lambda^{(j)}$ as a Lagrange multiplier, and then exploit the well-known duality property which allows to replace the unconstrained problem in \eqref{eq:subr_l} with its constrained formulation
\begin{equation}\label{eq:r_con}
\bm{r}^{(j+1)} \;{\in}\; 
\argmin{\bm{r} \in \mathcal{B}_{\delta}^{q}}
\left\{ \| \bm{r} - \bm{y}^{(j)} \|_2^2\right\} = \pi_{\mathcal{B}_{\delta}^{q}}\left(\bm{y}^{(j)}\right)\,,
\end{equation}
where $\pi_{\mathcal{B}_{\delta}^q}$ denotes the projection onto the $\ell_q$-ball $\mathcal{B}_{\delta}^q=\left\{\bm{r}\in\R^M\,:\,\|\bm{r}\|_q\leq \delta_q\right\}$, with $\delta_q$ given in \eqref{eq:GDP2}.

We remark that, although the presence of the regularisation parameter $\mu$ is not explicit in problem \eqref{eq:r_con}, it is actually embedded in the radius $\delta_q$.
%Under the adoption of the Generalised Discrepancy Principle \eqref{eq:GDP} for the automatic selection of $\mu$, problem \eqref{eq:subr_l} can be written in the following constrained form:
%
%\begin{equation}\label{eq:r_con}
%  \bm{r}^{(k+1)} \;{=}\; 
%\argmin{\bm{r} \in \mathcal{B}_{\delta}^{q}}
%\left\{ \| \bm{r} - \bm{q}^{(k)} \|_2^2\right\}\,,\quad\text{with}\quad \mathcal{B}_{\delta}^{q} = \left\{\bm{r}\in\R^M\,:\,\|r\|_q^q \leq \delta_q\right\}\,.
%\end{equation}
%Problem \eqref{eq:r_con} is a convex minimisation problem subject to convex constraint so that it can be equivalently written as
%\begin{equation}
%\label{eq:rproj1}
%  \bm{r}^{(k+1)} \;{=}\; 
%\argmin{\bm{r} \in \mathcal{B}_{\delta}^{q}}
%\left\{ \iota_{\mathcal{B}_{\delta}^q}(\bm{r}) \;{+}\; \| \bm{r} - \bm{q}^{(k)} \|_2^2\right\}  
%\label{eq:rproj2}
%\;{=}\;\pi_{\mathcal{B}_{\delta}^q}\left(\bm{q}^{(k)}\right)\,,
%\end{equation}
%with $\pi_{\mathcal{B}_{\delta}^q}$ denoting the projection onto the set $\ell_q$ ball $\mathcal{B}_{\delta}^q$.
%

\smallskip

When the underlying noise is AIU, i.e. $F_q$ is set as in \eqref{eq:fid2}, the $\bm{r}$-update can be expressed as the constrained minimisation problem in \eqref{eq:r_con}, where the constraint set is the $\ell_{\infty}$-ball with radius $\delta_{\infty}$ defined in \eqref{eq:GDP2}.
%\begin{eqnarray}
%\bm{r}^{(k+1)}
%&\;{\in}\; &
%\argmin{\bm{r} \in \R^M}
%\left\{\iota_{\mathcal{B}_{\delta}^{\infty}}(\bm{r})
%\:{+}\; \frac{\lambda^{(k)}}{2} \, \| \bm{r} - \bm{q}^{(k)} \|_2^2
%\right\}\;{=}\;\pi_{\mathcal{B}_{\delta}}\left(\bm{q}^{(k)}\right)\,,
%\label{eq:subr_linf}
%\\
%&\;{=}\;&
%\mathrm{prox}_{\mathrm{L}_q}^{\lambda^{(k)}}
%\left(\bm{q}^{(k)}\right), 
%\label{eq:subr_ll}
%\end{eqnarray}
%with $\mathcal{B}_{\delta}$ being the $\ell_{\infty}$-ball defined in ?, while $\bm{q}^{(k)}$ is given in \eqref{eq:qk}.
%It must be observed that when the corrupting noise is uniform, the sub-problem for $\bm{r}$ is immediately of the form \eqref{eq:rproj1}.

Note that the projections onto the $\ell_2$ and the $\ell_{\infty}$ balls can be efficiently computed by:
\begin{align}
\label{eq:proj2}
q=2   :&  \qquad\pi_{\mathcal{B}_{\delta}^2}\left(\bm{y}^{(j)}\right)\;{=}\;\left\{
\begin{array}{lc}
\bm{y}^{(j)}&\text{  if  }\|\bm{y}^{(j)}\|^2_2\leq \delta_2\\
\min\left(\delta_2,\lVert\bm{y}^{(j)}\rVert_2\right) \displaystyle{\frac{\bm{y}^{(j)}}{\lVert\bm{y}^{(j)}\rVert_2}}&\text{  otherwise}
\end{array}\right.\\
\label{eq:projInf}
q=+\infty :&  \qquad \pi_{\mathcal{B}_{\delta}^{\infty}}\left(\bm{y}^{(j)}\right) \;{=}\;\left\{\begin{array}{lc}
\bm{y}^{(j)}&\text{  if  }\|\bm{y}^{(j)}\|_{\infty}\leq \delta_{\infty}\\    \min\left(\max(\bm{y}^{(j)},-\delta_{\infty}),\delta_{\infty}\right) & \text{  otherwise}
\end{array} \right.
\end{align}
where all the operations have to be intended componentwise.

For $q=1$, the projection can be computed as follows:
\begin{align}
q=1   :&  \qquad\pi_{\mathcal{B}_{\delta}^1}\left(\bm{y}^{(j)}\right)\;{=}\;\left\{
\begin{array}{lc}
\bm{y}^{(j)}&\text{  if  }\|\bm{y}^{(j)}\|_1\leq \delta_1\\
\mathrm{sign}(\bm{y}^{(j)})\bm{x}&\text{  otherwise}
\end{array}\right.
\end{align}
where 
\begin{equation}
\bm{x}=\max(\bm{y}^{(j)}-\tau,0)\;\text{ and }\;\tau\in\R \;:\; \|\bm{x}\|_1 = \delta_1.   
\end{equation}
Setting a suitable $\tau$ for the problem at hand is possibly a very expensive task from the computational viewpoint. Nonetheless, in \cite{condat} a complexity linear $O(M)$ projection algorithm has been proposed, which improves previous $O(M^2)$ and $O(M\log M)$ strategies considered, e.g., in  \cite{ell1_2,ell1_1,ell1_3}.

%While the projection on the $\ell_2$ and on the $\ell_{\infty}$ can be performed very efficiently, the projection onto the $\ell_1$ ball is known to be very expensive. In fact, it amounts to seeking for $\tau\in\R$ such that $\|\bm{x}\|_1=\delta_1$, where $x_i = \max(q_i^{(k)}-\tau,0)$, $i=1,\ldots,M$. CONDAT.

%Notice that here the regularisation parameter $\mu$ is not fixed but it is updated along the ADMM iterations according to the discrepancy principle.

%In the following, we detail the solution of \eqref{eq:subr_ll} when a $\mathrm{L}_q$ fidelity term with $q=1,2,+\infty$ is adopted.

%When the data fidelity function $\mathcal{F}$ is convex, then the proximity operator in (\ref{eq:subr_ll}) is a singleton-valued function $\mathrm{prox}_{\mathcal{F}}^{\gamma}: \R^M \to \R^M$ and $\bm{r}^{(k+1)}$ is given by the unique global minimizer of the strongly convex cost function in (\ref{eq:subr_l}). 

%For many popular types of noises, the fidelity function $\mathrm{L}_q$ is not only convex but also leads to closed-form expressions of the proximity operator. As a particular and important For instance, for AWGN, AWLN, AWUN and PN, $\mathcal{F}$ takes the forms:
%

In Figure \ref{fig:balls}, we show the $\ell_q$-balls $\mathcal{B}_{\delta}^{q}(\bm{0})$ for the considered choices of $q$ in 2-dimensional settings. In the three plots, we also report the vector $\bm{y}^{(j)}$ in the case it does not belong to $\mathcal{B}_{\delta}^{q}(\bm{0})$, and the projection $\pi\left(\bm{y}^{(j)}\right)$ onto the ball.

\begin{figure}[!t]
	\centering   
	\begin{subfigure}{0.3\textwidth}
		\centering
		\includegraphics[height=3.8cm]{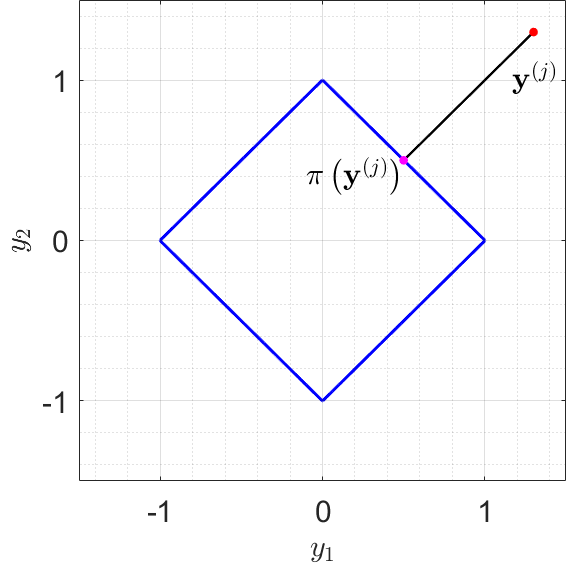}
		\caption{$q=1$}
		\label{fig:l1}
	\end{subfigure}
	\begin{subfigure}{0.3\textwidth}
		\centering
		\includegraphics[height=3.8cm]{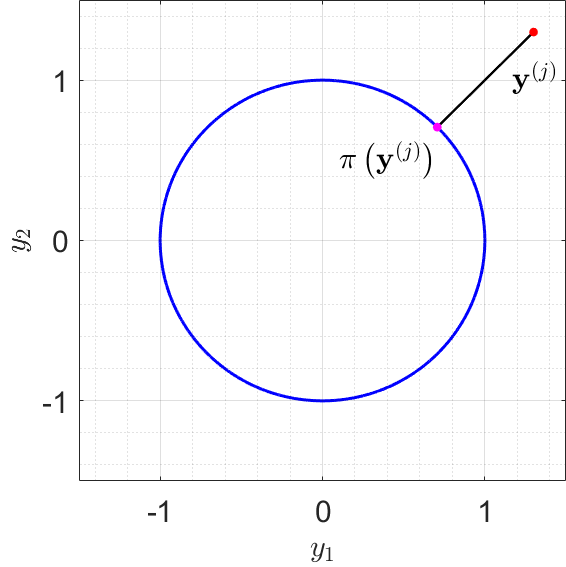}
		\caption{$q=2$}
		\label{fig:l2}
	\end{subfigure}
	\begin{subfigure}{0.3\textwidth}
		\centering
		\includegraphics[height=3.8cm]{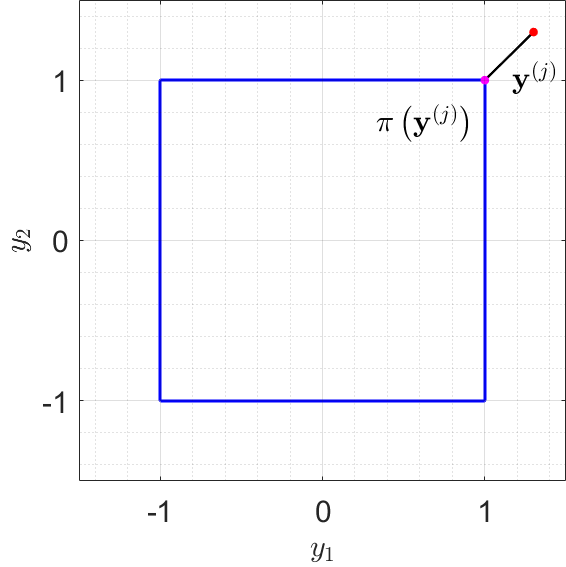}
		\caption{$q=+\infty$}
		\label{fig:linf}
	\end{subfigure}
	\caption{Projections on the unitary $\ell_q$ balls for $q=1$, $q=2$ and $q=+\infty$.}
	\label{fig:balls}
\end{figure}

\begin{comment}
\begin{equation}
\mathcal{F}(\bm{r};\bm{b}) 
\;{=}\, 
\left\{
\begin{array}{ll}
\frac{1}{2}\|\bm{r}-\bm{b}\|_2^2 
\\
\;\|\bm{r}-\bm{b}\|_1 
\\
\|\bm{r}-\bm{b}\|_\infty 
\\
%\displaystyle{\sum_{i=1}^M} 
%\left(r_i - b_i \ln r_i + \iota_{\R_+}(r_i)\right) 
\end{array}
\right.
\!\!\!\!{\Longrightarrow}\; 
\mathrm{prox}_{\mathcal{F}}^{\gamma}
\left(\bm{q}\right)
\;{=}\,
\left\{
\begin{array}{l}
\frac{\bm{b} + \gamma \bm{q}}{1+\gamma} \\
\bm{b}+(\bm{q}-\bm{b}) \mathrm{max}
\left\{
0,1-\frac{1}{\gamma\left| (\bm{q}-\bm{b}) \right|}
\right\}
\\
...
\\
%\bm{q} - \frac{1}{\gamma} + \frac{1}{2}
%\sqrt{
%\left(\bm{q} - \frac{1}{\gamma}\right)^2 
%+ 4 \frac{\bm{b}}{\gamma}
%}
\end{array}
\right.
\end{equation}
\end{comment}

The alternating scheme outlined in Section \ref{subsec:hierarchcal} requires to solve the $\bm{u}$-update via the ADMM until a fixed tolerance has been reached after having performed the $\bm{\Theta}$-update. As a result, the adoption of a \emph{pure} alternating scheme yields a computational burden that can be partially remedied by nesting the parameter estimation in the ADMM scheme, as formalised in the following Algorithm \ref{alg:1}. As clearly detailed in Section \ref{sec:parameter_estimation}, the estimation of the parameters involved in the expression of the $\WTV^{sv}_{\bm{p}}$ and $\WDTV^{sv}_{\bm{p}}$ regularisers represents a further computational bottleneck. Therefore, one can decide to further lighten the algorithmic scheme in Algorithm \ref{alg:1} by not performing the parameters update at each iteration $j$, but at every few iterations.

%\textcolor{blue}{We immediately notice a major computational drawback appearing in the first sub-problem \eqref{eq:ADMM_u}: the potential explicit dependence on $u$ of the parameter function $\bm{\Theta}$ which, in principle would require the solution of a demanding highly non-linear and non-convex optimisation problem. We propose to alleviate this inconvenience by means of the two possible strategies: whenever a once-for-all parameter estimation approach is used, we fix throughout the iterations $\bm{\Theta}(u)=\bm{\Theta}(u_0)$, i.e. we keep the parameters estimated from the given initial image $u_0$. Whenever an iterative parameter estimation approach is used, we linearise the estimation by fixing at any step  $k\geq 0$ $\bm{\Theta}(u)=\bm{\Theta}(u^{(k)})$, i.e. we use the estimation of parameters obtained in the previous iteration.}

\begin{algorithm}[H]\small
	\caption{ Joint ADMM-scheme for hypeparameter estimation and image reconstruction}
	\vspace{0.2cm}
	{\renewcommand{\arraystretch}{1.2}
		\renewcommand{\tabcolsep}{0.0cm}
		\vspace{-0.08cm}
		\begin{tabular}{lll}	\multicolumn{2}{l}{\textbf{inputs}:} & \;\;
			observed image $\,\bm{b}$, forward model operator $\,\bm{\A}$ \\	
			\multicolumn{2}{l}{\textbf{parameters}:} & \;\;
			radius $r>0$, discrepancy parameter $\tau=1$,\\
			\multicolumn{2}{l}{\phantom{\textbf{parameters}:}} & \;\; ADMM penalty parameters $\beta_g,\beta_r>0$ \\\multicolumn{2}{l}{\textbf{outputs}:$\;\:$} & \;\;
			estimated image $\,\bm{u}^*$ and parameters vector $\bm{\Theta^*}$ \\
		\end{tabular}
		\vspace{0.2cm}
		\begin{tabular}{lll}
			\textbf{$\bullet$} & \multicolumn{2}{l}{\;\textbf{ Initialisation}\textbf{:}\;\;$\bm{u}^{(0)}$, $\bm{\rho}_t^{(0)}= \bm{0}_{2N}$, $\bm{\rho}_r^{(0)}=\bm{0}_{M}$}\\
			%\end{tabular}
			%\vspace{-0.55cm}
			%\begin{tabular}{lll}
			\textbf{$\bullet$} & \multicolumn{2}{l}{\;\;\textbf{Nested alternating scheme:}}\vspace{0.1cm}\\
			& \multicolumn{2}{l}{\textbf{for}$\:$ \textit{j = 0, 1, 2, $\, \ldots \,$ until convergence}$\,$ \textbf{do}:} \vspace{0.15cm}\\
			&\multicolumn{2}{l}{$\quad\;\bf{\cdot}$   \emph{parameters update} }
			\vspace{-0.02cm} \\
			&\multicolumn{2}{l}{$\;\quad\;\;$   update $\bm{\Theta}^{(j+1)}\;\;\,\,$ as detailed in Section \ref{sec:thWTV}, \ref{sec:thWTVp} or \ref{sec:thWDTV} }
			\vspace{-0.02cm} \\
			&\multicolumn{2}{l}{$\quad\;\bf{\cdot}$   \emph{primal variables update} }
			\vspace{-0.02cm} \\
			& \multicolumn{2}{l}{$\quad\;\;$   update $\bm{u}^{(j+1)}\;\;\,\;\;\,$ by solving \eqref{eq:subu_ls}}
			\vspace{-0.02cm} \\
			& \multicolumn{2}{l}{$\quad\;\;$   update $\bm{g}^{(j+1)}\;\;\,\quad$  as detailed in Section \ref{sec:t_admm}}
			\vspace{-0.02cm} \\
			& \multicolumn{2}{l}{$\quad\;\;$   update $\bm{r}^{(j+1)}\;\;\,\;\;\;$ as detailed in Section \ref{sec:r_admm}}
			\vspace{-0.02cm} \\
			&\multicolumn{2}{l}{$\quad\;\bf{\cdot}$   \emph{dual variables update} }
			\vspace{-0.02cm} \\
			& \multicolumn{2}{l}{$\quad\;\;$   update $\bm{\rho}_{g}^{(j+1)}\;\;\,\quad$ by \eqref{eq:ADMM_sub_rho_t}}
			\vspace{-0.02cm} \\
			& \multicolumn{2}{l}{$\quad\;\;$   update $\bm{\rho}_r^{(j+1)}\;\;\,\quad$ by \eqref{eq:ADMM_sub_rho_r}}
			\vspace{-0.02cm} \\
			%	%
			& \multicolumn{2}{l}{\textbf{end$\;\:$for}} \vspace{0.11cm} 
		\end{tabular}
		%	\begin{tabular}{lll}	
		%	& \multicolumn{2}{l}{$u^* = u^{(k+1)}$, $\theta^* = \theta^{(k+1)}$}
		%	\end{tabular}		
	}
	\label{alg:1}
\end{algorithm}

\section{Applications to image restoration}
\label{sec:restoration}

In this section, we evaluate the performances of the space-variant regularisers discussed so far, namely the $\WTV$, the $\WTV^{sv}_{\bm{p}}$ and $\WDTV^{sv}_{\bm{p}}$ regularisers in comparison with the space-invariant TV \cite{ROF} and TV$_p$ \cite{tvpl2} regularisers. As an example, we will consider the problem of image deblurring, for which the forward linear operator $\bm{\A}\in\R^{N\times N}$ in \eqref{eq:linmod_2} models the action of a space-invariant blur kernel. 

%\begin{itemize}
%   \item the $\TV$ regulariser \cite{ROF};
%  \item the $\TV_p$ regulariser \cite{tvpl2}.
%\end{itemize}

\paragraph{Test images, quality measures and parameters}
In order to highlight the flexibility of the space-variant approach described in this work, the regularisers of interest will be tested on the restoration of images characterised by different global and local properties. More specifically, we will consider the  \texttt{geometric} image in Figure \ref{fig:test_geom}, which is purely piece-wise constant, the \texttt{skyscraper} image in Figure \ref{fig:test_sky}, which presents a mixture of piece-wise constant, piece-wise linear and textured features, and  the \texttt{stairs} image in Figure \ref{fig:test_stairs}, which is highly textured with fine oriented details. The three test images have all been corrupted by space-invariant Gaussian blur defined by a convolution kernel generated using the Matlab routine \texttt{fspecial} with parameters \texttt{band}\:{=}\:5 and \texttt{sigma}\:{=}\:1. The \texttt{band} parameter represents the side length (in pixels) of the square support of the kernel, whereas \texttt{sigma} is the standard deviation (in pixels) of the isotropic bivariate Gaussian distribution defining the kernel in continuous settings. Then, the blurred images have been degraded by AIGG noise realisations from different distributions with standard deviation $\sigma=0.1$. More specifically, we considered $q=1$ (Laplace noise) for the \texttt{geometric} test image, $q=2$ (Gaussian noise) for the \texttt{skyscraper} test image and $q=+\infty$ (uniform noise) for the \texttt{stairs} test image. The blur- and noise-corrupted images are displayed on the bottom row of Figure \ref{fig:test_rest}.
%The original images and the observed data are displayed in Fig.~\ref{fig:test_rest}. 

\begin{figure}
	\centering
	\begin{subfigure}{0.32\textwidth}
		\centering
		\includegraphics[height=3.5cm]{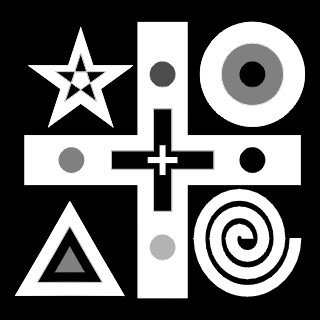}
		\caption{}
		\label{fig:test_geom}
	\end{subfigure}
	\begin{subfigure}{0.32\textwidth}
		\centering
		\includegraphics[height=3.5cm]{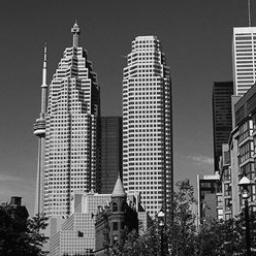}
		\caption{}
		\label{fig:test_sky}
	\end{subfigure}
	\begin{subfigure}{0.32\textwidth}
		\centering
		\includegraphics[height=3.5cm]{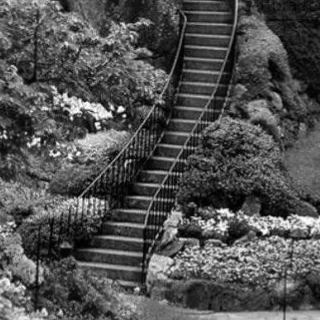}
		\caption{}
		\label{fig:test_stairs}
	\end{subfigure}\\
	\begin{subfigure}{0.32\textwidth}
		\centering
		\includegraphics[height=3.5cm]{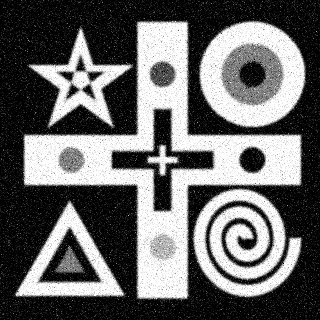}
		\caption{}
		\label{fig:data_geom__}
	\end{subfigure}
	\begin{subfigure}{0.32\textwidth}
		\centering
		\includegraphics[height=3.5cm]{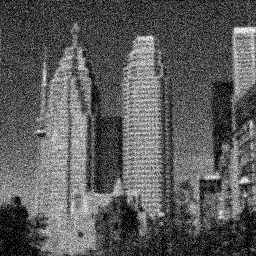}
		\caption{}
		\label{fig:data_sky__}
	\end{subfigure}
	\begin{subfigure}{0.32\textwidth}
		\centering
		\includegraphics[height=3.5cm]{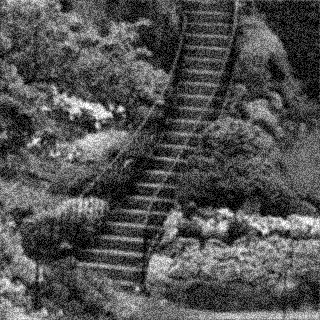}
		\caption{}
		\label{fig:data_stairs__}
	\end{subfigure}
	%\begin{subfigure}{0.32\textwidth}
	%\centering
	%\includegraphics[scale=0.38]{images/rest/sky_data_no_square.png}
	%\caption{}
	%\label{fig:test_im_4}
	%\end{subfigure}
	\caption{Original test images \texttt{geometric} ($320\times 320$), \texttt{skyscraper} $(256\times 256)$  and \texttt{stairs} ($320\times 320$) (top), and observed data corrupted by Gaussian blur and AIGG noise with $q=1$, $q=2$ and $q=+\infty$, respectively (bottom).}
	\label{fig:test_rest}
\end{figure}

The quality of the obtained restorations $\bm{u}^*$ versus the associated ground-truth image $\bm{u}$ is assessed by means of two scalar measures, the Improved Signal-to-Noise Ratio (ISNR)
\begin{equation}\mathrm{ISNR}(\bm{b}; \bm{u}; \bm{u}^*) :=10\log_{10}\left(\frac{\|\bm{b}-\bm{u}\|_2^2}{\|\bm{u}^*-\bm{u}\|_2^2}\right)\,,
\end{equation}
and the Structural Similarity Index (SSIM) \cite{ssim}. The larger the ISNR and SSIM values, the higher the restoration quality. For all tests, the ADMM iterations are stopped as soon as
\begin{equation}
\eta^{(j+1)} := \frac{\|\bm{u}^{(j+1)}-\bm{u}^{(j)}\|_2}{\|\bm{u}^{(j)}\|_2} < 10^{-5}\,,\quad j\in \mathbb{N}.
\end{equation}
The penalty parameters $\beta_g$, $\beta_r$ are manually set. 

The estimation of the hyperparameters in the space-variant regularisers $\WTV$, $\WTV^{sv}_{\bm{p}}$ and $\WDTV^{sv}_{\bm{p}}$ is performed by manually setting the radius $r$, so as to attain the highest ISNR and SSIM values. Moreover, for the WTV regulariser, the existence of a very efficient procedure for the computation of the $\{\alpha_i\}_i$ weights allows to update the $\bm{\alpha}$-map at each iteration of the ADMM-based scheme; in order to hold back the computational effort coming along with the estimation of the unknown $\{\alpha_i,\,p_i\}_i$ in the $\WTV^{sv}_{\bm{p}}$ and $\{\alpha_i,p_i,\theta_i,a_i\}_i$ in the $\WDTV^{sv}_{\bm{p}}$ regulariser, we update the maps of parameters every $30$ iterations. 

%\red{For what concerns the estimation of the global scalar $p$ in the $\TV_p$ regulariser, we refer to the estimation strategy outlined in \cite{tvpl2}.}

For what concerns the estimation of the local $p_i$ for the WTV$_{\bm{p}}^{sv}$ and the WDTV$_{\bm{p}}^{sv}$, as well as of the global $p$ in the TV$_p$ regulariser, we fix the compact set $[\epsilon,R]$ of Propositions \ref{prop:prop_p1} and \ref{prop:ex_scale}, equal to $[0.5,2]$. Notice the the choice of the lower bound allows the $\bm{u}$-estimation problem \eqref{eq:ADMM_u} to result in non-convex regularisers. This implies that a particular attention has to be put in the design of a suitable initial guess,  which can prevent the performed hypermodels to get trapped in bad local minima.

%s we are allowing non-convex configurations for the $\TV_p$-L$_q$, $\WTV^{sv}_{\bm{p}}$-L$_q$ and $\WDTV^{sv}_{\bm{p}}$-L$_q$ models, a particular attention has to be paid to the selection of a suitable initial guess 

We initialise Algorithm \ref{alg:1} using a suitable initialisation minimising noise whiteness for a standard Tikhonov-L$_2$ problem as proposed recently in \cite{Lanzaetna}.

\paragraph{Restoration of \texttt{geometric}}

First, we discuss the performance of the considered regularisers for the restoration of the \texttt{geometric} test image. The restored images are shown in Figure \ref{fig:geom_rest_new}, while the achieved ISNR and SSIM values are reported in Table \ref{tab:vals_rest}. Notice that, in general, the TV regulariser is well-suited for the restoration of piece-wise constant images; however, as discussed in \ref{sec:features_TV}, it also suffers from several drawbacks. Our results confirm that using instead a TV$_p$ regulariser ($p=0.5$) reduces such artefacts. Overall, the three considered space-variant regulariseres appear to be more effective than plain TV.
%for this example, the $\WTV$ returns the best results, while the  $\WTV^{sv}_{\bm{p}}$ pays the price of a strongly non-convex configuration which is reflected into small distorsions in the output restoration. Finally, embedding directional information in the $\WDTV^{sv}_{\bm{p}}$ regularisation term provides further improvements in the final restoration. 

\begin{table}[!t]
	\centering
	\renewcommand{\arraystretch}{1.1}
	\renewcommand{\tabcolsep}{0.3cm}
	\begin{tabular}{c|c|c|c|c|c}\hline\hline
		&$\TV$&$\TV_p$&$\WTV$&$\WTV^{sv}_{\bm{p}}$&$\WDTV^{sv}_{\bm{p}}$\\
		\hline\hline
		\multicolumn{6}{c}{$\qquad$\texttt{geometric}}\\
		\hline
		ISNR&8.8499 & 9.0568 & 9.5567 & 9.6041 & 10.2188\\
		SSIM&0.9227&0.9225&0.9343&0.9346&0.9388\\
		\hline
		\multicolumn{6}{c}{$\qquad$\texttt{skyscraper}}\\
		\hline
		ISNR&2.3239&2.5775&2.7906&2.9894&3.2083\\
		SSIM&0.6255&0.6432&0.6711&0.6789&0.7166\\
		\hline
		\multicolumn{6}{c}{$\qquad$\texttt{stairs}}\\
		\hline
		ISNR&3.9417&4.5251&4.6836&5.0718&5.2031\\
		SSIM&0.6515&0.6912&0.6879&0.7149&0.7307\\
	\end{tabular}
	
	\caption{ISNR and SSIM values achieved by the considered regularisers for the three test images corrupted by blur and different AIGG noises.}
	\label{tab:vals_rest}
\end{table}

In Figure \ref{fig:geom_maps}, we show the output maps of parameters for the $\WTV$, $\WTV^{sv}_{\bm{p}}$ and $\WDTV^{sv}_{\bm{p}}$ regularisers, obtained with $r=1$, $r=3$ and $r=1$, respectively. For all three regularisers, the $\bm{\alpha}$-maps present higher weights in the background, while showing that weaker regularisation is performed along the profiles of the geometrical figures. Notice that the $\bm{p}$-values in the WTV$_{\bm{p}}^{sv}$ and in the WDTV$_{\bm{p}}^{sv}$ approach 2 in the background, which combined with the high regularisation weights allow for an effective smoothing and noise removal therein. Finally, the $\bm{\theta}$ and $\bm{a}$ maps in the bottom row of Figure \ref{fig:geom_maps} show that the estimator detects a clear directionality in correspondence of the figure profiles, where the angles $\bm{\theta}$ have been accurately estimated and $\bm{a}$ assume small values.

%has been detected by the p, the $\bm{a}$-values appear to be smaller in correspondence of edges, i.e. when a local dominant orientation is detected, while the angles in the $\bm{\theta}$-map indicate the directions of the figures profiles.

\begin{figure}
	\centering
	\renewcommand{\tabcolsep}{0.02cm}
	\resizebox{\textwidth}{!}{
		\begin{tabular}{cccc}
			\includegraphics[height = 2.5cm]{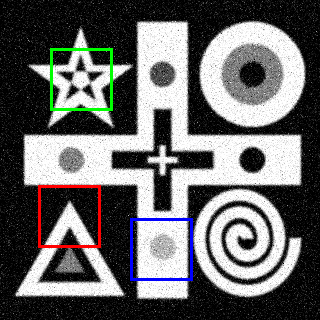}&\includegraphics[height = 2.5cm]{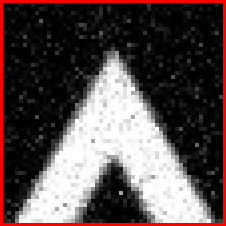}&\includegraphics[height = 2.5cm]{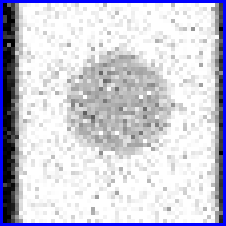}&\includegraphics[height = 2.5cm]{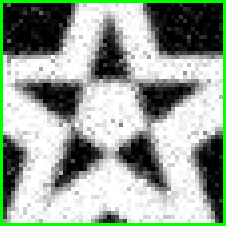}\\ 
			\includegraphics[height = 2.5cm]{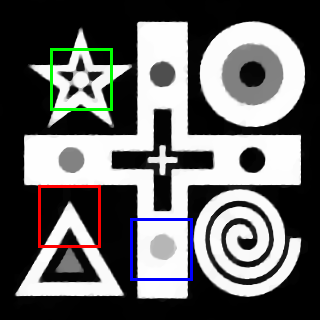}&
			\includegraphics[height = 2.5cm]{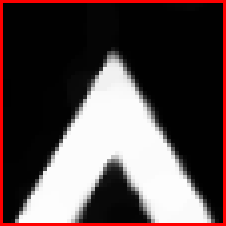}&\includegraphics[height = 2.5cm]{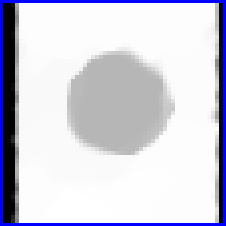}&\includegraphics[height = 2.5cm]{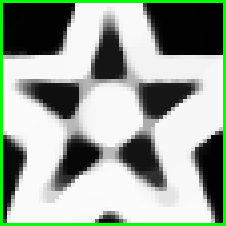}\\
			\includegraphics[height = 2.5cm]{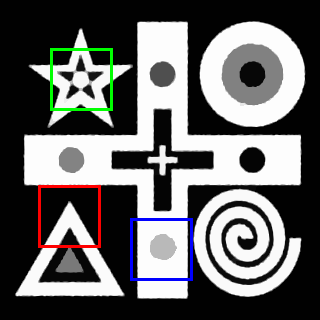}&
			\includegraphics[height = 2.5cm]{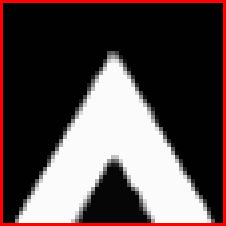}&\includegraphics[height = 2.5cm]{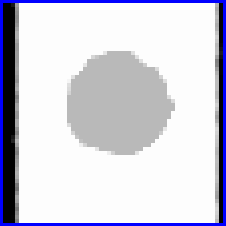}&\includegraphics[height = 2.5cm]{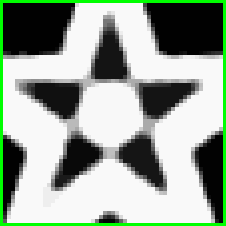}\\
			\includegraphics[height = 2.5cm]{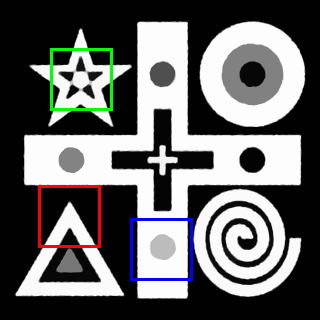}&
			\includegraphics[height = 2.5cm]{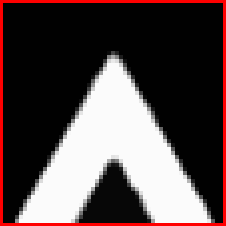}&\includegraphics[height = 2.5cm]{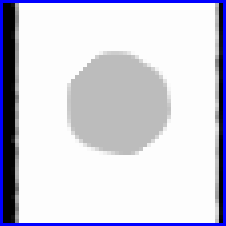}&\includegraphics[height = 2.5cm]{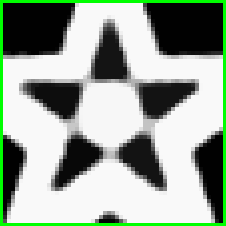}\\
			\includegraphics[height = 2.5cm]{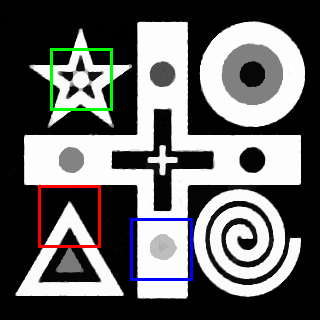}&
			\includegraphics[height = 2.5cm]{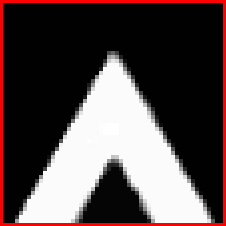}&\includegraphics[height = 2.5cm]{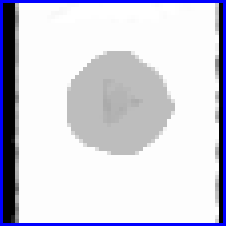}&\includegraphics[height = 2.5cm]{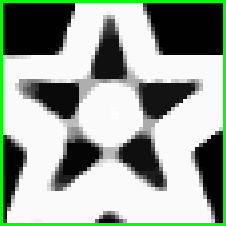}\\
			\includegraphics[height = 2.5cm]{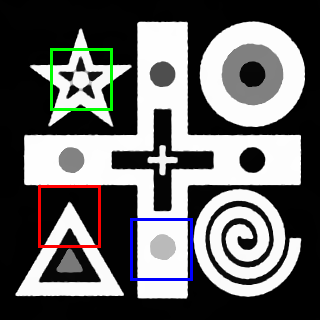}&
			\includegraphics[height = 2.5cm]{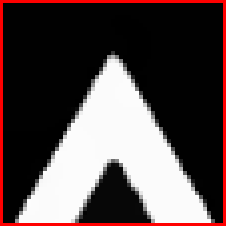}&\includegraphics[height = 2.5cm]{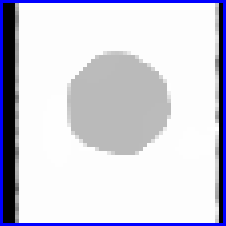}&\includegraphics[height = 2.5cm]{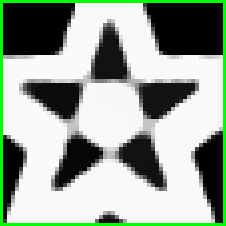}\\
	\end{tabular}}
	\caption{From top to bottom: for the test image \texttt{geometric}, observed image $b$, performance of the TV, the TV$_p$ (with output $p=0.5$), the WTV, the WTV$_{\bm{p}}^{sv}$ and the WDTV$_{\bm{p}}^{sv}$ regularisers with the respective close-up(s).}
	\label{fig:geom_rest_new}
\end{figure}

\begin{figure}
	\centering
	\renewcommand{\tabcolsep}{0.01cm}
	\resizebox{\textwidth}{!}{
		\begin{tabular}{cccc}
			\includegraphics[height = 2.5cm]{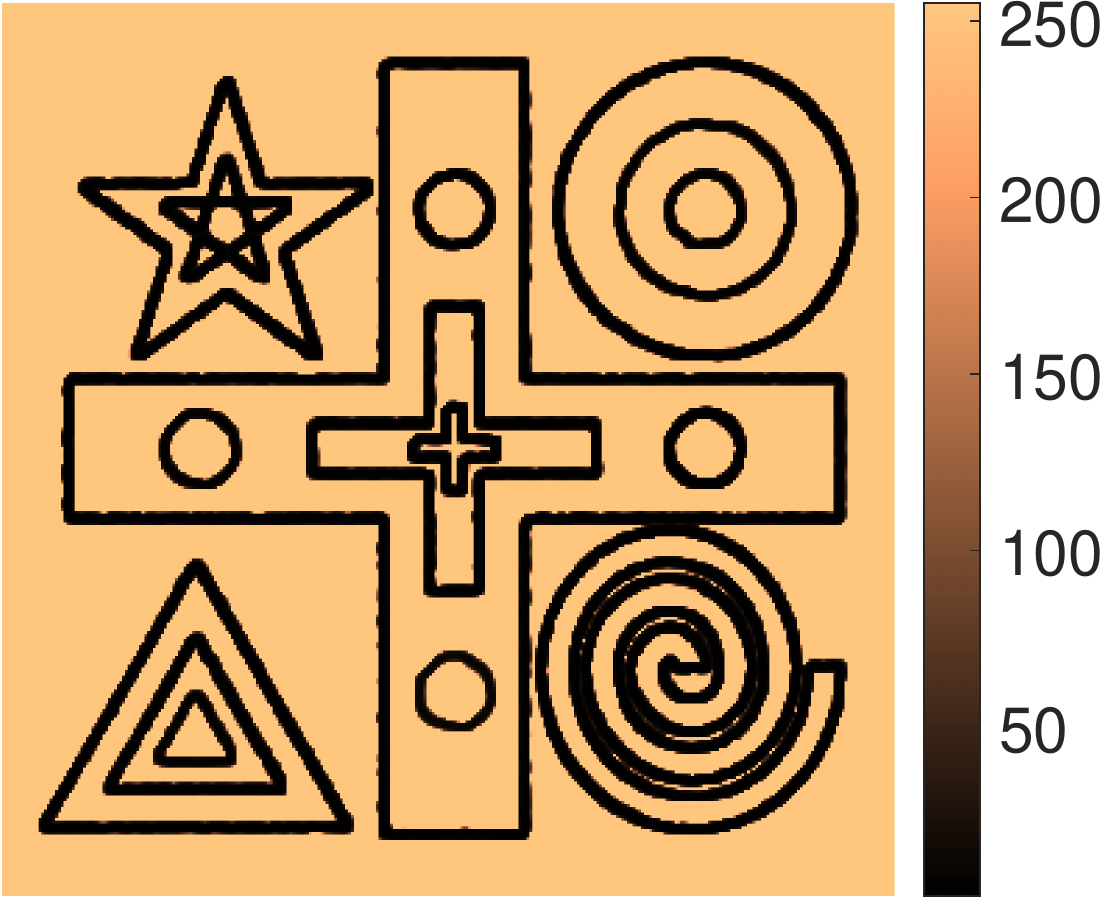}&&&\\ 
			$\bm{\alpha}$&&&\\
			\includegraphics[height = 2.5cm]{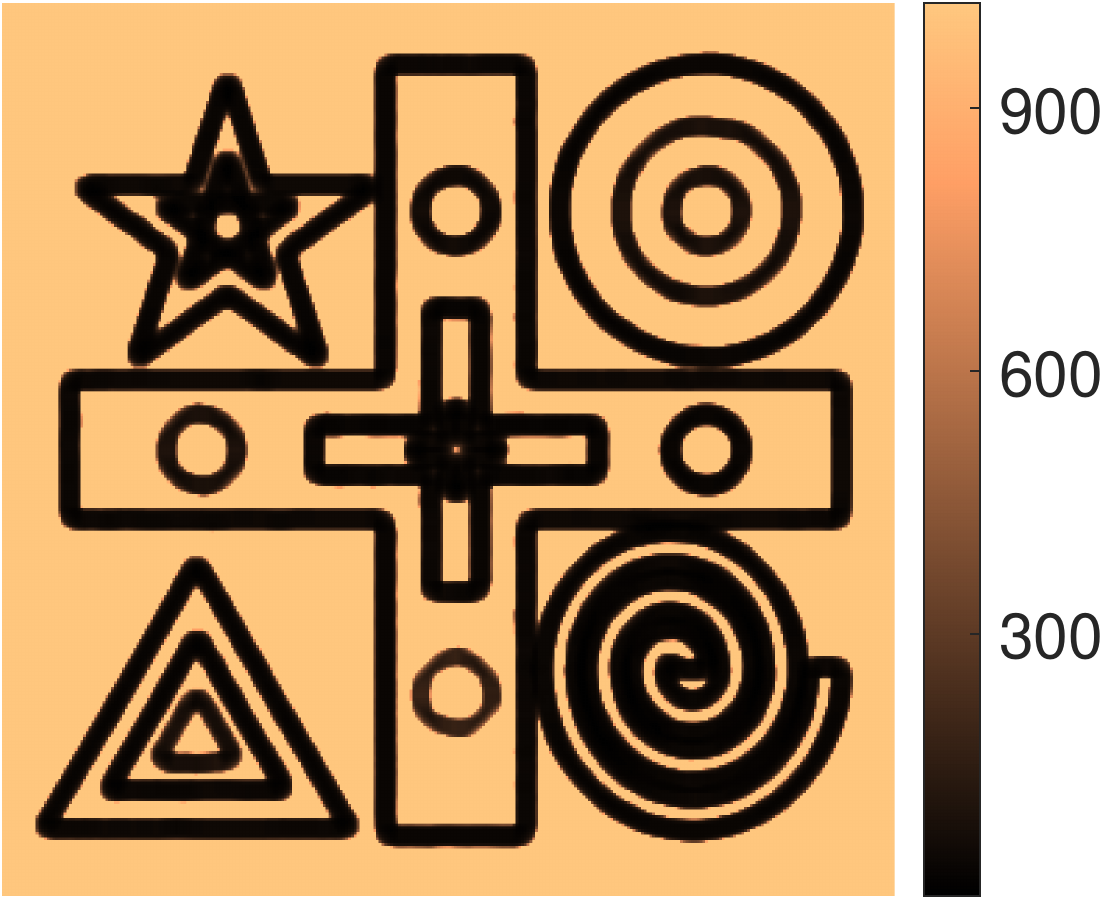}&
			\includegraphics[height = 2.5cm]{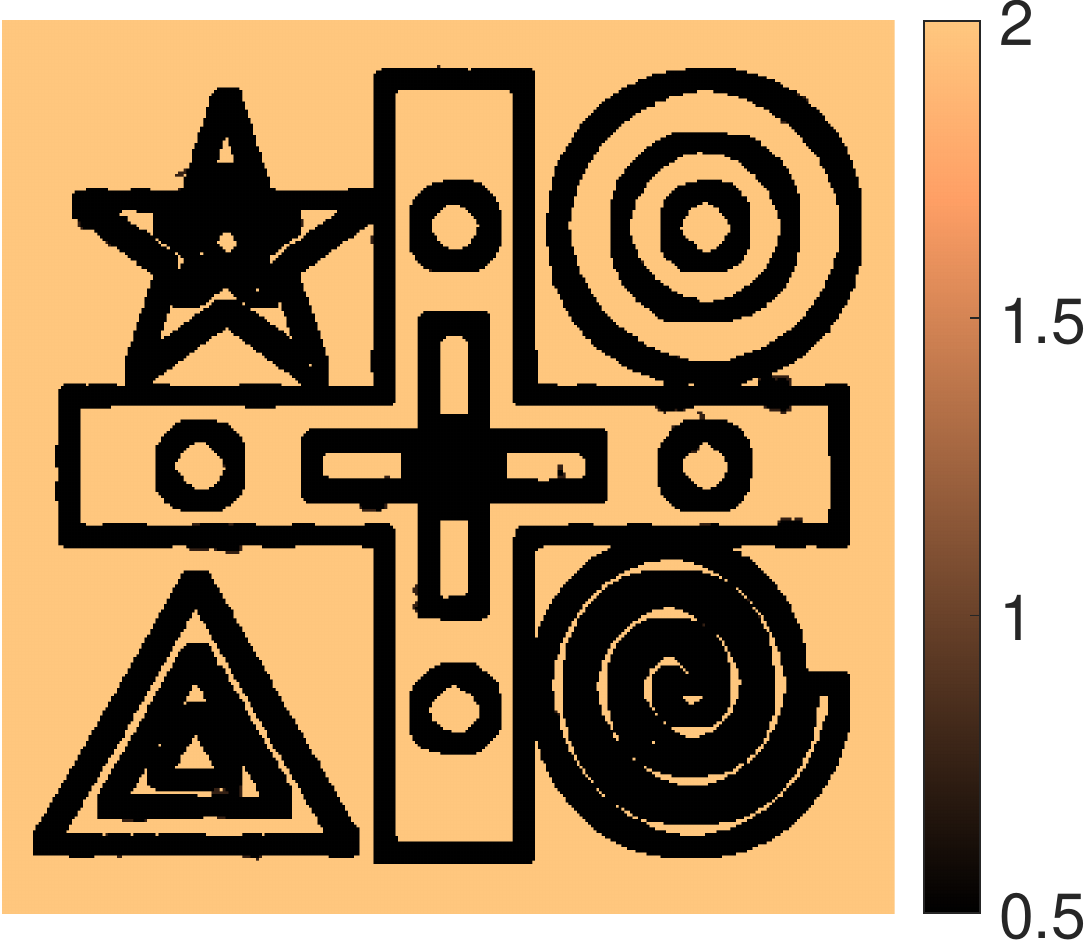}&&\\
			$\bm{\alpha}$&$\bm{p}$&&\\
			\includegraphics[height = 2.5cm]{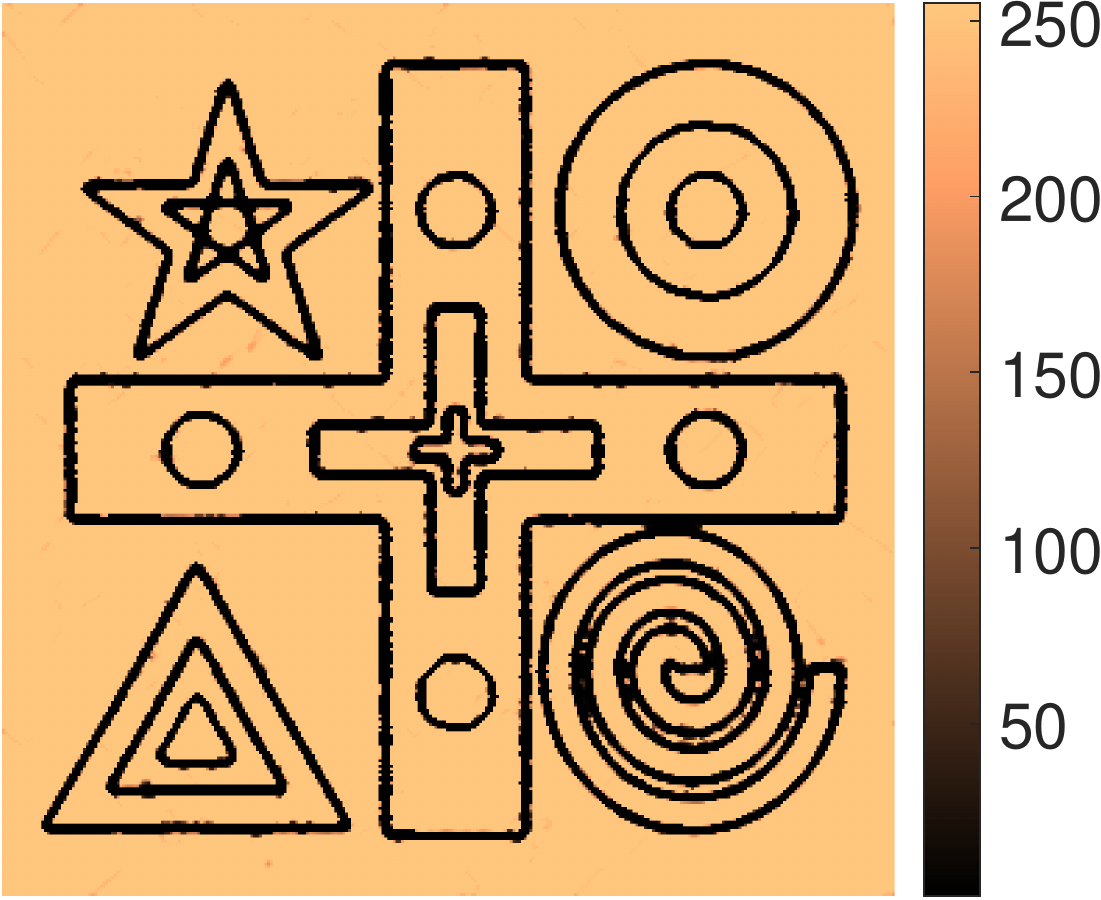}&
			\includegraphics[height = 2.5cm]{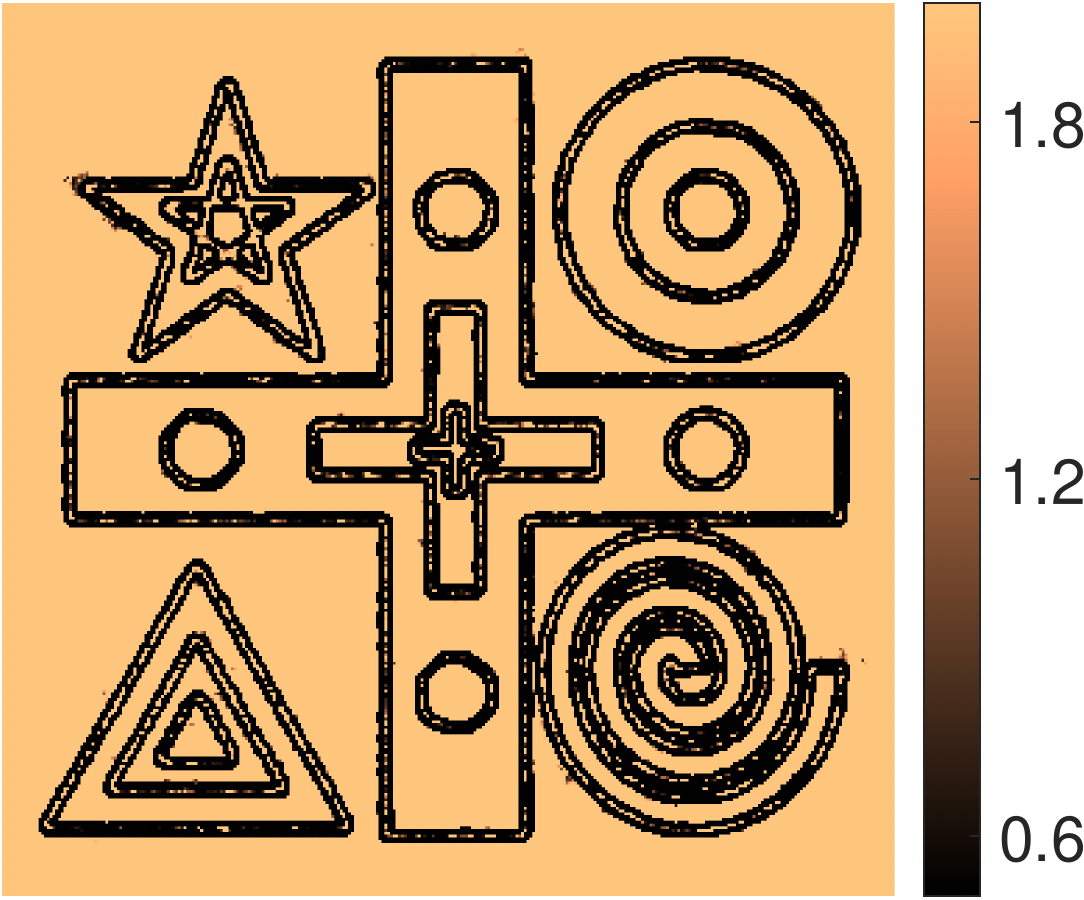}&\includegraphics[height = 2.5cm]{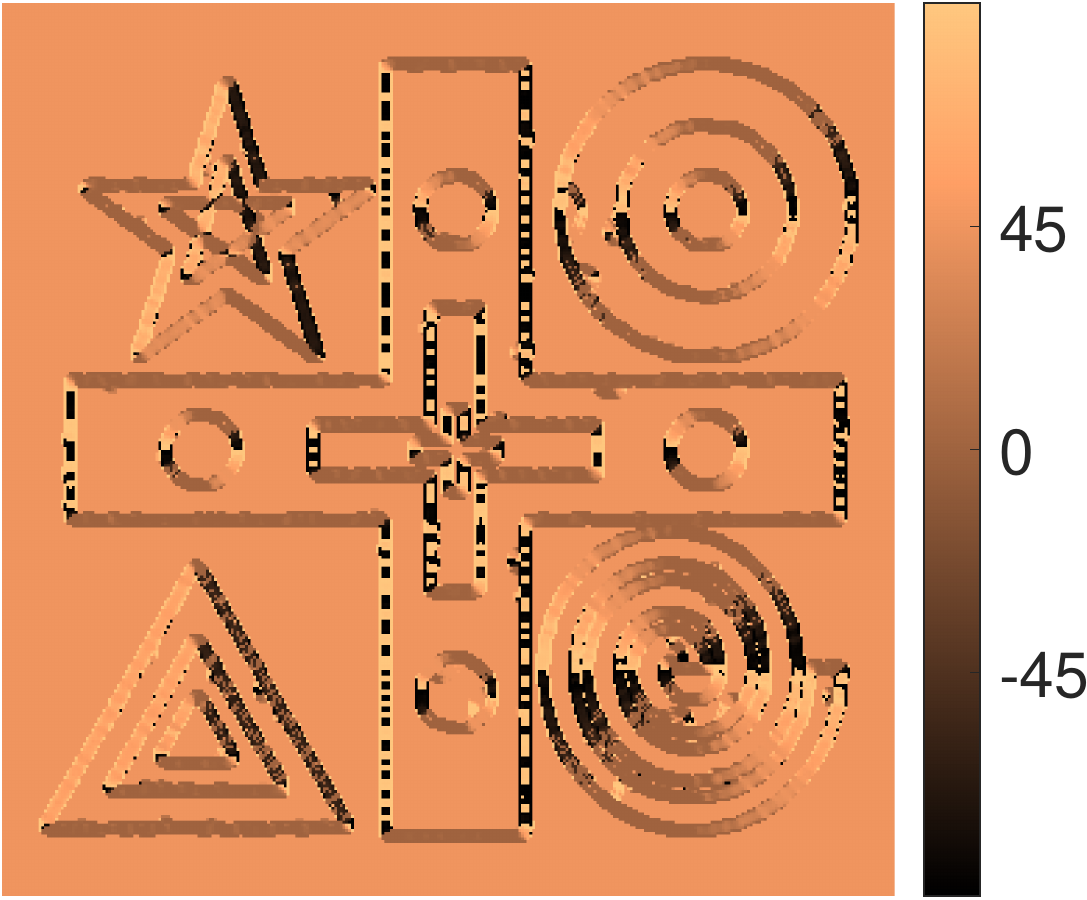}&
			\includegraphics[height = 2.5cm]{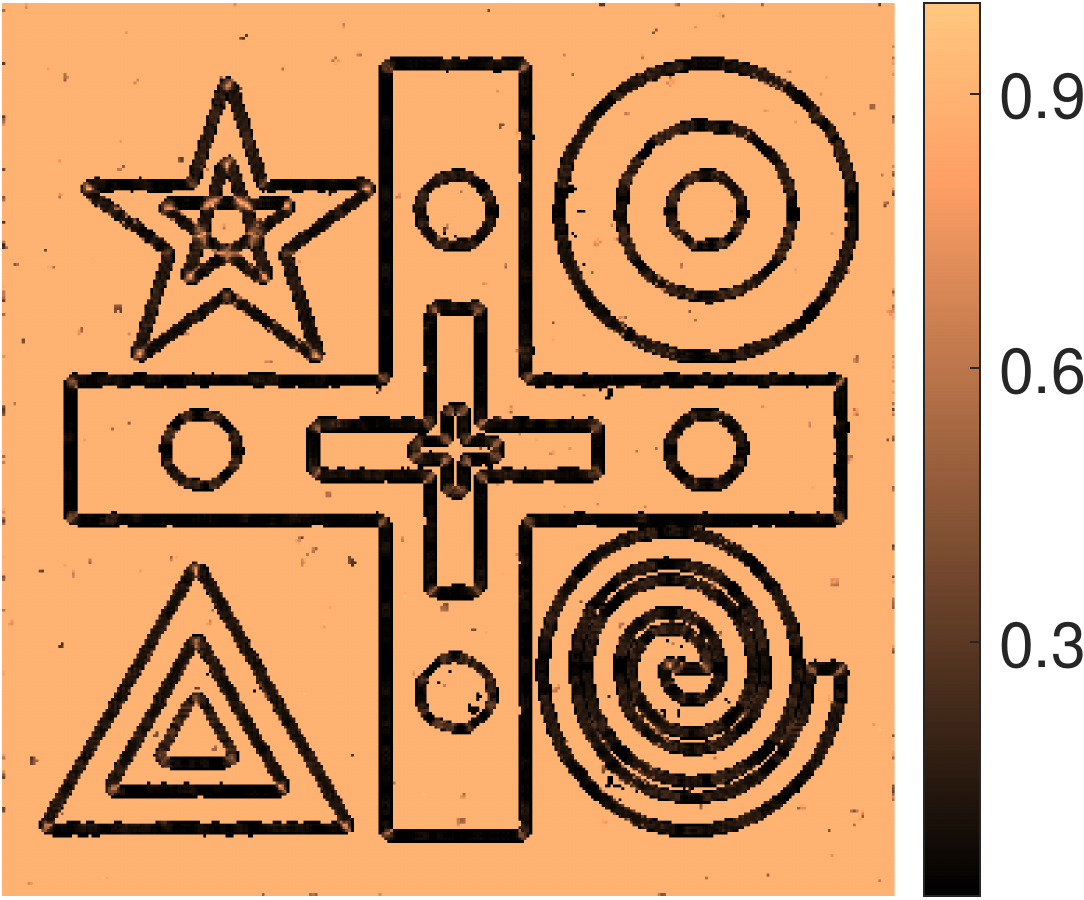}
			\\
			$\bm{\alpha}$&$\bm{p}$&$\bm{\theta}$&$\bm{a}$\\
	\end{tabular}}
	\caption{From top to bottom: for the test image \texttt{geometric}, output maps of the parameters for the $\WTV$ ($r=1$), the $\WTV^{sv}_{\bm{p}}$ ($r=3$) and the $\WDTV^{sv}_{\bm{p}}$ ($r=1$) regularisers.}
	\label{fig:geom_maps}
\end{figure}

\paragraph{Restoration of \texttt{skyscraper}}

We now consider the restoration of the test image \texttt{skyscraper}, which, due to its composite nature, is expected to largely benefit from a space-variant approach. From the restored images and the selected details in Figure \ref{fig:sky_rest_new}, one can clearly notice how each additional space-variant parameter effectively contributes in gradually improving the output result, as also reflected in the ISNR and SSIM values reported in Table \ref{tab:vals_rest}.
In Figure \ref{fig:maps_sky}, we show the output map of parameters for the $\WTV$, $\WTV^{sv}_{\bm{p}}$ and $\WDTV^{sv}_{\bm{p}}$ regularisers, computed for $r=15$, $r=15$ and $r=3$, respectively.
%For the last case, the selection of a significant smaller radius $r$ for the parameter estimation is motivated by the presence in the image of fine oriented textures that are less likely to be detected when employing larger radii.

As a general comment, we highlight that the weights $\alpha_i$ assume larger values on the background so that a strong regularisation is performed regardless of the corresponding $p_i$; in fact, the $\bm{p}$-maps for the WTV$_{\bm{p}}^{sv}$ and the WDTV$_{\bm{p}}^{sv}$ regularisers appear to be different in this region. From the $\bm{\theta}$-map reported in the bottom row of Figure \ref{fig:maps_sky}, we observe that also in this case the estimator is capable of detecting the direction of the buildings profile as well as the horizontal oriented texture. Finally, the $\bm{a}$ values indicate a stronger dominance in terms of directionality along the edges of the buildings.

\begin{figure}
	\centering
	\renewcommand{\tabcolsep}{0.02cm}
	\resizebox{\textwidth}{!}{
		\begin{tabular}{cccc}
			\includegraphics[height = 2.5cm]{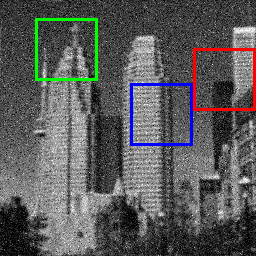}&\includegraphics[height = 2.5cm]{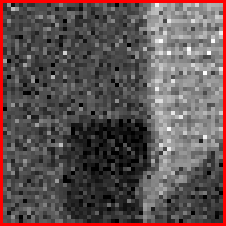}&\includegraphics[height = 2.5cm]{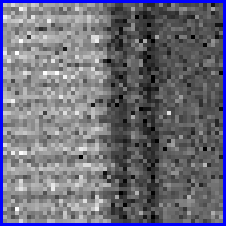}&\includegraphics[height = 2.5cm]{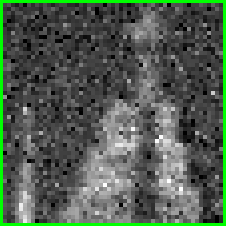}\\ 
			\includegraphics[height = 2.5cm]{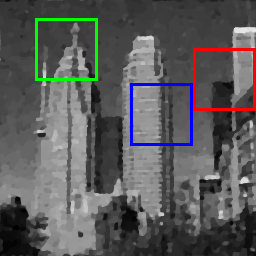}&
			\includegraphics[height = 2.5cm]{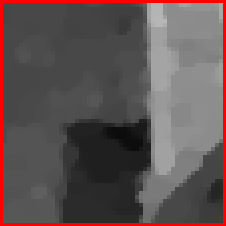}&
			\includegraphics[height = 2.5cm]{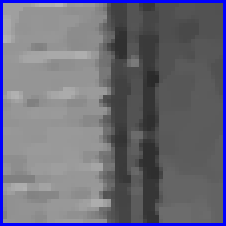}&
			\includegraphics[height = 2.5cm]{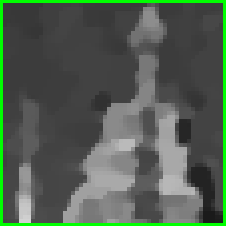}\\
			\includegraphics[height = 2.5cm]{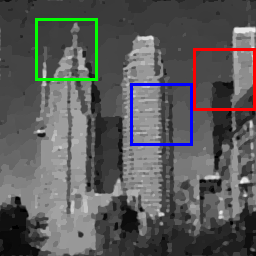}&
			\includegraphics[height = 2.5cm]{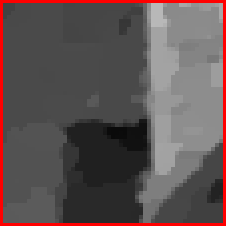}&
			\includegraphics[height = 2.5cm]{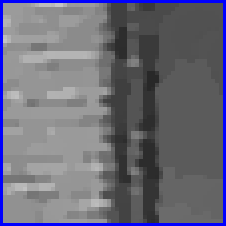}&
			\includegraphics[height = 2.5cm]{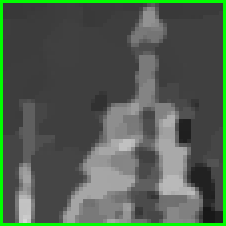}\\
			\includegraphics[height = 2.5cm]{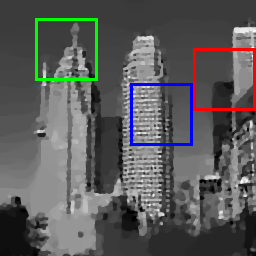}&
			\includegraphics[height = 2.5cm]{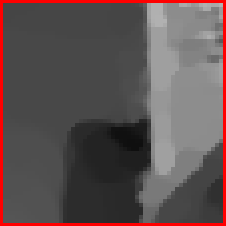}&
			\includegraphics[height = 2.5cm]{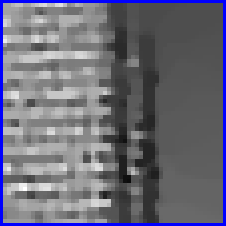}&
			\includegraphics[height = 2.5cm]{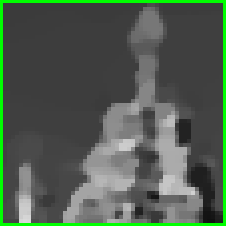}\\
			\includegraphics[height = 2.5cm]{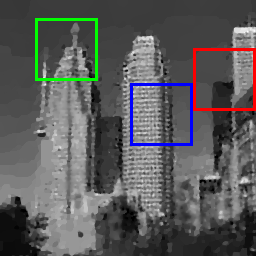}&
			\includegraphics[height = 2.5cm]{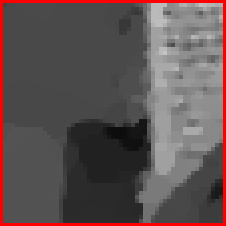}&
			\includegraphics[height = 2.5cm]{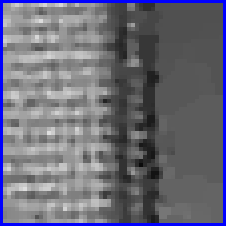}&
			\includegraphics[height = 2.5cm]{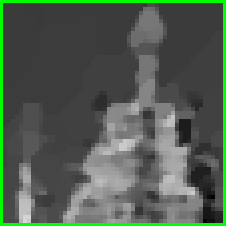}\\
			\includegraphics[height = 2.5cm]{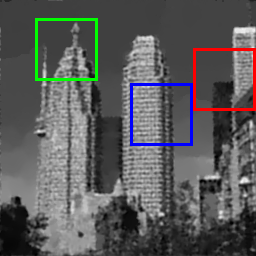}&
			\includegraphics[height = 2.5cm]{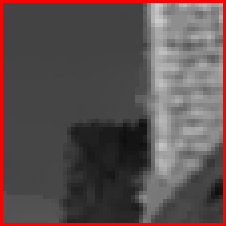}&
			\includegraphics[height = 2.5cm]{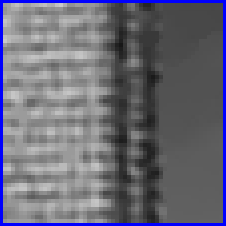}&
			\includegraphics[height = 2.5cm]{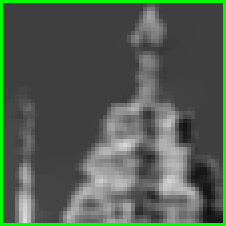}
	\end{tabular}}
	\caption{From top to bottom: for the test image \texttt{skyscraper}, observed image $b$, performance of the TV, the TV$_p$ (with output $p=0.5$), the WTV, the WTV$_{\bm{p}}^{sv}$ and the WDTV$_{\bm{p}}^{sv}$ regularisers with the respective close-up(s).}
	\label{fig:sky_rest_new}
\end{figure}

\begin{figure}
	\centering
	\renewcommand{\tabcolsep}{0.01cm}
	\resizebox{\textwidth}{!}{
		\begin{tabular}{cccc}
			\includegraphics[height = 2.5cm]{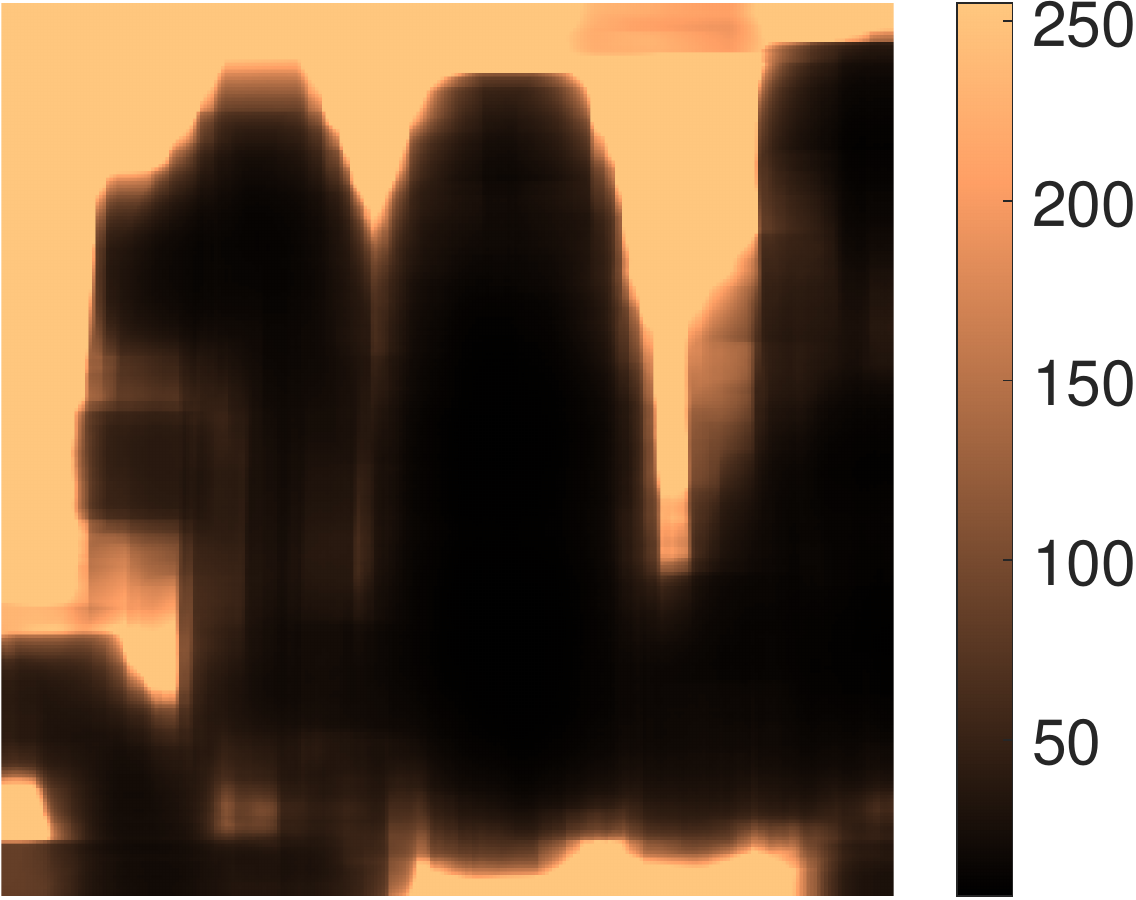}&&&\\ 
			$\bm{\alpha}$&&&\\
			\includegraphics[height = 2.5cm]{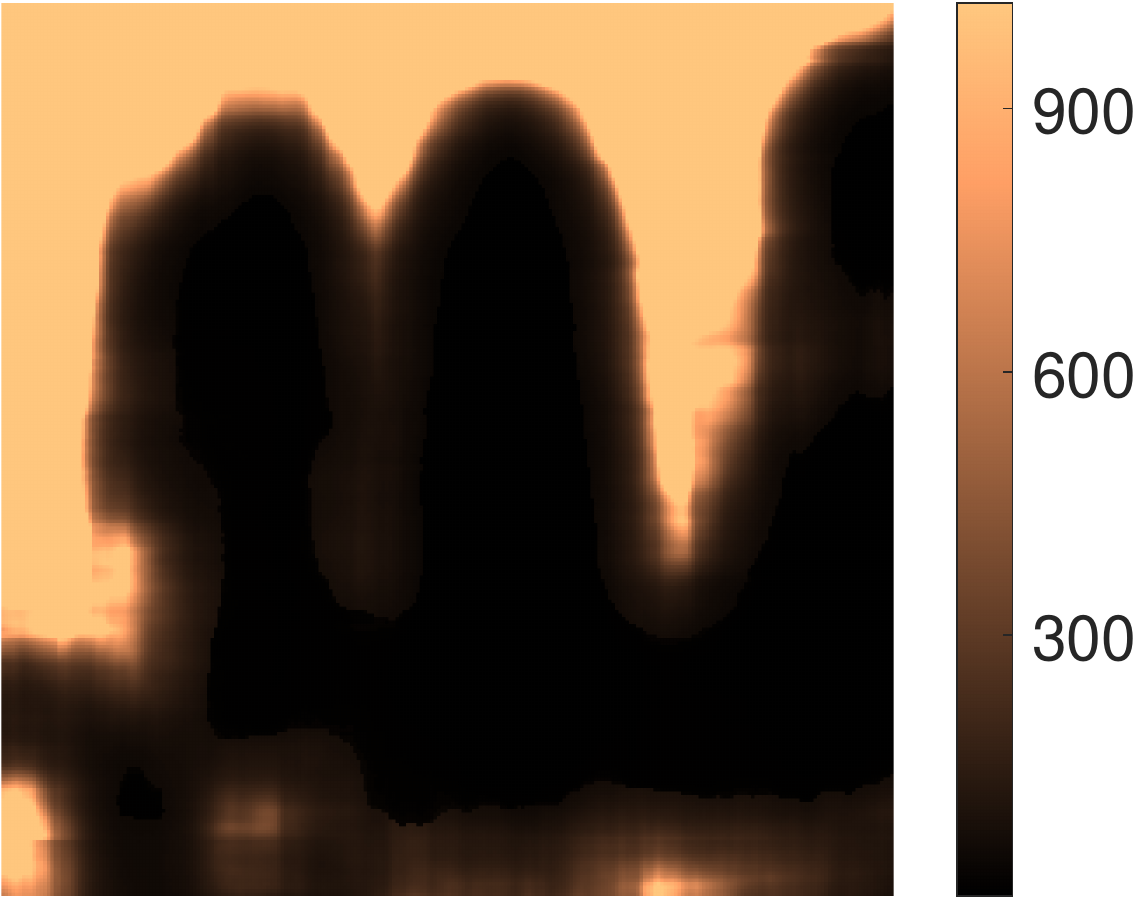}&
			\includegraphics[height = 2.5cm]{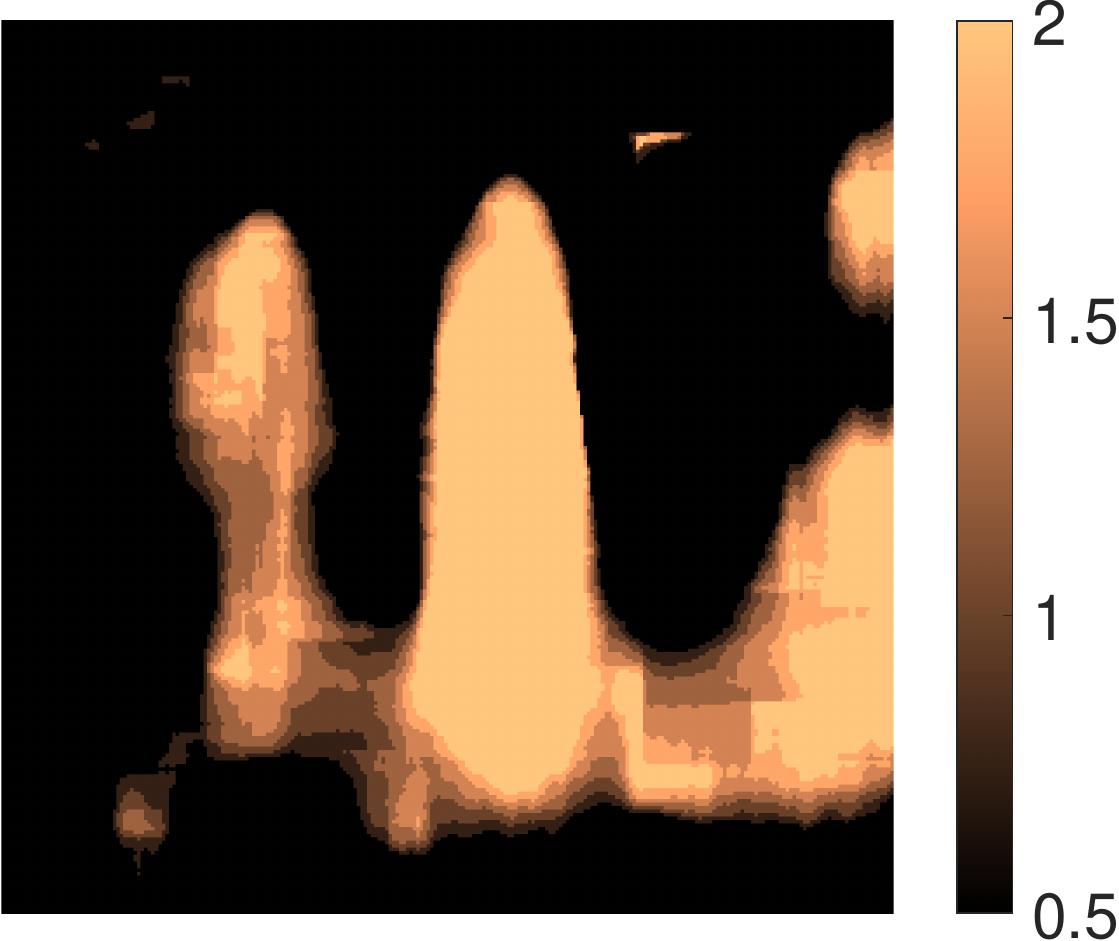}&&\\
			$\bm{\alpha}$&$\bm{p}$&&\\
			\includegraphics[height = 2.5cm]{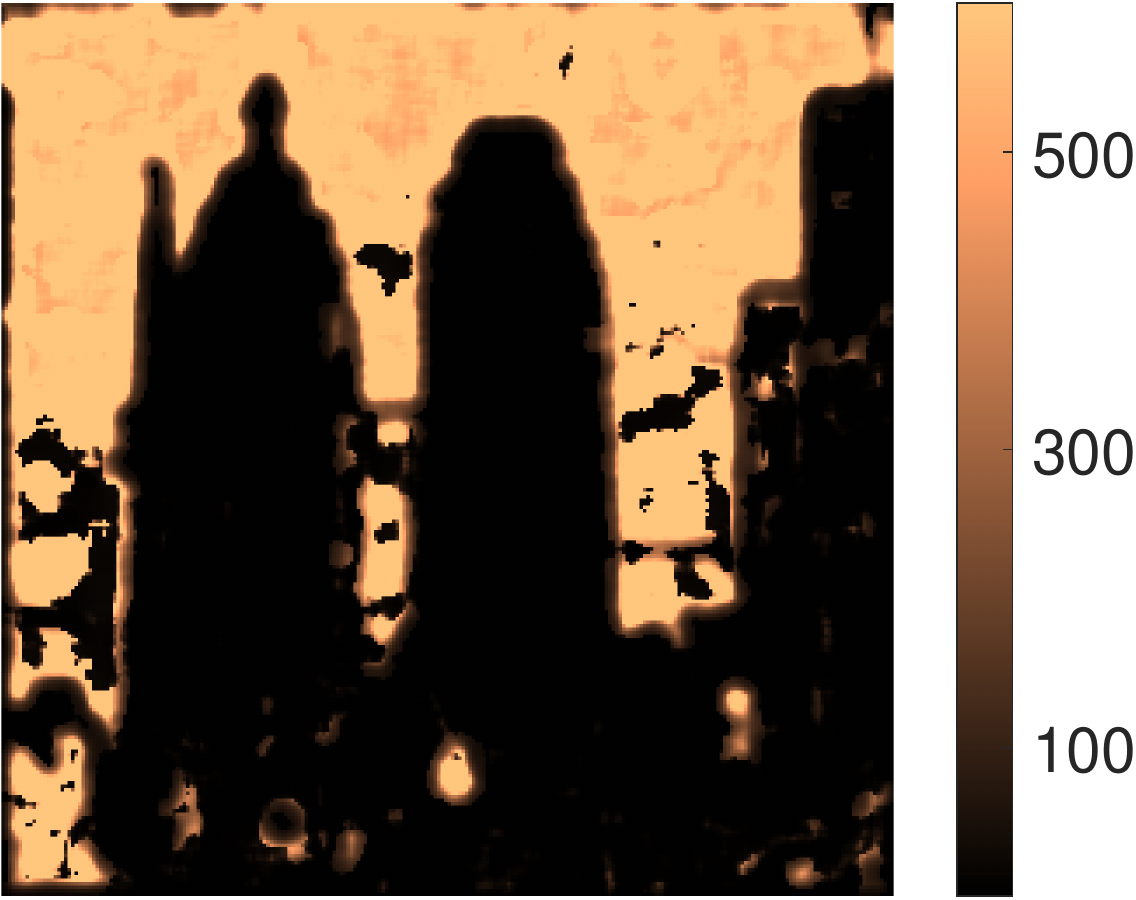}&
			\includegraphics[height = 2.5cm]{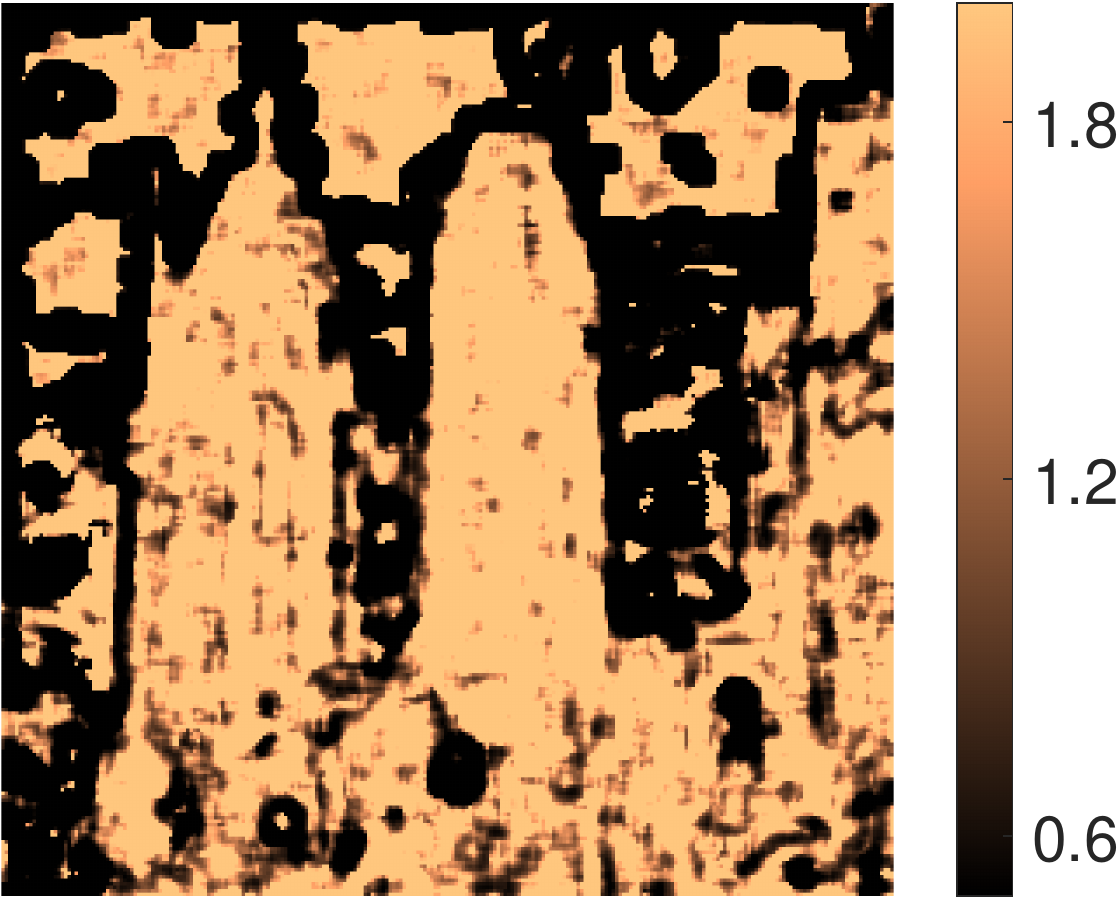}&\includegraphics[height = 2.5cm]{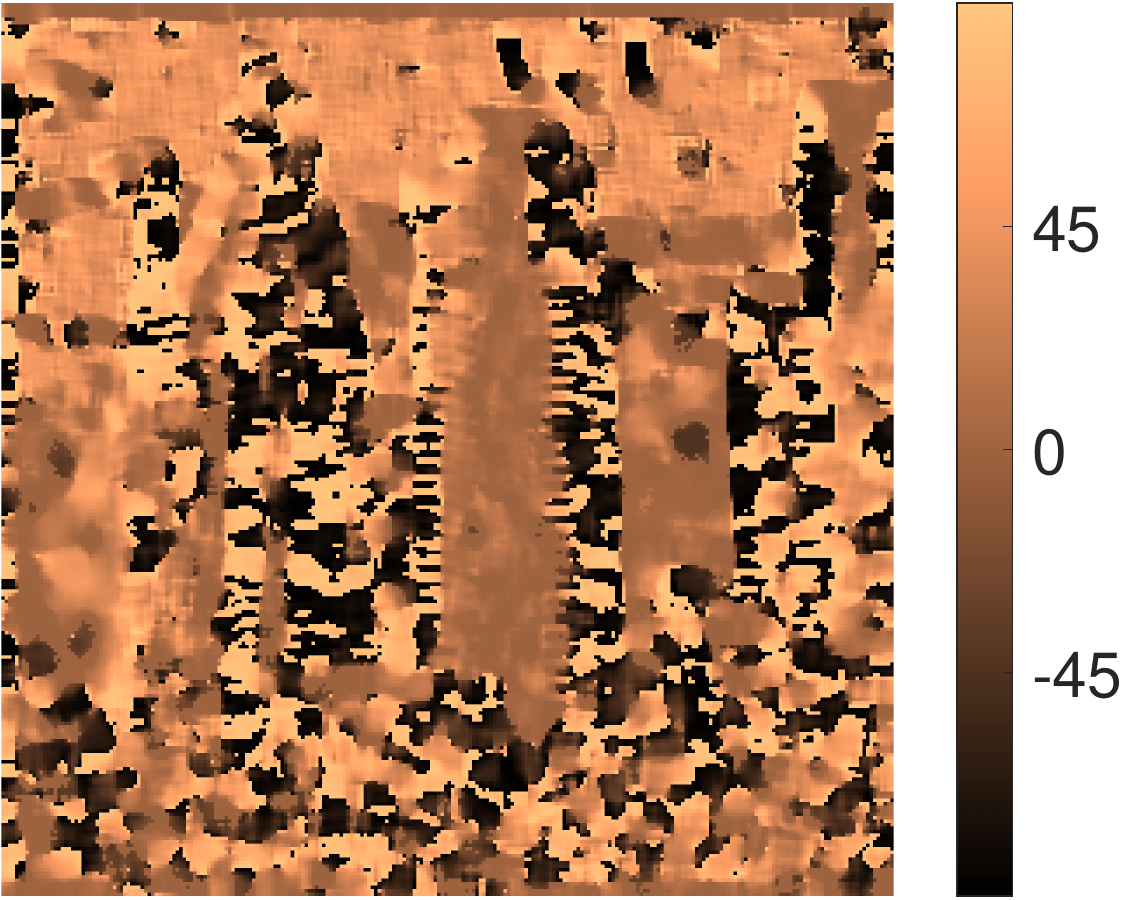}&
			\includegraphics[height = 2.5cm]{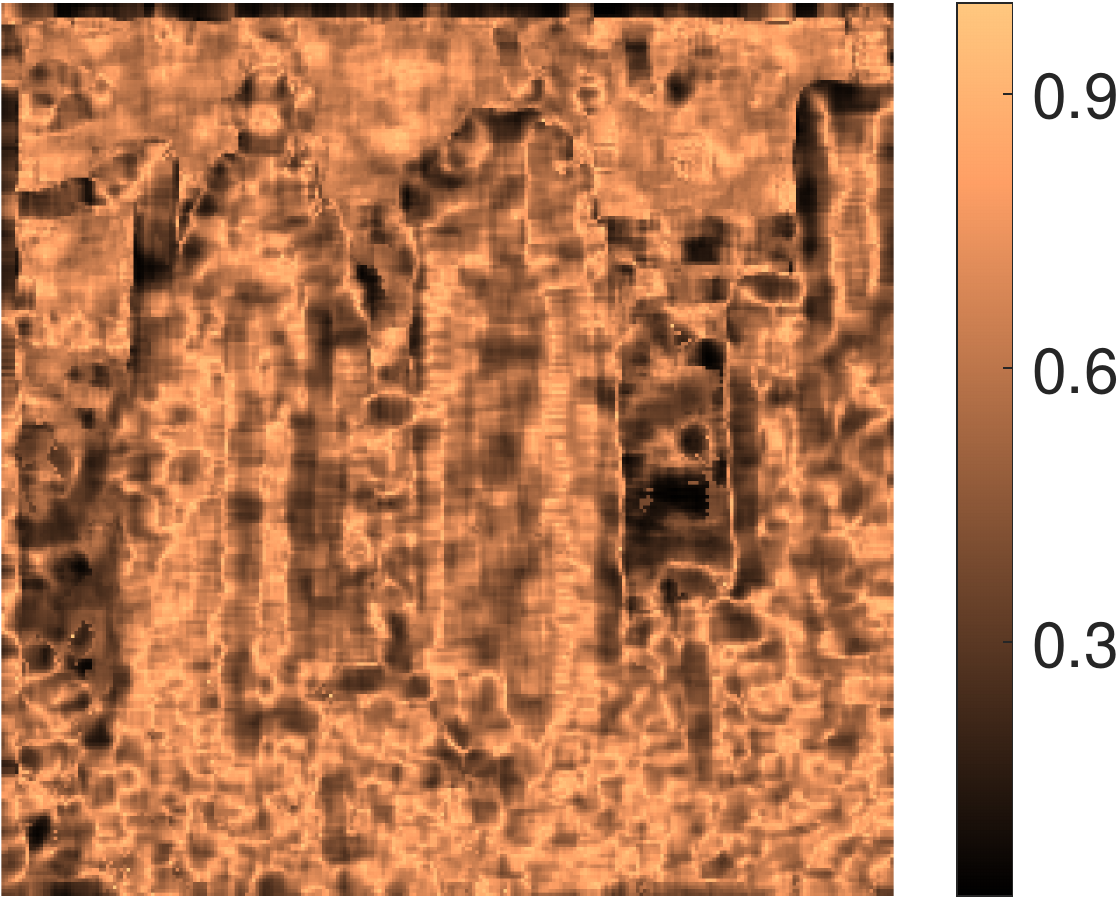}
			\\
			$\bm{\alpha}$&$\bm{p}$&$\bm{\theta}$&$\bm{a}$\\
	\end{tabular}}
	\caption{From top to bottom: for the test image \texttt{skyscraper}, output maps of the parameters for the $\WTV$ ($r=15$), $\WTV^{sv}_{\bm{p}}$ (r=15) and $\WDTV^{sv}_{\bm{p}}$ ($r=3$) regularisers.}
	\label{fig:maps_sky}
\end{figure}

%\begin{table}[!t]
%\centering
%\begin{tabular}{c|c|c|c|c|c||c}\hline\hline
%	&$\TV$-L$_q$&$\TV_p$-L$_q$&$\WTV$-L$_q$&$\WTV_p$-L$_q$&$\DTV_p$-L$_q$&\\
%	\hline\hline
%	\multicolumn{6}{c}{$\qquad$\texttt{geometric}}&\\
%	\hline
%	ISNR&7.9550&8.0045&9.0999&9.3861&9.7568&\multirow{2}{*}{\STAB{\rotatebox[origin=c]{90}{$q=2$}}}\\
%	SSIM&0.9051&0.9184&0.9188&0.9201&0.9270&\\
%	\hline
%	\multicolumn{6}{c}{$\qquad$\texttt{skyscraper}}\\
%	\hline
%		ISNR&2.5056&2.2522&2.9388&&&\multirow{2}{*}{\STAB{\rotatebox[origin=c]{90}{$q=1$}}}\\
%		SSIM&0.6364&0.6172&0.6836&&\\
%\end{tabular}

%%%	    \multicolumn{6}{c}{$\qquad$\texttt{geometric}}\\
%   %  & &&\texttt{geometric}&&\\

%	\hline
%%%	 \multicolumn{6}{c}{}\\
%	\noalign{\smallskip}\hline\noalign{\smallskip}
%\hline
%\hline
%   \multicolumn{6}{c}{$\qquad$\texttt{skyscraper}}\\
%  & &&\texttt{geometric}&&\\
%  \hline\hline\noalign{\smallskip}
%	&$\TV$-L$_1$&$\TV_p$-L$_1$&$\WTV$-L$_1$&$\WTV_p$-L$_1$&$\DTV_p$-L$_1$\\
%	\hline
%	ISNR&2.5056&2.2522&2.9388&&\\
%	SSIM&0.6364&0.6172&0.6836&&\\
%\end{tabular}
%\caption{ISNR and SSIM values achieved by the considered models for the three test images corrupted by blur and different noises.}
%\label{tab:1_}
%\end{table}
%
\paragraph{Restoration of \texttt{stairs}}

%\red{One could claim that, so far, the $\WTV$-L$_q$ models represent the optimal trade-off in terms of image restoration quality and computational efficiency. In fact, the $\WTV$ regularisation term, whose weights $\alpha_i$ can be updated very efficiently, does not suffer from the typical drawbacks of non-convex regularisation; at the same time, the output restorations are quite accurate from a visual perspective on very different images.}

In this final test, we consider the highly textured image \texttt{stairs}. From the ISNR and SSIM values reported in Table~\ref{tab:vals_rest} and from the restored images displayed in Figure \ref{fig:stairs_rest}, we notice that the WTV regulariser is outperformed by the space-invariant TV$_p$ regulariser, with output estimated $p=1.56$, in terms of SSIM. In fact, the TV$_p$ performs a type of regularisation which, although global, appears to be more suitable for describing the image of interest. 
%is easy to notice that the WTV-L$_{\infty}$ model is outperformed by the global TV$_p$-L$_{\infty}$, with $p=1.56$. 
A further improvement is achieved by the $\WTV^{sv}_{\bm{p}}$ regulariser which preserves the textured regions in the image while smoothing out the limited piece-wise constant parts. Finally, the $\WDTV^{sv}_{\bm{p}}$ regularisation term slightly refines the output by driving the regularisation along the local directionalities.

The maps of the parameters for the space-variant WTV with radius $r=2$, and for $\WTV^{sv}_{\bm{p}}$ and $\WDTV^{sv}_{\bm{p}}$ with radius $r=1$ are shown in Figures \ref{fig:maps_stairs}. From the $\bm{p}$-maps for $\WTV^{sv}_{\bm{p}}$ and $\WDTV^{sv}_{\bm{p}}$, one can observe that values of $p_i$ equal or close to $2$ are spread out all over the image, thus indicating that a Tikhonov-type of regularisation, combined with the suitable local weights, is more effective in dealing with this sort of images, due to the presence of large textured regions where the distribution of gradients is thus very spread. The directions in the central part of the image are precisely detected, as shown in the $\bm{\theta}$-map, as well as the confidence in the estimation, represented by $\bm{a}$, which appear to be particularly relevant along the steps.

\begin{figure}
	\centering
	\renewcommand{\tabcolsep}{0.02cm}
	\resizebox{\textwidth}{!}{
		\begin{tabular}{cccc}
			\includegraphics[height = 2.5cm]{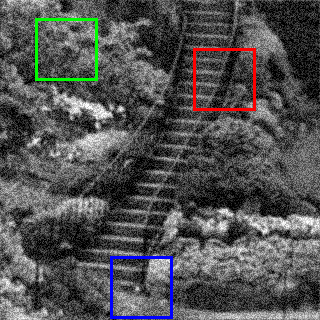}&\includegraphics[height = 2.5cm]{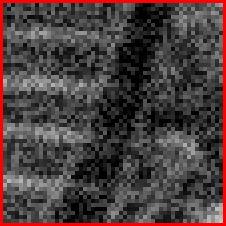}&\includegraphics[height = 2.5cm]{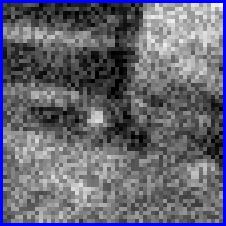}&\includegraphics[height = 2.5cm]{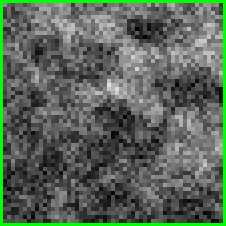}\\ 
			\includegraphics[height = 2.5cm]{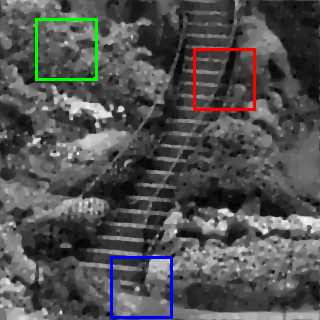}&
			\includegraphics[height = 2.5cm]{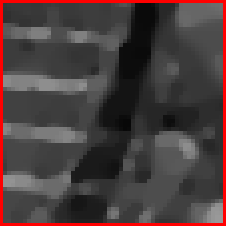}&
			\includegraphics[height = 2.5cm]{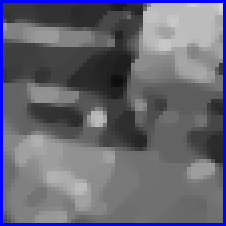}&
			\includegraphics[height = 2.5cm]{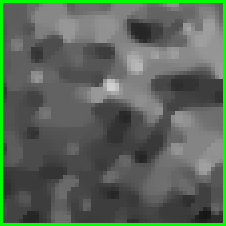}\\
			\includegraphics[height = 2.5cm]{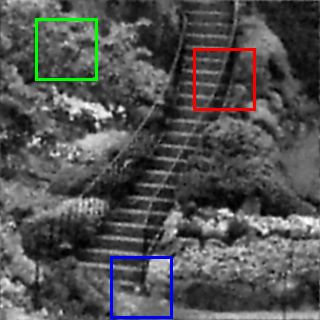}&
			\includegraphics[height = 2.5cm]{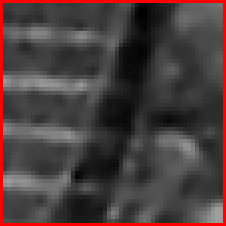}&
			\includegraphics[height = 2.5cm]{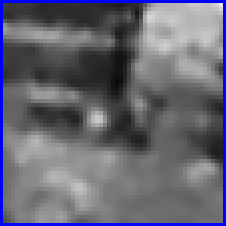}&
			\includegraphics[height = 2.5cm]{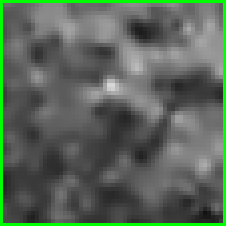}\\
			\includegraphics[height = 2.5cm]{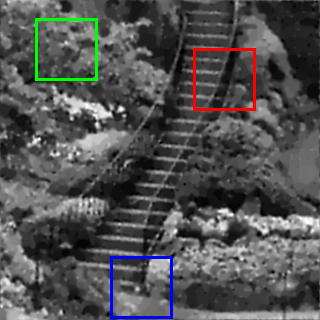}&
			\includegraphics[height = 2.5cm]{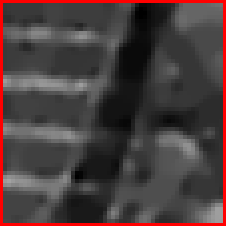}&
			\includegraphics[height = 2.5cm]{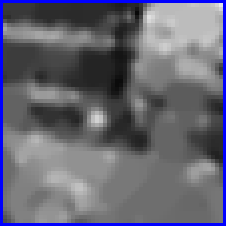}&
			\includegraphics[height = 2.5cm]{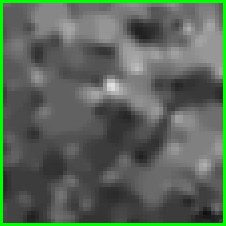}\\
			\includegraphics[height = 2.5cm]{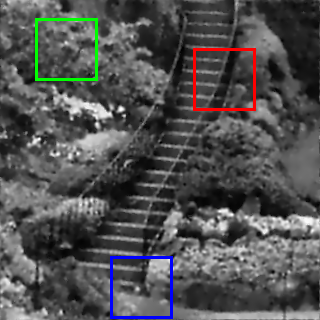}&
			\includegraphics[height = 2.5cm]{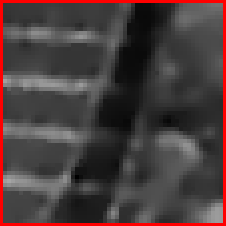}&
			\includegraphics[height = 2.5cm]{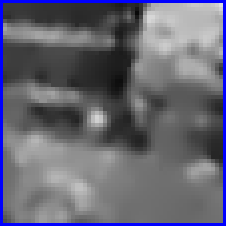}&
			\includegraphics[height = 2.5cm]{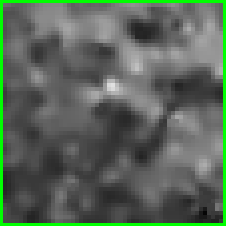}\\
			\includegraphics[height = 2.5cm]{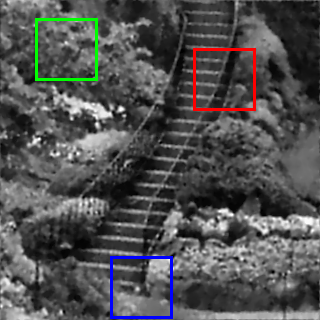}&
			\includegraphics[height = 2.5cm]{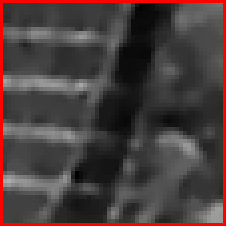}&
			\includegraphics[height = 2.5cm]{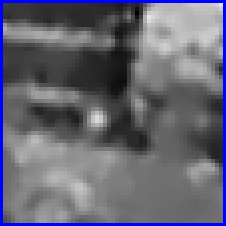}&
			\includegraphics[height = 2.5cm]{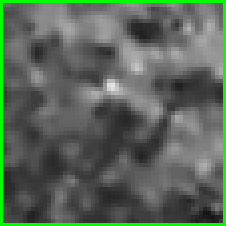}
	\end{tabular}}
	\caption{From top to bottom: for the test image \texttt{stairs}, observed image $b$, performance of the TV, the TV$_p$ (with output $p=1.56$), the WTV, the WTV$_{\bm{p}}^{sv}$ and the WDTV$_{\bm{p}}^{sv}$ regularisers with the respective close-up(s).}
	\label{fig:stairs_rest}
\end{figure}

\begin{figure}
	\centering
	\renewcommand{\tabcolsep}{0.01cm}
	\resizebox{\textwidth}{!}{
		\begin{tabular}{cccc}
			\includegraphics[height = 2.5cm]{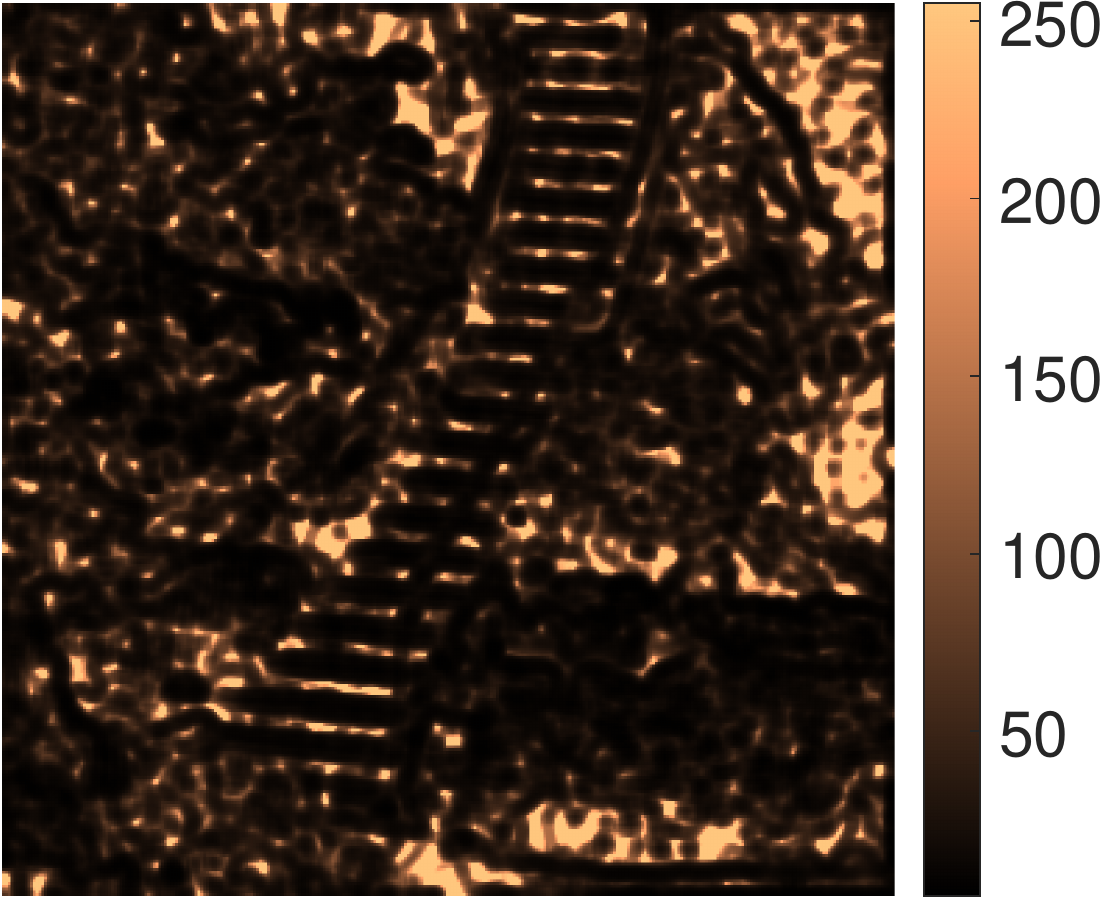}&&&\\ 
			$\bm{\alpha}$&&&\\
			\includegraphics[height = 2.5cm]{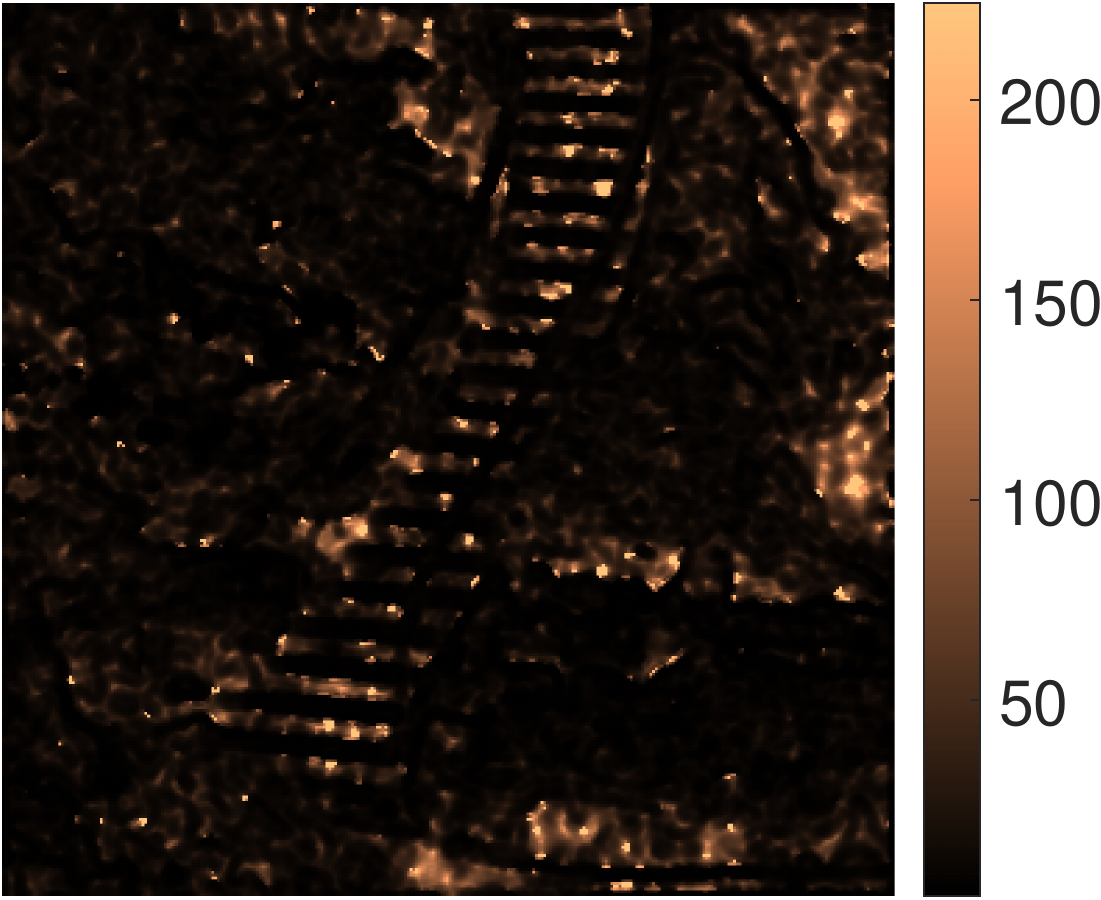}&
			\includegraphics[height = 2.5cm]{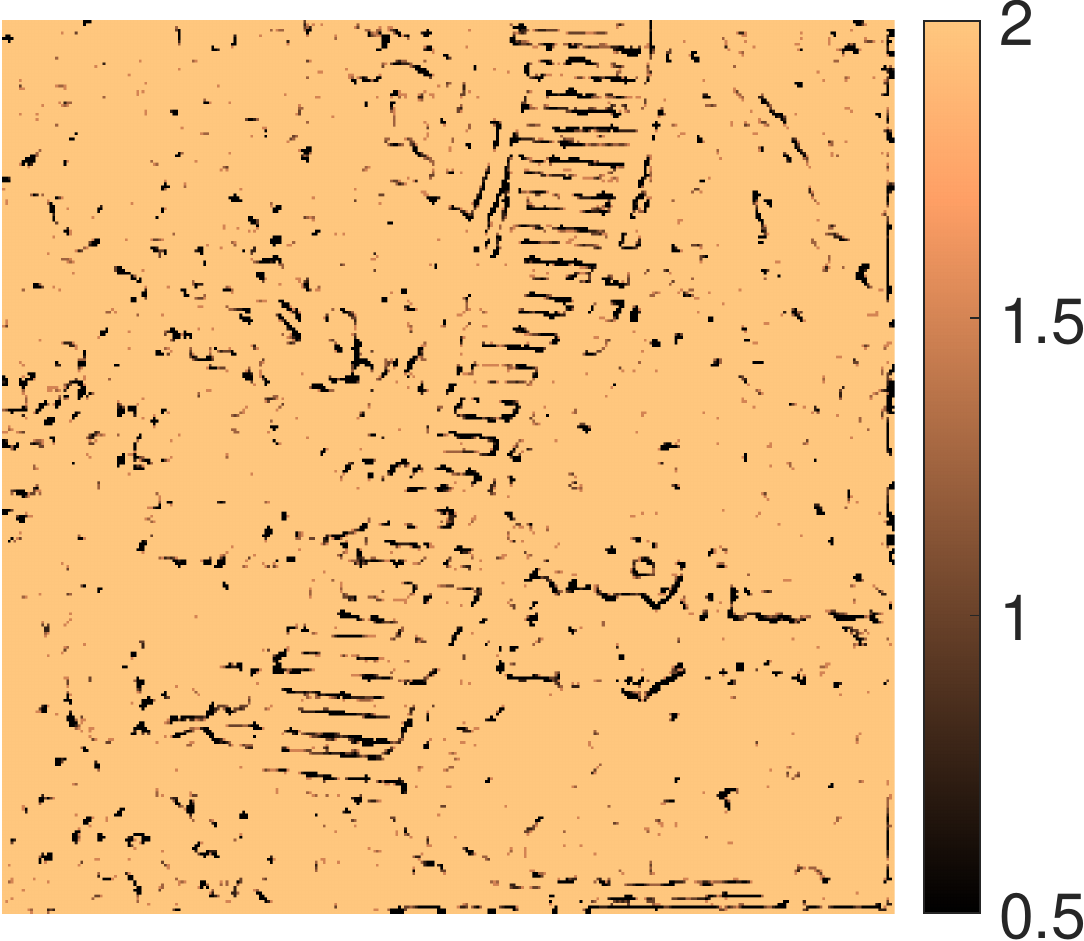}&&\\
			$\bm{\alpha}$&$\bm{p}$&&\\
			\includegraphics[height = 2.5cm]{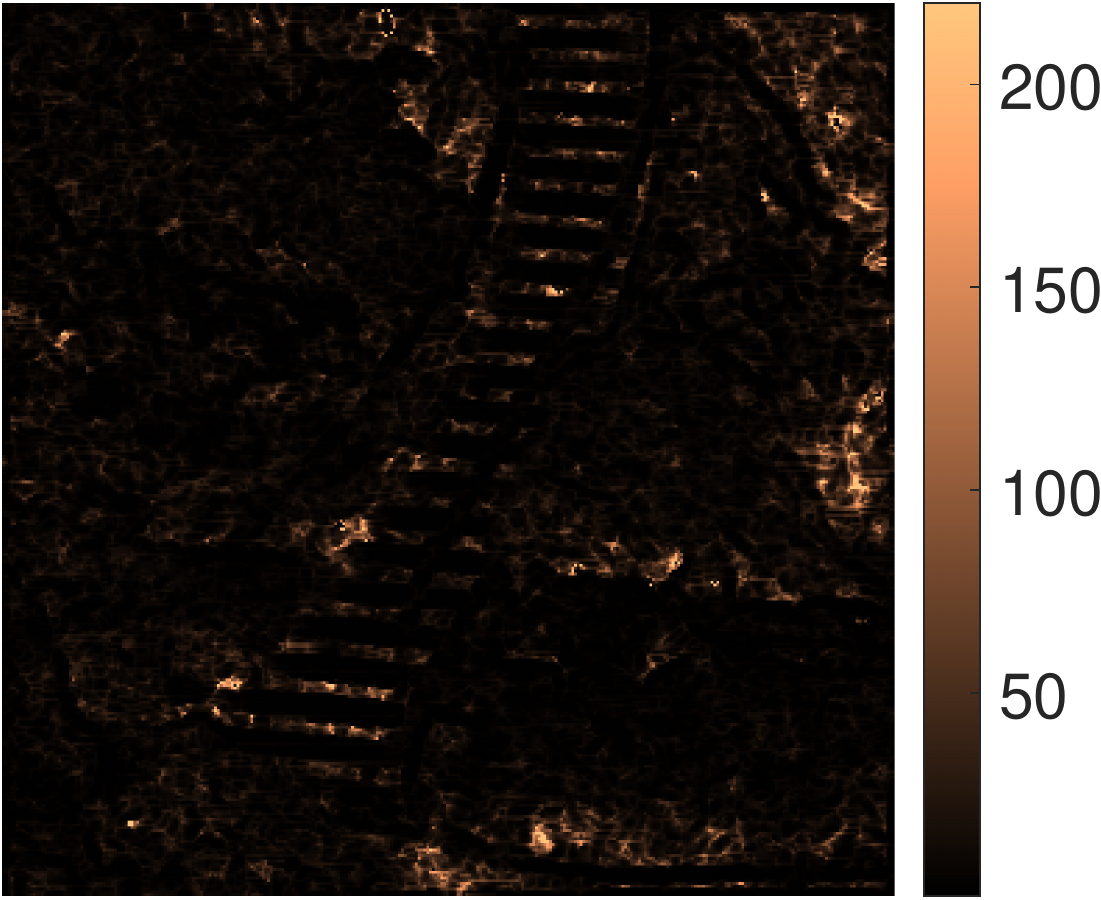}&
			\includegraphics[height = 2.5cm]{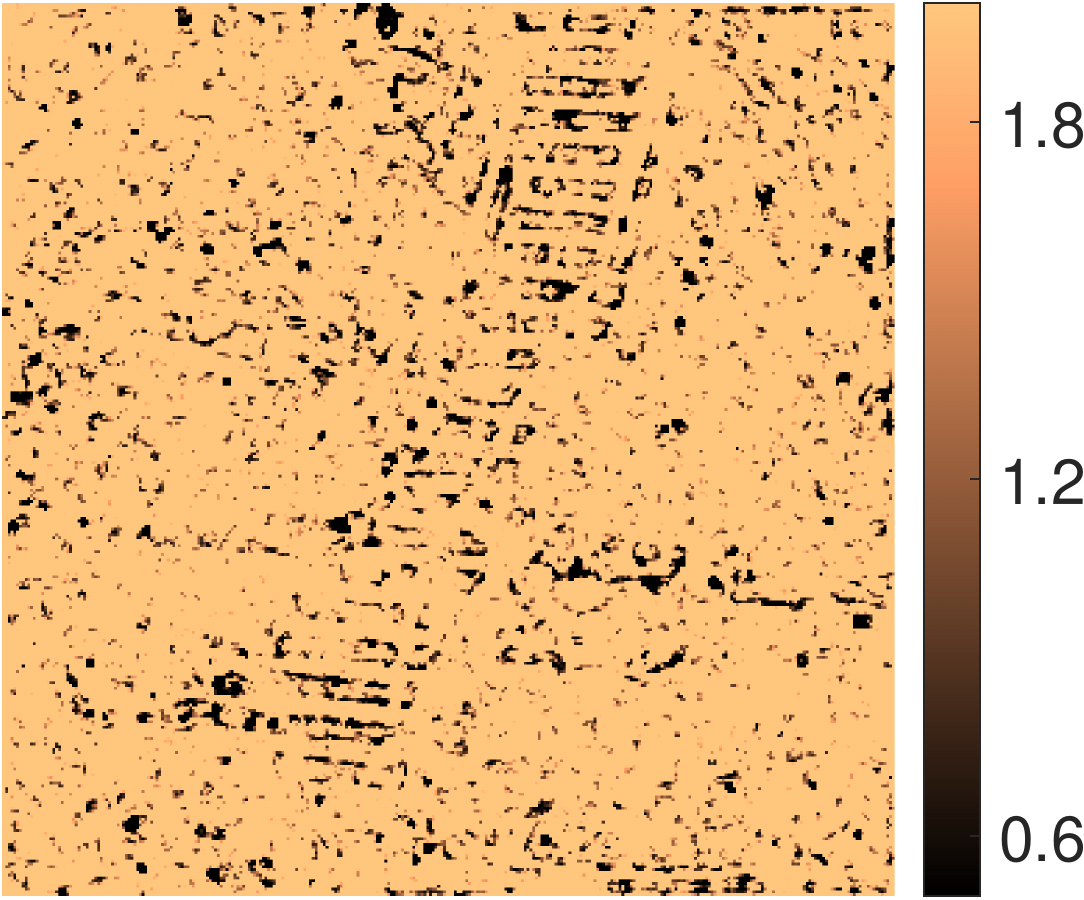}&\includegraphics[height = 2.5cm]{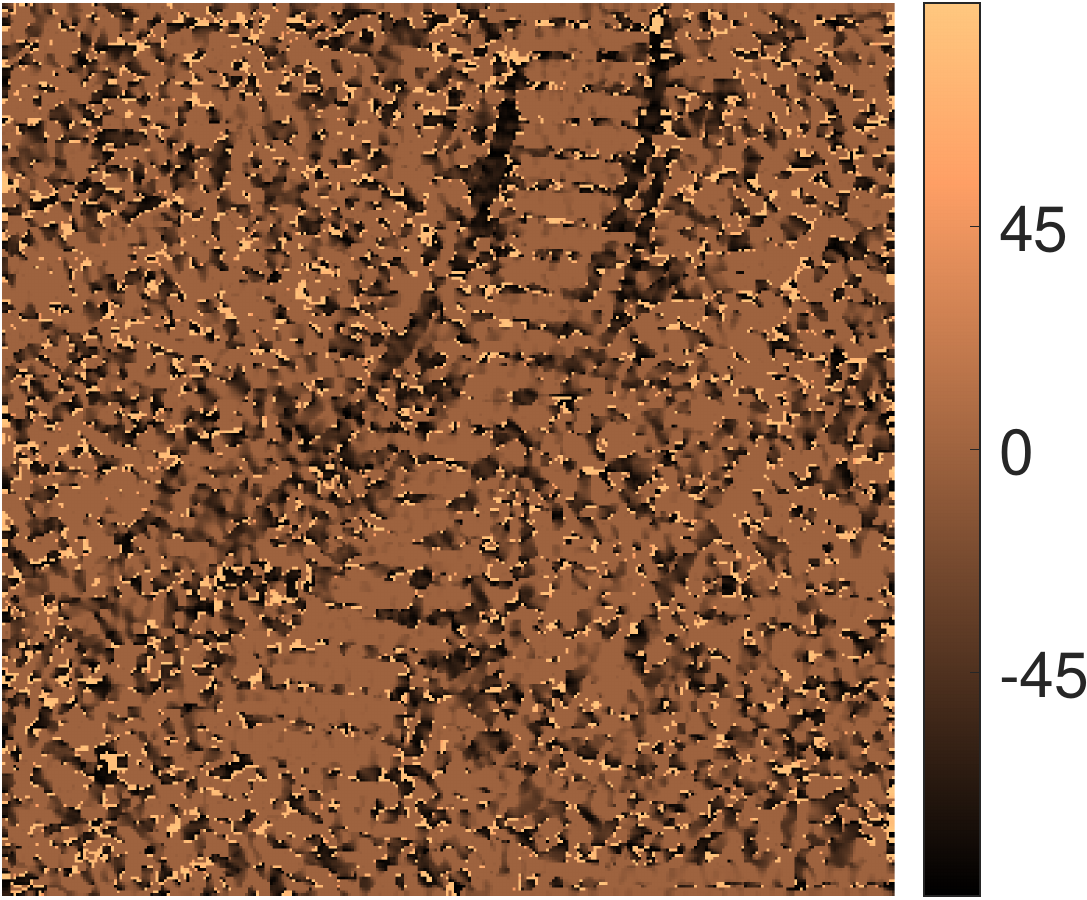}&
			\includegraphics[height = 2.5cm]{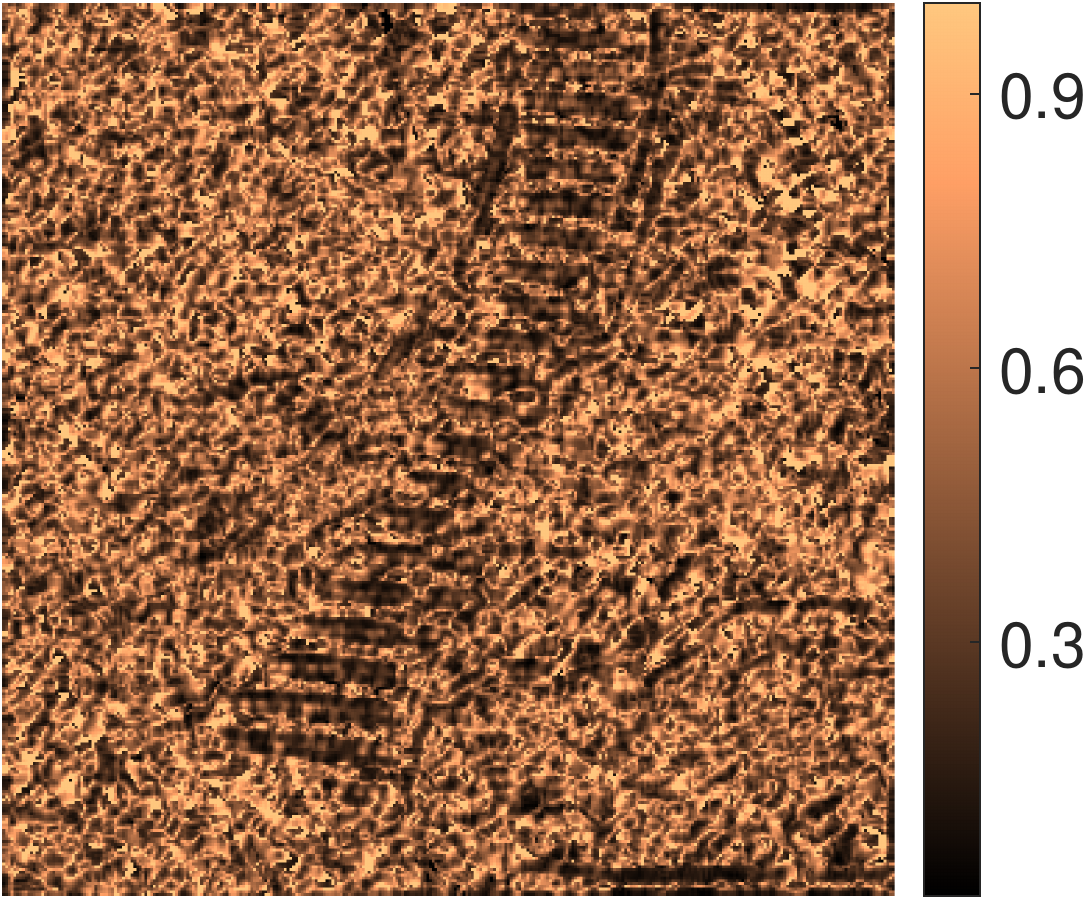}
			\\
			$\bm{\alpha}$&$\bm{p}$&$\bm{\theta}$&$\bm{a}$\\
	\end{tabular}}
	\caption{From top to bottom: for the test image \texttt{stairs}, output maps of the parameters for the $\WTV$ ($r=2$), $\WTV^{sv}_{\bm{p}}$ (r=1) and $\WDTV^{sv}_{\bm{p}}$ ($r=1$) regularisers.}
	\label{fig:maps_stairs}
\end{figure}

\section{User's guide to space-variance}

We conclude this work by addressing several issues pointed out so far, with the purpose of guaranteeing a more conscious use of the derived hypermodels.
\begin{enumerate}
	\item \emph{The best space-adaptive regulariser}. As one could expect, the question about which regulariser performs \emph{best} in absolute terms does not have a unique answer. The adoption of more and more general regularisation terms does not always pay back as the overall performance has to be evaluated with respect to the trade-off between quality of the restorations and computational effort. In this perspective, on the restoration of the \texttt{geometric} and \texttt{skyscraper} images, the WTV has returned remarkably good results while keeping the computational times low, in light of the closed form-expressions existing for the $\bm{\alpha}$-update and for the $\bm{g}$-subproblem in the ADMM-based scheme. However, for the \texttt{stairs} image, the typical shortage of a TV-type regularisation, even if weighted, has emerged.
	
	The selection of the regulariser to employ should thus be motivated by the application of interest and, ultimately, by the processed data.

	\item \emph{The optimal radius $r$}. In the previous tests, the radius $r$ involved in the estimate of the parameters has been selected so as to maximise the ISNR and the SSIM of the final restorations. One can notice that its choice somehow reflects the scale of the structures to preserve, and that it can be set differently for different space-variant regularisers. In this sense, a useful example is given by the test image \texttt{skyscraper} which presents textured objects in the foreground with a smooth and constant background. For the WTV and the WTV$_{\bm{p}}^{sv}$ regularisers, a large value of $r$ can easily catch the `dual' nature of the image. However, the texture on the foreground buildings is fine-scale so that to detect the local directionalities with the WDTV$_{\bm{p}}^{sv}$ regulariser, a smaller radius has to be selected. 
	
	\item \emph{The curse of non-convexity}. The convergence of the outlined numerical scheme aimed at solving a possibly non-convex problem interlaced with a parameter estimation step is a very delicate issue that has not found a theoretical response yet. However, when the ADMM penalty parameter $\beta_g$, $\beta_r$ are set in a suitable manner - typically, $\beta_g,\beta_r\approx 10^4,10^5$ - empirical convergence is observed. 
	%We remark here that, although under the adoption of a WTV regularisation term the $k$-th $\bm{u}$-subproblem is convex, the iterative update of the parameters makes the overall numerical scheme still sensitive to the choice of the penalty parameters, that again can be set as $\beta_g,\beta_r\approx 10^4$.
\end{enumerate}

%For the last case, the selection of a significant smaller radius $r$ for the parameter estimation is motivated by the presence in the image of fine oriented textures that are less likely to be detected when employing larger radii.

%\red{One could claim that, so far, the $\WTV$-L$_q$ models represent the optimal trade-off in terms of image restoration quality and computational efficiency. In fact, the $\WTV$ regularisation term, whose weights $\alpha_i$ can be updated very efficiently, does not suffer from the typical drawbacks of non-convex regularisation; at the same time, the output restorations are quite accurate from a visual perspective on very different images.}
%{\color{red}   sottosezione di istruzioni per l'uso. titolo?}

%{\color{red} Non volevamo mettere anche qualche possibile estensione ad altri operatori lineari nel regolarizzatore? (Hessian-Schatten norm, Wavelets...)?}

\section{Challenges}
\label{sec:challenges}

We conclude this review by listing in the following some challenging future research directions which could enrich this work from both the theoretical and the applied point of view. Each of the following items has to be intended not as a straightforward extension of the framework presented here, but rather as an intersection with some related mathematical fields (analysis, optimisation, numerical and linear algebra, medical imaging...) favouring the development of new and unexplored research.

\begin{enumerate}
	
	\item The detailed analytical study  in an infinite-dimensional framework of the non-smooth, non-convex and space-variant regularisation models discussed in this work is expected to provide more insights on the structure of the expected solutions by means of duality tools, functional calculus and non-standard Lebesgue/Sobolev calculus in spaces with variable exponents.
	
	\item The development of a rigorous theoretical framework guaranteeing convergence to (at least) stationary points for the non-convex ADMM Algorithm \ref{alg:1} is highly non-trivial and practically made challenging due to the parameter estimation performed jointly along the iterations. Note that even for the IAS algorithm \cite{Calvetti_2015} where parameters and iterates are updated sequentially, a convergence proof in general non-convex scenarios is still missing, as only partial results in convex (quadratic) cases are available.
	%\item If $s$ is the number of parameters to estimate for every pixel $i,~i=1,\ldots, N$ in the image, the total number of quantities to be determined is $sN$. In principle, these parameters should be estimated in each iteration of the algorithm which is quite challenging in the absence of a closed form expression. 

	\item The use of a similar space-variant modelling for more general regularisers defined, for instance, in terms of wavelet expansions \cite{Candes2011} and higher-order differential operators (see, e.g., \cite{Schatten2013}).
	
	\item Similarly as for the case of anisotropic diffusion, we expect that the use of suitable adaptive discretisation stencils \cite{ChambolleUpwind2011,Mirebeau,Condat2017,chambolle2021} built on the estimated local directional information and/or relying on the suitable definition of appropriate transfer operators and staggered grids \cite{ParMasSch18applied}  could improve upon the quality of the numerical reconstructions by describing anisotropy on the image grid in a more precise way.
	
	\item Following \cite{Calvetti_2020}, we believe that the use of a hierarchical modelling with \emph{informative} hyperpriors encoding, for instance, local smoothness/sparsity beliefs on the solution could refine the hyperparameter selection strategy and thus, overall, the quality of the reconstruction. 
	%As preliminarily shown in Section \ref{sec:perona_malik}, it is expected that upon an appropriate choice of hyperprior distributions (for instance, a generalised $\Gamma$ distribution), well-known imaging models can be derived and interpreted in a Bayesian setting, making it amenable for combined analytical-statistical optimisation solvers. 
	
	\item The comparison of the proposed ML-type parameter estimation procedure described in Section \ref{sec:parameter_estimation} with the recent approach based on empirical Bayes estimation proposed in  \cite{PereyraPartI,PereyraPartII} could lead to new hybrid hyperparameter selection strategies relying on the sole observation of the given corrupted image $\bm{b}$. These ideas are expected indeed to speed up the performance of Algorithm \ref{alg:1}.

	\item As an obvious field of applications due to the recent use of analogous models
	in medical imaging problems such as MRI, PET and CT applications  
	\cite{Ehrhardt_2014,Ehrhardt_2016,Ehrhardt2016,Tovey2019}, we expect that the use of a structural and adaptive modelling could significantly improve the quality of the reconstructions and favour, at the same time, the exploitation of structural information in multimodal image analysis.
	
	\item We wonder whether the flexible underlying statistical modelling proposed in this work could be used in a GAN framework \cite{Goodfellow2014} from two different perspectives: firstly, to compare qualitatively (and/or quantitatively, provided that a good quality measure is used \cite{Borji2019}) the highly parametric model-driven BGGD-type distribution of the target image with the target distribution estimated in the generative step; secondly, to provide an alternative way of estimating the hyperparameters, thus replacing the parameter estimation step \eqref{eq:sub_th} by a fully data-driven approach.
	
\end{enumerate}

\section{Conclusions} \label{sec:conc}

In this work, we described a journey across time and various fields of applied mathematics with the intent of reviewing the many features of the exemplar and  probably \emph{the} most popular image regularisation model over the last thirty years, the TV functional. After recalling its genesis, its main features and shortcomings in Section \ref{sec:intro} and having fixed some notations in Section \ref{sec:prel}, we described in Section \ref{sec:Bayes} how the rigidity of existing TV-type image regularisation models can be overcome within the setting of non-stationary Markov Random Fields whose capability of describing local image features (i.e. scale, shape and directionality)  endows the corresponding prior distributions with more flexibility and degrees of freedom. Analytically, we show in Section \ref{sec:map} that these non-stationary priors can be put in close correspondence to a large class of  space-variant image regularisation models which have been thoroughly studied over the last decades with the intent of improving upon well-known TV drawbacks. In their analytical form, the dependence of these models on local information (i.e. amount of regularisation, sharpness and anisotropy) is then showed in Section \ref{sec:geom} to correspond geometrically to change at each pixel the definition of the constrain set of dual functions and changing their alignment accordingly. Having provided an expression of the corresponding data models in Section \ref{sec:joint}, we then describe in Section \ref{sec:parameter_estimation} a Maximum-Likelihood type automatic parameter estimation strategy, motivated by the underlying Bayesian formulation resulting in the definition of appropriate variational Bayesian hypermodels. Finally, the joint statistical-analytical procedure is embedded into a general alternating minimisation scheme in Section \ref{sec:admm} and validated in Section \ref{sec:restoration} on some exemplar image restoration models. The flexibility of the proposed approach and the accuracy and robustness of the estimator considered for the automatic selection of hyperparameters show good adaptation to both geometrical and texture image information and pave the way to new challenging research directions as finally described in Section \ref{sec:challenges}.

\medskip

The incredible potential offered nowadays by the possibility of combining together different fields of applied mathematics with the intent of improving and making more data-adaptive the performance of TV, makes the use of such, often taught out-of-date, image regularisation model still interesting for the whole applied mathematics community. The swan song of TV is still far away in time. Due to both the profound understanding of this powerful, yet simple, image regularisation model carried out over years and the recent advances in large-scale data exploitation and numerical optimisation favouring the development of its many extensions, we expect that the descendent models stemmed from TV have still much to say and could at the same time enrich and be enriched by the increasingly popular interest towards data-driven approaches showed by the analytical, signal-processing, statistical and optimisation communities.

\section*{Acknowledgements}
The authors are warmly grateful to Laure Blanc-F\'eraud and Xavier Descombes for their precious suggestions and comments on the statistical derivation of our models and to Simone Parisotto for his advice on their geometrical interpretation.
Research of AL, MP, FS was supported by the “National Group for Scientific
Computation (GNCS-INDAM)” and by ex60 project by the University of Bologna
“Funds for selected research topics”. LC acknowledges the support received by the EU H2020 RISE NoMADS, GA 777826 and  the UCA JEDI IDEX grant DEP ``Attractivité du territoire".

\appendix

\section{Generalised Gaussian distributions}

\begin{definition}[GG cumulative distribution function]
	The cumulative distribution function (cdf) of a scalar random variable $X\sim \mathrm{GG}(\eta,\gamma,s)$ reads
	\begin{equation}
	F_X(x) \,\;{=}\;\, \frac{1}{2} \;{+}\; 
	\frac{\mathrm{sign}(x-\eta)}
	{2 \,\Gamma\left(1/s\right)} \;\,
	\underline{\Gamma}\left( 
	\frac{1}{s} \, , \, \gamma^s\,|x-\eta|^s
	\right) \, ,
	\label{eq:GN_cdf}
	\end{equation}
	with $\underline{\Gamma}$ the lower incomplete Gamma function defined in \eqref{eq:Ginc} and $\Gamma$ the Gamma function defined in \eqref{eq:Gamma}.
\end{definition}

\begin{lemma}
	\label{prop:Gamma_sum}
	If $\,X_i \sim \mathrm{Gamma}(\nu,z_i)$, $i = 1,\ldots,M$, are independent random variables, then it holds true that
	\begin{equation}
	Y \,\;{=}\;\, \sum_{i=1}^M X_i \,\;{\sim}\;\, 
	\mathrm{Gamma}\left(\nu,\sum_{i=1}^n z_i\right).
	\label{eq:Gamma_sum}
	\end{equation}
\end{lemma}

\begin{lemma}
	\label{lem:12}
	If $\,X \sim \mathrm{GG}(0,\gamma,s)$, then it holds true that
	\begin{equation}
	Y \,\;{=}\;\, g(X) \,\;{=}\;\, | X |^s \,\;{\sim}\;\, 
	\mathrm{Gamma}(\nu,z), \quad \nu = \frac{1}{\gamma^s}, \, z = \frac{1}{s} \, .
	\label{eq:Yq_pdf}
	\end{equation}
\end{lemma}

\begin{proof}
	We have:
	\begin{eqnarray}
	F_Y(y) 
	&\;{=}\;& 
	P\left( Y \;{\in}\; ]-\infty,y] \right)
	\;{=}\; 
	P\left( Y \in [0,y] \right)
	\label{eq:FY1}\\
	&\;{=}\;& 
	P\left( X \in g^{-1}\left([0,y]\right) \right)
	\;{=}\;
	P\left( X \in \left[-y^{1/s},+y^{1/s}\right] \right) 
	\label{eq:FY2}\\
	&\;{=}\;&
	1 - 2 \, F_X\left( -y^{1/s} \right)
	\;{=}\;\,
	\frac{1}{2 \, \Gamma\left(1/s\right)} \,\, 
	\underline{\Gamma}\left( \frac{1}{s} \, , \, \gamma^s\,y\right)
	\label{eq:FY3}
	\end{eqnarray}
	where the first and second equality in (\ref{eq:FY1}) come from the definition of cdf and from noticing that $Y$ can not assume negative values, respectively, $g^{-1}\left([0,y]\right)$ in (\ref{eq:FY2}) denotes the preimage of interval $[0,y]$ under the function $g$ defined in \eqref{eq:Yq_pdf}, 
	(\ref{eq:FY3}) follows from the pdf of $X$ being an even function and, then, from replacing the expression of the GG cdf given in (\ref{eq:GN_cdf}) for $F_X$.
	
	The pdf $\,\mathbb{P}_Y$ can be obtained by differentiating the cdf $F_Y$ in (\ref{eq:FY3}). To this aim, first we rewrite $F_Y$ in the following equivalent composite form:
	\begin{equation}
	F_Y(y) \,\;{=}\;\, F_2\left(F_1(y)\right), \quad F_1(y) \;{=}\; \gamma^s\,y \, , \quad F_2(w) \;{=}\; 
	\frac{1}{2 \, \Gamma\left(1/s\right)} \, 
	\int_0^w t^{\frac{1}{s}-1} e^{-t} d t \, ,
	\end{equation}
	where we also replaced the explicit expression of the lower incomplete Gamma function $\underline{\Gamma}$ given in (\ref{eq:Ginc}). By applying the chain rule of differentiation, we have
	\begin{eqnarray}
	\mathbb{P}_Y(y) 
	&\,\;{=}\;\,& 
	\frac{d}{dy} F_Y
	\,\;{=}\;\,
	\frac{d}{dy} F_1(y) \;{\times}\; 
	\frac{d}{dw} F_2\left(\gamma^s\,y\right)
	\\
	&\,\;{=}\;\,& 
	\gamma^s \;{\times}\; 
	\frac{1}{2 \, \Gamma\left(1/s\right)} \, 
	\left(\gamma^s\,y\right)^{\frac{1}{s}-1} 
	\exp\left( - \gamma^s\,y\right) 
	\\
	&\,\;{=}\;\,&
	\frac{\gamma}{2  \, \Gamma\left(1/s\right)} \,\,\,
	y^{\frac{1}{s}-1} \,
	\exp\left( - \gamma^s\, y\right) \, .
	\label{eq:pY_fin}
	\end{eqnarray}
	A simple one-to-one reparameterisation of (\ref{eq:pY_fin}), 
	namely $s =  1 / z$, $\gamma = (1/\nu)^s$, together with the recall of definition (\ref{eq:Gamma_pdf}), leads to (\ref{eq:Yq_pdf}) and thus completes the proof. 
\end{proof}

%{\color{red} vedere se togliere enunciato visto che uguale a prima nella sezione di notazioni.}

\begin{comment}
\begin{proposition}
\label{prop:Lq_pdf}
If $\,X_i \sim GG(0,\gamma,s)$, $i = 1,\ldots,n$, are independent random variables, then we have
%
\begin{equation}
Y 
\,\;{=}\;\, 
\left\| \left(X_1,\ldots,X_n\right) \right\|_s^s
\,\;{=}\;\, 
\sum_{i=1}^n \left|X_i\right|^s
\,\;{\sim}\;\, 
\mathrm{Gamma}(\nu,z), \quad \nu = 1/\gamma^s, \, z = \frac{n}{s}.
\label{eq:Lq_pdf}
\end{equation}
%
In particular, the random variable $Y$ is characterised by mean and variance
%
\begin{equation}
\mu_Y \,\;{=}\;\, \frac{n}s \, \frac{1}{\gamma^s} \, , \quad\quad
\sigma_Y^2 \,\;{=}\;\, \frac{n}{s} \, \frac{1}{\gamma^{2s}} \, .
\label{eq:Lq_moms}
\end{equation}
%
\end{proposition}
\end{comment}

%\begin{proof}
%Statement (\ref{eq:Lq_pdf}) follows straightforwardly from Lemma \ref{lem:12} and Lemma \ref{prop:Gamma_sum}, whereas (\ref{eq:Lq_moms}) comes from (\ref{eq:Lq_pdf}) and the well-known expressions for the mean and variance of a Gamma-distributed random variable.
%\end{proof}

\vspace{0.2cm}
\textbf{Proof of Proposition \ref{prop:Lq_pdf_pap}}. Statement (\ref{eq:Lq_pdf}) follows straightforwardly from Lemma \ref{prop:Gamma_sum} and Lemma \ref{lem:12}, whereas (\ref{eq:Lq_moms}) comes from (\ref{eq:Lq_pdf}) and the well-known expressions for the mean and variance of a Gamma-distributed random variable. \hfill $\square$

\bibliographystyle{acm}
\bibliography{references}

\end{document}